\documentclass{article}
\usepackage{amsmath}
\usepackage{amssymb}
\usepackage{amsthm}
\usepackage{graphicx}
\usepackage{amsfonts}
\usepackage{hyperref}
\usepackage{float}

\usepackage{color}
\def\M{{\cal  M}}

\def\f12{\frac 1 2}
\def\hh{\mathcal{H}^{+}}

\def\z{\zeta}
\def\a{\alpha}
\def\b{\beta}

\def\f12{\frac 1 2}

\newcommand{\nabb}{\mbox{$\nabla \mkern-13mu /$\,}}
\newcommand{\lapp}{\mbox{$\triangle \mkern-13mu /$\,}}

\newtheorem{definition}{Definition}[section]

\newtheorem{remark}{Remark}[section]
\newtheorem{lemma}{Lemma}[subsection]
\newtheorem{theorem}{Theorem}[section]
\newtheorem{proposition}{Proposition}[subsection]

\newtheorem{corollary}{Corollary}[section]

\newtheorem{mytheo}{Theorem}
\begin{document}

\title{The Wave Equation on  Extreme Reissner-Nordstr\"{o}m Black Hole Spacetimes: Stability and Instability Results}
\author{Stefanos Aretakis\thanks{University of Cambridge,
Department of Pure Mathematics and Mathematical Statistics,
Wilberforce Road, Cambridge, CB3 0WB, United Kingdom}}
\date{October 18, 2010}
\maketitle

\begin{abstract}

We consider solutions to the linear wave equation $\Box_{g}\psi=0$ on a suitable globally hyperbolic subset of an  extreme Reissner-Nordstr\"{o}m spacetime, arising from regular initial data prescribed on a Cauchy hypersurface $\Sigma_{0}$ crossing the future event horizon $\mathcal{H}^{+}$.  We obtain boundedness, decay, non-decay and blow-up results. Our estimates hold up to and including $\mathcal{H}^{+}$. The fundamental new aspect of this problem is the degeneracy of the redshift on the event horizon $\mathcal{H}^{+}$. Several new analytical features of degenerate horizons are also presented.
\end{abstract}

\tableofcontents

\section{Introduction}
\label{sec:Introduction}

\textit{Black holes} are one of the most celebrated predictions of General Relativity, and perhaps of Mathematical Physics in general. A particularly interesting but  peculiar example of a  black hole spacetime is given by the so-called  \textit{extreme Reissner-Nordstr\"{o}m} metric, which in local coordinates $(t,r,\theta,\phi)$ takes the form 
\begin{equation}
g=-Ddt^{2}+\frac{1}{D}dr^{2}+r^{2}g_{\scriptstyle\mathbb{S}^{2}},
\label{1g1}
\end{equation}
where 
\begin{equation*}
D=D\left(r\right)=\left(1-\frac{M}{r}\right)^{2},
\end{equation*}
$g_{\scriptstyle\mathbb{S}^{2}}$ is the standard metric on $\mathbb{S}^{2}$ and $M>0$ (see also Appendix \ref{sec:OnTheGeometryOfReissnerNordstrOM}).   This spacetime has been the object of considerable study in the physics literature (see for instance the very recent \cite{marolf}), but the analysis of its waves has not been adequately studied from a mathematical point of view.   In this paper, we shall attempt a more or less complete treatment of  the wave equation
\begin{equation}
\Box_{g}\psi=0
\label{1eq}
\end{equation}
on  extreme Reissner-Nordstr\"{o}m. Our main results (see Section \ref{sec:TheMainTheorems}) include:
\begin{enumerate}
 \item Local integrated decay of energy, up to and including the event horizon $\mathcal{H}^{+}$ (Theorem \ref{th1}).
	\item Energy and pointwise uniform boundedness of solutions, up to and including  $\mathcal{H}^{+}$ (Theorems \ref{t2}, \ref{t5}).
	\item Energy and pointwise decay of solutions,  up to and including  $\mathcal{H}^{+}$ (Theorems \ref{t4}, \ref{t6}, \ref{theo8}).
    \item Non-decay and blow-up estimates for higher order derivatives of solutions along  $\mathcal{H}^{+}$ (Theorem \ref{theo9}).
\end{enumerate}
The latter blow-up estimates are in sharp contrast with the non-extreme case, for which decay holds for all higher order derivatives of $\psi$ along $\mathcal{H}^{+}$.
Several new analytical features of degenerate event horizons are also presented.  The above theorems, in particular, resolve Open Problem 4 (for  extreme Reissner-Nordstr\"{o}m) from Section 8 of \cite{md}. The fundamentally new aspect of this problem is the degeneracy of the redshift on the event horizon $\mathcal{H}^{+}$.

\subsection{Preliminaries}
\label{sec:Preliminaries}

Before we discuss in detail our results, let us present the distinguishing properties  of extreme Reissner-Nordstr\"{o}m and put this spacetime in the context of previous results.

\subsubsection{Extreme Black Holes}
\label{sec:ExtremeBlackHoles}

We  briefly describe here  the geometry of the horizon of extreme Reissner-Nordstr\"{o}m. (For a nice introduction to the relevant notions, we refer the reader to \cite{haw}). The event horizon $\mathcal{H}^{+}$ corresponds to $r=M$, where the $(t,r)$ coordinate system \eqref{1g1} breaks down. The coordinate vector field $\partial_{t}$, however, extends to a regular null Killing vector field $T$ on $\mathcal{H}^{+}$. The integral curves of $T$ on $\mathcal{H}^{+}$ are in fact affinely parametrized:
\begin{equation}
\nabla_{T}T=0.
\label{1t1}
\end{equation}
More generally, if an event horizon admits a Killing  tangent vector field $T$ for which \eqref{1t1} holds, then the horizon is called degenerate and the black hole \textit{extreme}. In other words, a black hole is called extreme if the surface gravity vanishes on $\mathcal{H}^{+}$(see Section \ref{sec:RedshiftEffectAndSurfaceGravityOfH}). Under suitable circumstances, the notion of extreme black holes can in fact be defined even  in case the spacetime does not admit a Killing field (see \cite{price}). 

The extreme  Reissner-Nordstr\"{o}m corresponds to the $M=e$ subfamily of the two parameter  Reissner-Nordstr\"{o}m family with parameters  mass $M>0$ and  charge $e>0$. It  sits between the non-extreme black hole case $e<M$ and the so-called naked singularity case $M<e$. Note that the physical relevance of the black hole notion rests in the expectation  that black holes are ``stable'' objects in the context of the dynamics of the Cauchy problem for the Einstein equations. On the other hand, the so-called weak cosmic censorship conjecture suggests that naked singularities are dynamically unstable (see the discussion in \cite{neo}). That is to say, extreme black holes are expected to have both stable and unstable properties; this makes their analysis very interesting and challenging.

\subsubsection{Linear Waves}
\label{sec:LinearWaves}

 The first step in understanding the dynamic  stability (or instability) of a spacetime  is by considering the wave equation \eqref{1eq}. This is precisely the motivation of the present paper. Indeed, to show stability one would have  to prove that solutions of the wave equation decay sufficiently fast. For potential future applications all methods should be robust and the resulting estimates   quantitative. Robust means that the methods still apply when the background metric is replaced by a nearby metric, and quantitative means that any estimate for $\psi$ must be in terms of uniform constants and (weighted Sobolev) norms of the initial data.  Note also that it is essential to obtain non-degenerate estimates for $\psi$ on $\mathcal{H}^{+}$ and to consider initial data that  do not vanish on the horizon. As we shall see,  the issues at the horizon  turn out to be the most challenging part in understanding the evolution of waves on  extreme Reissner-Nordstr\"{o}m.

\subsubsection{Previous Results for Waves on Non-Extreme Black Holes}
\label{sec:PreviousResultsForNonExtremeBlackHoles}

The wave equation \eqref{1eq} on black hole spacetimes  has  long been studied beginning with the pioneering work of Regge and Wheeler \cite{RW} for Schwarzschild. Subsequently, a series of heuristic and numerical arguments were put forth for obtaining decay results for $\psi$ (see \cite{other2, price72}). However, the first complete quantitative result (uniform boundedness) was obtained only in 1989 by Kay and Wald \cite{wa1}, extending the restricted result of  \cite{drimos}. Note that the proof of \cite{wa1} heavily depends on the exact symmetries of the Schwarzschild spacetime.

During the last decade, the wave equation on black hole spacetimes  has become a very active area in mathematical physics. As regards the Schwarzschild spacetime,   ``$X$ estimates" providing local integrated energy decay (see Section \ref{sec:MorawetzAndXEstimates} below) were derived   in \cite{blu0,blu3,dr3}. Note that \cite{dr3} introduced a vector field estimate which captures in a stable manner the so-called  \textit{redshift effect}, which allowed the authors to obtain quantitative pointwise estimates on the horizon $\mathcal{H}^{+}$.  Refinements for Schwarzschild were achieved in \cite{dr5} and \cite{tataru1}. Similar estimates to \cite{blu0} were derived in \cite{blu1} for the whole parameter range of Reissner-Nordstr\"{o}m including the extreme case. However, these estimates  degenerate on $\mathcal{H}^{+}$ and require  the initial data to be supported away from $\mathcal{H}^{+}$.  
 
The first boundedness result for solutions of the wave equation on slowly rotating Kerr ($\left|a\right|\ll M$) spacetimes was proved in \cite{dr7} and decay results  were derived in \cite{blukerr,md,tataru2}. Decay results for general subextreme Kerr  spacetimes ($\left|a\right|<M$) are proven in \cite{megalaa}. Two new methods were presented recently for obtaining sharp decay of energy flux and pointwise decay on black hole spacetimes; see \cite{new,tataru3}. For results on the coupled wave equation see \cite{price}. For other results see \cite{other1,finster1,kro}. For an exhaustive list of references, see \cite{md}.

Note that all previous arguments for obtaining boundedness and  decay results on non-extreme black hole  spacetimes near the horizon would break down in our case (see Sections \ref{sec:RedshiftEffectAndSurfaceGravityOfH}, \ref{sec:TheSpacetimeTermKN}, \ref{sec:ConservationLawsOnDegenerateEventHorizons}). The reason for this is precisely the degeneracy of the redshift on $\mathcal{H}^{+}$.

\subsection{Overview of Results and Techniques}
\label{sec:OverviewOfMethodsAndResults}

We next describe the results we prove in this paper.   All our estimates are with respect to $L^{2}$ norms and we mainly use the robust vector field method.  Many of our results hold for more general (spherically symmetric) extreme black hole spacetimes (but we shall not pursue this here). We deduce that in some aspects the waves on extreme Reissner-Nordstr\"{o}m  have stable properties but in other aspects they appear to be unstable; this result is consistent with our discussion above about extreme black holes.

\subsubsection{Zeroth Order Morawetz and $X$ Estimates}
\label{sec:MorawetzAndXEstimates}

\noindent Our analysis begins with local $L^{2}$ spacetime estimates. We refer to local spacetime estimates controlling the derivatives of $\psi$  as ``$X$ estimates" and $\psi$ itself  as ``zeroth order Morawetz estimate". Both these types of estimates have a long history (see \cite{md}) beginning with the seminal work of Morawetz \cite{mor2} for the wave equation on Minkowski spacetime. They arise from the spacetime term of energy currents $J^{X}_{\mu}$ associated to a vector field $X$. For Schwarzschild,  such estimates appeared in \cite{blu0,blu3,blu2,dr3,dr5}  and for  Reissner-Nordstr\"{o}m in \cite{blu1}. The biggest difficulty in deriving an $X$ estimate for black hole spacetimes has to do with the \textit{trapping effect}. Indeed, from a continuity argument one can infer the existence of null geodesics which neither cross $\mathcal{H}^{+}$ nor terminate at $\mathcal{I}^{+}$. In our case these geodesics lie on a hypersurface of constant radius (see Section \ref{sec:PhotonSphereAndTrappingEffect}) known as the \textit{photon sphere}. From the analytical point of view, trapping affects the derivatives tangential to the photon sphere  and any non-degenerate spacetime estimate must lose (tangential) derivatives (i.e.~must require high regularity for $\psi$).
 
In this paper, we first (making minimal use of the spherical decomposition)  derive a zeroth order Morawetz estimate for $\psi$ which does not degenerate at the photon sphere.  For the case $l\geq 1$ (where $l$ is related to the eigenvalues of the spherical Laplacian, see Section \ref{sec:EllipticTheoryOnMathbbS2}) our work is inspired by \cite{dr5} and for $l=0$ we present a method which is 
robust and uses only geometric properties of the domain of outer communications. Our argument applies for a wider class of black hole spacetimes and, in particular, it applies for Schwarzschild. Note that no unphysical conditions are imposed on the initial data  which, in particular, are not required to be compactly supported or supported away from $\mathcal{H}^{+}$. Once this Morawetz estimate is established, we then show how to derive a degenerate (at the photon sphere) X estimate which does not require higher regularity and  a non-degenerate X estimate (for which we need, however, to commute with the  Killing $T$). These estimates, however, degenerate on $\mathcal{H}^{+}$; this degeneracy will be dropped later.  See Theorem \ref{th1} of Section \ref{sec:TheMainTheorems}.

\subsubsection{Uniform Boundedness of Non-Degenerate Energy}
\label{sec:UniformBoundednessOfEnergy}

The vector field $T=\partial_{v}$ is causal and Killing and  the energy flux of the current $J_{\mu}^{T}$ is non-negative definite (and bounded) but degenerates on the horizon (see Section \ref{sec:TheVectorFieldTextbfM}). Moreover, in view of the lack of redshift along $\mathcal{H}^{+}$, the divergence of the energy current $J_{\mu}^{N}$ associated to the redshift vector field $N$, first introduced in \cite{dr3}, is not positive definite near $\mathcal{H}^{+}$ (see Section \ref{sec:TheVectorFieldN}). For this reason we appropriately modify $J_{\mu}^{N}$ so the new bulk term is non-negative definite near $\mathcal{H}^{+}$. Note that the arising boundary terms can be bounded using  Hardy-like inequalities. It is important here to mention that a Hardy inequality (in the first form presented in Section \ref{sec:HardyInequalities}) allows us to bound the  local $L^{2}$ norm of $\psi$ on hypersurfaces crossing $\mathcal{H}^{+}$ using the (conserved) degenerate energy of $T$. Note also that  the bulk term is not positive far away from $\mathcal{H}^{+}$ and so to control these terms we use the $X$ and zeroth order Morawetz estimates. See Theorem \ref{t2} of Section \ref{sec:TheMainTheorems}.

\subsubsection{Conservation Laws on $\mathcal{H}^{+}$}
\label{sec:ConservationLawsOnMathcalH}

\noindent Although the bulk term of the modified redshift current is non-negative definite, it degenerates on $\mathcal{H}^{+}$ with respect to the derivative \textbf{transversal} to $\mathcal{H}^{+}$. This degeneracy is a characteristic feature of degenerate event horizons. 
Indeed, in Section \ref{sec:ConservationLawsOnDegenerateEventHorizons}, we show that the lack of redshift along $\mathcal{H}^{+}$ gives rise to a series of conservation laws (see Theorem \ref{t3} of Section \ref{sec:TheMainTheorems}). These laws apply for waves which are supported on the angular frequency $l$, i.e.~for which $\psi_{k}=0$ for all $k\neq l$, where $\psi_{k}$ is the projection of $\psi$ (viewed as an $L^{2}$ function on the spheres of symmetry) on the eigenspace $V_{k}$ of the spherical Laplacian  $\lapp$ (see Section \ref{sec:EllipticTheoryOnMathbbS2}). According to these laws a linear combination of the transversal derivatives of $\psi$ of order at most $l+1$ is conserved along the null geodesics of $\mathcal{H}^{+}$. As we shall see, it is these conserved quantities that allow us to infer the unstable properties of extreme black holes, and thus, understanding their structure is crucial and essential.  A consequence of these laws is that for generic initial data, the derivatives  transversal to $\mathcal{H}^{+}$  for waves $\psi$ which are supported on the  low angular frequencies  \textbf{do not decay}, and if the order of the derivatives is sufficiently high then they in fact \textbf{blow up} along $\mathcal{H}^{+}$. The genericity here refers to data for which certain quantities do not vanish on the horizon.

Based on these conservation laws, we will also prove that the Schwarzschild boundedness argument of Kay and Wald \cite{wa1} cannot be applied in the extreme case, i.e.~for generic $\psi$, there does not exist a Cauchy hypersurface $\Sigma$ crossing $\mathcal{H}^{+}$ and a solution $\tilde{\psi}$ such that 
\begin{equation*}
T\tilde{\psi}=\psi
\end{equation*}
in the  causal future of $\Sigma$.

\subsubsection{Sharp Higher Order $L^{2}$ Estimates}
\label{sec:SharpHigherOrderL2Estimates}

\noindent We next return to the problem of retrieving the derivative tranversal to $\mathcal{H}^{+}$  in the $X$ estimate in a neighbourhood of $\mathcal{H}^{+}$. More generally, we establish  $L^{2}$ estimates of the derivatives of $\psi$; see Theorem \ref{theorem3} of Section \ref{sec:TheMainTheorems}. In view of the conservation laws one expects to derive $k'$th order ($k\geq 1)$ $L^{2}$ estimates close to $\mathcal{H}^{+}$ only if $\psi_{l}=0$ for all $l\leq k$. However, on top of this low frequency obstruction comes another new feature of degenerate event horizons. Indeed, to obtain such estimates, one needs to require higher regularity for $\psi$ and commute  with the vector field transversal to $\mathcal{H}^{+}$ . This shows that $\mathcal{H}^{+}$ exhibits phenomena characteristic of trapping (see also the discussion in Section \ref{sec:TrappingEffectOnMathcalH}). Then by using  appropriate modifications and the Hardy inequalities we obtain the sharpest possible result. See Section \ref{sec:HigherOrderEstimates}. Note that although (an appropriate modification of) the redshift current can be used as a multiplier for all angular frequencies, the redshift vector field $N$ can only be used as a commutator\footnote{The redshift vector field was used as a commutator for the first time in \cite{dr7}.} for $\psi$ supported on the frequencies $l\geq 1$.

\subsubsection{Energy and Pointwise Decay}
\label{sec:EnergyAndPointwiseDecay}

\noindent Using an adaptation of  methods developed in the recent \cite{new}, we  obtain energy and pointwise decay for $\psi$. See Theorems \ref{t4} and \ref{t6} of Section \ref{sec:TheMainTheorems}.  In \cite{new}, a general framework is provided for obtaining decay results. The ingredients necessary for applying the framework are: 1) good asymptotics of the metric towards null infinity, 2) uniform boundedness of energy and 3) integrated local energy decay (where the spacetime integral of energy should be controlled by the energy of $\psi$ and $T\psi$). We first verify that extreme  Reissner-Nordstr\"{o}m satisfies the first hypothesis. However, in view of the trapping on the event horizon $\mathcal{H}^{+}$, it turns out that the method described in \cite{new} can not be directly used to yield decay results in the extreme case. Indeed, the third hypothesis of \cite{new} is not satisfied in extreme Reissner-Nordstr\"{o}m.  For this reason, we obtain several  hierarchies of estimates in an appropriate neighbourhood of $\mathcal{H}^{+}$. These estimates avoid multipliers or commutators with weights in $t$, following the philosophy of \cite{new}. Our method applies to   black hole spacetimes where trapping is present on $\mathcal{H}^{+}$ (including, in particular, a wide class of extreme black holes). Pointwise estimates follow then by commuting with the generators of so(3) and Sobolev estimates. We shall also see that the low angular frequencies decay more slowly than the higher ones. See Section \ref{sec:PointwiseEstimates}.

\subsubsection{Higher Order Pointwise Estimates}
\label{sec:HigherOrderPointwiseEstimates1}

\noindent In order to provide a complete picture of the behaviour of waves $\psi$, it remains to derive pointwise estimates for all derivatives of $\psi$.  We show that if $\psi$ is supported on the angular frequency $l$, then the derivatives transversal to $\mathcal{H}^{+}$  of $\psi$ decay if the order is at most $l$. If the order is $l+1$, then  for generic initial data this derivative converges along $\mathcal{H}^{+}$ to a non-zero number. As before, by generic initial data we mean data for which certain quantities do not vasish on $\mathcal{H}^{+}$. If, moreover, the order is at least $l+2$, then for generic initial data these derivatives blow up asymptotically along $\mathcal{H}^{+}$.  On the other hand, for $T^{m}\psi$  one needs to consider the order to be at least $l+2+m$  so the derivatives blow up. See Theorems \ref{theo8} and \ref{theo9} of Section \ref{sec:TheMainTheorems}. Therefore, the $T$ derivatives\footnote{It is also shown that $T\psi$ decays faster than $\psi$.} slightly counteract the action of the derivatives transversal to $\mathcal{H}^{+}$. We conclude this paper by deriving similar decay, non-decay and blow-up results for the higher order energy. In particular, this  shows that one can commute with the redshift vector field at most $l$ times for   $\psi$ supported on the frequency $l$.  See Section \ref{sec:HigherOrderPointwiseEstimates}.  

\subsection{Remarks on the Analysis of Extreme Black Holes}
\label{sec:EventHorizonVsPhotonSphere}

We conclude this introductory section by discussing several new features of degenerate event horizons. 

\subsubsection{Dispersion vs Redshift}
\label{sec:DispersionVsRedshift}

In \cite{md}, it was shown that for a wide variety of non-extreme black holes, the redshift on $\mathcal{H}^{+}$ suffices to yield uniform boundedness of the non-degenerate energy without any need of understanding the dispersion properties of $\psi$. However, in the extreme case, the degeneracy of the redshift makes the understanding of the dispersion of $\psi$ essential even for the problem of  boundedness. In particular, one has to derive spacetime integral estimates for $\psi$ and its derivatives.

\subsubsection{Trapping Effect on $\mathcal{H}^{+}$}
\label{sec:TrappingEffectOnMathcalH}

According to the results of Sections \ref{sec:CommutingWithAVectorFieldTransversalToMathcalH} and \ref{sec:HigherOrderEstimates}, in order to obtain $L^{2}$ estimates in neighbourhoods of $\mathcal{H}^{+}$ one must require higher regularity for $\psi$ and commute with the vector field transversal to $\mathcal{H}^{+}$ (which is not Killing). This loss of a derivative is characteristic  of trapping. Geometrically, this is related to the fact that the null generators of $\mathcal{H}^{+}$ viewed as integrals curves of the Killing vector field $T$ are affinely parametrized. 

The trapping properties of the photon sphere have different analytical flavour. Indeed, in order to obtain $L^{2}$ estimates in regions which include the photon sphere, one needs to commute with  either $T$ or the generators of the Lie algebra so(3) (note that all these vector fields are Killing). Only the high angular frequencies are trapped on the photon sphere (and for the low frequencies no commutation is required) while all the angular frequencies are trapped (in the above sense) on $\mathcal{H}^{+}$.

\subsubsection{Instability of Degenerate Horizons}
\label{sec:InstabilityOfDegenerateHorizons}

Our setting is appropriate for understading the dynamic formation of extreme black holes from gravitational collapse:

 \begin{figure}[H]
	\centering
		\includegraphics[scale=0.14]{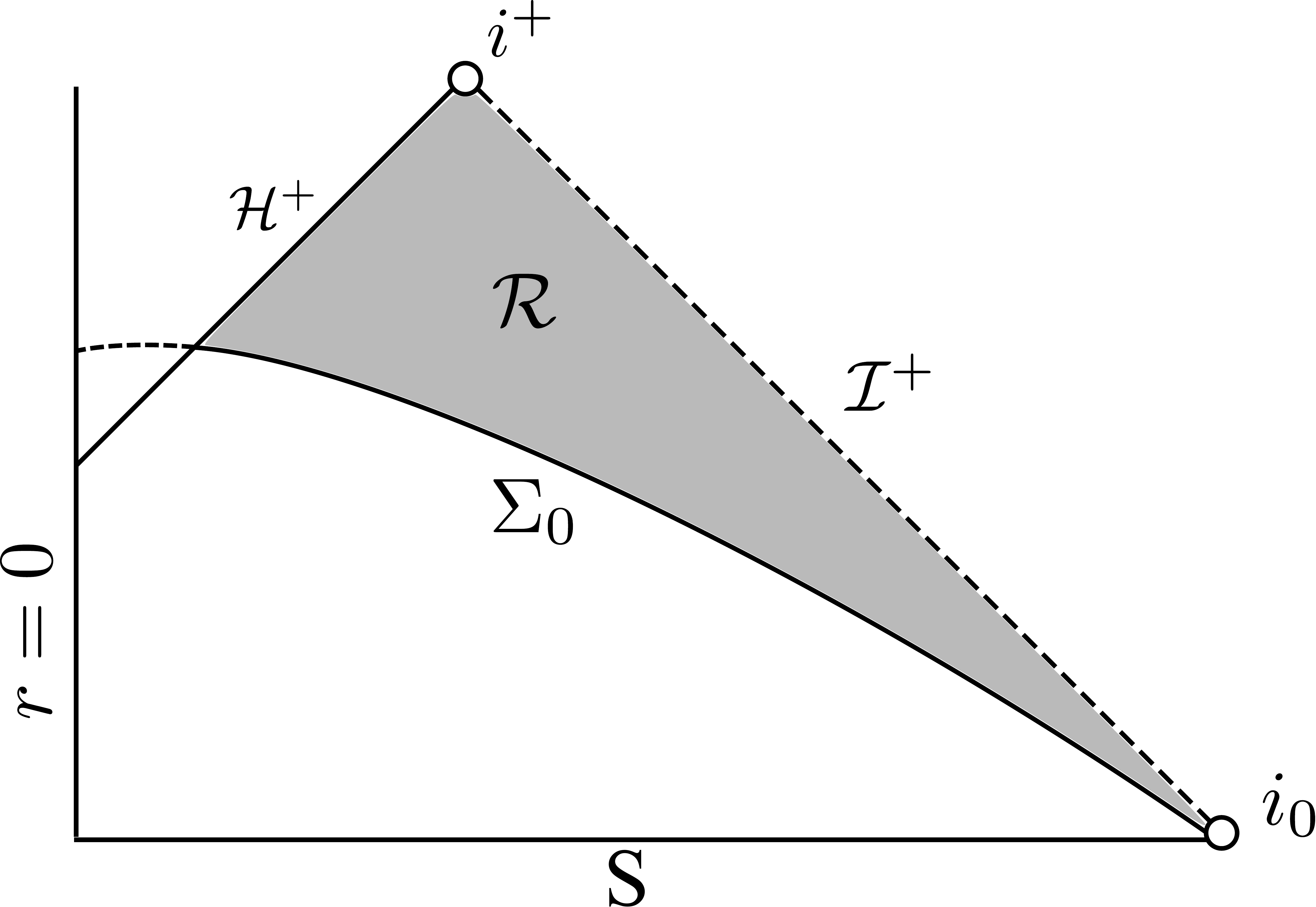}
	\label{fig:collapse}
\end{figure}
One can construct spacetimes (for instance solutions of the Einstein-Maxwell-Scalar field system) such that the initial hypersurface $S\sim\mathbb{R}^{3}$ is complete and asymptotically flat with one end and the spacetime in region $\mathcal{R}$ coincides with the exterior region of extreme R-N. Consider arbitrary initial data $ID_{\Sigma_{0}}$ on $\Sigma_{0}$ for the wave equation (in particular, suppose that their support includes $\mathcal{H}^{+}$).
By (extending and) solving backwards, we obtain a new initial data set $ID_{S}$ on the initial hypersurface $S$. 
The initial data  $ID_{\Sigma_{0}}$ and $ID_{S}$ are isomorphic. Indeed, if one considers a scalar perturbation which corresponds to the data set $ID_{S}$, then the trace of the  perturbation on $\Sigma_{0}$ corresponds to the  data   $ID_{\Sigma_{0}}$. Hence, the support of  our initial data on $\Sigma_{0}$ should include $\mathcal{H}^{+}$. That is to say, our results represent a poor man's linearisation of the non-linear problem.

In the linear level,  the low angular frequencies not only  are they trapped on $\mathcal{H}^{+}$, but their evolution is governed by  conservation laws along null generators. These laws imply that for generic initial data (i.e.~data for which certain quantities do not vanish on $\mathcal{H}^{+}\cap\Sigma_{0}$), the higher order derivatives transversal to $\mathcal{H}^{+}$ blow up along $\mathcal{H}^{+}$. These differential operators are translation invariant and do not depend on the choice of a coordinate system. The blow-up of these geometric quantities suggests that extreme black holes are dynamically unstable\footnote{Price's law conjectures that all ``parameters'' of the exterior spacetime other than the mass, charge and angular momentum --so-caled ``hair''-- should decay polynomially along the event horizon or null infinity. This decay suggests that black holes with regular event horizons are stable from the point of view of far away observers and can form dynamicaly in collapse. Clearly, our results prove that Price's law does not hold in extreme R-N.}. We hope that the methods we develop in this paper will be useful for proving such a result in the nonlinear setting.

\subsection{Open Problems}
\label{sec:FutureWork}

An important problem is that of understanding the solutions of the wave equation on the extreme Kerr spacetime. This spacetime is not spherically symmetric and there is no globally causal Killing field in the domain of outer communications (in particular, $T$ becomes spacelike close to the event horizon). Recent results \cite{megalaa} overcome these difficulties for the whole subextreme range of Kerr. The extreme case remains open.

One could  consider the problem of the wave equation coupled with the Einstein-Maxwell equations. Then decay for the scalar field was proven in the deep work of Dafermos and Rodnianski  \cite{price}. Again these results hold for non-extreme black holes. For extreme black holes even boundedness of waves remains open. 

\subsection{Addendum: Published Version}
\label{add}
 
After its original appearance on the arxiv, this work was subsequently improved with various new results and split into two (now published) parts \cite{aretakis1,aretakis2}.

\section{Geometry of  Extreme Reissner-Nordstr\"{o}m \\ Spacetime}
\label{sec:GeometryOfExtremeReissnerNordstromSpacetime}

The extreme Reissner-Nordstr\"{o}m family is a one parameter subfamily sitting inside the two parameter Reissner-Nordstr\"{o}m family of 4-dimensional Lorentzian manifolds $(\mathcal{N}_{M,e},g_{M,e})$, where the parameters are the mass $M>0$ and the electromagnetic charge $e\geq 0$. The extreme case corresponds to $M=e$. In the ingoing Eddington-Finkelstein coordinates $(v,r)$ the metric of the extreme Reissner-Nordstr\"{o}m takes the form
 \begin{equation}
g=-Ddv^{2}+2dvdr+r^{2}g_{\scriptstyle\mathbb{S}^{2}},
\label{RN1}
\end{equation}
where 
\begin{equation*}
D=D\left(r\right)=\left(1-\frac{M}{r}\right)^{2}
\end{equation*}
and $g_{\scriptstyle\mathbb{S}^{2}}$ is the standard metric on $\mathbb{S}^{2}$. The Penrose diagram (see also Appendix \ref{sec:PenroseDiagrams}) of the spacetime $\mathcal{N}$ covered by this coordinate system for $v\in\mathbb{R},r\in\mathbb{R}^{+}$ is
 \begin{figure}[H]
	\centering
		\includegraphics[scale=0.14]{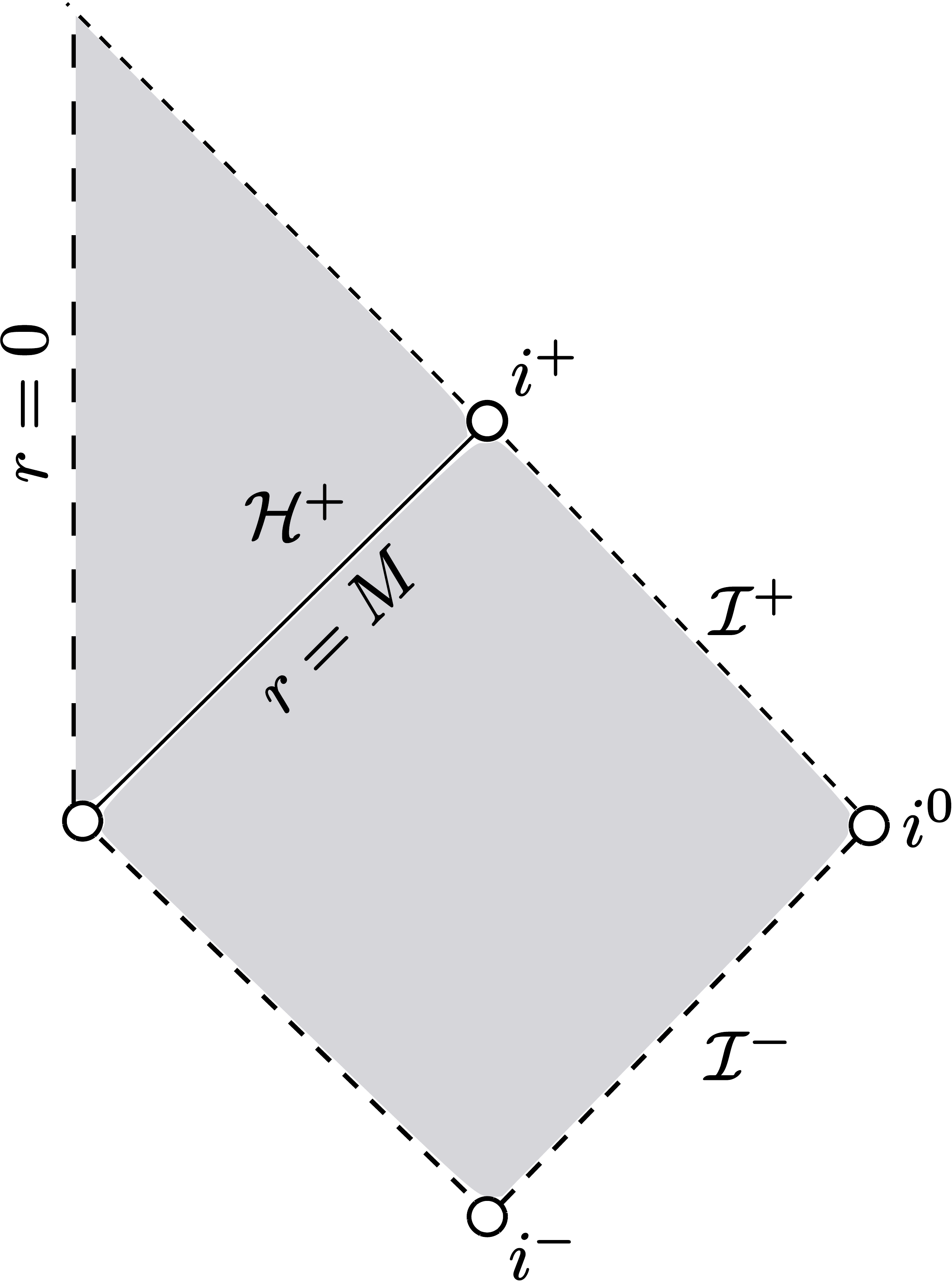}
	\label{fig:ern001}
\end{figure}
We will refer to the hypersurface  $r=M$ as the event horizon (and denote it by $\mathcal{H}^{+}$) and the region $r\leq M$ as the black hole region. The region where $M<r$ corresponds to  the domain of outer communications.

In view of the existence of the timelike curvature singularity $\left\{r=0\right\}$ `inside' the black hole (thought of here as a singular boundary of the black hole region) and its unstable behaviour,  we are only interested in studying the wave equation in the domain of outer communications \textbf{including the horizon} $\mathcal{H}^{+}$. Note that the study of the horizon is of fundamental importance. The horizon determines the existence of the black hole and therefore any attempt to prove the nonlinear stability of the exterior of black holes must come to terms with the structure of the horizon.

We consider a connected asymptotically flat SO(3)-invariant spacelike hypersurface $\Sigma_{0}$ in $\mathcal{N}$ terminating at $i^{0}$ with boundary  such that $\partial\Sigma_{0}=\Sigma_{0}\cap\mathcal{H}^{+}$. We also assume that if $n$ is its future directed unit normal and $T=\partial_{v}$ then there exist positive constants $C_{1}<C_{2}$ such that 
\begin{equation*}
\begin{split}
&C_{1}<-g\left(n,n\right)<C_{2},\\
&C_{1}<-g\left(n,T\right)<C_{2}.
\end{split}
\end{equation*}
Let $\mathcal{M}$ be the domain of dependence of $\Sigma_{0}$. Then, using the coordinate system $(v,r)$ we have
\begin{equation}
\mathcal{M}=\left(\left(-\infty,+\infty\right)\times\left[M\right.\!,+\infty\left.\right)\times\mathbb{S}^{2}\right)\cap J^{+}\left(\Sigma_{0}\right),
\label{Mern}
\end{equation}where  $J^{+}\left(\Sigma_{0}\right)$ is the causal future of $\Sigma_{0}$ (which by our convention includes $\Sigma_{0}$). Note  that $\mathcal{M}$ is a manifold with stratified (piecewise smooth) boundary  $\partial\mathcal{M}=(\mathcal{H}^{+}\cap\mathcal{M})\cup\Sigma_{0}$.
\begin{figure}[H]
	\centering
		\includegraphics[scale=0.14]{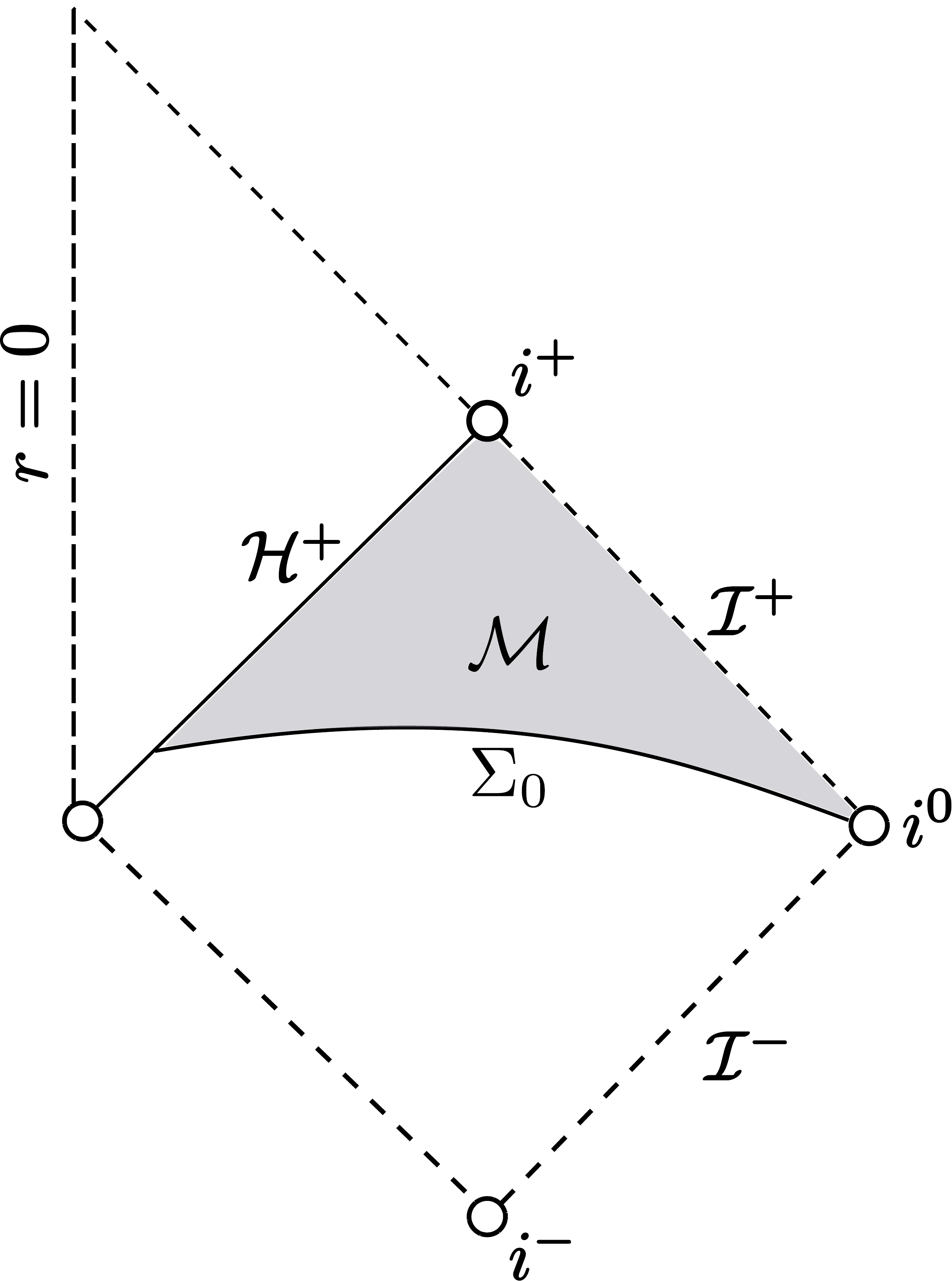}
	\label{fig:exrnw1}
\end{figure}

\subsection{The Foliations $\Sigma_{\tau}$ and $\tilde{\Sigma}_{\tau}$}
\label{sec:TheFoliationsSigmaTauAndTildeSigmaTau}

We  consider the foliation $\Sigma_{\tau}=\varphi_{\tau}(\Sigma_{0})$, where $\varphi_{\tau}$ is the flow of $T=\partial_{v}$. Of course, since $T$ is Killing, the hypersurfaces $\Sigma_{\tau}$ are all isometric to $\Sigma_{0}$. 
 \begin{figure}[H]
	\centering
		\includegraphics[scale=0.14]{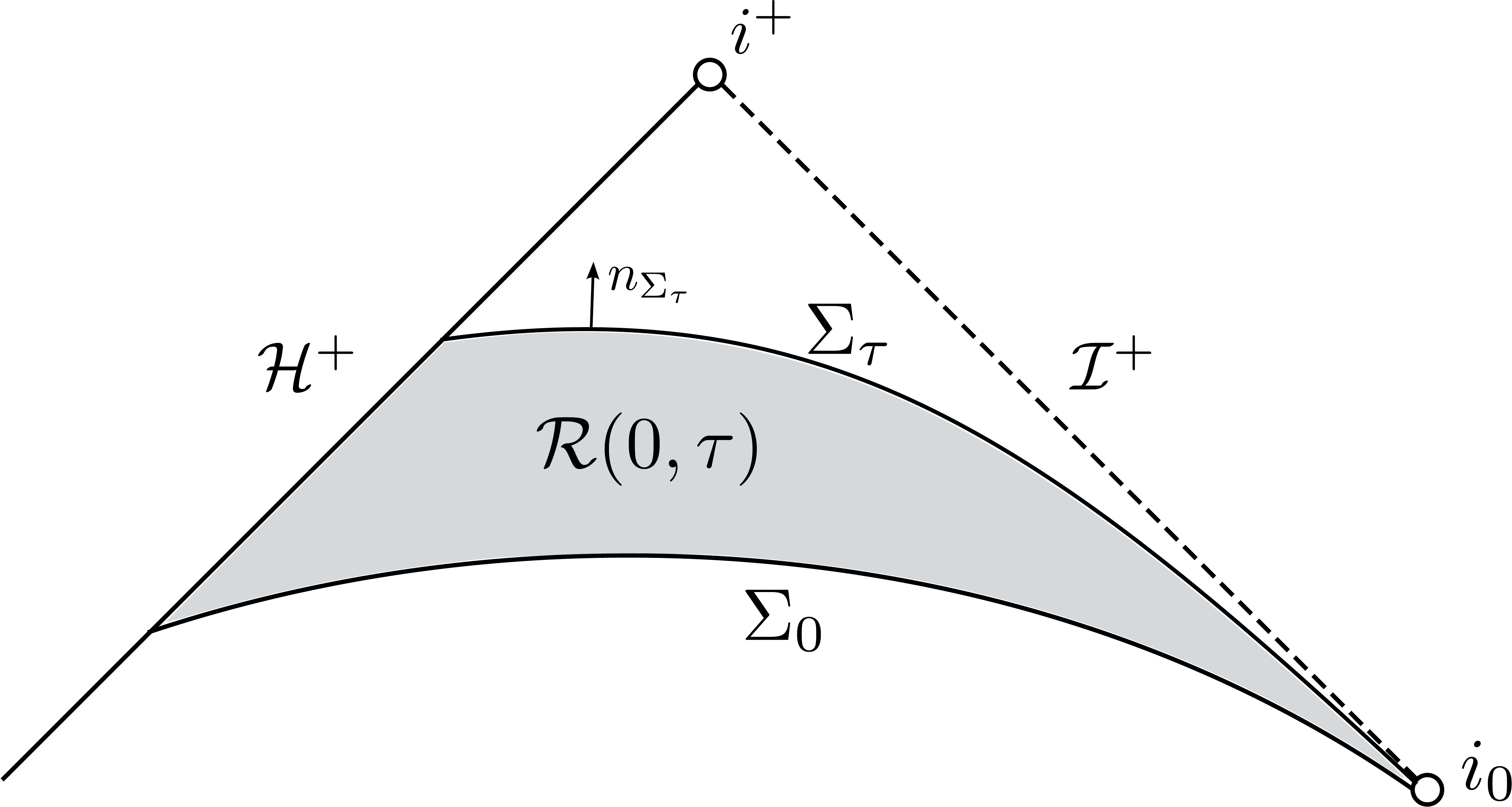}
	\label{fig:ernt1}
\end{figure}
We define the region
\begin{equation*} 
\mathcal{R}(0,\tau)=\cup_{0\leq \tilde{\tau}\leq \tau}\Sigma_{\tilde{\tau}}. 
\end{equation*}
On $\Sigma_{\tau}$ we have an induced Lie propagated coordinate system $(\rho,\omega)$ such that $\rho\in[\left.\right.\!\!\M,+\infty)$ and $\omega\in\mathbb{S}^{2}$. These coordinates are defined such that if $Q\in\Sigma_{\tau}$ and $Q=(v_{Q},r_{Q},\omega_{Q})$  then $\rho=r_{Q}$ and $\omega=\omega_{Q}$. Our assumption on the normal $n_{\Sigma_{0}}$ (and thus for $n_{\Sigma_{\tau}})$ implies that there exists a bounded function $g_{1}$ such that  
\begin{equation*}
\partial_{\rho}=g_{1}\partial_{v}+\partial_{r}.
\end{equation*}
This defines a coordinate system since $[\partial_{\rho},\partial_{\theta}]=[\partial_{\rho},\partial_{\phi}]=0$. Moreover, the volume form of $\Sigma_{\tau}$ is
\begin{equation}
dg_{\scriptstyle\Sigma_{\tau}}=V\rho^{2}d\rho d\omega,
\label{volumeformfoliation}
\end{equation}
 where $V$ is a positive bounded function.
 \begin{figure}[H]
	\centering
		\includegraphics[scale=0.14]{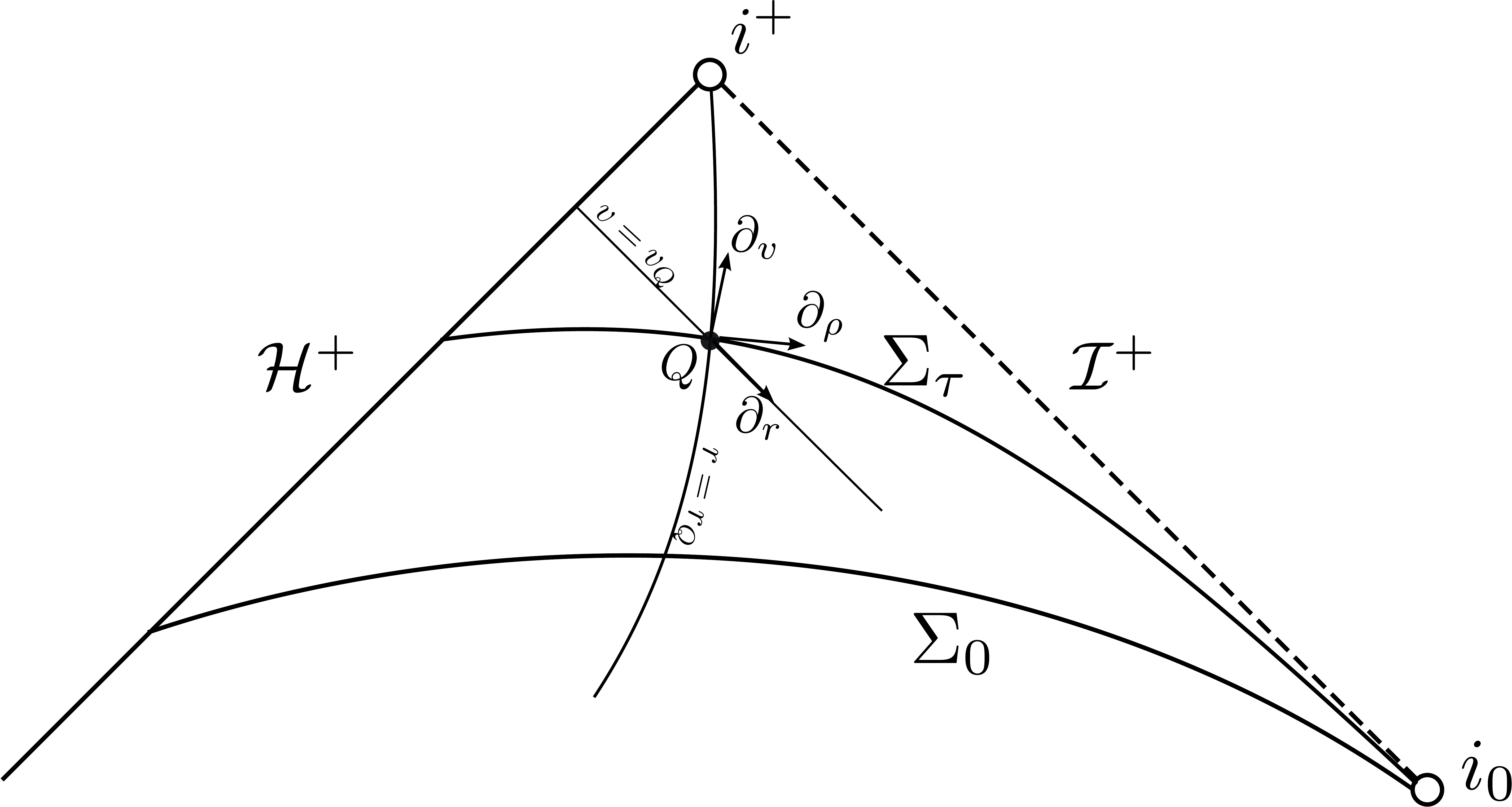}
	\label{fig:ernt2}
\end{figure}
Another foliation is $\tilde{\Sigma}_{\tau}$ which, instead of terminating at $i^{0}$,  terminates at $\mathcal{I}^{+}$ and thus  ``follows" the waves to the future. 
 \begin{figure}[H]
	\centering
		\includegraphics[scale=0.14]{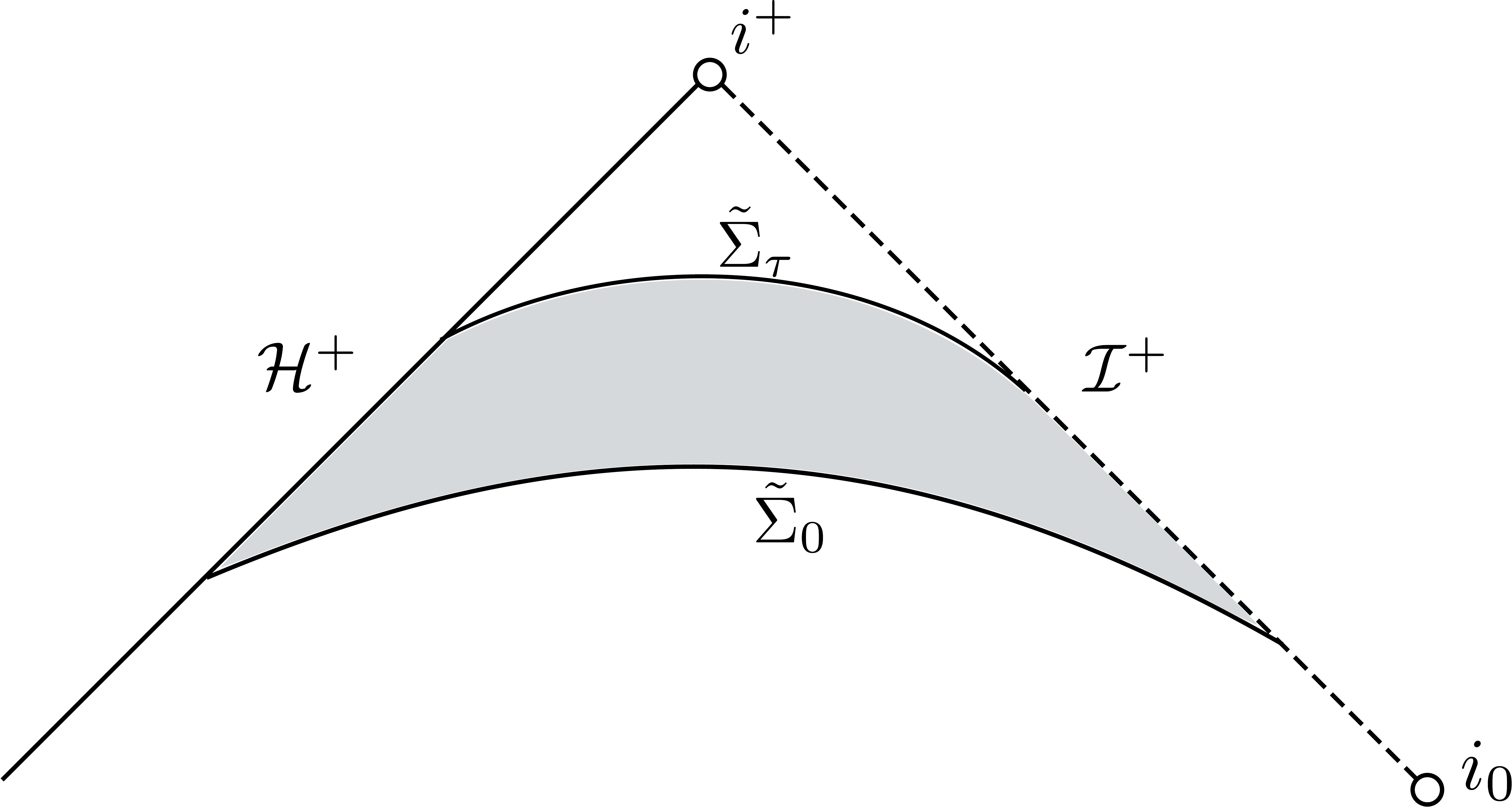}
	\label{fig:ernt3}
\end{figure}
One can similarly define an induced coordinate system on $\tilde{\Sigma}_{\tau}$. Note that only local elliptic estimates (see Appendix \ref{sec:EllipticEstimates}) are to be applied on $\tilde{\Sigma}_{\tau}$.

\subsection{The Photon Sphere and Trapping Effect}
\label{sec:PhotonSphereAndTrappingEffect}

One can easily see that there exist orbiting future directed null geodesics, i.e.~null geodesics that neither cross the horizon $\mathcal{H}^{+}$ nor meet null infinity $\mathcal{I}^{+}$. A class of such  geodesics $\gamma$ is of the form
\begin{equation*}
\begin{split}
\gamma:&\mathbb{R}\rightarrow\mathcal{M}\\
&\tau\mapsto \gamma\left(\tau\right)=\left(t\left(\tau\right),Q,\frac{\pi}{2},\phi\left(\tau\right)\right).
\end{split}
\end{equation*}
The conditions $\nabla_{\overset{.}{\gamma}}\overset{.}{\gamma}=0$ and $g\left(\overset{.}{\gamma},\overset{.}{\gamma}\right)=0$ imply that 
\begin{equation*}
Q=2M,
\end{equation*}
which is the radius of the so called \textit{photon sphere}. The $t,\phi$ depend linearly on $\tau$.
 \begin{figure}[H]
	\centering
		\includegraphics[scale=0.13]{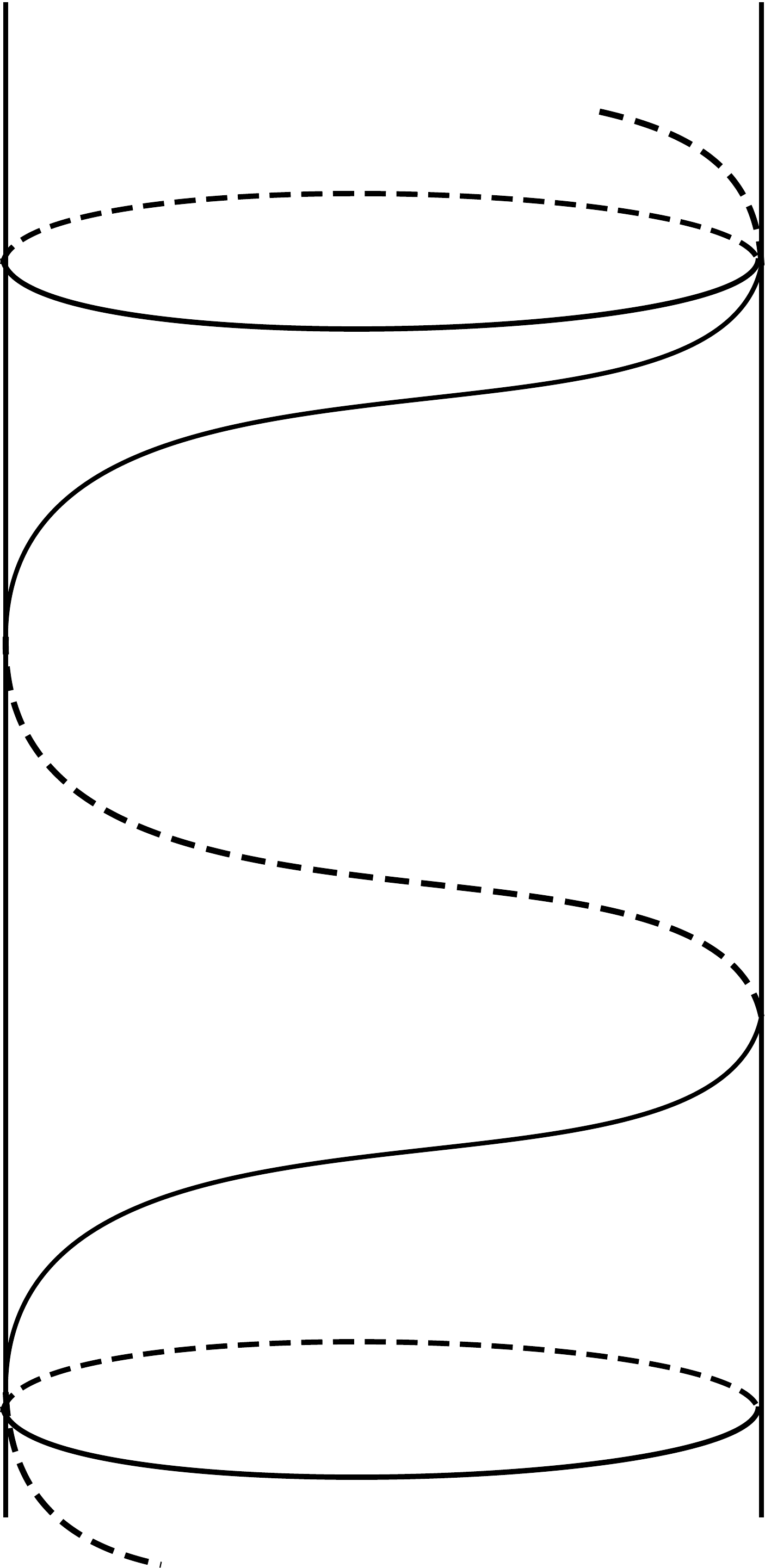}
	\label{fig:ps1}
\end{figure}

In the subextreme case we have $Q=\frac{3M}{2}\left(1+\sqrt{1-\frac{8e^{2}}{9M^{2}}}\right)$. In fact, from any point there is a family of null geodesics which asymptotically converge to the photon sphere to the future. The physical interpretation of the photon sphere will be crucial in what follows. Indeed, the existence of this ``sphere'' (which is in fact a 3-dimensional timelike hypersurface) implies that the energy of some photons is not scattered to null  infinity or  the black hole region. This is the so called \textit{trapping effect}. Note, in comparison, that in Minkowski spacetime all future directed null geodesics meet future  null infinity $\mathcal{I}^{+}$. As we shall see, this effect forces us to require higher regularity for the waves in order to achieve decay results.

\subsection{The Redshift Effect and Surface Gravity of $\mathcal{H}^{+}$}
\label{sec:RedshiftEffectAndSurfaceGravityOfH}
As we have already seen, the vector field $\partial_{v}$ becomes null on the event horizon $\mathcal{H}^{+}$ and is also  tangent to it. Therefore,  $\mathcal{H}^{+}$ is a null hypersurface. In view of the symmetry of the Levi-Civita connection, one  easily sees that if $N=\nabla f$, where the null hypersurface is given by $f=0$, then $N$ generates geodesics and  satisfies the equation 
\begin{equation}
\nabla_{N}N=\lambda N.
\label{sg0}
\end{equation}
A similar equation holds if there exists a Killing field $V$ which is normal to the null hupersurface. In this case, we  have
\begin{equation}
\nabla_{V}V=\kappa V
\label{sg}
\end{equation}
on the hypersurface. Since $V$ is Killing, the function $\kappa$ is constant along the integral curves of $V$. This can be seen by taking the pushforward of \eqref{sg} via the flow of $V$ and noting that since the flow of $V$ consists of isometries, the pushforward of the Levi-Civita connection is the same connection. The quantity $\kappa$ is called the \textit{surface gravity}\footnote{This plays a significant role in  black hole ``thermodynamics" (see also~\cite{wald} and~\cite{poisson}).} of the null hypersurface. Note that, in Riemannian geometry, any Killing field that satisfies \eqref{sg} must have $\kappa=0$; this is not the case for Lorentzian manifolds, however. It is the equation $g\left(V,V\right)=0$ that allows $\kappa$ not to be zero. In  the Reissner-Nordstr\"{o}m family, the surface gravities of the two horizons $\left\{r=r_{-}\right\}$ and $\left\{r=r_{+}\right\}$ are given by 
\begin{equation}
\kappa_{\pm}=\frac{r_{\pm}-r_{\mp}}{2r^{2}_{\pm}}=\frac{1}{2}\left.\frac{dD\left(r\right)}{dr}\right|_{r=r_{\pm}},
\label{sgrn}
\end{equation}
where $D$ is given by \eqref{d} of Appendix \ref{sec:OnTheGeometryOfReissnerNordstrOM}. Note that in  extreme Reissner-Nordstr\"{o}m spacetime we have $D\left(r\right)=\left(1-\frac{M}{r}\right)^{2}$ and so $r_{+}=r_{-}=M$ which implies that the surface gravity vanishes. In general, a horizon whose surface gravity vanishes is called \textit{degenerate}.

Physically, the surface gravity is related to the so-called \textit{redshift effect} that is observed along (and close to) $\mathcal{H}^{+}$. According to this effect, the wavelength of radiation close to $\mathcal{H}^{+}$ becomes longer as $v$ increases and thus the radiation gets less energetic.  This effect has a long history in the heuristic analysis of waves but only in the last decade has it  been used mathematically. For example,  Price's law (see \cite{price}) and the stability and instability of Cauchy horizons in appropriate setting (see \cite{d1}) were proved using heavily this effect. (Note that for the latter, one also needs to use the dual blueshift effect which is present in the interior of black holes.) Moreover, a second proof of the boundedness of waves in Schwarzschild  (the first proof was given in \cite{wa1}) was established in  \cite{dr3} based on redshift, whereas in \cite{md} it is proved that under some geometric assumptions, the positivity of surface gravity suffices to prove boundedness of waves without understanding the trapping.

\section{The Cauchy Problem for the Wave Equation}
\label{sec:TheCauchyProblemForTheWaveEquation}

We consider solutions of the Cauchy problem of the wave equation \eqref{1eq} with initial data 
\begin{equation}
\left.\psi\right|_{\Sigma_{0}}=\psi_{0}\in H^{k}_{\operatorname{loc}}\left(\Sigma_{0}\right), \left.n_{\Sigma_{0}}\psi\right|_{\Sigma_{0}}=\psi_{1}\in H^{k-1}_{\operatorname{loc}}\left(\Sigma_{0}\right),
\label{cd}
\end{equation}
where the hypersurface $\Sigma_{0}$ is as defined in Section \ref{sec:GeometryOfExtremeReissnerNordstromSpacetime} and $n_{\Sigma_{0}}$ denotes the future unit normal of $\Sigma_{0}$.  In view of the global hyperbolicity of $\mathcal{M}$, there exists a unique solution to the above equation. Moreover, as long as $k\geq 1$, we have that for any spacelike hypersurface  $S$
\begin{equation*}
\left.\psi\right|_{S}\in H^{k}_{\operatorname{loc}}\left(S\right), \left.n_{S}\psi\right|_{S}\in H^{k-1}_{\operatorname{loc}}\left(S\right).
\end{equation*}
In this paper we will be interested in the case where $k\geq 2$.  Moreover, we assume that 
\begin{equation}
\lim_{x\rightarrow i^{0}}r\psi^{2}(x)=0.
\label{condition}
\end{equation}
 For simplicity, from now on, \textbf{when we say ``for all solutions $\psi$ of the wave equation" we will assume that $\psi$ satisfies the above conditions}. Note that for obtaining sharp decay results we will have to consider even higher regularity for $\psi$.

\section{The Main Theorems}
\label{sec:TheMainTheorems}

We consider the Cauchy problem for the wave equation (see Section \ref{sec:TheCauchyProblemForTheWaveEquation}) on the extreme Reissner-Nordstr\"{o}m spacetime. This spacetime is partially covered by the coordinate systems $(t,r)$, $(t,r^{*})$, $(v,r)$ and $(u,v)$  described in Appendix \ref{sec:ConstructingTheExtentionOfReissnerNordstrOM}. Recall that $M$ is a positive parameter and  $D=D(r)=\left(1-\frac{M}{r}\right)^{2}$. Recall also that the horizon $\mathcal{H}^{+}$ is located at $\left\{r=M\right\}$ and the photon sphere at $\left\{r=2M\right\}$.

We denote $T=\partial_{v}=\partial_{t}$, where $\partial_{v}$ corresponds to the system $(v,r)$ and $\partial_{t}$ corresponds to $(t,r)$.  From now on, $\partial_{v},\partial_{r}$ are the vector fields corresponding to $(v,r)$, unless otherwise stated. Note that $\partial_{r^{*}}=\partial_{t}$ on $\mathcal{H}^{+}$ and therefore it is not transversal to $\mathcal{H}^{+}$, whereas $\partial_{r}$ is transversal to $\mathcal{H}^{+}$.

The foliations $\Sigma_{\tau}$ and $\tilde{\Sigma}_{\tau}$  are defined in Sections \ref{sec:TheFoliationsSigmaTauAndTildeSigmaTau} and \ref{sec:EnergyDecay} and the current $J^{V}$ associated to the vector field $V$ is defined in Section \ref{sec:TheCurrentsJKAndMathcalE}. For reference, we mention that close to  $\mathcal{H}^{+}$ we have 
\begin{equation*}
J_{\mu}^{T}(\psi)n^{\mu}_{\Sigma}\sim \, (T\psi)^{2}+\left(1-\frac{M}{r}\right)^{2}(\partial_{r}\psi)^{2}+\left|\nabb\psi\right|^{2},
\end{equation*}
which degenerates on $\mathcal{H}^{+}$ whereas
\begin{equation*}
J_{\mu}^{n}(\psi)n^{\mu}_{\Sigma}\sim \, (T\psi)^{2}+(\partial_{r}\psi)^{2}+\left|\nabb\psi\right|^{2},
\end{equation*}
which does not degenerate on $\mathcal{H}^{+}$. As regards the foliation $\tilde{\Sigma}_{\tau}$, we have for $r$ sufficiently large
\begin{equation*}
J_{\mu}^{T}(\psi)n^{\mu}_{\tilde{\Sigma}}\sim (\partial_{v}\psi)^{2}+\left|\nabb\psi\right|^{2},
\end{equation*}
where $\partial_{v}$ corresponds to the null coordinate system $(u,v)$.
The Fourier decomposition of $\psi$ on $\mathbb{S}^{2}(r)$ is discussed in Section \ref{sec:EllipticTheoryOnMathbbS2}, where it is also defined what it means for a function to be supported on a given range of angular frequencies. The notation $\psi_{l}$ is also introduced in Section \ref{sec:EllipticTheoryOnMathbbS2}.  Note that all the integrals are considered with respect to the volume form. The initial data are assumed to be as in Section \ref{sec:TheCauchyProblemForTheWaveEquation} and  sufficiently regular  such that the right hand side of the estimates below are all finite. Then we have the following

\begin{mytheo}(\textbf{Morawetz and $X$ estimates})
There exists a  constant $C>0$ which depends  on $M$ and $\Sigma_{0}$ such that for all solutions $\psi$ of the wave equation the following estimates hold
\begin{enumerate}
	\item 
\textbf{Non-Degenerate Zeroth Order Morawetz Estimate:}
	 \begin{equation*}
	\begin{split}
	\displaystyle\int_{\mathcal{R}(0,\tau)}{\frac{1}{r^{4}}\psi^{2}}\leq C\displaystyle\int_{\Sigma_{0}}{J^{n_{\Sigma_{0}}}_{\mu}(\psi)n^{\mu}_{\Sigma_{0}}}.
		\end{split}
	\end{equation*}
	\item 
\textbf{$X$ Estimate with Degeneracy at $\mathcal{H}^{+}$ and Photon Sphere:}
	\begin{equation*}
	\begin{split}
&\displaystyle\int_{\mathcal{R}(0,\tau)}\!\!\!{\left(\frac{\sqrt{D}}{r^{4}}(\partial_{r^{*}}\psi)^{2}+\frac{(r-2M)^{2}}{r^{6}}\left((\partial_{t}\psi)^{2}+\left|\nabb\psi\right|^{2}\right)\right)}\\&\ \ \ \ \ \ \ \ \ \ \ \ \ \ \leq C\displaystyle\int_{\Sigma_{0}}{J_{\mu}^{T}(\psi)n^{\mu}_{\Sigma_{0}}}.
	\end{split}
	\end{equation*}
		\item \textbf{$X$ Estimate with Degeneracy at $\mathcal{H}^{+}$:}
		\begin{equation*}
	\begin{split}
&\displaystyle\int_{\mathcal{R}(0,\tau)}\!\!\!{\left(\frac{1}{r^{4}}\left(\partial_{t}\psi\right)^{2}+\frac{\sqrt{D}}{r^{4}}\left(\partial_{r^{*}}\psi\right)^{2}+\frac{1}{r}\left|\nabb\psi\right|^{2}\right)}
\\&\ \ \ \ \ \ \ \ \ \ \ \ \ \ \ \ \ \ \ \leq C\displaystyle\int_{\Sigma_{0}}{\left(J_{\mu}^{T}(\psi)n^{\mu}_{\Sigma_{0}}+J_{\mu}^{T}(T\psi)n^{\mu}_{\Sigma_{0}}\right)}.
		\end{split}
	\end{equation*}
		\item \textbf{Non-Degenerate $X$ Estimate with Commutation:}
		
		 If, in addition,  $\psi$  is supported on the angular frequencies $l\geq 1$ then
\begin{equation*}
	\begin{split}
&\displaystyle\int_{\mathcal{R}(0,\tau)}\!\!\!{\left(\frac{1}{r^{4}}(\partial_{r}\psi)^{2}+\frac{1}{r^{4}}(\partial_{t}\psi)^{2}+\frac{1}{r}\left|\nabb\psi\right|^{2}\right)}\\&\ \ \ \ \ \ \leq C\displaystyle\int_{\Sigma_{0}}{J_{\mu}^{n_{\Sigma_{0}}}(\psi)n^{\mu}_{\Sigma_{0}}}+C\displaystyle\int_{\Sigma_{0}}{J_{\mu}^{n_{\Sigma_{0}}}(n_{\Sigma_{0}}\psi)n^{\mu}_{\Sigma_{0}}}.
	\end{split}
	\end{equation*}		
	\end{enumerate}
	\label{th1}
\end{mytheo}

\begin{mytheo}(\textbf{Uniform boundedness of Non-Degenerate Energy}) 
There exists a  constant $C>0$ which depends  on $M$ and $\Sigma_{0}$ such that for all solutions $\psi$ of the wave equation we have
\begin{equation*}
\displaystyle\int_{\Sigma_{\tau}}{J_{\mu}^{n_{\Sigma_{\tau}}}(\psi)n^{\mu}_{\Sigma_{\tau}}}\leq C\displaystyle\int_{\Sigma_{0}}{J_{\mu}^{n_{\Sigma_{0}}}(\psi)n^{\mu}_{\Sigma_{0}}}.
\end{equation*}
\label{t2}
\end{mytheo}

\begin{mytheo}(\textbf{Conservation Laws along $\mathcal{H}^{+}$})
There exist constants $\a_{i}^{j},j=0,1,...,l-1,\, i=0,1,...,j+1$, which depend on $M$ and $l$ such that for all  solutions $\psi$ of the wave equation which are supported on the frequency $l$ we have
\begin{equation*}
\partial_{r}^{j}\psi=\sum_{i=0}^{j+1}{\a_{i}^{j}\partial_{v}\partial_{r}^{i}\psi},
\end{equation*}
on $\mathcal{H}^{+}$. Moreover, there exist constants $\b_{i},i=0,1,...,l$, which depend on $M$ and $l$ such that the quantity
\begin{equation*}
H_{l}[\psi]=
\partial_{r}^{l+1}\psi+\sum_{i=0}^{l}{\b_{i}\partial_{r}^{i}\psi}
\end{equation*}
is conserved along the null geodesics of $\mathcal{H}^{+}$ and thus does not decay for generic initial data.
\label{t3}
\end{mytheo}

\begin{mytheo}(\textbf{Higher Order $L^{2}$ Estimates: Trapping on $\mathcal{H}^{+}$})
There exists $r_{0}$ such that $M<r_{0}<2M$ and a  constant $C>0$ which depends on $M$, $k$ and $\Sigma_{0}$ such that if $\mathcal{A}=\left\{M\leq r\leq r_{0}\right\}$ then   for all solutions $\psi$ of the wave equation with $\psi_{l}=0$ for all $l\leq k-1$, $k\in\mathbb{N}$, the following holds
\begin{equation*}
\begin{split}
&\displaystyle\int_{\Sigma_{\tau}\cap\mathcal{A}}{\left(\partial_{v}\partial_{r}^{k}\psi\right)^{2}+\left(\partial_{r}^{k+1}\psi\right)^{2}+\left|\nabb\partial_{r}^{k}\psi\right|^{2}}\\
+&\displaystyle\int_{\mathcal{H}^{+}}{\left(\partial_{v}\partial_{r}^{k}\psi\right)^{2}+\chi_{k}\left|\nabb\partial_{r}^{k}\psi\right|^{2}}\\
+&\displaystyle\int_{\mathcal{A}}{\left(\partial_{v}\partial_{r}^{k}\psi\right)^{2}+\sqrt{D}\left(\partial_{r}^{k+1}\psi\right)^{2}+\left|\nabb\partial_{r}^{k}\psi\right|^{2}}\\
\leq  &C\sum_{i=0}^{k}\displaystyle\int_{\Sigma_{0}}{J_{\mu}^{n_{\Sigma_{0}}}\left(T^{i}\psi\right)n^{\mu}_{\Sigma_{0}}}+C\sum_{i=1}^{k}\int_{\Sigma_{0}\cap\mathcal{A}}{J_{\mu}^{n_{\Sigma_{0}}}\left(\partial^{i}_{r}\psi\right)n^{\mu}_{\Sigma_{0}}},
\end{split}
\end{equation*}
where $\chi_{k}=1$ if $\psi_{k}=0$ and    $\chi_{k}=0$ otherwise.
\label{theorem3}
\end{mytheo}

\begin{mytheo}(\textbf{Energy Decay})
Consider the foliation $\tilde{\Sigma}_{\tau}$ as defined in Section \ref{sec:EnergyDecay}. Let
\begin{equation*}
\begin{split}
I^{T}_{\tilde{\Sigma}_{\tau}}(\psi)=&\int_{\tilde{\Sigma}_{\tau}}{J^{n_{\tilde{\Sigma}_{\tau}}}_{\mu}(\psi)n^{\mu}_{\tilde{\Sigma}_{\tau}}}+
\int_{\tilde{\Sigma}_{\tau}}{J^{T}_{\mu}(T\psi)n^{\mu}_{\tilde{\Sigma}_{\tau}}}+\int_{\tilde{\Sigma}_{\tau}}{r^{-1}\left(\partial_{v}\phi\right)^{2}}
\end{split}
\end{equation*}
and
\begin{equation*}
\begin{split}
I^{n_{\tilde{\Sigma}_{\tau}}}_{\tilde{\Sigma}_{\tau}}(\psi)=&\int_{\tilde{\Sigma}_{\tau}}{J^{n_{\tilde{\Sigma}_{\tau}}}_{\mu}(\psi)n^{\mu}_{\tilde{\Sigma}_{\tau}}}+
\int_{\tilde{\Sigma}_{\tau}}{J^{n_{\tilde{\Sigma}_{\tau}}}_{\mu}(T\psi)n^{\mu}_{\tilde{\Sigma}_{\tau}}}\\
&+\int_{\mathcal{A}\cap\tilde{\Sigma}_{\tau}}{J^{n_{\tilde{\Sigma}_{\tau}}}_{\mu}(\partial_{r}\psi)n^{\mu}_{\tilde{\Sigma}_{\tau}}}+\int_{\tilde{\Sigma}_{\tau}}{r^{-1}\left(\partial_{v}(r\psi)\right)^{2}},
\end{split}
\end{equation*}
where $\mathcal{A}$ is as defined in Theorem \ref{theorem3}.  Here $\partial_{v}$ corresponds to the null system $(u,v)$.
There exists a constant $C$ that depends  on the mass $M$  and $\tilde{\Sigma}_{0}$  such that:
\begin{itemize}
	\item For all solutions $\psi$ of the wave equation we have 
\begin{equation*}
\displaystyle\int_{\tilde{\Sigma}_{\tau}}J^{T}_{\mu}(\psi)n_{\tilde{\Sigma}_{\tau}}^{\mu}\leq CE_{1}\frac{1}{\tau^{2}},
\end{equation*}
where
\begin{equation*}
E_{1}(\psi)=I^{T}_{\tilde{\Sigma}_{0}}(T\psi)+\displaystyle\int_{\tilde{\Sigma}_{0}}{J^{n_{\tilde{\Sigma}_{0}}}_{\mu}(\psi)n^{\mu}_{\tilde{\Sigma}_{0}}}+\displaystyle\int_{\tilde{\Sigma}_{0}}{\left(\partial_{v}(r\psi)\right)^{2}}.
\end{equation*}
\end{itemize}
\begin{itemize}
	\item For all solutions $\psi$ to the wave equation which are supported on the frequencies $l\geq 1$ we have
\begin{equation*}
\begin{split}
\displaystyle\int_{\tilde{\Sigma}_{\tau}}{J^{n_{\tilde{\Sigma}_{\tau}}}_{\mu}(\psi)n^{\mu}_{\tilde{\Sigma}_{\tau}}}\,\leq\, CE_{2}\frac{1}{\tau},
\end{split}
\end{equation*}
where
\begin{equation*}
\begin{split}
E_{2}(\psi)=I^{n_{\tilde{\Sigma}_{0}}}_{\tilde{\Sigma}_{0}}(\psi).
\end{split}
\end{equation*}
\end{itemize}
\begin{itemize}
	\item  For all solutions $\psi$ to the wave equation which are supported on the frequencies $l\geq 2$ we have
\begin{equation*}
\begin{split}
\displaystyle\int_{\tilde{\Sigma}_{\tau}}{J^{n_{\tilde{\Sigma}_{\tau}}}_{\mu}(\psi)n^{\mu}_{\tilde{\Sigma}_{\tau}}}\,\leq\, CE_{3}\frac{1}{\tau^{2}},
\end{split}
\end{equation*}
where
\begin{equation*}
\begin{split}
E_{3}(\psi)=I^{n_{\tilde{\Sigma}_{0}}}_{\tilde{\Sigma}_{0}}(\psi)+I^{n_{\tilde{\Sigma}_{0}}}_{\tilde{\Sigma}_{0}}(T\psi)+\int_{\mathcal{A}\cap\tilde{\Sigma}_{0}}{J^{n_{\tilde{\Sigma}_{0}}}_{\mu}(\partial_{r}\partial_{r}\psi)n^{\mu}_{\tilde{\Sigma}_{0}}}+\int_{\tilde{\Sigma}_{0}}{(\partial_{v}(r\psi))^{2}}.
\end{split}
\end{equation*}
\end{itemize}
\label{t4}
\end{mytheo}

\begin{mytheo}(\textbf{Pointwise Boundedness})
There exists a constant $C$ which depends on $M$ and $\Sigma_{0}$ such that for all solutions $\psi$ of the wave equation we have
\begin{equation*}
\begin{split}
\left|\psi\right|\leq C\cdot \sqrt{E_{4}},
\end{split}
\end{equation*}
everywhere in $\mathcal{R}$, where
\begin{equation*}
E_{4}=\sum_{\left|k\right|\leq 2}{\int_{\Sigma_{0}}{J_{\mu}^{n_{\Sigma_{0}}}(\Omega^{k}\psi)n^{\mu}_{\Sigma_{0}}}}.
\end{equation*}
Alternatively, we have
\begin{equation*}
\left|\psi\right|\leq C\sqrt{\tilde{E_{4}}},
\end{equation*}
everywhere in $\mathcal{R}$,
where
\begin{equation*}
\tilde{E_{4}}=\!\!\int_{\Sigma_{0}}{J_{\mu}^{n_{{\Sigma}_{0}}}(\psi)n^{\mu}_{\Sigma_{0}}}+C\!\!\displaystyle\int_{\Sigma_{0}}{J_{\mu}^{n_{{\Sigma}_{0}}}(n_{{\Sigma}_{0}}\psi)n^{\mu}_{\Sigma_{0}}}.
\end{equation*}
\label{t5}
\end{mytheo}

\begin{mytheo}(\textbf{Pointwise Decay})
Fix $R_{0}$ such that $M<R_{0}$ and let $\tau\geq 1$. Let $E_{1}, E_{2}, E_{3}$ be the quantities as defined in Theorem \ref{t4}. Then, there exists a constant $C$ that depends  on the mass $M$, $R_{0}$  and $\tilde{\Sigma}_{0}$  such that:
\begin{itemize}
	\item  For all solutions $\psi$ to the wave equation  we have
	\begin{equation*}
\left|\psi\right|\leq C \sqrt{E_{5}}\frac{1}{\sqrt{r}\cdot \tau},\ \  \left|\psi\right|\leq C\sqrt{E_{5}}\frac{1}{r\cdot\sqrt{\tau}}
\end{equation*}
in $\left\{R_{0}\leq r\right\}$, where
\begin{equation*}
E_{5}=\sum_{\left|k\right|\leq 2}{E_{1}\left(\Omega^{k}\psi\right)}.
\end{equation*}
\end{itemize}
\begin{itemize}
	\item For all  solutions $\psi$ of the wave equation we have 
\begin{equation*}
\begin{split}
\left|\psi\right|\leq C\sqrt{E_{6}}\frac{1}{\tau^{\frac{3}{5}}}
\end{split}
\end{equation*}
in $\left\{M\leq r\leq R_{0}\right\}$, where 
\begin{equation*}
E_{6}=E_{1}+E_{4}(\psi)+E_{4}(T\psi)+E_{5}+\!\sum_{\left|k\right|\leq 2}\!\!\left(E_{2}{\left(\Omega^{k}\psi\right)}\!+\!E_{3}{\left(\Omega^{k}\psi\right)}\right)+\left\|\partial_{r}\psi\right\|_{L^{\infty}\left(\tilde{\Sigma}_{0}\right)}^{2}\!. 
\end{equation*}
\end{itemize}
\begin{itemize}
	\item For all solutions $\psi$ to the wave equation which are supported on the frequencies $l\geq 1$ we have
\begin{equation*}
\begin{split}
\left|\psi\right|\leq C\sqrt{E_{7}}\frac{1}{\tau^{\frac{3}{4}}}
\end{split}
\end{equation*}
in $\left\{M\leq r\leq R_{0}\right\}$, where 
\begin{equation*}
E_{7}=E_{5}+\sum_{\left|k\right|\leq 2}E_{2}{\left(\Omega^{k}\psi\right)}+\sum_{\left|k\right|\leq 2}E_{3}{\left(\Omega^{k}\psi\right)}.
\end{equation*}
\end{itemize}
\begin{itemize}
\item For all solutions $\psi$ to the wave equation which are supported on the frequencies $l\geq 2$ we have
\begin{equation*}
\begin{split}
\left|\psi\right|\leq C\sqrt{E_{8}}\frac{1}{\tau},
\end{split}
\end{equation*}
in $\left\{M\leq r\leq R_{0}\right\}$, where 
\begin{equation*}
\begin{split}
E_{8}=\sum_{\left|k\right|\leq 2}E_{3}{\left(\Omega^{k}\psi\right)}.
\end{split}
\end{equation*}
\end{itemize}
\label{t6}
\end{mytheo}

\begin{mytheo}(\textbf{Higher Order  Estimates I: Decay})
Fix $R_{1}$ such that $R_{1}>M$ and let $\tau\geq 1$. Let also $k,l\in\mathbb{N}$. Then there exists a constant $C$ which depend on $M,l, R_{1}$ and $\tilde{\Sigma}_{0}$ such that the following holds: For all solutions $\psi$ of the wave equation which are supported on the  angular frequencies greater or equal to $l$, there exist norms $\tilde{E}_{k,l}, E_{k,l}$ of the initial data of $\psi$ such that
\begin{itemize}
	\item $\displaystyle\int_{\tilde{\Sigma}_{\tau}\cap\left\{M\leq r\leq R_{1}\right\}}{J_{\mu}^{N}(\partial_{r}^{k}\psi)n_{\tilde{\Sigma}_{\tau}}^{\mu}}\leq C\tilde{E}_{k,l}^{2}\frac{1}{\tau^{2}}$ for all $k\leq l-2$,
	\item $\displaystyle\int_{\tilde{\Sigma}_{\tau}\cap\left\{M\leq r\leq R_{1}\right\}}{J_{\mu}^{N}(\partial_{r}^{l-1}\psi)n_{\tilde{\Sigma}_{\tau}}^{\mu}}\leq C\tilde{E}_{l-1,l}^{2}\frac{1}{\tau}$,
	
	\item $\displaystyle\left|\partial_{r}^{k}\psi\right|\leq CE_{k,l}\displaystyle\frac{1}{\tau}$ in $\left\{M\leq r\leq R_{1}\right\}$ for all $k\leq l-2$,
	\item $\displaystyle\left|\partial_{r}^{l-1}\psi\right|\leq CE_{l-1,l}\displaystyle\frac{1}{\tau^{\frac{3}{4}}}$ in $\left\{M\leq r\leq R_{1}\right\}$,
		\item $\displaystyle\left|\partial_{r}^{l}\psi\right|\leq CE_{l,l}\displaystyle\frac{1}{\tau^{\frac{1}{4}}}$ in $\left\{M\leq r\leq R_{1}\right\}$.
\end{itemize}
\label{theo8}
\end{mytheo}

\begin{mytheo}(\textbf{Higher Order  Estimates II: Non-Decay and Blow-up})

Fix $R_{1}$ such that $R_{1}>M$.  Let also $k,l\in\mathbb{N}$. Then there exists a positive constant $c$ which depends only on $M,l,k$ such that for all  solutions $\psi$ to the wave equation which are supported on the frequency $l$ we have 
\begin{itemize}
	\item  
$\partial_{r}^{l+1}\psi(\tau, \theta, \phi)\rightarrow H_{l}[\psi](\theta, \phi)
$
along $\hh$ and generically $H_{l}[\psi](\theta,\phi)\neq 0$ almost everywhere on $\mathbb{S}^{2}_{0}$.

\item
$
\left|\partial_{r}^{l+k}\psi\right|(\tau,\theta,\phi)\geq c \left|H_{l}[\psi](\theta,\phi)\right|\tau^{k-1}
$
asymptotically on $\mathcal{H}^{+}$, for $k\geq 2$.

\item
$
\displaystyle\int_{\tilde{\Sigma}_{\tau}}{J_{\mu}^{N}(\partial_{r}^{k}\psi)n_{\tilde{\Sigma}_{\tau}}^{\mu}}\longrightarrow +\infty 
$
as $\tau\rightarrow +\infty$ for all $k\geq l+1$.

\end{itemize}
\label{theo9}
\end{mytheo}

\section{The Vector Field Method}
\label{sec:TheVectorFieldMethod}

For understanding the evolution of waves we will use the so-called vector field method. This is a geometric and robust method and involves mainly $L^{2}$ estimates. The main idea is to construct appropriate (0,1) currents and use Stokes' theorem (see Appendix \ref{sec:StokesTheoremOnLorentzianManifolds}) in appropriate regions. For a nice recent exposition see \cite{ali}.

Given a (0,1) current $P_{\mu}$ we have the continuity equation
\begin{equation}
\int_{\Sigma_{0}}{P_{\mu}n^{\mu}_{\Sigma_{0}}}=\int_{\Sigma_{\tau}}{P_{\mu}n^{\mu}_{\Sigma_{\tau}}}+\int_{\mathcal{H}^{+}\left(0,\tau\right)}{P_{\mu}n^{\mu}_{\mathcal{H}^{+}}}+\int_{\mathcal{R}\left(0,\tau\right)}{\nabla^{\mu}P_{\mu}}
\label{div}
\end{equation}
where all the integrals are with respect to the \textit{induced volume form} and the unit normals $n_{\scriptstyle\Sigma_{\tau}}$ are future directed. The normal to the horizon $\mathcal{H}^{+}$ is also future directed and can be chosen arbitrarily. Then the corresponding volume form on $\mathcal{H}^{+}$ is defined such that \eqref{div} holds. For definiteness, from now on we  consider that $n_{\mathcal{H}^{+}}=\partial_{v}=T$. Furthermore, all the integrals  are to be considered  with respect to the induced volume form and thus we omit writing the measure.

\subsection{The Compatible Currents $J$, $K$ and the Current $\mathcal{E}$}
\label{sec:TheCurrentsJKAndMathcalE}

We usually consider currents $P_{\mu}$ that depend on the geometry of $\left(\mathcal{M},g\right)$ and are such that both $P_{\mu}$ and $\nabla^{\mu}P_{\mu}$ depend only on the 1-jet of $\psi$. This can be achieved by using the wave equation to make all second order derivatives disappear and end up with something that highly depends on the geometry of the spacetime.  There is a general method for producing such currents using the energy momentum tensor \textbf{T}. Indeed, the Lagrangian structure of the wave equation gives us the following energy momentum tensor 
\begin{equation}
\textbf{T}_{\mu\nu}\left(\psi\right)=\partial_{\mu}\psi\partial_{\nu}\psi-\frac{1}{2}g_{\mu\nu}\partial^{a}\psi\partial_{a}\psi,
\label{tem}
\end{equation}
which is a symmetric divergence free (0,2) tensor. We will in fact consider this tensor for general functions $\psi:\mathcal{M}\rightarrow\mathbb{R}$ in which case we have the identity
\begin{equation}
Div\textbf{T}\left(\psi\right)=\left(\Box_{g}\psi\right)d\psi.
\label{divT}
\end{equation}
Since  \textbf{T} is a $\left(0,2\right)$ tensor we need to  contract it with vector fields of $\mathcal{M}$. It is here where the geometry of $\mathcal{M}$ makes its appearance. We have the following definition
\begin{definition}
Given a vector field $V$  we define the $J^{V}$ current by
\begin{equation}
J^{V}_{\mu}(\psi)=\textbf{T}_{\mu\nu}(\psi)V^{\nu}.
\label{jcurrent}
\end{equation}
\end{definition}
We say that we use the vector field $V$ as a \textit{multiplier}\footnote{the name comes from the fact that the tensor \textbf{T} is multiplied by \textit{V}.} if we apply \eqref{div} for the current $J^{V}_{\mu}$. The divergence of this current is
\begin{equation}
\operatorname{Div}(J)=\operatorname{Div}\left(\textbf{T}V\right)=\operatorname{Div}\left(\textbf{T}\right)V+\textbf{T}\left(\nabla V\right),
\label{divtv}
\end{equation}
where $\left(\nabla V\right)^{ij}=\left(g^{ki}\nabla _{k}V\right)^{j}=\left(\nabla ^{i}V\right)^{j}$. If $\psi$ is a solution of the wave equation then $\operatorname{Div}\left(\textbf{T}\right)=0$ and, therefore, $\nabla^{\mu}J_{\mu}^{V}$ is an expression of the 1-jet of $\psi$.
\begin{definition} 
Given a vector field $V$  the scalar  current $K^{V}$ is defined by
\begin{equation}
K^{V}\left(\psi\right)=\textbf{T}\left(\psi\right)\left(\nabla V\right)=\textbf{T}_{ij}\left(\psi\right)\left(\nabla ^{i}V\right)^{j}.
\label{K}
\end{equation}
\end{definition}
Note that from the symmetry of the energy momentum tensor $\textbf{T}$ we have
\begin{equation*}
K^{V}\left(\psi\right)=T_{\mu\nu}\left(\psi\right)\pi^{\mu\nu}_{V},
\end{equation*} 
where $\pi^{\mu\nu}_{V}=(\mathcal{L}_{V}g)^{\mu\nu}$ is the deformation tensor of $V$. Clearly if $\psi$ satisfies the wave equation then
\begin{equation*}
K^{V}(\psi)=\nabla^{\mu}J^{V}_{\mu}(\psi).
\end{equation*}
Thus if we use Killing vector fields as multiplier then the divergence vanishes and so we obtain a conservation law. This is partly the content of a deep theorem of Noether\footnote{According to this theorem, any continuous family of isometries gives rise to a conservation law.}.
The first term of the right hand side of \eqref{divtv} does not vanish when we commute the wave equation with a vector field that is not Killing (or, more generally, when $\psi$ does satisfy the wave equation). Therefore, we also have the following definition.
\begin{definition} 
Given a vector field $V$ we define the scalar  current $\mathcal{E}^{V}$ by
\begin{equation}
\mathcal{E}^{V}\left(\psi\right)=Div\left(\textbf{T}\right)V=\left(\Box_{g}\psi\right)d\psi\left(V\right)=\left(\Box_{g}\psi\right)V\left(\psi\right).
\label{E}
\end{equation}
\end{definition}

\subsection{The Hyperbolicity of the Wave Equation}
\label{sec:TheHyperbolicityOfTheWaveEquation}
Note that equation \eqref{tem} provides us with an energy momentum tensor for any pseudo-Riemannian manifold. However, Lorentzian manifolds admit timelike vectors and the hyperbolicity of the wave equation is captured by the following proposition
\begin{proposition}
Let $V_{1},V_{2}$ be two future directed timelike vectors. Then the quadratic expression $\textbf{T}\left(V_{1},V_{2}\right)$ is positive definite in $d\psi$. By continuity, if one of these vectors is null then $\textbf{T}\left(V_{1},V_{2}\right)$ is non-negative definite in $d\psi$.
\label{hyperb}
\end{proposition}
\begin{proof}
Consider a point $p\in\mathcal{M}$ and the normal coordinates around this point and that without loss of generality $V_{2}=(1,0,0,...,0)$. Then the proposition is an application of the Cauchy-Schwarz inequality.
\end{proof}
An important application of the vector field method and the hyperbolicity of the wave equation is the domain of dependence property. 

As we shall see, the exact dependence of $\textbf{T}$ on the derivatives of $\psi$ will be crucial later. For a general computation see Appendix \ref{sec:TheHyperbolicityOfTheWaveEquation1}.

\section{Hardy Inequalities}
\label{sec:HardyInequalities}

In this section we establish three Hardy inequalities which as we shall see will be very crucial for obtaining sharp estimates.  They do not require $\psi$ to satisfy the wave equation and so we only assume that $\psi$ satifies the regularity assumptions described in Section \ref{sec:TheCauchyProblemForTheWaveEquation} and \eqref{condition}. From now on,  $\Sigma_{\tau}$  is the foliation introduced  in Section \ref{sec:TheFoliationsSigmaTauAndTildeSigmaTau} (however, these inequalities hold also for the foliation $\tilde{\Sigma}_{\tau}$).  

\begin{proposition}(\textbf{First Hardy Inequality}) For all functions $\psi$ which satisfy the regularity assumptions of Section \ref{sec:TheCauchyProblemForTheWaveEquation}  we have
\begin{equation*}
\int_{\Sigma_{\tau}}{\frac{1}{r^{2}}\psi^{2}}\leq C\int_{\Sigma_{\tau}}{D[(\partial_{v}\psi)^{2}+(\partial_{r}\psi)^{2}]},
\end{equation*}
where the constant  $C$ depends only on $M$ and $\Sigma_{0}$.
\label{firsthardy}
\end{proposition}
\begin{proof}
Consider the induced coordinate system $(\rho,\omega)$ introduced in Section \ref{sec:TheFoliationsSigmaTauAndTildeSigmaTau}. We use the 1-dimensional identity
\begin{equation*}
\int_{M}^{+\infty}(\partial_{\rho}h)\psi^{2}d\rho=\left[h\psi^{2}\right]_{r=M}^{+\infty}-2\int_{M}^{+\infty}h\psi\partial_{\rho}\psi d\rho
\end{equation*}
with $h=r-M$. In view of the assumption on $\psi$, the boundary terms vanish. Thus, Cauchy-Schwarz gives
\begin{equation*}
\begin{split}
\int_{\left\{r\geq M\right\}}{\psi^{2}d\rho}=&-2\int_{\left\{r\geq M\right\}}{(r-M)\psi\partial_{\rho}\psi d\rho}\\
\leq &2\left(\int_{\left\{r\geq M\right\}}{\psi^{2}d\rho}\right)^{\frac{1}{2}}\left(\int_{\left\{r\geq M\right\}}{(r-M)^{2}(\partial_{\rho}\psi)^{2}}d\rho\right)^{\frac{1}{2}}.
\end{split}
\end{equation*}
Therefore,
\begin{equation*}
\begin{split}
\int_{\left\{r\geq M\right\}}{\psi^{2}d\rho}\leq 4\int_{\left\{r\geq M\right\}}{(r-M)^{2}(\partial_{\rho}\psi)^{2}d\rho}.
\end{split}
\end{equation*}
Integrating this inequality over $\mathbb{S}^{2}$ gives us
\begin{equation*}
\begin{split}
\int_{\mathbb{S}^{2}}\int_{\left\{r\geq M\right\}}{\frac{1}{\rho^{2}}\psi^{2}}\rho^{2}d\rho d\omega\leq 4\int_{\mathbb{S}^{2}}\int_{\left\{r\geq M\right\}}{D(\partial_{\rho}\psi)^{2}\rho^{2}d\rho d\omega}.
\end{split}
\end{equation*}
The result follows from the fact that each $\Sigma_{\tau}$ is diffeomorphic to $\mathbb{S}^{2}\times [M,\left.\right.\!\! +\infty)$ and the boundedness of the factor $V$ in the volume form \eqref{volumeformfoliation} of $\Sigma_{\tau}$.
\end{proof}
The importance of the above inequality lies on the weights. The weight of $(\partial_{\rho}\psi)^{2}$ vanishes to second order on $\mathcal{H}^{+}$  but does not degenerate at infinity\footnote{Note that if we replace $r-M$ with another function $g$ then we will  not be able to make this weight degenerate fast enough without obtaining non-trinial boundary terms.} whereas the weight of $\psi$ degenerate at infinity but not at $r=M$. Similarly, one might derive estimates for the non-extreme case\footnote{These estimates turn out to be stronger since the weight of $\psi$ may diverge at $r=M$ in an integrable manner.}. We also mention that the right hand side is bounded by the (conserved) flux of $T$ through $\Sigma_{\tau}$ (see Section \ref{sec:TheVectorFieldTextbfM}).

\begin{proposition}(\textbf{Second Hardy Inequality})  Let  $r_{0}\in (M,2M)$. Then for all functions $\psi$ which satisfy the regularity assumptions of Section \ref{sec:TheCauchyProblemForTheWaveEquation}  and any positive number $\epsilon$ we have
\begin{equation*}
\int_{\mathcal{H}^{+}\cap\Sigma_{\tau}}{\psi^{2}}\leq \epsilon\int_{\Sigma_{\tau}\cap\left\{r\leq r_{0}\right\}}{(\partial_{v}\psi)^{2}+(\partial_{r}\psi)^{2}}\ \,+C_{\epsilon}\int_{\Sigma_{\tau}\cap\left\{r\leq r_{0}\right\}}{\psi^{2}},
\end{equation*}
where the constant $C_{\epsilon}$ depends on $M$, $\epsilon$, $r_{0}$ and $\Sigma_{0}$. 
\label{secondhardy}
\end{proposition}
\begin{proof}
We use as before the 1-dimensional identity
\begin{equation*}
\int_{M}^{r_{0}}(\partial_{\rho}h)\psi^{2}d\rho=\left[h\psi^{2}\right]_{r=M}^{r=r_{0}}-2\int_{M}^{r_{0}}h\psi\partial_{\rho}\psi d\rho
\end{equation*}
with $h=r-r_{0}$. Then 
\begin{equation*}
\begin{split}
(r_{0}-M)\psi^{2}(M)=&\int_{M}^{r_{0}}{\psi^{2}+2h\psi\partial_{\rho}\psi d\rho}\\
\leq & \epsilon\int_{M}^{r_{0}}{(\partial_{\rho}\psi)^{2}}d\rho +\int_{M}^{r_{0}}{\left(1+\frac{g^{2}}{\epsilon}\right)\psi^{2}d\rho},
\end{split}
\end{equation*}
for any $\epsilon >0$. By integrating over $\mathbb{S}^{2}$ and noting that the $\rho^{2}$ factor that appears in the volume form of $\Sigma\cap\left\{r\leq r_{0}\right\}$ is bounded we obtain the required result.
\end{proof}
The previous two inequalities concern the hypersurfaces $\Sigma_{\tau}$ that cross $\mathcal{H}^{+}$. The next estimate concerns spacetime neighbourhoods of $\mathcal{H}^{+}$.

\begin{proposition} (\textbf{Third Hardy Inequality})  Let  $r_{0},r_{1}$ be such that  $M<r_{0}<r_{1}<2M$. We define the regions
\begin{equation*}
\begin{split}
&\mathcal{A}=\mathcal{R}(0,\tau)\cap\left\{M\leq r\leq r_{0}\right\},\\
&\mathcal{B}=\mathcal{R}(0,\tau)\cap\left\{r_{0}\leq r\leq r_{1}\right\}.
\end{split}
\end{equation*}
Then for all functions $\psi$ which satisfy the regularity assumptions of Section \ref{sec:TheCauchyProblemForTheWaveEquation}   we have
\begin{equation*}
\int_{\mathcal{A}}{\psi^{2}}\leq C\int_{\mathcal{B}}{\psi^{2}}\ \,+C\int_{\mathcal{A}\cup\mathcal{B}}{D[(\partial_{v}\psi)^{2}+(\partial_{r}\psi)^{2}]},
\end{equation*}
where the constant $C$ depends on $M$, $r_{0}$, $r_{1}$  and $\Sigma_{0}$.
\label{thirdhardy}
\end{proposition}
\begin{proof}
 \begin{figure}[H]
	\centering
		\includegraphics[scale=0.14]{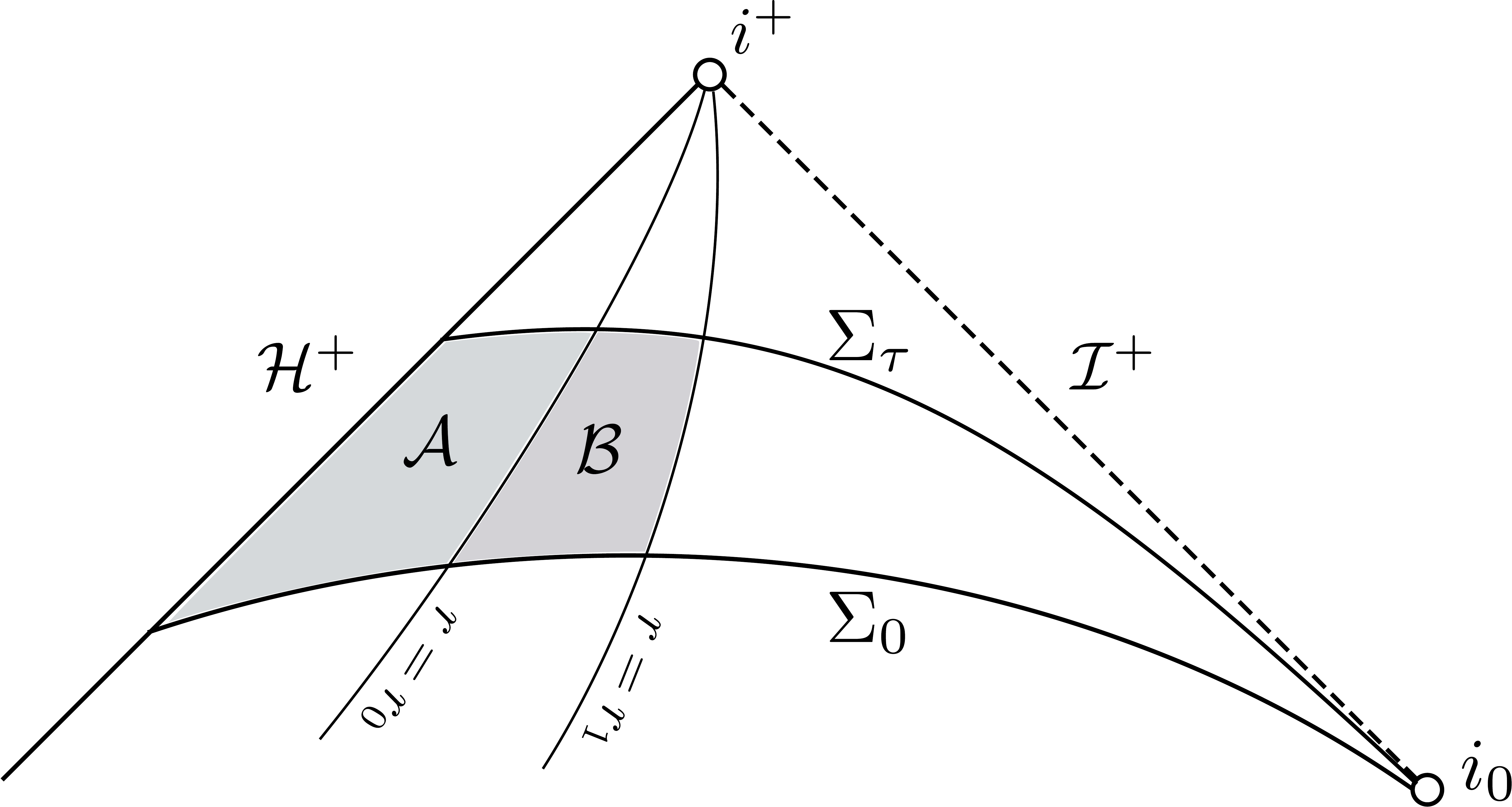}
	\label{fig:ernt4}
\end{figure}
We again use  the 1-dimensional identity
\begin{equation*}
\int_{M}^{r_{1}}(\partial_{\rho}h)\psi^{2}d\rho=\left[h\psi^{2}\right]_{r=M}^{r=r_{1}}-2\int_{M}^{r_{1}}h\psi\partial_{\rho}\psi d\rho
\end{equation*}
with $h$ such that $h=2(\rho-M)$ in  $[M,r_{0}]$ and $h(r_{1})=0$.   Then
\begin{equation*}
\begin{split}
2\int_{M}^{r_{0}}{\psi^{2}d\rho}\leq \int_{M}^{r_{0}}{\psi^{2}d\rho}+\int_{r_{0}}^{r_{1}}{\left(1-(\partial_{\rho}g)\right)\psi^{2} d\rho}+\int_{M}^{r_{1}}{g^{2}(\partial_{\rho}\psi)^{2} d\rho}.
\end{split}
\end{equation*}Thus, integrating over $\mathbb{S}^{2}$ we obtain
\begin{equation*}
\begin{split}
\int_{\Sigma_{\tilde{\tau}}\cap\mathcal{A}}{\psi^{2}}\leq C \int_{\Sigma_{\tilde{\tau}}\cap\left(\mathcal{A}\cup\mathcal{B}\right)}{D[(\partial_{v}\psi)^{2}+(\partial_{r}\psi)^{2}]}+C\int_{\Sigma_{\tilde{\tau}}\cap\mathcal{B}}{\psi^{2}},
\end{split}
\end{equation*}
for all $\tilde{\tau}\geq 0$, where the constant $C$ depends on $M$, $r_{0}$, $r_{1}$ and $\Sigma_{0}$. Therefore, integrating over $\tilde{\tau}\in [0,\tau]$ and using coarea formula
\begin{equation*}
\begin{split}
\int_{0}^{\tau}\left(\int_{\Sigma_{\tilde{\tau}}}{f}\right)d\tilde{\tau}\ \sim\,\int_{\bigcup_{0\leq\tilde{\tau}\leq\tau}\Sigma_{\tilde{\tau}}}{f}
\end{split}
\end{equation*}
completes the proof of the proposition.

\end{proof}

\section{Elliptic Theory on $\mathbb{S}^{2}(r), r>0$}
\label{sec:EllipticTheoryOnMathbbS2}
																									
In view of the symmetries of the spacetime, it is important to understand the behaviour of functions on the orbits of the action of SO(3). Clearly, these orbits are isometric to $\mathbb{S}^{2}(r)$ for  $r>0$. The most important operator on $\mathbb{S}^{2}(r)$ for our applications is the spherical Laplacian (i.e.~the Laplacian induced by the standard metric on $\mathbb{S}^{2}(r)$) which will be denoted by $\lapp$. In local coordinates it is given by
\begin{equation*}
\lapp\psi=\frac{1}{r^{2}\sin\theta}{\left[\partial_{\theta}(\sin\theta\partial_{\phi}\psi)+\frac{1}{\sin\theta}\partial_{\phi}\partial_{\phi}\psi\right]}.
\end{equation*}
Clearly, $\lapp\psi$ is a linear second order differential operator 
\begin{equation*}
\lapp:H^{2}(\mathbb{S}^{2}(r))\rightarrow L^{2}(\mathbb{S}^{2}(r)).
\end{equation*}
For the definition of the Sobolev space $H^{2}$ see Appendix \ref{sec:SobolevSpaces}. However, $\mathbb{S}^{2}(r)$ is compact without boundary  and therefore any constant map on the sphere lies in the kernel of $\lapp$ and thus $\lapp$ is not invertible. Observing now that 
\begin{equation}
\int_{\mathbb{S}^{2}(r)}{(-\lapp+I)\psi\cdot \psi}=\int_{\mathbb{S}^{2}(r)}{\left|\nabb\psi\right|^{2}}+\int_{\mathbb{S}^{2}(r)}{\psi^{2}},
\label{s2impo}
\end{equation}
where $\nabb$ is the induced gradient on the sphere and $I$ the identity map, we see that  (-\lapp+I) is injective. It is not difficult to see that \eqref{s2impo} implies that $O=(-\lapp+I)$ is surjective. Thus we have a well-defined inverse
\begin{equation*}
O^{-1}:L^{2}(\mathbb{S}^{2}(r))\rightarrow H^{2}(\mathbb{S}^{2}(r)).
\end{equation*}
Using Stokes' theorem on $\mathbb{S}^{2}(r)$ it is easy to see that $O^{-1}$ is self-adjoint.  Moreover, by Rellich theorem, the image $H^{2}(\mathbb{S}^{2}(r))$ of $O^{-1}$ is compactly imbedded in its domain $L^{2}(\mathbb{S}^{2}(r))$. Therefore, $O^{-1}$ is also compact and thus if we view $T$ as an endomorphism of $L^{2}(\mathbb{S}^{2}(r))$ then we can apply the spectral theorem and so we have a complete orthonormal basis of $L^{2}(\mathbb{S}^{2}(r))$ of eigenvectors of $O^{-1}$ with discrete eigenvalues $\mu_{i}\searrow 0$. Hence the numbers $-\frac{1}{\mu_{i}}+1$ are eigenvalues of $\lapp$. By separating the angular variables $\theta, \phi$ we can compute explicitly these eigenvalues. It turns out that the eigenvalues are equal to $\frac{-l(l+1)}{r^{2}}, l\in\mathbb{N}$. The dimension of the eigenspaces $E^{l}$ is equal to $2l+1$ 
and the corresponding eigenvectors are denoted by $Y^{m,l}, -l\leq m\leq l$ and called \textit{spherical harmonics}. We have $L^{2}(\mathbb{S}^{2}(r))=\displaystyle\oplus_{l\geq 0}E^{l}$ and, therefore, any function $\psi\in L^{2}(\mathbb{S}^{2}(r))$ can be written as
\begin{equation}
\psi=\sum_{l=0}^{\infty}\sum_{m=-l}^{l}\psi_{m,l}Y^{m,l}.
\label{sd}
\end{equation}
The right hand side converges to $\psi$ in  $L^{2}$ of the sphere and under stronger regularity assumptions the convergence is pointwise.  Let us denote by $\psi_{l}$ the projection of $\psi$ onto  $E^{l}$, i.e. 
\begin{equation*}
\psi_{l}=\sum_{m=-l}^{l}\psi_{m,l}Y^{m,l}.
\end{equation*}
\begin{proposition}(\textbf{Poincar\'e inequality}) If $\psi\in L^{2}\left(\mathbb{S}^{2}(r)\right)$ and $\psi_{l}=0$ for all $l\leq L-1$ for some finite natural number $L$ then we have
\begin{equation*}
\frac{L\left(L+1\right)}{r^{2}}\int_{\mathbb{S}^{2}(r)}{\psi^{2}}\leq \int_{\mathbb{S}^{2}(r)}\left|\nabb\psi\right|^{2}
\label{poincare}
\end{equation*}
and equality holds if and only if $\psi_{l}=0$ for all $l\neq L$.
\label{poin}
\end{proposition}
\begin{proof} 
We have
\begin{equation*}
\psi=\sum_{l\geq L}{\psi_{l}}.
\end{equation*}
Therefore, 
\begin{equation*}
\begin{split}
\int_{\mathbb{S}^{2}(r)}{\left|\nabb\psi\right|^{2}}&=-\displaystyle\int_{\mathbb{S}^{2}(r)}{\psi\cdot\lapp\psi}=\int_{\mathbb{S}^{2}(r)}{\left(\sum_{l'\geq L}{\psi_{l'}}\right)\left(\sum_{l\geq L}{\frac{l\left(l+1\right)}{r^{2}}\psi_{l}}\right)}\\
&=\int_{\mathbb{S}^{2}(r)}{\sum_{l\geq L}{\frac{l\left(l+1\right)}{r^{2}}\psi_{l}^{2}}}\\
&\geq\frac{L\left(L+1\right)}{r^{2}}\int_{\mathbb{S}^{2}(r)}{\sum_{l\geq L}{\psi_{l}^{2}}}=\frac{L\left(L+1\right)}{r^{2}}\int_{\mathbb{S}^{2}(r)}{\psi^{2}},
\end{split}
\end{equation*}
where we have used the orthogonality of distinct eigenspaces.
\end{proof}
This inequality is very useful, since it gives us an estimate of the zeroth order term. However, it is rather restrictive since it requires $\psi$ not to have highly symmetric components in its spherical decomposition. This should  be contrasted with the Hardy inequalities where no such restriction is present. However, as we shall see in Section \ref{sec:CommutingWithAVectorFieldTransversalToMathcalH}, the Poincar\'e inequality is very crucial in establishing sharp results. In particular, we will see that the behaviour of the spherical decomposition determines the evolution of waves close to $\mathcal{H}^{+}$.

 Note also that  each eigenspace $E^{l}$ is finite dimensional (and so complete) and so closed. Therefore, $L^{2}(\mathbb{S}^{2}(r))=E^{l}\oplus \left(E^{l}\right)^{\bot}$. For example, we have that $\psi$ can be written uniquely as
\begin{equation*}
\psi=\psi_{0}+\psi_{\geq 1},
\end{equation*}
where $\psi_{\geq 1}\in\left(E^{l}\right)^{\bot}=\displaystyle\oplus_{l\geq 1}E^{l}$.

Returning to our 4-dimensional problem, if $\psi$ is  sufficiently regular in $\mathcal{R}$ then its restriction at each sphere can be written as in \eqref{sd}, where $\psi_{m,l}=\psi_{m,l}(v,r)$ and $Y^{m,l}=Y^{m,l}(\theta,\phi)$. We will not worry about the convergence of the series since we may assume that $\psi$ is sufficiently regular\footnote{Indeed, we may work only with functions which are smooth and such that  the non-zero terms in \eqref{sd} are finitely many. Then, since all of our results are quantitative and all the constants involved do not depend on $\psi$, by a density argument we may lower the regularity of $\psi$ requiring only certain norms depending on the initial data of $\psi$ to be finite.}. First observe that for each summand in \eqref{sd} we have
\begin{equation*}
\Box_{g}\psi_{m,l}Y^{m,l}=\left(S\psi_{m,l}-\frac{l(l+1)}{r^{2}}\psi_{m,l}\right)Y^{m,l},
\end{equation*}
where $S$ is an operator on the quotient $\mathcal{M}/\text{SO(3)}$. Therefore, if $\psi$ satisfies the wave equation then in view of the linear independence of $Y^{m,l}$'s the terms $\psi_{m,l}$ satisfy $S\psi_{m,l}=\frac{l(l+1)}{r^{2}}\psi_{m,l}$ and, therefore, each summand also satisfies the wave equation. This implies that if $\psi_{m,l}=0$ initially then $\psi_{m,l}=0$ everywhere. From now on, \textbf{we will say that the wave $\psi$ is supported on the angular frequencies  $l\geq L$ if $\psi_{i}=0,i=0,...,L-1$ initially (and thus everywhere). Similarly, we will also say that $\psi$ is supported on the angular frequency  $l=L$ if $\psi\in E^{L}$. }

Another important observation is that the operators $\partial_{v}$ and $\partial_{r}$ are endomorphisms of $E^{l}$ for all $l$. Indeed, $\partial_{v}$ commutes with $\lapp$ and if $\psi\in E^{l}$ then
\begin{equation*}
\begin{split}
\lapp\partial_{r}\psi &=\partial_{r}\lapp\psi+\frac{2}{r}\lapp\psi=-\frac{l(l+1)}{r^{2}}\partial_{r}\psi
\end{split}
\end{equation*}
and thus $\partial_{r}\psi\in E^{l}$. Therefore, if $\psi$ is supported on frequencies $l\geq L$ then the same holds for $\partial_{r}\psi$.

\section{The Vector Field $T$}
\label{sec:TheVectorFieldTextbfM}
One can easily see that $\partial_{v}=\partial_{t}=\partial_{t^{*}}$ in the intersection of the corresponding coordinate systems. Here $\partial_{t}$ corresponds to the coordinate basis vector field of either  $\left(t,r\right)$ or $\left(t,r^{*}\right)$. Recall that the region $\mathcal{M}$ where we want to understand the behaviour of waves is covered by the system $\left(v,r,\theta,\phi\right)$. Therefore we define $T=\partial_{v}$. It can be easily seen from \eqref{RN1} that $T$ is Killing and timelike everywhere\footnote{Note that in the subextreme range $T$ becomes spacelike in the region bounded by the two horizons, which however coincide in the extreme case.} except on the horizon where it is null.
\subsection{Uniform Boundedness of Degenerate Energy}
\label{sec:TheCurrentKT}
Recall that 
\begin{equation*}
K^{T}=T_{\mu\nu}\pi_{T}^{\mu\nu},
\end{equation*}
where $\pi_{T}^{\mu\nu}$ is the deformation tensor of $T$. Since $T$ is Killing, its deformation tensor is zero and so 
\begin{equation}
K^{T}=0.
\label{KT}
\end{equation}
Therefore, the 
divergence identity in the region $\mathcal{R}(0,\tau)$ gives us the following conservation law
\begin{equation}
\int_{\Sigma_{\tau}}{J^{T}_{\mu}n^{\mu}_{\Sigma_{\tau}}}+\int_{\mathcal{H}^{+}}{J^{T}_{\mu}n^{\mu}_{{\mathcal{H}^{+}}}}=\int_{\Sigma_{0}}{J^{T}_{\mu}n^{\mu}_{\Sigma_{0}}}.
\label{t1}
\end{equation}
Since $T$ is null on the horizon $\mathcal{H}^{+}$ we have that $J_{\mu}^{T}n^{\mu}_{{\mathcal{H}^{+}}}\geq 0$. More presicely, since $n^{\mu}_{{\mathcal{H}^{+}}}=T$, from \eqref{GENERALT} (See Appendix \ref{sec:TheHyperbolicityOfTheWaveEquation1}) we have
\begin{equation*}
J_{\mu}^{T}n^{\mu}_{{\mathcal{H}^{+}}}=(\partial_{v}\psi)^{2},
\end{equation*}
thus proving the following proposition:
\begin{proposition}
For all solutions $\psi$ of the wave equation we have
\begin{equation}
\int_{\Sigma_{\tau}}{J^{T}_{\mu}(\psi)n^{\mu}_{\Sigma_{\tau}}}\leq \int_{\Sigma_{0}}{J^{T}_{\mu}(\psi)n^{\mu}_{\Sigma_{0}}}.
\label{boundT}
\end{equation}
\label{ubdenergy}
\end{proposition}

We know by Proposition \ref{hyperb} that $J^{T}_{\mu}n^{\mu}_{\Sigma_{\tau}}$ is non-negative definite. However, we need to know the exact way $J^{T}_{\mu}n^{\mu}_{\Sigma_{\tau}}$ depends on $d\psi$.  Note  that   $\omega_{n}$ (see Appendix \ref{sec:TheHyperbolicityOfTheWaveEquation1}) is strictly positive and  uniformly bounded. However
\begin{equation*}
\omega_{T}=-\frac{1}{2}g(T,T)=\frac{1}{2}D,
\end{equation*}
which clearly vanishes (to second order) at the horizon. Therefore, \eqref{GENERALT} implies 
\begin{equation*}
\begin{split}
J^{T}_{\mu}n^{\mu}\, \sim\ \left(\partial_{v}\psi\right)^{2}+D\left(\partial_{r}\psi\right)^{2}+\left|\nabb\psi\right|^{2},
\end{split}
\end{equation*}
where the constants in $\sim$ depend on the mass $M$ and $\Sigma_{0}$. Note that all these relations are invariant under the flow of $T$.

On $\mathcal{H}^{+}$,  the energy estimate \eqref{boundT} degenerates with respect to the transversal derivative $\partial_{r}\psi$. It is exactly this that does not allow us to use estimate $\eqref{boundT}$ to obtain the boundedness result for the waves in the whole region $\mathcal{R}$. However, if we restrict our attention to the region where $r\geq r_{0}>M$ (i.e.~away from the horizon) then commutating the wave equation with $T$ and estimate $\eqref{boundT}$ in conjunction with elliptic and Sobolev estimates  give us the boundedness of $\psi$ in this region. This result is not satisfactory since it provides no information about the behaviour of waves on the horizon (which is the boundary of the black hole) and so it is not sufficient for non-linear stability problems.

\section{Integrated Weighted Energy Decay}
\label{sec:TheVectorFieldTextbfX}

It turns out that even proving uniform boundedness of $\psi$ requires a strong result such as integrated decay. The first result we prove in this direction is that there exists a constant $C$ which depends on $M$ and $\Sigma_{0}$ such that
\begin{equation*}
\int_{\mathcal{R}}\chi\cdot \left(J_{\mu}^{T}(\psi)n^{\mu}_{\Sigma}+\psi^{2}\right) \leq C\int_{\Sigma_{0}}J_{\mu}^{T}(\psi)n^{\mu}_{\Sigma_{0}},
\end{equation*}
where the weight $\chi=\chi (r)$ degenerates only at $\mathcal{H}^{+}$, at $r=2M$ (photon sphere) and at infinity. The degeneracy at the photon sphere is expected in view of  trapping. In particular, as we shall see, such an estimate degenerates only with respect to the derivatives tangential to the photon sphere. This degeneracy can be overcome at the expense of commuting with $T$. The degeneracy at $\mathcal{H}^{+}$ is due to the lack of redshift and can be overcome by commuting with (the non Killing) vector field $\partial_{r}$ and by imposing conditions on the spherical decomposition of $\psi$ (see Section \ref{sec:CommutingWithAVectorFieldTransversalToMathcalH}). In Section \ref{sec:ConservationLawsOnDegenerateEventHorizons} we will see that these conditions are necessary. The degeneracy at infinity will be dropped in Section \ref{sec:EnergyDecay} where we will use a vector field more adapted to the neighbourhoods of $\mathcal{I}^{+}$ (which do not contain $i^{0})$.

Our approach for obtaining the above estimate is by first deriving the weighted $L^{2}$ estimate for $\psi$ and then for the derivatives. Note that no unphysical restriction on the initial data is required (in particular they are not required to be  supported away $\mathcal{H}^{+}$).

\subsection{The Vector Field $X$}
\label{sec:TheVectorFieldX}

We are looking for a vector field which gives rise to a current whose divergence is non-negative (upon integration on the spheres of symmetry).  Our work here is inspired by \cite{dr5}. We will be working with vector fields of the form $X=f\left(r^{*}\right)\partial_{r^{*}}$ (for the coordinate system $(t,r^{*})$ see Appendix \ref{sec:ConstructingTheExtentionOfReissnerNordstrOM}).

\subsubsection{The Spacetime Term $K^{X}$}
\label{sec:TheSpacetimeTermKX}
 We compute
\begin{equation*}
K^{X}=\left(\frac{f'}{2D}+\frac{f}{r}\right)\left(\partial_{t}\psi\right)^{2}+\left(\frac{f'}{2D}-\frac{f}{r}\right)\left(\partial_{r^{*}}\psi\right)^{2}+\left(-\frac{f'+2fH}{2}\right)\left|\nabb\psi\right|^{2},
\end{equation*}
where $f'=\frac{df\left(r^{*}\right)}{dr^{*}}$ and $H=\frac{1}{2}\frac{dD\left(r\right)}{dr}$. Note that all the derivatives from now on will be considered with respect to  $r^{*}$  unless otherwise stated.

Unfortunately, the trapping obstruction does not allow us to obtain a positive definite current $K^{X}$ so easily. Indeed, if we assume that all these coefficients are positive then we have 
\begin{equation*}
\begin{split}
&f'>0,f<0\\
&-\frac{fH}{D}>\frac{f'}{2D}>-\frac{f}{r}\Rightarrow\\
&\left(-\frac{H}{D}+\frac{1}{r}\right)f>0\Rightarrow\\
&\left(\frac{H}{D}-\frac{1}{r}\right)<0,\\
\end{split}
\end{equation*}
but the quantity $\left(\frac{H}{D}-\frac{1}{r}\right)$ changes sign exactly at the radius $Q$ of the photon sphere. Therefore, there is no way to make all the above four coefficients positive. For simplicity, let us define
\begin{equation*}
\frac{P\left(r\right)}{r^{2}}=-H\cdot r+D=\frac{r^{2}-3Mr+2e^{2}}{r^{2}}.
\end{equation*}

\subsection{The Case $l\geq 1$}
\label{sec:TheCaseL0}
We first consider the case where $\psi$ is supported on the frequencies $l\geq 1$.

\subsubsection{The Currents $J^{X,1}_{\mu}$ and $K^{X,1}$}
\label{sec:TheCurrentJXGMuAndEstimatesForKXG}

To assist in overcoming the obstacle of the photon sphere we shall introduce zeroth order terms in order to modify the coef{}ficient of $\left(\partial_{t}\psi\right)^{2}$ and  create another which is more flexible. Let us consider the current
\begin{equation*}
J_{\mu}^{X,g,h,w}=J_{\mu}^{X}+g\left(r\right)\psi\nabla _{\mu}\psi+h\left(r\right)\psi^{2}\left(\nabla_{\mu}w\right),
\end{equation*}
where $g,h,w$ are functions on $\mathcal{M}$. Then we have
\begin{equation*}
\begin{split}
\tilde{K}^{X}&\!=\!\nabla^{\mu}J_{\mu}^{X,g,h,w}=K^{X}+\nabla^{\mu}\left(g\psi\nabla _{\mu}\psi\right)+\nabla^{\mu}\left(h\psi^{2}\left(\nabla_{\mu}w\right)\right)\\
&\!=\!K^{X}+\left(\nabla^{\mu}g\right)\psi\nabla _{\mu}\psi+g\left(\nabla^{a}\psi\nabla_{a}\psi\right)+\left(\nabla^{\mu}h\right)\psi^{2}\left(\nabla_{\mu}w\right)+\\
&\ \ \ \ +h\psi^{2}\left(\Box_{g} w\right)+2h\psi\left(\nabla^{\mu}\psi\right)\left(\nabla_{\mu}w\right)\\
&\!=\!K^{X}\!+g\left(\nabla^{a}\psi\nabla_{a}\psi\right)\!+\left(\nabla^{\mu}g+2h\nabla^{\mu}w\right)\psi\nabla_{\mu}\psi+\left(\nabla_{\mu}h\nabla^{\mu}w+h\left(\Box_{g} w\right)\right)\psi^{2}.\\
\end{split}
\end{equation*}
By taking $h=1,g=2G,w=-G$ we make the coefficient of $\psi\nabla_{\mu}\psi$ vanish. Therefore, let us define
\begin{equation}
J_{\mu}^{X,1}\overset{.}{=}J^{X}_{\mu}+2G\psi\left(\nabla_{\mu}\psi\right)-\left(\nabla_{\mu}G\right)\psi^{2}.
\label{modG}
\end{equation}
Then
\begin{equation*}
\begin{split}
K^{X,1}&\overset{.}{=}\nabla^{\mu}J_{\mu}^{X,1}=K^{X}+2G\left(\nabla^{a}\psi\nabla_{a}\psi\right)-\left(\Box_{g} G\right)\psi^{2}\\
&=K^{X}+2G\left(\left(-\frac{1}{D}\right)\left(\partial_{t}\psi\right)^{2}+\left(\frac{1}{D}\right)\left(\partial_{r^{*}}\psi\right)^{2}+\left|\nabb\psi\right|^{2}\right)-\left(\Box_{g} G\right)\psi^{2}.\\
\end{split}
\end{equation*}
Therefore, if we take $G$ such that 
\begin{equation}
\begin{split}
&\left(2G\right)\left(-\frac{1}{D}\right)=-\left(\frac{f'}{2D}+\frac{f}{r}\right)\Rightarrow\\
&G=\frac{f'}{4}+\frac{f\cdot D}{2r}
\label{G}
\end{split}
\end{equation}
then
\begin{equation*}
\begin{split}
K^{X,1}&=\frac{f'}{D}\left(\partial_{r^{*}}\psi\right)^{2}+\frac{f\cdot P}{r^{3}}\left|\nabb\psi\right|^{2}-\left(\Box_{g} G\right)\psi^{2}\\
&=\frac{f'}{D}\left(\partial_{r^{*}}\psi\right)^{2}+\frac{f\cdot P}{r^{3}}\left|\nabb\psi\right|^{2}\\
&\ \ \ -\left(\frac{1}{4D}f'''+\frac{1}{r}f''+\frac{D'}{D\cdot r}f'+\left(\frac{D''}{2D\cdot r}-\frac{D'}{2r^{2}}\right)f\right)\psi^{2}
\end{split}
\end{equation*}
Note that, since $f$ must be bounded and $f'$ positive, $-f'''$ must become negative and therefore has the wrong sign. Note also the factor $D$ at the denominator which degenerates at the horizon. One way to overcome this would be to make $f$ approach appropriately the horizon and spatial infinity and $f'$ sufficiently concave at the photon sphere. But then, we would need $\psi_{\mu,l}=0$ for sufficiently large $l$ in order to compensate the loss in the compact intermediate regions. However, we want to avoid this restriction on $\psi$. The way that turns out to work is by borrowing from the coefficient of $\left(\partial_{r^{*}}\psi\right)^{2}$ which is accomplished by introducing a third current,  as in \cite{dr5}.

\subsubsection{The Current $J^{X,2}_{\mu}$ and Estimates for $K^{X,2}$}
\label{sec:TheCurrentJX2MuAndEstimatesForKX2}

We define
\begin{equation*}
\begin{split}
&J_{\mu}^{X,2}\left(\psi\right)=J_{\mu}^{X,1}\left(\psi\right)+\frac{f'}{D\cdot f}\beta \psi^{2} X_{\mu},
\end{split}
\end{equation*}
where $\beta$ will be a function of $r^{*}$ to be defined below.  Since
\begin{equation*}
\begin{split}
Div\left(fV\right)=fDiv\left(V\right)+V\left(f\right),
\end{split}
\end{equation*}
and
\begin{equation*}
\begin{split}
Div\left(\partial_{r^{*}}\right)=\frac{D'}{D}+\frac{2D}{r},
\end{split}
\end{equation*}
we have
\begin{equation*}
\begin{split}
K^{X,2}=&K^{X,1}+Div\left(\frac{f'}{D}\beta\psi^{2}\partial_{r^{*}}\right)\\
=&\frac{f'}{D}\left(\partial_{r^{*}}\psi+\beta\psi\right)^{2}+\frac{f\cdot P}{r^{3}}\left|\nabb{\psi}\right|^{2}+\\
&+\left[-\frac{1}{4}\frac{f'''}{D}+\left(\beta-\frac{D}{r}\right)\frac{f''}{D}+\left(\beta'-\beta^{2}+\frac{2D}{r}\beta-\frac{D'}{r}\right)\frac{f'}{D}+\left(\frac{D'}{2r^{2}}-\frac{D''}{2D\cdot r}\right)f\right]\psi^{2}.
\end{split}
\end{equation*}
Note that the coefficient of $f\psi^{2}$ is independent of the choice of the function $\beta$. Let us now take
\begin{equation*}
\begin{split}
\beta=\frac{D}{r}-\frac{x}{\a^{2}+x^{2}},
\end{split}
\end{equation*}
where $x=r^{*}-\a-\sqrt{\a}\ $ and $\a>0$ a sufficiently large number to be chosen appropriately. As we shall see, the reason for introducing this shifted coordinate $x$ is that we want the origin $x=0$ to be far away from the photon sphere. Then
\begin{equation*}
\begin{split}
K^{X,2}=&\frac{f'}{D}\left(\partial_{r^{*}}\psi+\beta\psi\right)^{2}+\frac{f\cdot P}{r^{3}}\left|\nabb{\psi}\right|^{2}+\\
&+\left[-\frac{1}{4}\frac{f'''}{D}-\frac{x}{\a^{2}+x^{2}}\frac{f''}{D}-\frac{\a^{2}}{\left(\a^{2}+x^{2}\right)^{2}}\frac{f'}{D}+\left(\frac{D'}{2r^{2}}-\frac{D''}{2D\cdot r}\right)f\right]\psi^{2}.
\end{split}
\end{equation*}
We clearly need to choose a function $f$ that is stricly increasing and changes sign at the photon sphere. So, if we choose this function (which we call $f^{\a}$) such that
\begin{equation*}
\begin{split}
\left(f^{\a}\right)'=\frac{1}{\a^{2}+x^{2}}, \ \  f^{\a}\left(r^{*}=0\right)=f^{\a}\left(r=2M\right)=0,
\end{split}
\end{equation*}
then
\begin{equation*}
\begin{split}
F:=-\frac{1}{4}\frac{\left(f^{\a}\right)'''}{D}-\frac{x}{\a^{2}+x^{2}}\frac{\left(f^{\a}\right)''}{D}-\frac{\a^{2}}{\left(\a^{2}+x^{2}\right)^{2}}\frac{\left(f^{\a}\right)'}{D}=\frac{1}{2D}\frac{x^{2}-\a^{2}}{\left(x^{2}+\a^{2}\right)^{3}}.
\end{split}
\end{equation*}
If we define $X^{\a}=f^{\a}\partial_{r^{*}}$ and
\begin{equation*}
\begin{split}
I=\left(\frac{D'}{2r^{2}}-\frac{D''}{2D\cdot r}\right)f^{\a}
\end{split}
\end{equation*}
then
\begin{equation*}
\begin{split}
K^{X^{\a},2}\left(\psi\right)=\frac{\left(f^{\a}\right)'}{D}\left(\partial_{r^{*}}\psi+\beta\psi\right)^{2}+\frac{f^{\a}\cdot P}{r^{3}}\left|\nabb{\psi}\right|^{2}+\left(F+I\right)\psi^{2}.
\end{split}
\end{equation*}

\subsubsection{Nonnegativity of $K^{X^{\a},2}$}
\label{sec:NonnegativityOfK2}

\begin{proposition}
There exists a constant $C$ which depends only on $M$ such that for all solutions $\psi$ to the wave equation which are supported on the frequencies  $l\geq 1$  we have
\begin{equation}
\begin{split}
\int_{\mathbb{S}^{2}}{\left(\frac{P\cdot\left(r-2M\right)}{r^{4}}\left|\nabb\psi\right|^{2}+\frac{1}{D}\frac{1}{(\left(r^{*}\right)^{2}+1)^{2}}\psi^{2}\right)}\leq C\int_{\mathbb{S}^{2}}{K^{X^{\a},2}\left(\psi\right)}.
\label{Xa2psi}
\end{split}
\end{equation}
\label{xa2prop}
\end{proposition}
\begin{proof}
We have
\begin{equation*}
\begin{split}
K^{X^{\a},2}\geq\frac{f^{\a}\cdot P}{r^{3}}\left|\nabb{\psi}\right|^{2}+\left(F+I\right)\psi^{2}.
\end{split}
\end{equation*}
For the term $I$ we have
\begin{equation*}
\begin{split}
I=\left(\frac{D'}{2r^{2}}-\frac{D''}{2D\cdot r}\right)f^{\a}=\frac{f^{\a}}{2r}\left(\frac{D'}{r}-\frac{D''}{D}\right).
\end{split}
\end{equation*}
However,
\begin{equation*}
\begin{split}
\frac{D'}{r}-\frac{D''}{D}=&\frac{D\partial_{r}D}{r}-\partial_{r}\left(D\partial_{r}D\right)=\frac{D\partial_{r}D}{r}-\left(\partial_{r}D\right)^{2}-D\partial_{rr}D \\
=& D^{\frac{3}{2}}\frac{2M}{r^{3}}-D\frac{4M^{2}}{r^{4}}-D\left[D^{\frac{1}{2}}\left(-\frac{4M}{r^{3}}\right)+\frac{2M^{2}}{r^{4}}\right]\\
=& \left(1-\frac{M}{r}\right)^{2}\left(1-\frac{2M}{r}\right)\frac{6M}{r^{3}}\\
=&\left(1-\frac{M}{r}\right)^{2}\left(1-\frac{2M}{r}\right)\frac{6M}{r^{3}}\\
=& \left(1-\frac{M}{r}\right)\frac{6M}{r^{5}}P.
\end{split}
\end{equation*}
Therefore,
\begin{equation*}
\begin{split}
I=\left(1-\frac{M}{r}\right)\frac{3M}{r^{6}}f^{\a}\cdot P.
\end{split}
\end{equation*}
Therefore, $I\geq 0$ and vanishes (to second order) at the horizon and the photon sphere. Note now that the term $F$ is positive  $x$ is not in the interval $\left[-\a, \a\right]$. On the other hand, the photon sphere is far away from the region where $x\in\left[-\a, \a\right]$ so $I$ is positive there. However, $I$ behaves like $\frac{1}{r^{4}}$ and so it is not sufficient to compensate the negativity of $F$. That is why we need to borrow from the coefficient of $\nabb\psi$ which behaves like $\frac{1}{r}$. Indeed, the Poincar\'e inequality gives us
\begin{equation*}
\begin{split}
\int_{\mathbb{S}^{2}}{\frac{2}{r^{2}}\psi^{2}}\leq\int_{\mathbb{S}^{2}}{\left|\nabb\psi\right|^{2}},
\end{split}
\end{equation*}
and, therefore, it suffices to prove that 
\begin{equation*}
\begin{split}
2\frac{f^{\a}\cdot P}{r^{5}}+F> 0
\end{split}
\end{equation*}
for all $x\in\left[-\a,\a\right]$ or, equivalently, $r^{*}\in\left[\sqrt{\a}, 2\a+\sqrt{\a}\right]$. Then
\begin{equation*}
\begin{split}
F=\frac{1}{D}\frac{x^{2}-\a^{2}}{2\left(x^{2}+\a^{2}\right)^{3}}=\Delta_{\a}\frac{x^{2}-\a^{2}}{2\left(x^{2}+\a^{2}\right)^{3}}
\end{split}
\end{equation*}
with $1\leq\Delta_{\a}\rightarrow 1$ as $\a\rightarrow +\infty$. Moreover, since $r^{*}\geq r$ for big $r$, by taking $\a$ sufficiently big we have
\begin{equation*}
\begin{split}
2\frac{f^{\a}\cdot P}{r^{5}}\geq 2\delta_{\a}\frac{f^{\a}}{\left(r^{*}\right)^{3}}
\end{split}
\end{equation*}
with $1\geq\delta_{\a}\rightarrow 1$ as $\a\rightarrow +\infty$. Therefore, we need to establish that 
\begin{equation*}
\begin{split}
\Pi :=\frac{\left(\a^{2}-x^{2}\right)\left(x+\a+\sqrt{\a}\right)^{3}}{4\left(x^{2}+\a^{2}\right)^{3}f^{\a}}<\frac{\delta_{\a}}{\Delta_{\a}}.
\end{split}
\end{equation*}
for all $x\in\left[-\a,\a\right]$. In view of the asymptotic behaviour of the constants $\Delta_{\a},\delta_{\a}$ it suffices to prove that the right hand side of the above inequality is strictly less than 1. 
\begin{lemma}
Given the function $f^{\a}\left(r^{*}\right)=f^{\a}\left(r^{*}\left(x\right)\right)$ defined above, we have the following: If $\ -\a\leq x\leq 0$ then
\begin{equation*}
\begin{split}
f^{\a}>\frac{x+\a}{2\a^{2}}.
\end{split}
\end{equation*}
Also, if $\ 0\leq x \leq \a$ then 
\begin{equation*}
\begin{split}
f^{\a}>\frac{x+\a}{x^{2}+\a^{2}}-\frac{1}{2\a}.
\end{split}
\end{equation*}
\end{lemma}

\begin{proof}
For $\ -\a\leq x\leq 0$ we have
\begin{equation*}
\begin{split}
f^{\a}\left(x\right)=\int_{-\a-\sqrt{\a}}^{x}{\frac{1}{\tilde{x}^{2}+\a^{2}}d\tilde{x}}>\int_{-\a}^{x}{\frac{1}{\tilde{x}^{2}+\a^{2}}d\tilde{x}}>\int_{-\a}^{x}{\frac{1}{2\a^{2}}d\tilde{x}}=\frac{x+\a}{2\a^{2}}.
\end{split}
\end{equation*}
Now for $\ 0\leq x \leq \a$ we have
\begin{equation*}
\begin{split}
f^{\a}\left(x\right)&>\int_{-\a}^{x}{\frac{1}{\tilde{x}^{2}+\a^{2}}d\tilde{x}}>\int_{-\a}^{-x}{\frac{1}{\tilde{x}^{2}+\a^{2}}d\tilde{x}}+\int_{-x}^{x}{\frac{1}{x^{2}+\a^{2}}d\tilde{x}}\\&=\int_{-\a}^{x}{\frac{1}{x^{2}+\a^{2}}d\tilde{x}}-\int_{-\a}^{-x}{\left(\frac{1}{x^{2}+\a^{2}}-\frac{1}{\tilde{x}^{2}+\a^{2}}\right)d\tilde{x}}\\
&>\frac{x+a}{x^{2}+\a^{2}}-\frac{1}{2a}.
\end{split}
\end{equation*}
\end{proof}
Consequently, if $\ -\a\leq x\leq 0$ and $x=-\lambda\a$ then $\ 0\leq\lambda\leq 1$ and
\begin{equation*}
\begin{split}
\Pi<&\frac{\left(\a-x\right)\left(x+\a+\sqrt{\a}\right)^{3}\a^{2}}{2\left(x^{2}+\a^{2}\right)^{3}}=\frac{\left(1-\lambda\right)\left(1+\lambda+\a^{-\frac{1}{2}}\right)^{3}}{2\left(\lambda^{2}+1\right)^{3}}\\
=&d_{\a}\frac{\left(1-\lambda\right)\left(1+\lambda\right)^{3}}{2\left(\lambda^{2}+1\right)^{3}}<d_{\a}\frac{2}{3}<\frac{9}{10},
\end{split}
\end{equation*}
since $\ d_{\a}\rightarrow 1$ as $\a\rightarrow +\infty$. Similarly, if $\ 0\leq x\leq \a$ and $x=\lambda\a$ then $\ 0\leq\lambda\leq 1$ and
\begin{equation*}
\begin{split}
\Pi&<\frac{\a\left(\a^{2}-x^{2}\right)\left(x+\a+\sqrt{\a}\right)^{3}}{2\left(x^{2}+\a^{2}\right)^{2}\left(\a^{2}+2\a x-x^{2}\right)}<\frac{\a\left(\a^{2}-x^{2}\right)\left(x+\a+\sqrt{\a}\right)^{3}}{2\left(x^{2}+\a^{2}\right)^{3}}\\
&=\tilde{d_{\a}}\frac{\left(1-\lambda^{2}\right)\left(1+\lambda\right)^{3}}{2\left(1+\lambda^{2}\right)^{3}}<\tilde{d_{\a}}\frac{8}{10}<\frac{9}{10},
\end{split}
\end{equation*}
since $\tilde{d_{\a}}\rightarrow 1$ as $\a\rightarrow +\infty$.
\end{proof}
Note, however, that although the coef{}ficient of $\psi$ does not degenerate on the photon sphere, the coef{}ficient of the angular derivatives vanishes at the photon sphere to second order. Having this estimate for $\psi$, we  obtain estimates for its derivatives (in Section \ref{sec:SpacetimeL2EstimateForPsi} we derive a similar estimate for $\psi$ for the case $l=0$).

\begin{proposition}
There exists a positive constant $C$ which depends only on $M$ such that for all solutions $\psi$ of the wave equation which are supported on the frequencies $l\geq 1$ we have
\begin{equation*}
\begin{split}
\int_{\mathbb{S}^{2}}{\!\left(\frac{(r-M)}{r^{4}}(\partial_{t}\psi)^{2}+\frac{(r-M)}{r^{4}}\psi^{2}\right)}\leq C\int_{\mathbb{S}^{2}}{\sum_{i=0}^{1}{K^{X^{\a},2}\left(T^{i}\psi\right)}}.
\end{split}
\end{equation*}
\label{mesaio}
\end{proposition}
\begin{proof}
From \eqref{Xa2psi} we have that the coef{}ficient of $\psi^{2}$ does not degenerate at the photon sphere. The weights at infinity are given by the Poincar\'e inequality. Commuting the wave equation with $T$ completes the proof of the proposition.
\end{proof}

\subsubsection{The Lagrangian Current $L_{\mu}^{f}$}
\label{sec:TheAuxiliaryCurrentsAnd}
In order to retrieve the remaining derivatives we consider the Lagrangian\footnote{The name Lagrangian comes from the fact that if $\psi$ satisfies the wave equation and $\mathcal{L}$ denotes the Lagrangian that corresponds to the wave equation, then $\mathcal{L}(\psi,d\psi,g^{-1})=g^{\mu\nu}\partial_{\mu}\psi\partial_{\nu}\psi=\operatorname{Div}(\psi\nabla_{\mu}\psi)$} current
\begin{equation*}
\begin{split}
L_{\mu}^{f}=f\psi\nabla_{\mu}\psi,
\end{split}
\end{equation*}
where
\begin{equation*}
\begin{split}
f=\frac{1}{r^{3}}D^{\frac{3}{2}}.
\end{split}
\end{equation*}

\begin{proposition}
There exists a positive constant $C$ which depends only on $M$ such that for all solutions $\psi$ of the wave equation which are supported on the frequencies $l\geq 1$ we have
\begin{equation*}
\begin{split}
&\int_{\mathbb{S}^{2}}{\left(\frac{\sqrt{D}}{r^{3}}\left(\partial_{t}\psi\right)^{2}+\frac{\sqrt{D}}{2r^{3}}\left(\partial_{r^{*}}\psi\right)^{2}+\frac{D^{3/2}}{r^{3}}\left|\nabb\psi\right|^{2}\right)}
\\&\ \ \ \ \ \ \ \ \ \ \ \ \ \ \ \ \ \ \ \leq \int_{\mathbb{S}^{2}}{\left(\operatorname{Div}(L_{\mu}^{f})+C\sum_{i=0}^{1}{K^{X^{\a},2}\left(T^{i}\psi\right)}\right)}.
\label{rparagogos}
\end{split}
\end{equation*}
\label{rparagogosprop}
\end{proposition}

\begin{proof}
We have
\begin{equation*}
\begin{split}
\operatorname{Div}(L_{\mu}^{f})=&f\nabla^{\mu}\psi\nabla_{\mu}\psi +\nabla^{\mu}f\psi\nabla_{\mu}\psi \\
=&-\frac{f}{D}(\partial_{t}\psi)^{2}+\frac{f}{D}(\partial_{r^{*}}\psi)^{2}+f\left|\nabb\psi\right|^{2}+\frac{f'}{D}\psi(\partial_{r^{*}}\psi)\\
=&-\frac{\sqrt{D}}{r^{3}}\left(\partial_{t}\psi\right)^{2}+\frac{\sqrt{D}}{r^{3}}\left(\partial_{r^{*}}\psi\right)^{2}+\frac{D^{3/2}}{r^{3}}\left|\nabb\psi\right|^{2}+\frac{3D}{r^{4}}\left(\frac{M}{r}-3\sqrt{D}\right)\psi\partial_{r^{*}}\psi\\
\geq &\frac{\sqrt{D}}{r^{3}}\left(\partial_{r^{*}}\psi\right)^{2}+\frac{D^{3/2}}{r^{3}}\left|\nabb\psi\right|^{2}-\frac{\sqrt{D}}{r^{3}}\left(\partial_{t}\psi\right)^{2}-\frac{1}{\epsilon}\frac{D}{r^{4}}\psi^{2}-\epsilon\frac{D}{r^{4}}(\partial_{r^{*}}\psi)^{2},
\end{split}
\end{equation*}
where $\epsilon >0$ is such that $\epsilon\sqrt{D}<\frac{r}{2}$. Therefore, in view of Proposition \ref{mesaio}, we have the required result.
\end{proof}

\begin{proposition}
There exists a positive constant $C$ which depends only on $M$ such that for all solutions $\psi$ of the wave equation which are supported on the frequencies $l\geq 1$ we have
\begin{equation*}
\begin{split}
&\int_{\mathbb{S}^{2}}{\left(\frac{\sqrt{D}}{r^{3}}\left(\partial_{t}\psi\right)^{2}+\frac{\sqrt{D}}{2r^{3}}\left(\partial_{r^{*}}\psi\right)^{2}+\frac{\sqrt{D}}{r}\left|\nabb\psi\right|^{2}+\frac{\sqrt{D}}{r^{3}}\psi^{2}\right)}
\\&\ \ \ \ \ \ \ \ \ \ \ \ \ \ \ \ \ \ \ \leq \int_{\mathbb{S}^{2}}{\left(\operatorname{Div}(L_{\mu}^{f})+C\sum_{i=0}^{1}{K^{X^{\a},2}\left(T^{i}\psi\right)}\right)}.
\label{olaparagogos}
\end{split}
\end{equation*}
\label{finspacetimeprop}
\end{proposition}
\begin{proof}
Immediate from Propositions \ref{xa2prop} and \ref{rparagogosprop}.
\end{proof}

\subsubsection{The Current $J_{\mu}^{X^{d},1}$}
\label{sec:TheCurrentJMuXD1}

In case we allow some degeneracy at the photon sphere  we then  can obtain similar estimates without commuting the wave equation with $T$. This will be very useful whenever our analysis is local. We define the function $f^{d}$ such that 
\begin{equation*}
\begin{split}
\left(f^{d}\right)'=\frac{1}{(r^{*})^{2}+1},\ \ \ \ f^{d}(r^{*}=0)=0.
\end{split}
\end{equation*}
\begin{proposition}
There exists a positive constant $C$ which depends only on $M$ such that for all solutions $\psi$ of the wave equation which are supported on the frequencies $l\geq 1$ we have
\begin{equation*}
\begin{split}
\frac{1}{C}\int_{\mathbb{S}^{2}}{\left(\frac{1}{r^{2}}(\partial_{r^{*}}\psi)^{2}+\frac{P\cdot (r-2M)}{r^{4}}\left|\nabb\psi\right|^{2}\right)}\leq \int_{\mathbb{S}^{2}}{\left(K^{X^{d},1}(\psi)+CK^{X^{\a},2}\left(\psi\right)\right)},
\end{split}
\end{equation*}
where $X^{d}=f^{d}\partial_{r^{*}}$ and the current $J_{\mu}^{X,1}$ is as defined in Section \ref{sec:TheCurrentJXGMuAndEstimatesForKXG}.
\label{nocomm1prop}
\end{proposition}
\begin{proof}
Recall that 
\begin{equation*}
\begin{split}
K^{X,1}&=\frac{(f^{d})'}{D}\left(\partial_{r^{*}}\psi\right)^{2}+\frac{f^{d}\cdot P}{r^{3}}\left|\nabb\psi\right|^{2}\\
&\ \ \ -\left(\frac{1}{4D}(f^{d})'''+\frac{1}{r}(f^{d})''+\frac{D'}{D\cdot r}(f^{d})'+\left(\frac{D''}{2D\cdot r}-\frac{D'}{2r^{2}}\right)f^{d}\right)\psi^{2}.
\end{split}
\end{equation*}
Note that the coef{}ficient of $\psi^{2}$ vanishes to first order on $\mathcal{H}^{+}$ (see also Lemma \ref{rnrstar}) and behaves like $\frac{1}{r^{4}}$ for large $r$. Note also that the coefficient of $(\partial_{r^{*}}\psi)^{2}$ converges to $M^{2}$ (see again Lemma \ref{rnrstar}). The result now follows from Proposition \ref{xa2prop}.
\end{proof}
In order to retrieve the $\partial_{t}$-derivative we introduce the current
\begin{equation*}
\begin{split}
L_{\mu}^{h^{d}}=h^{d}\psi\nabla_{\mu}\psi,
\end{split}
\end{equation*}
where $h^{d}$ is such that:
\begin{equation*}
\begin{split}
&\text{For } M\leq r\leq r_{0}<2M,  \   h^{d}=-\frac{1}{(r^{*})^{2}+1},\\
&\text{For } r_{0}< r <2M, \   h^{d}< 0,\\
&\text{For } r=2M, \ h^{d}=0 \text{ to second order},\\
&\text{For } 2M<r\leq r_{1}, h^{d}<0,\\
&\text{For } r_{1}\leq r, \   h^{d}=-\frac{1}{r^{2}}.\\
\end{split}
\end{equation*}
\begin{proposition}
There exists a positive constant $C$ which depends only on $M$ such that for all solutions $\psi$ of the wave equation which are supported on the frequencies $l\geq 1$ we have
\begin{equation*}
\begin{split}
&\frac{1}{C}\int_{\mathbb{S}^{2}}{\left(\frac{P\cdot (r-2M)}{r^{5}}(\partial_{t}\psi)^{2}+\frac{1}{r^{2}}(\partial_{r^{*}}\psi)^{2}+\frac{P\cdot (r-2M)}{r^{4}}\left|\nabb\psi\right|^{2}\right)}\\&\ \ \ \ \ \ \ \ \ \ \ \ \ \ \leq \int_{\mathbb{S}^{2}}{\left(\operatorname{Div}(L_{\mu}^{h^{d}})+K^{X^{d},1}(\psi)+CK^{X^{\a},2}\left(\psi\right)\right)}.
\end{split}
\end{equation*}
\label{nocomm2prop}
\end{proposition}
\begin{proof}
We have as before
\begin{equation*}
\begin{split}
\operatorname{Div}(L_{\mu}^{h^{d}})=-\frac{h^{d}}{D}(\partial_{t}\psi)^{2}+\frac{h^{d}}{D}(\partial_{r^{*}}\psi)^{2}+h^{d}\left|\nabb\psi\right|^{2}+\frac{(h^{d})'}{D}\psi(\partial_{r^{*}}\psi).\\
\end{split}
\end{equation*}
Since the coef{}ficient of $\psi(\partial_{r^{*}}\psi)$ vanishes to first on the horizon, the Cauchy-Schwarz inequality and Proposition \ref{nocomm1prop} imply the result.
\end{proof}
Finally, we obtain:
\begin{proposition}
There exists a positive constant $C$ which depends only on $M$ such that for all solutions $\psi$ of the wave equation which are supported on the frequencies $l\geq 1$ we have
\begin{equation*}
\begin{split}
&\frac{1}{C}\int_{\mathbb{S}^{2}}{\left(\frac{P\cdot (r-2M)}{r^{5}}(\partial_{t}\psi)^{2}+\frac{1}{r^{2}}(\partial_{r^{*}}\psi)^{2}+\frac{P\cdot (r-2M)}{r^{4}}\left|\nabb\psi\right|^{2}+\frac{(r-M)}{r^{4}}\psi^{2}\right)}\\&\ \ \ \ \ \ \ \ \ \ \ \ \ \ \leq \int_{\mathbb{S}^{2}}{\left(\operatorname{Div}(L_{\mu}^{h^{d}})+K^{X^{d},1}(\psi)+CK^{X^{\a},2}\left(\psi\right)\right)}.
\end{split}
\end{equation*}
\label{nocomm3prop}
\end{proposition}
\begin{proof}
Immediate from Propositions \ref{xa2prop} and \ref{nocomm2prop}.
\end{proof}

\subsection{The Case $l=0$}
\label{sec:TheCaseLO}

In the case $l=0$ the wave is not trapped. Indeed, if we define 
\begin{equation*}
f^{0}=-\frac{1}{r^{3}},\ \ \ X^{0}=f^{0}\partial_{r^{*}},
\end{equation*}
then we have: 
\begin{proposition}
For all spherically symmetric solutions $\psi$ of the wave equation  we have 
\begin{equation*}
\frac{1}{r^{4}}(\partial_{t}\psi)^{2}+\frac{5}{r^{4}}(\partial_{r^{*}}\psi)^{2}=K^{X^{0}}.
\end{equation*}
\label{1l0}
\end{proposition}
\begin{proof}
Immediate from the expression of $K^{X}$ and the above choice of $f=f^{0}$.
\end{proof}

\subsection{The Boundary Terms}
\label{sec:TheBoundaryTerms}

We now  control the boundary terms.

\subsubsection{Estimates for $J_{\mu}^{X}n^{\mu}_{S}$}
\label{sec:EstimatesForJMuXNMuSigma}

\begin{proposition}
Let $X=f\partial_{r^{*}}$ where $f=f(r^{*})$ is bounded and  $S$ be either a $SO(3)$ invariant spacelike (that may cross $\mathcal{H}^{+}$) or a $SO(3)$ invariant null hypersurface. Then there exists a uniform constant $C$ that depends  on $M$, $S$ and the function $f$ such that for all $\psi$ we have
\begin{equation*}
\begin{split}
\left|\int_{S}{J^{X^{i}}_{\mu}\left(\psi\right)n^{\mu}_{S}}\right|\leq C\int_{S}{J_{\mu}^{T}\left(\psi\right)n^{\mu}_{S}}.
\end{split}
\end{equation*}
\label{boun1}
\end{proposition}
\begin{proof}

We work using the coordinate system $\left(v,r,\theta,\phi\right)$. First note that 
\begin{equation*}
\begin{split}
&\frac{\partial}{\partial r^{*}}=\frac{\partial v}{\partial r^{*}}\frac{\partial}{\partial v}+\frac{\partial r}{\partial r^{*}}\frac{\partial}{\partial r}=\partial_{v}+D\partial_{r}\Rightarrow\\
&X=f\partial_{v}+f\cdot D\partial_{r}.
\end{split}
\end{equation*}
Now
\begin{equation*}
\begin{split} 
J^{X}_{\mu}n^{\mu}_{S}&=\textbf{T}_{\mu\nu}(X)^{\nu}n^{\mu}_{S}=\textbf{T}_{\mu v}fn^{\mu}_{S}+\textbf{T}_{\mu r}f D n^{\mu}_{S}\\
&=\textbf{T}_{vv}f n^{v}+\textbf{T}_{vr}\left(fn^{r}+f D n^{v}\right)+\textbf{T}_{rr}f D n^{r}\\
&=\left(fn^{v}\right)\left(\partial_{v}\psi\right)^{2}+D f n^{v}\left(\partial_{v}\psi\right)\left(\partial_{r}\psi\right)+\\
&\ \ \ \ +D\left[\frac{1}{2}Dfn^{v}-\frac{1}{2}fn^{r}-\frac{1}{2}Dfn^{v}+fn^{r}\right]\left(\partial_{r}\psi\right)^{2}+\\
&\ \ \ \ +\left[\frac{1}{2}Dfn^{v}-\frac{1}{2}Dfn^{v}-\frac{1}{2}fn^{r}\right]\left|\nabb\psi\right|^{2}.
\end{split}
\end{equation*}
The result now follows from the boundedness of $f$.
\end{proof}

\subsubsection{Estimates for $J_{\mu}^{X^{\a},i}\,n_{S}^{\mu},\, i=1,2$}
\label{sec:EstimatesForJMuXA2NSigmaMu}

\begin{proposition}
There exists a uniform constant $C$ that depends  on $M$ and $S$ such that
\begin{equation*}
\begin{split}
\left|\int_{S}{J^{X^{\a},i}_{\mu}\left(\psi\right)n^{\mu}_{S}}\right|\leq C\int_{S}{J_{\mu}^{T}\left(\psi\right)n^{\mu}_{S}},\, i=1,2,
\end{split}
\end{equation*}
where $S$ is as in Proposition \ref{boun1}.
\label{boun2}
\end{proposition}
\begin{proof}
It suffices to prove
\begin{equation*}
\begin{split}
\left|\int_{S}{\left(2G^{\a}\psi\left(\nabla_{\mu}\psi\right)-\left(\nabla_{\mu}G^{\a}\right)\psi^{2}+\left(\frac{(f^{\a})'}{D}\beta (\partial_{r^{*}})_{\mu}\right)\psi^{2}\right)n^{\mu}_{S}}\,\right|\leq B\int_{S}{J_{\mu}^{T}\left(\psi\right)n^{\mu}_{S}}.
\end{split}
\end{equation*}
For this we first prove the following lemma that is true only in the case of  extreme Reissner-Nordstr\"{o}m (and not in the subextreme range).
\begin{lemma}
The function
\begin{equation*}
\begin{split}
F=\frac{1}{D}\frac{1}{((r^{*})^{2}+1)}
\end{split}
\end{equation*}
is bounded in $\mathcal{R}\cup\mathcal{H}^{+}$.
\label{rnrstar}
\end{lemma}
\begin{proof}
For the tortoise coordinate $r^{*}$ we have
\begin{equation*}
\begin{split}
r^{*}(r)=r+2M\ln(r-M)-\frac{M^{2}}{r-M}+C.
\end{split}
\end{equation*}
Clearly,  $F\rightarrow 0$ as $r\rightarrow +\infty$. Moreover, in a neighbourhood of $\mathcal{H}^{+}$
\begin{equation*}
\begin{split}
F\sim\frac{r^{2}}{(r-M)^{2}}\frac{1}{\left[\frac{M^{4}}{(r-M)^{2}}+r^{2}+4M^{2}(\ln(r-M))^{2}\right]}\rightarrow M^{2}<\infty.
\end{split}
\end{equation*}
This implies the required result.
\end{proof}
An immediate collorary of this lemma is that the functions
\begin{equation*}
\begin{split}
G_{1}^{\a}=r\frac{G^{\a}}{D} \text{  and  } \frac{(f^{\a})'}{D}
\end{split}
\end{equation*}
are bounded in $\mathcal{R}\cup\mathcal{H}^{+}$.
Furthermore, for $M\ll r$ we have $(f^{\a})'\sim\frac{1}{r^{2}}$, $D\sim1$ and $\partial_{r}D\sim\frac{1}{r^{2}}$.
Therefore,
\begin{equation*}
\begin{split}
G_{2}^{\a}=r^{2}\nabla_{r}G^{\a}=r^{2}\frac{1}{D}\left(\frac{(f^{\a})''}{4}+\frac{(f^{\a})'D}{2r}+\frac{f^{\a}D\partial_{r}D
}{2r}-\frac{f^{\a}D}{2r^{2}}\right)
\end{split}
\end{equation*}
is bounded. Finally, $\nabla_{v}G^{\a}=0$.
The above bounds and the first Hardy inequality complete the proof of the proposition. 
\end{proof}

\subsection{A Degenerate $X$ Estimate}
\label{sec:ADegenateXEstimate}

We first obtain an estimate which does not lose derivatives but degenerates at the photon sphere.

\begin{theorem}
There exists a constant $C$ which depends  on $M$ and $\Sigma_{0}$ such that for all solutions $\psi$ of the wave equation we have
\begin{equation}
\begin{split}
&\int_{\mathcal{R}(0,\tau)}\!\!\!{\left(\frac{1}{r^{4}}(\partial_{r^{*}}\psi)^{2}+\frac{P\cdot (r-2M)}{r^{7}}\left((\partial_{t}\psi)^{2}+\left|\nabb\psi\right|^{2}\right)\right)}\\&\ \ \ \ \ \ \ \ \ \ \ \ \ \ \leq C\int_{\Sigma_{0}}{J_{\mu}^{T}(\psi)n^{\mu}_{\Sigma_{0}}}.
\end{split}
\label{degX}
\end{equation}
\label{degXprop}
\end{theorem}
\begin{proof}
We first decompose $\psi$  as 
\begin{equation*}
\psi=\psi_{\geq 1}+\psi_{0}
\end{equation*}
We apply Stokes' theorem for the current 
\begin{equation*}
J_{\mu}^{d}(\psi_{\geq 1})=J_{\mu}^{X^{d},1}(\psi_{\geq 1})+J_{\mu}^{X^{\a},2}(\psi_{\geq 1})+L_{\mu}^{h^{d}}(\psi_{\geq 1})
\end{equation*}
in the spacetime region $\mathcal{R}(0,\tau)$ and use Propositions \ref{nocomm3prop} and \ref{boun2}. We also apply Stokes' theorem in $\mathcal{R}(0,\tau)$ for the current $J_{\mu}^{X^{0}}(\psi_{0})$ and use Proposition \ref{boun1} and by adding these two estimates we obtain the required result.

\end{proof}

\subsection{A non-Degenerate $X$ Estimate}
\label{sec:ANonDegenerateXEstimate}

We  derive an $L^{2}$ estimate which does not degenerate at the photon sphere but requires higher regularity for $\psi$.
\begin{theorem}
There exists a constant $C$ which depends  on $M$ and $\Sigma_{0}$ such that for all solutions $\psi$ of the wave equation we have
\begin{equation}
\begin{split}
&\int_{\mathcal{R}(0,\tau)}\!\!\!{\left(\frac{\sqrt{D}}{r^{4}}\left(\partial_{t}\psi\right)^{2}+\frac{\sqrt{D}}{2r^{4}}\left(\partial_{r^{*}}\psi\right)^{2}+\frac{\sqrt{D}}{r}\left|\nabb\psi\right|^{2}\right)}
\\&\ \ \ \ \ \ \ \ \ \ \ \ \ \ \ \ \ \ \ \leq C\int_{\Sigma_{0}}{\left(J_{\mu}^{T}(\psi)n^{\mu}_{\Sigma_{0}}+J_{\mu}^{T}(T\psi)n^{\mu}_{\Sigma_{0}}\right)}.
\label{x}
\end{split}
\end{equation}
\label{xtheo}
\end{theorem}
\begin{proof}
We again decompose  $\psi$ as 
\begin{equation*}
\psi=\psi_{\geq 1}+\psi_{0}
\end{equation*}
and apply Stokes' theorem for the current 
\begin{equation*}
J_{\mu}(\psi_{\geq 1})=L_{\mu}^{f}(\psi_{\geq 1})+CJ_{\mu}^{X^{\a},2}(\psi_{\geq 1})+CJ_{\mu}^{X^{\a},2}(T\psi_{\geq 1})
\end{equation*}
in the spacetime region $\mathcal{R}(0,\tau)$, where $C$ is the constant of Proposition \ref{finspacetimeprop}, and use Propositions \ref{finspacetimeprop} and \ref{boun2}. Finally, we  apply Stokes' theorem in $\mathcal{R}(0,\tau)$ for the current $J_{\mu}^{X^{0}}(\psi_{0})$ and use Proposition \ref{boun1}. Adding these two estimates completes the proof.
\end{proof}

\subsection{Zeroth Order Morawetz Estimate for $\psi$}
\label{sec:SpacetimeL2EstimateForPsi}

We now prove weighted $L^{2}$ estimates of the wave $\psi$ itself. In Section \ref{sec:NonnegativityOfK2}, we obtained such an estimate for $\psi_{\geq 1}$. Next we derive a similar  estimate for the zeroth spherical harmonic $\psi_{0}$.  We first prove the following lemma

\begin{lemma}
Fix $R>2M$. There exists a constant $C$ which depends  on $M$, $\Sigma_{0}$ and $R$ such that  for all spherically symmetric solutions $\psi$ of the wave equation 
\begin{equation*}
\int_{\left\{r=R\right\}\cap\mathcal{R}(0,\tau)}{(\partial_{t}\psi)^{2}+(\partial_{r^{*}}\psi)^{2}}\leq C_{R}\int_{\Sigma_{0}}{J^{T}_{\mu}(\psi)n^{\mu}_{\Sigma_{0}}}.
\end{equation*}
\label{lemmal0}
\end{lemma}
\begin{proof}
Consider the region
\begin{equation*}
\mathcal{F}(0,\tau)=\left\{\mathcal{R}(0,\tau)\cap\left\{M\leq r\leq R\right\}\right\}.
\end{equation*}
By applying the vector field $X=\partial_{r^{*}}$ as a multiplier in the region $\mathcal{F}(0,\tau)$ we obtain
\begin{equation*}
\begin{split}
\int_{\mathcal{H}^{+}}{J_{\mu}^{X}(\psi)n^{\mu}_{\mathcal{H}^{+}}}+\int_{\mathcal{F}}{K^{X}}+\int_{\Sigma_{\tau}}{J_{\mu}^{X}(\psi)n^{\mu}_{\Sigma_{\tau}}}+\int_{\left\{r=R\right\}\cap\mathcal{R}(0,\tau)}{J_{\mu}^{X}(\psi)n^{\mu}_{\mathcal{F}}}=\int_{\Sigma_{0}}{J_{\mu}^{X}(\psi)n^{\mu}_{\Sigma_{0}}},
\end{split}
\end{equation*}
where $n^{\mu}_{\mathcal{F}}$ denotes the unit normal vector to $\left\{r=R\right\}$ pointing in the interior of $\mathcal{F}$. 
 \begin{figure}[H]
	\centering
		\includegraphics[scale=0.14]{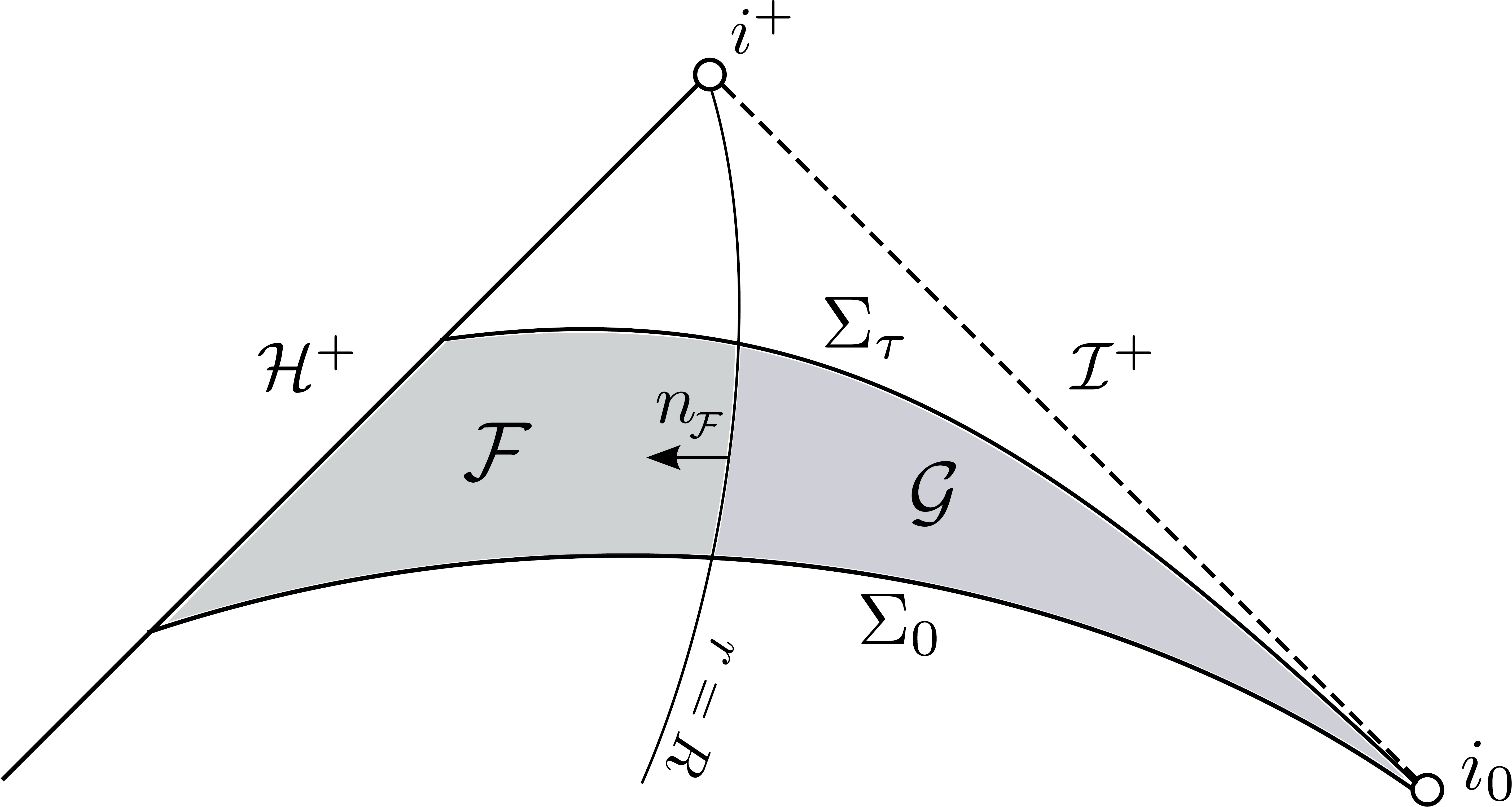}
	\label{fig:ernt5}
\end{figure}
In view of the spatial compactness of the region $\mathcal{F}$, the corresponding spacetime integral can be estimated using  Propositions \ref{1l0} and \ref{boun1}. The boundary integrals over $\Sigma_{0}$ and $\Sigma_{\tau}$ can be estimated using Proposition \ref{boun1}. Moreover, for spherically symmetric waves $\psi$ we have
\begin{equation*}
J_{\mu}^{X}(\psi)n^{\mu}_{\mathcal{F}}=-\frac{1}{\sqrt{D}}\textbf{T}_{r^{*}r^{*}}=-\frac{1}{2\sqrt{D}}\left((\partial_{t}\psi)^{2}+(\partial_{r^{*}}\psi)^{2}\right),
\end{equation*}
which completes the proof.
\end{proof}
Consider now the region
\begin{equation*}
\mathcal{G}=\mathcal{R}\cap\left\{R\leq r\right\}.
\end{equation*}
\begin{proposition}
Fix $R>2M$. There exists a constant $C$ which depends  on $M$, $\Sigma_{0}$ and $R$ such that  for all spherically symmetric solutions $\psi$ of the wave equation  
\begin{equation*}
\int_{\mathcal{G}}{\frac{1}{r^{4}}\psi^{2}} +\int_{\left\{r=R\right\}}{\psi^{2}}\leq C_{R}\int_{\Sigma_{0}}{J^{T}_{\mu}(\psi)n^{\mu}_{\Sigma_{0}}}.
\end{equation*}
\label{fl0}
\end{proposition}
\begin{proof}
By applying Stokes' theorem for the current $J_{\mu}^{X,1}(\psi)$ in the region $\mathcal{G}(0,\tau)$, where again $X=\partial_{r^{*}}$ (and, therefore, $f=1$), we obtain
\begin{equation*}
\int_{\mathcal{G}(0,\tau)}{K^{X,1}(\psi)}+\int_{\left\{r=R\right\}\cap\mathcal{G}(0,\tau)}{J_{\mu}^{X,1}(\psi)n^{\mu}_{\mathcal{G}}}\leq C\int_{\Sigma_{0}}{J^{T}_{\mu}(\psi)n^{\mu}_{\Sigma_{0}}}
\end{equation*}
where we have used Proposition \ref{boun2} to estimate the boundary integral over $\Sigma_{\tau}\cap\mathcal{G}$. Note that again $n^{\mu}_{\mathcal{G}}$ denotes the unit normal vector to $\left\{r=R\right\}$ pointing in the interior of $\mathcal{G}$. For $f=1$ we have
\begin{equation*}
K^{X,1}(\psi)=I\psi^{2},
\end{equation*}
where $I>0$ and $I\sim \frac{1}{r^{4}}$ for $r\geq R>2M$. If  $G$ is the function defined in Section \ref{sec:TheCurrentJXGMuAndEstimatesForKXG} then $G=\frac{D}{2r}$ and therefore for sufficiently large $R$ we have $\partial_{r^{*}}G<0$. Then,
\begin{equation*}
J_{\mu}^{X,1}(\psi)n^{\mu}_{\mathcal{G}}=J_{\mu}^{X}n^{\mu}_{\mathcal{G}}+2G\psi(\nabla_{\mu}\psi)n^{\mu}_{\mathcal{G}}-(\nabla_{\mu}G)\psi^{2}n^{\mu}_{\mathcal{G}}.
\end{equation*}
Since $n_{\mathcal{G}}=\frac{1}{\sqrt{D}}\partial_{r^{*}}$, by applying Cauchy- Schwarz for the second term on the right hand side (and since the third term is positive) we obtain
\begin{equation*}
J_{\mu}^{X,1}(\psi)n^{\mu}_{\mathcal{F}}\sim\  \psi^{2}-\frac{1}{\epsilon}((\partial_{t}\psi)^{2}+(\partial_{r^{*}}\psi)^{2}),
\end{equation*}
for a sufficiently small $\epsilon$. Lemma \ref{lemmal0} completes the proof.
\end{proof}
It remains to obtain a (weighted) $L^{2}$ estimate for $\psi$ in the region $\mathcal{F}$.
\begin{proposition}Fix sufficiently large $R>2M$.  Then, there exists a constant $C$ which depends  on $M$, $\Sigma_{0}$ and $R$ such that  for all spherically symmetric solutions $\psi$ of the wave equation  
\begin{equation*}
\int_{\mathcal{F}}{D\psi^{2}}\leq C_{R}\int_{\Sigma_{0}}{J^{T}_{\mu}(\psi)n^{\mu}_{\Sigma_{0}}}.
\end{equation*}
\label{el0}
\end{proposition}
\begin{proof}
By applying Stokes' theorem for the current
\begin{equation*}
J_{\mu}^{H}(\psi)=(\nabla_{\mu}H)\psi^{2}-2H\psi\nabla_{\mu}\psi,
\end{equation*}
where $H=(r-M)^{2}$, we obtain
\begin{equation*}
\int_{\mathcal{F}(0,\tau)}{\nabla^{\mu}J_{\mu}^{H}}=\int_{\partial{\mathcal{F}(0,\tau)}}{J_{\mu}^{H}n^{\mu}_{\partial{\mathcal{F}}}}.
\end{equation*}
All the boundary integrals can be estimated using Propositions \ref{boun2} and \ref{fl0}. Note also that 
\begin{equation*}
\nabla^{\mu}J_{\mu}^{H}=(\Box_{g}H)\psi^{2}-\frac{2H}{D}((\partial_{r^{*}}\psi)^{2}-(\partial_{t}\psi)^{2}).
\end{equation*}
Since $\Box_{g}H\sim D$ for $M\leq r \leq R$ and $\frac{H}{D}$ is bounded, the result follows in view of the spatial compactness of $\mathcal{F}$ and Proposition \ref{1l0}.
\end{proof}
We finally have the following zeroth order Morawetz  estimate:
\begin{theorem}
There exists a constant $C$ that depends  on $M$ and $\Sigma_{0}$ such that for all solutions $\psi$ of the wave equation we have
\begin{equation}
\int_{\mathcal{R}}{\frac{D}{r^{4}}\psi^{2}}\leq C\int_{\Sigma_{0}}{J^{T}_{\mu}(\psi)n^{\mu}_{\Sigma_{0}}}.
\label{moraw}
\end{equation}
\label{morawetz}
\end{theorem}
\begin{proof}
Write 
\begin{equation*}
\psi=\psi_{0}+\psi_{\geq 1}
\end{equation*}
and use Propositions \ref{xa2prop}, \ref{el0} and \ref{fl0}.
\end{proof}
Clearly, the lemma used for Proposition \ref{fl0} holds stricly for the case $l=0$. In general, one could have argued by averaging  the $X$ estimate \ref{degX} in $R$, which would imply that there exists a value $R_{0}$ of $r$ such that all the derivatives are controlled on the hypersurface of  constant radius $R_{0}$. This makes our argument work for all $\psi$ without recourse to the spherical decomposition.

Note that the above Morawetz estimate holds for $\psi$ that satisfy the regularity assumptions of Section \ref{sec:TheCauchyProblemForTheWaveEquation}.

\begin{remark}
The key property used in Proposition \ref{fl0} is that the d' Alembertian of $\frac{1}{r}$ is always negative, something  not true for larger powers of $\frac{1}{r}$. Note that this unstable behaviour of $\frac{1}{r}$ is expected since it is the static solution of the wave equation in Minkowski.
\end{remark}

\subsection{Discussion}
\label{sec:Discussion}

In \cite{dr5}, a non-degenerate $X$ estimate is established for  Schwarzschild. However, there one needs to commute with the generators of the Lie algebra so(3). Moreover, the boundary terms  could not be controlled by the flux of $T$  but one needed a  small portion of the redshift estimate. In our case, the structure of $r^{*}$ allowed us to bound these terms using only the $T$ flux. 

For a nice exposition of previous work on $X$ estimates see \cite{md}.

\section{The Vector Field $N$}
\label{sec:TheVectorFieldN}

It is clear that in order to obtain an estimate for the non-degenerate energy of a local observer we need to use timelike multipliers at the horizon. Then uniform boundedness of  energy would  follow provided we can control the spacetime terms which arise. For a suitable class of non-degenerate black hole spacetimes, not only have the bulk terms the right sign close to $\mathcal{H}^{+}$ but they in fact control the non-degenerate energy. Indeed, in  \cite{md} the following  is proved 
\begin{proposition}
Let $\mathcal{H}^{+}$ be a Killing horizon  with positive surface gravity and let $V$ be the Killing vector field tangent to $\mathcal{H}^{+}$. Then there exists a $\phi^{V}_{\tau}$-invariant vector field $N$ on $\mathcal{H}^{+}$ and  constants $b,B>0$  such that for all functions $\psi$ we have
\begin{equation*}
bJ^{N}_{\mu}(\psi)n^{\mu}_{\Sigma_{\tau}}\leq K^{N}(\psi)\leq BJ^{N}_{\mu}(\psi)n^{\mu}_{\Sigma_{\tau}} 
\end{equation*}
on $\mathcal{H}^{+}$.
\label{redshiftmd}
\end{proposition}
The construction of the above vector field does not require the global existence of a causal Killing field and  the positivity of the surface gravity suffices. Under suitable circumstances, one can prove that the $N$ flux is uniformly bounded  without  understanding the structure of the trapping (i.e.~no $X$ or Morawetz estimate is required). 

However, in our case, in view of the lack of redshift along $\mathcal{H}^{+}$ the situation is completely different. Indeed, we will see  that in  extreme Reissner-Nordstr\"{o}m   a vector field satisfying the properties of Proposition~\ref{redshiftmd} does not exist. Not only  will we show that there is no $\phi_{\tau}$-invariant vector field $N$ satisfying the properties of Proposition~\ref{redshiftmd} but we will in fact prove that there is no $\phi_{\tau}$-invariant timelike vector field $N$ such that 
\begin{equation*}
K^{N}\left(\psi\right)\geq 0
\end{equation*}
on  $\mathcal{H}^{+}$.  Our resolution to this problem uses appropriate modification of $J_{\mu}^{N}$ and the Hardy inequalities and thus is still robust. 

\subsection{The Effect of Vanishing Redshift on Linear Waves }
\label{sec:TheSpacetimeTermKN}
Let us first try to understand the  current $K$ associated to a future directed timelike $\phi_{T}$-invariant vector field $N$ in a neighbourhood of the horizon $\mathcal{H}^{+}$. If $N=f_{v}\left(r\right)\partial_{v}+f_{r}\left(r\right)\partial_{r}$ then we obtain 
\begin{equation*}
\begin{split}
K^{N}\left(\psi\right)=&F_{vv}\left(\partial_{v}\psi\right)^{2}+F_{rr}\left(\partial_{r}\psi\right)^{2}+F_{\scriptsize\nabb}\left|\nabb\psi\right|^{2}+F_{vr}\left(\partial_{v}\psi\right)\left(\partial_{r}\psi\right),
\end{split}
\end{equation*}
where the coef{}ficients are given by
\begin{equation}
\begin{split}
&F_{vv}=\left(\partial_{r}f_{v}\right),\\
&F_{rr}=D\left[\frac{\left(\partial_{r}f_{r}\right)}{2}-\frac{f_{r}}{r}\right]-\frac{f_{r}D'}{r},  \\
&F_{\scriptsize\nabb}=-\frac{1}{2}\left(\partial_{r}f_{r}\right),\\
&F_{vr}=D\left(\partial_{r}f_{v}\right)-\frac{2f_{r}}{r}.\\
\label{list}
\end{split}
\end{equation}
Note that since $N=f_{v}\partial_{v}+f_{r}\partial_{r}$ we have 
\begin{equation}
\begin{split}
&g\left(N,N\right)=-D\left(f_{v}\right)^{2}+2f_{v}f_{r},\\
&g\left(N,T\right)=-Df_{v}+f_{r},
\label{Mtifut}
\end{split}
\end{equation}
and so $f_{r}\left(r=M\right)$ can not be zero (otherwise the vector field $N$ would not be timelike). Therefore, looking back at the list \eqref{list} we see that the coefficient of $\left(\partial_{r}\psi\right)^{2}$ vanishes on the horizon $\mathcal{H}^{+}$ whereas the coefficient of $\partial_{v}\psi\partial_{r}\psi$ is equal to $-\frac{2f_{r}\left(M\right)}{M}$ which is not zero. Therefore, $K^{N}\left(\psi\right)$ is linear with respect to $\partial_{r}\psi$ on the horizon $\mathcal{H}^{+}$ and thus it necessarily fails to be non-negative definite. This linearity is a characteristic feature of the geometry of the event horizon of  extreme Reissner-Nordstr\"{o}m and degenerate black hole spacetimes more generally. This proves that  a vector field satisfying the properties of Proposition~\ref{redshiftmd} does not exist.

\subsection{A Locally Non-Negative Spacetime Current}
\label{sec:TheCurrentJMuNDeltaFrac12AndEstimatesForItsDivergence}

In view of the above discussion, we need to modify the bulk term by introducing new terms that will counteract the presence of $\partial_{v}\psi\partial_{r}\psi$.  We define 
\begin{equation}
J_{\mu}\equiv J_{\mu}^{N,h}=J_{\mu}^{N}+h\left(r\right)\psi\nabla _{\mu}\psi
\label{Jmodified}
\end{equation}
where $g$ is a function on $\mathcal{M}$. Then we have
\begin{equation*}
\begin{split}
K&\equiv K^{N,h}=\nabla^{\mu}J_{\mu}\\
&=K^{N}+\nabla^{\mu}\left(h\psi\nabla _{\mu}\psi\right)\\
&=K^{N}+\left(\nabla^{\mu}h\right)\psi\nabla _{\mu}\psi+h\left(\nabla^{a}\psi\nabla_{a}\psi\right),\\
\end{split}
\end{equation*}
provided $\psi$ is a solution of the wave equation. Let us suppose that $h\left(r\right)=\frac{f_{r}\left(M\right)}{M}$. Then
\begin{equation}
\begin{split}
K^{N,h}\left(\psi\right)&=K^{N}\left(\psi\right)+h\left(\nabla^{a}\psi\nabla_{a}\psi\right)\\
&=K^{N}\left(\psi\right)+h\left(2\partial _{v}\psi\partial_{r}\psi+D\left(\partial _{r}\psi\right)^{2}+\left|\nabb\psi\right|^{2}\right)\\
&=F_{vv}\left(\partial_{v}\psi\right)^{2}+\left[F_{rr}+hD\right]\left(\partial_{r}\psi\right)^{2}+\left[-\frac{\partial_{r}f_{r}}{2}+h\right]\left|\nabb\psi\right|^{2}\\
&\ \ \ \ +\left[F_{vr}+2h\right]\left(\partial_{v}\psi\partial_{r}\psi\right).\\
\end{split}
\label{knn}
\end{equation}
Note that by taking $h$ to be constant  we managed to have no zeroth order terms in the  current $K$. Let us denote the above coef{}ficients of $\left(\partial_{a}\psi\partial_{b}\psi\right)$ by $G_{ab}$ where $a,b\in\left\{v,r,\nabb\right\}$ and define the vector field $N$ in the region $M\leq r\leq \frac{9M}{8}$ to be such that 
\begin{equation}
\begin{split}
&f_{v}(r)=16r,\\
&f_{r}(r)=-\frac{3}{2}r+M
\end{split}
\label{n}
\end{equation}
and, therefore,
\begin{equation*}
h=-\frac{1}{2}.
\end{equation*}
Clearly, $N$ is timelike future directed vector field. We have the following
\begin{proposition}
For all functions $\psi$,  the current $K^{N, -\frac{1}{2}}(\psi)$ defined by \eqref{knn} is non-negative definite in the region $\mathcal{A}=\left\{M\leq r\leq \frac{9M}{8}\right\}$ and, in particular, there is a positive constant $C$ that depends only on $M$ such that
\begin{equation*}
K^{N,-\frac{1}{2}}(\psi)\geq C\left(\left(\partial_{v}\psi\right)^{2}+\sqrt{D}\left(\partial_{r}\psi\right)^{2}+\left|\nabb\psi\right|^{2}\right).
\end{equation*}
\label{knprop}
\end{proposition}
\begin{proof}
We first observe that the coef{}ficient $G_{\scriptsize\nabb}$ of $\partial_{v}\psi\partial_{r}\psi$  is equal to
\begin{equation*}
G_{\scriptsize\nabb}=\frac{1}{4}.
\end{equation*}
Clearly, $G_{vv}=16$ and $G_{rr}$ is non-negative since the factor of the dominant term $D'$ (which vanishes to first order on $\mathcal{H}^{+}$) is positive. As regards the coef{}ficient of the mixed term we have
\begin{equation*}
G_{vr}=16D+2\sqrt{D}=\epsilon_{1}\cdot\epsilon_{2},
\end{equation*}
where
\begin{equation*}
\begin{split}
&\epsilon_{1}=\sqrt{D},\\
&\epsilon_{2}=16\sqrt{D}+2.
\end{split}
\end{equation*}
We will show that in region $\mathcal{A}$ we have
\begin{equation*}
\begin{split}
&\epsilon_{1}^{2}\leq G_{rr},\\
&\epsilon_{2}^{2}<G_{vv}.
\end{split}
\end{equation*}
Indeed, we set $\lambda =\frac{M}{r}$ we have
\begin{equation*}
\begin{split}
&\epsilon_{1}^{2}\leq G_{rr}\Leftrightarrow\\
&\left(1-\frac{M}{r}\right)\leq\left(1-\frac{M}{r}\right)\left[\frac{\partial_{r}f_{r}\left(r\right)}{2}-\frac{f_{r}\left(r\right)}{r}+\frac{f_{r}\left(M\right)}{M}\right]+\left(-f_{r}\left(r\right)\frac{M}{r^{2}}\right)\overset{\eqref{n}}{\Leftrightarrow}\\
&\left(1-\lambda\right)\leq\left(1-\lambda\right)\left(\frac{1}{4}-\lambda\right)-\lambda\left(-\frac{3}{2}+\lambda\right)\Leftrightarrow\\
&\lambda \leq \frac{3}{5},
\end{split}
\end{equation*}
which holds. Note also that $\epsilon_{1}^{2}$ vanishes to second order on $\mathcal{H}^{+}$ whereas $G_{rr}$ vanishes to first order. Therefore, in region $\mathcal{A}$ we have $G_{rr}-\epsilon_{1}^{2}\sim \sqrt{D}$. Similarly, 
\begin{equation*}
\begin{split}
&\epsilon_{2}^{2}< G_{vv}\Leftrightarrow\\
&16\left[\left(1-\frac{M}{r}\right)+2\right]^{2}< 16\Leftrightarrow\\
&\lambda >\frac{7}{8},
\end{split}
\end{equation*}
which again holds. So in $\mathcal{A}$ we have
\begin{equation*}
\begin{split}
K^{N,-\frac{1}{2}}&=G_{vv}\left(\partial_{v}\psi\right)^{2}+G_{rr}\left(\partial_{r}\psi\right)^{2}+G_{rv}\left(\partial_{v}\psi\partial_{r}\psi\right)+G_{\scriptsize\nabb}\left|\nabb\psi\right|^{2}\\
&=\left(G_{vv}-\epsilon_{2}^{2}\right)\left(\partial_{v}\psi\right)^{2}+\left(G_{rr}-\epsilon_{1}^{2}\right)\left(\partial_{r}\psi\right)^{2}+\\
&\ \ \ \ 
+\left(\epsilon_{2}\left(\partial_{v}\psi\right)\right)^{2}+\left(\epsilon_{1}\left(\partial_{r}\psi\right)\right)^{2}+\left(\epsilon_{2}\left(\partial_{v}\psi\right)\right)\left(\epsilon_{1}\left(\partial_{r}\psi\right)\right)+G_{\scriptsize\nabb}\left|\nabb\psi\right|^{2}.
\end{split}
\end{equation*}
The above inequalities, the compactness of $\left[M, \frac{9M}{8}\right]$ and that
\begin{equation*}
 a^{2}+ab+b^{2}=\frac{1}{2}\left(a^{2}+b^{2}\right)+\frac{1}{2}\left(a+b\right)^{2}\geq 0
 \end{equation*}
for all $a,b\in\mathbb{R}$, complete the proof. 

\end{proof}

\subsection{The Cut-off $\delta$ and the Current $J_{\mu}^{N,\delta,-\frac{1}{2}}$}
\label{sec:TheCutOffDeltaAndTheCurrentJMuNDeltaFrac12}

Clearly, the current $K^{N,-\frac{1}{2}}$ will not be non-negative far away from $\mathcal{H}^{+}$ and thus is not useful there. For this reason we extend $f_{v},f_{r}$ such that 
\begin{equation*}
\begin{split}
&f_{v}\left(r\right)>0 \ \text{for all}\  r\geq M \ \text{and}\  f_{v}\left(r\right)=1\  \text{for all}\  r\geq \frac{8M}{7}, \\
&f_{r}\left(r\right)\leq 0\  \text{for all}\  r\geq M\  \text{and}\  f_{r}\left(r\right)=0\  \text{for all}\  r\geq \frac{8M}{7}.
\end{split}
\end{equation*}
Clearly with the above choices  we have $g\left(N,N\right)<0,g\left(N,T\right)<0$ everywhere and so $N$ is a future directed timelike $\phi_{\tau}$-invariant vector field.

Similarly, the modification term in the current $J^{N,-\frac{1}{2}}_{\mu}$ is not useful for consideration far away from  $\mathcal{H}^{+}$. That is why we introduce a smooth cut-off function $\delta:[\left.M,+\infty\right.)\rightarrow\mathbb{R}$ such that $\delta\left(r\right)=1,r\in\left[M,\frac{9M}{8}\right]$ and $\delta\left(r\right)=0,r\in\left[\left.\frac{8M}{7},
+\infty\right.\right)$. Then we  consider the currents
\begin{equation}
\begin{split}
&J_{\mu}^{N,\delta,-\frac{1}{2}}\overset{.}{=}J_{\mu}^{N}-\frac{1}{2}\delta\psi\nabla_{\mu}\psi,\\
&K^{N,\delta,-\frac{1}{2}}\overset{.}{=}\nabla^{\mu}J_{\mu}^{N,\delta,-\frac{1}{2}}.\\
\end{split}
\end{equation}
We now consider the three regions $\mathcal{A},\mathcal{B},\mathcal{C}$
 \begin{figure}[H]
	\centering
		\includegraphics[scale=0.14]{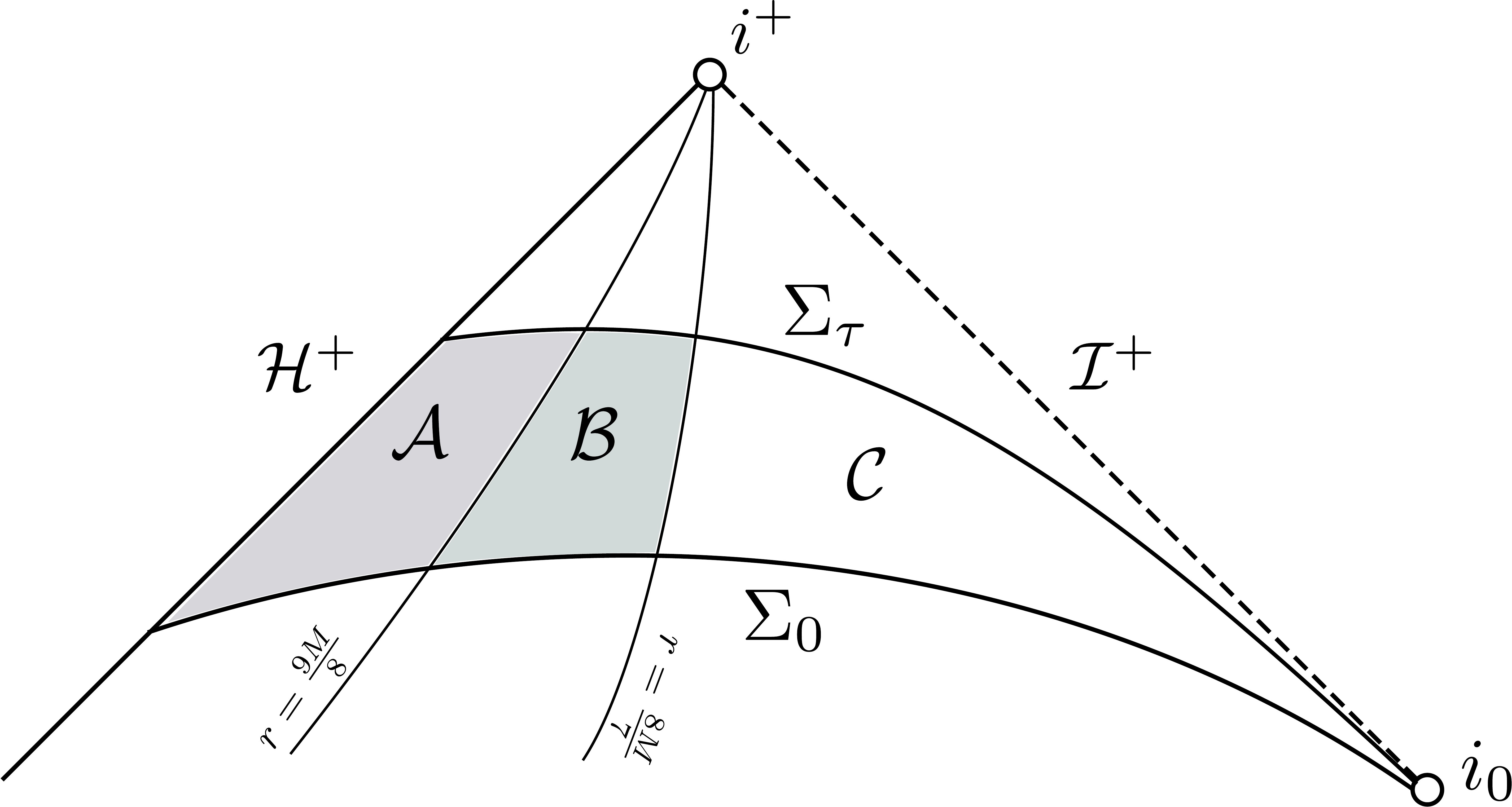}
	\label{fig:ernt6}
\end{figure}
In  region $\mathcal{C}=\left\{r\geq\frac{8M}{7}\right\}$ where $\delta=0$ and $ N=T$, we have
\begin{equation*}
K^{N,\delta,-\frac{1}{2}}= 0.
\end{equation*}
However, this spacetime current, which depends on the 1-jet of $\psi$,  will generally be  negative in  region $\mathcal{B}$ and thus   will be controlled by the $X$ and Morawetz estimates (which, of course, are non-degenerate in $\mathcal{B}$).

We next control the Sobolev norm
\begin{equation*}
\left\|\psi\right\|^{2}_{\dot{H}^{1}\left(\Sigma_{\tau}\right)}\lesssim\, \int_{\Sigma_{\tau}}{\left(\partial_{v}\psi\right)^{2}+\left(\partial_{r}\psi\right)^{2}+\left|\nabb\psi\right|^{2}}.
\end{equation*}
Since $N$ and $n^{\mu}_{\Sigma_{\tau}}$ are timelike everywhere in $\mathcal{R}$ and since $\omega_{N},\omega_{n}$ are positive and uniformly bounded, \eqref{GENERALT} of Appendix \ref{sec:TheHyperbolicityOfTheWaveEquation1} implies 
\begin{equation*}
J_{\mu}^{N}(\psi)n^{\mu}_{\Sigma_{\tau}}\sim \, \left(\partial_{v}\psi\right)^{2}+\left(\partial_{r}\psi\right)^{2}+\left|\nabb\psi\right|^{2},
\end{equation*}
where, in view of the $\varphi_{T}$-invariance of $\Sigma_{\tau}$ and $N$, the constants in  $\sim$ depend only on $M$ and $\Sigma_{0}$. Therefore, it suffices to estimate the flux of $N$ through $\Sigma_{\tau}$. For this we have:
\begin{proposition}
There exists a  constant $C>0$ which depends  on $M$ and $\Sigma_{0}$ such that for all functions $\psi$ 
\begin{equation}
\int_{\Sigma_{\tau}}{J_{\mu}^{N}(\psi)n^{\mu}}\leq 2\int_{\Sigma_{\tau}}{J_{\mu}^{N,\delta,-\frac{1}{2}}(\psi)n^{\mu}}+C\int_{\Sigma_{\tau}}{J_{\mu}^{T}(\psi)n^{\mu}}.
\label{nboundary}
\end{equation}
\label{nb}
\end{proposition}
\begin{proof}

We have
\begin{equation*}
\begin{split}
J_{\mu}^{N,\delta,-\frac{1}{2}}n^{\mu}&=J_{\mu}^{N}n^{\mu}-\frac{1}{2}\delta\psi\partial_{\mu}\psi n^{\mu}=J_{\mu}^{N}n^{\mu}-\frac{1}{2}\delta\psi\partial_{v}\psi n^{v}-\frac{1}{2}\delta\psi\partial_{r}\psi n^{r}\\
&\geq J_{\mu}^{N}n^{\mu}-\epsilon\delta(\partial_{v}\psi)^{2}-\epsilon\delta(\partial_{r}\psi)^{2}-\frac{\delta}{\epsilon}\psi^{2}\\
&\geq  \frac{1}{2}J_{\mu}^{N}n^{\mu}-\frac{\delta}{\epsilon}\psi^{2}
\end{split}
\end{equation*}
for a sufficiently small $\epsilon$. The result follows from the first Hardy inequality.  
\end{proof}
\begin{corollary}
There exists a  constant $C>0$ which depends  on $M$ and $\Sigma_{0}$ such that for all functions $\psi$ 
\begin{equation}
\int_{\Sigma_{\tau}}{J_{\mu}^{N,\delta ,-\frac{1}{2}}(\psi)n^{\mu}}\leq C\int_{\Sigma_{\tau}}{J_{\mu}^{N}(\psi)n^{\mu}}.
\label{nboundary1}
\end{equation}
\label{nb1}
\end{corollary}
\begin{proof}

Note that
\begin{equation*}
\begin{split}
J_{\mu}^{N,\delta,-\frac{1}{2}}n^{\mu}&=J_{\mu}^{N}n^{\mu}-\frac{1}{2}\delta\psi\partial_{\mu}\psi n^{\mu}\\
&\leq J_{\mu}^{N}n^{\mu}+\delta(\partial_{v}\psi)^{2}+\delta(\partial_{r}\psi)^{2}+C\delta\psi^{2}\\
\end{split}
\end{equation*}
and use first the Hardy inequality.
\end{proof}

\subsection{Lower Estimate for an Integral over $\mathcal{H}^{+}$}
\label{sec:EstimatesForAnIntegralOverMathcalH}

Lastly, we need to estimate the integral $\displaystyle\int_{\mathcal{H}^{+}}{J_{\mu}^{N,\delta,-\frac{1}{2}}(\psi)n^{\mu}_{\mathcal{H}^{+}}}$. Recall that  $n_{\scriptstyle\mathcal{H}^{+}}=T$. We have
\begin{proposition}
There exist   positive constants $C$ and $C_{\epsilon}$  such that for any $\epsilon >0$ and all functions $\psi$ 
\begin{equation*}
\int_{\mathcal{H}^{+}}{J_{\mu}^{N,\delta,-\frac{1}{2}}(\psi)n^{\mu}_{\mathcal{H}^{+}}}\geq C\int_{\mathcal{H}^{+}}{J_{\mu}^{N}(\psi)n^{\mu}_{\mathcal{H}^{+}}}\ -C_{\epsilon}\int_{\Sigma_{\tau}}{J_{\mu}^{T}(\psi)n^{\mu}_{\Sigma_{\tau}}}\ -\epsilon\int_{\Sigma_{\tau}}{J_{\mu}^{N}(\psi)n^{\mu}_{\Sigma_{\tau}}},
\end{equation*}
\label{nhb}
where $C$ depends on $M$, $\Sigma_{0}$ and $C_{\epsilon}$ depends on $M$, $\Sigma_{0}$ and $\epsilon$.
\end{proposition}
\begin{proof}
On $\mathcal{H}^{+}$ we have $\delta =1$. Therefore,
\begin{equation*}
\begin{split}
J_{\mu}^{N,\delta,-\frac{1}{2}}n^{\mu}_{\mathcal{H}^{+}}&=J_{\mu}^{N}n^{\mu}_{\mathcal{H}^{+}}-\frac{1}{2}\psi\partial_{\mu}\psi T^{\mu}\\
&=J_{\mu}^{N}n^{\mu}_{\mathcal{H}^{+}}-\frac{1}{2}\psi\partial_{v}\psi\\
\end{split}
\end{equation*}
However,
\begin{equation*}
\begin{split}
\int_{\mathcal{H}^{+}}{-2\psi\partial_{v}\psi}&=\int_{\mathcal{H}^{+}}{-\partial_{v}\psi^{2}}\\
&=\int_{\mathcal{H}^{+}\cap\Sigma_{0}}{\psi^{2}}-\int_{\mathcal{H}^{+}\cap\Sigma_{\tau}}{\psi^{2}}.
\end{split}
\end{equation*}
From the first and second Hardy inequality  we have
\begin{equation*}
\int_{\mathcal{H}^{+}\cap\Sigma}{\psi^{2}}\leq C_{\epsilon}\int_{\Sigma}{J_{\mu}^{T}n^{\mu}_{\Sigma}}+\epsilon\int_{\Sigma}{(\partial_{v}\psi)^{2}+(\partial_{r}\psi)^{2}},
\end{equation*}
Therefore,
\begin{equation*}
\int_{\mathcal{H}^{+}\cap\Sigma}{\psi^{2}}\leq \frac{1}{\epsilon}\int_{\Sigma}{J_{\mu}^{T}n^{\mu}_{\Sigma}}+\epsilon\int_{\Sigma}{J_{\mu}^{N}n^{\mu}_{\Sigma}},
\end{equation*}
which completes the proof.
\end{proof}

\section{Uniform Boundedness of Local Observer's Energy}
\label{sec:UniformBoundednessOfLocalObserverSEnergy}
We have all tools in place in order to prove the following theorem
\begin{theorem}
There exists a  constant $C>0$ which depends  on $M$ and $\Sigma_{0}$ such that for all solutions $\psi$ of the wave equation 
\begin{equation}
\int_{\Sigma_{\tau}}{J_{\mu}^{N}(\psi)n^{\mu}_{\Sigma_{\tau}}}\leq C\int_{\Sigma_{0}}{J_{\mu}^{N}(\psi)n^{\mu}_{\Sigma_{0}}}.
\label{nener}
\end{equation}
\label{nenergy}
\end{theorem}
\begin{proof}
Stokes' theorem for the current $J_{\mu}^{N,\delta, -\frac{1}{2}}$ in region $\mathcal{R}(0,\tau)$ gives us
\begin{equation*}
\int_{\Sigma_{\tau}}{J_{\mu}^{N,\delta, -\frac{1}{2}}n^{\mu}_{\Sigma_{\tau}}}+\int_{\mathcal{R}}{K^{N,\delta, -\frac{1}{2}}}+\int_{\mathcal{H}^{+}}{J_{\mu}^{N,\delta, -\frac{1}{2}}n^{\mu}_{\mathcal{H}^{+}}}= \int_{\Sigma_{0}}{J_{\mu}^{N,\delta, -\frac{1}{2}}n^{\mu}_{\Sigma_{0}}}.
\end{equation*}
First observe that the right hand side is controlled by the right hand side of \eqref{nboundary1}. As regards the left hand side, the boundary integrals can be estimated using Propositions \ref{nb} and \ref{nhb}. The spacetime term is non-negative (and thus has the right sign) in region $\mathcal{A}$ , vanishes in region $\mathcal{C}$ and can be estimated in the spatially compact region $\mathcal{B}$ (which does not contain the photon sphere) by the X estimate \eqref{degX} and Morawetz estimate \eqref{moraw}. The result follows from the boundedness of $T$-flux through $\Sigma_{\tau}$.
\end{proof}

\begin{corollary}
There exists a  constant $C>0$ which depends  on $M$ and $\Sigma_{0}$ such that for all solutions $\psi$ of the wave equation 
\begin{equation}
\int_{\mathcal{A}}{K^{N,-\frac{1}{2}}(\psi)}\leq C\int_{\Sigma_{0}}{J_{\mu}^{N}(\psi)n^{\mu}_{\Sigma_{0}}}.
\label{nk}\end{equation}\label{nkcor}
\end{corollary}
This corollary and \eqref{x} give us a spacetime integral where the only weight that locally degenerates (to first order) is that of the derivative tranversal to $\mathcal{H}^{+}$. Recall that in the  subextreme Reissner-Nordstr\"{o}m case there is no such degeneration. In the next section, this degeneracy is removed provided $\psi_{0}=0$. This condition is necessary as is shown in Section \ref{sec:ConservationLawsOnDegenerateEventHorizons}.

\begin{corollary}
There exists a  constant $C>0$ which depends  on $M$ and $\Sigma_{0}$ such that for all solutions $\psi$ of the wave equation 
\begin{equation}
\int_{\mathcal{H}^{+}}{J_{\mu}^{N}(\psi)n^{\mu}_{\mathcal{H}^{+}}}\leq C\int_{\Sigma_{0}}{J_{\mu}^{N}(\psi)n^{\mu}_{\Sigma_{0}}}.
\label{nhh}
\end{equation}
\end{corollary}
One application of the above theorem is the following  Morawetz estimate which does not degenerate at $\mathcal{H}^{+}$.
\begin{proposition}
There exists a  constant $C>0$ which depends  on $M$ and $\Sigma_{0}$ such that for all solutions $\psi$ of the wave equation 
\begin{equation}
\int_{\mathcal{A}}{\psi^{2}}\leq C\int_{\Sigma_{0}}{J_{\mu}^{N}(\psi)n^{\mu}_{\Sigma_{0}}}.
\label{nodmoraw}
\end{equation}
\label{moranondeg}
\end{proposition}
\begin{proof}
The third Hardy inequality gives us
\begin{equation*}
\int_{\mathcal{A}}{\psi^{2}}\leq C\int_{\mathcal{B}}{\psi^{2}}+C\int_{\mathcal{A}\cup\mathcal{B}}{D\left((\partial_{v}\psi)^{2}+(\partial_{r}\psi)^{2},\right)}
\end{equation*}
where $C$ is a uniform positive constant that depends only on $M$ and the regions $\mathcal{A}, \mathcal{B}$. The integral over $\mathcal{B}$ of $\psi^{2}$ can be estimated using \eqref{moraw} and the last integral on the right hand side can be estimated using \eqref{degX} and Propositions  \ref{knprop} and Corollary \ref{nkcor}.
\end{proof}

\begin{remark}
Note that all the above estimates hold if we replace the foliation $\Sigma_{\tau}$ with the foliation $\tilde{\Sigma}_{\tau}$ which terminates at $\mathcal{I}^{+}$ since the only difference is a boundary integral over $\mathcal{I}^{+}$ of the right sign that arises every time we apply Stokes' theorem. The remaining local estimates are exactly the same. However, the advantage of $\tilde{\Sigma}_{\tau}$ over $\Sigma_{\tau}$ will become apparent in Section \ref{sec:EnergyDecay}.
\end{remark}

\section{Commuting with a Vector Field Transversal to $\mathcal{H}^{+}$}
\label{sec:CommutingWithAVectorFieldTransversalToMathcalH}

If we commute the wave equation with $T=\partial_{v}$ then we obtain  the previous estimates  for the second order derivatives of $\psi$ which involve  $\partial_{v}$. Next we commute the wave equation $\Box_{g}\psi=0$ with the transversal to the horizon vector field $\partial_{r}$ aiming at controlling all the second derivatives of $\psi$ (on the spacelike hypersurfaces and the spacetime region up to and including the horizon $\mathcal{H}^{+}$). Such commutations first appeared in \cite{dr7} and were used in a more general setting for subextreme black holes in \cite{md}.

\subsection{Commutation with the Vector Field $\partial_{r}$}
\label{sec:CommutationWithTheVectorFieldPartialR}

We  compute the commutator $\left[\Box_{g},\partial_{r}\right]$. First note that 
\begin{equation*}
\Box_{g}\left(\partial_{r}\psi\right)=D\partial_{r}\partial_{r}\partial_{r}\psi+2\partial_{v}\partial_{r}\partial_{r}\psi+\frac{2}{r}\partial_{v}\partial_{r}\psi+R\partial_{r}\partial_{r}\psi+\lapp\partial_{r}\psi,
\end{equation*}
where $R=D'+\frac{D}{2r}$, $D'=\frac{d D}{dr}$, and
\begin{equation*}
\begin{split}
\partial_{r}\left(\Box_{g}\psi\right)&=D\partial_{r}\partial_{r}\partial_{r}\psi+2\partial_{r}\partial_{v}\partial_{r}\psi+\frac{2}{r}\partial_{r}\partial_{v}\psi+R\partial_{r}\partial_{r}\psi+\partial_{r}\lapp\psi\\
&\ \ \ \ +D'\partial_{r}\partial_{r}\psi-\frac{2}{r^{2}}\partial_{v}\psi+R'\partial_{r}\psi.\\
\end{split}
\end{equation*}
Therefore,
\begin{equation*}
\begin{split}
\Box_{g}\left(\partial_{r}\psi\right)-\partial_{r}\left(\Box_{g}\psi\right)=\lapp\partial_{r}\psi-\partial_{r}\lapp\psi-D'\partial_{r}\partial_{r}\psi+\frac{2}{r^{2}}\partial_{v}\psi-R'\partial_{r}\psi.
\end{split}
\end{equation*}
Since
\begin{equation}
\begin{split}
\left[\lapp,\partial_{r}\right]\psi=\frac{2}{r}\lapp\psi,
\label{comsphr}
\end{split}
\end{equation}
we obtain
\begin{equation*}
\begin{split}
\left[\Box_{g},\partial_{r}\right]\psi=-D'\partial_{r}\partial_{r}\psi+\frac{2}{r^{2}}\partial_{v}\psi-R'\partial_{r}\psi+\frac{2}{r}\lapp\psi.
\end{split}
\end{equation*}
In case  $\psi$ solves the wave equation $\Box_{g}\psi=0$ we have
\begin{equation}
\Box_{g}\left(\partial_{r}\psi\right)=D'\partial_{r}\partial_{r}\psi+\frac{2}{r^{2}}\partial_{v}\psi-R'\partial_{r}\psi+\frac{2}{r}\lapp\psi.
\label{comm1}
\end{equation}

\subsection{The Multiplier $L$ and the Energy Identity}
\label{sec:TheMultiplierLAndTheEnergyIdentity}

For any solution $\psi$ of the wave equation we have complete control of the second order derivatives of $\psi$ away from  $\mathcal{H}^{+}$ since
\begin{equation}
\begin{split}
\left\|\psi\right\|^{2}_{\overset{.}{H}^{2}\left(\Sigma_{\tau}\cap\left\{M<r_{0}\leq r\leq r_{1}<2M\right\}\right)}&=\left\|\left|\nabla\nabla\psi\right|\right\|^{2}_{L^{2}\left(\Sigma_{\tau}\cap\left\{M<r_{0}\leq r\right\}\right)}\\ &\leq C\int_{\Sigma_{0}}{J^{T}_{\mu}\left(\psi\right)n^{\mu}_{\Sigma_{0}}}+C\int_{\Sigma_{0}}{J^{T}_{\mu}\left(T\psi\right)n^{\mu}_{\Sigma_{0}}},
\end{split}
\label{hypersurawayH}
\end{equation}
where $C$ depends  on $M$ and $\Sigma_{0}$ and $\nabla\nabla$ denotes a 2-tensor on the Riemannian manifold $\Sigma_{\tau}$ and $\left|-\right|$ denotes its induced norm (see Appendix \ref{sec:EllipticEstimates}). Note that we have used the first Hardy inequality to control the zeroth order terms. Similarly, away from $\mathcal{H}^{+}$, we  have control of the bulk integrals of the second order derivatives since
\begin{equation}
\left\|\partial_{a}\partial_{b}\psi\right\|^{2}_{L^{2}\left(\mathcal{R}\left(0,\tau\right)\cap\left\{M<r_{0}\leq r\leq r_{1}<2M\right\}\right)}\leq C\int_{\Sigma_{0}}{J^{T}_{\mu}\left(\psi\right)n^{\mu}_{\Sigma_{0}}}+C\int_{\Sigma_{0}}{J^{T}_{\mu}\left(T\psi\right)n^{\mu}_{\Sigma_{0}}}
\label{bulkawayH}
\end{equation}
where $C$ depends  on $M$, $r_{0}$, $r_{1}$ and $\Sigma_{0}$ and $a,b\in\left\{v,r,\nabb \right\}$. This is proved using \eqref{moraw} and local elliptic estimates (appendix \ref{sec:EllipticEstimates}). 

In order to understand the behaviour of the second order derivatives of the wave $\psi$ in a neighbourhood of  $\mathcal{H}^{+}$ we will use the vector field method.  We will construct an appropriate future directed timelike $\phi_{\tau}$-invariant vector field 
\begin{equation*}
L=f_{v}\partial_{v}+f_{r}\partial_{r},
\end{equation*} 
which will be used as our multiplier. In view of our discussion above, the vector field $L$ will be of interest only in a neighbourhood of the horizon. That is why, $L$ will be spatially compactly supported. More precisely, we will construct $L$ such that $L=\textbf{0}$ in $r\geq r_{1}$ and $L$ timelike in the region $M\leq r < r_{1}$. In particular, we will be interested in the region $M\leq r\leq r_{0}<r_{1}$. Note that $r_{0},r_{1}$ are constants which will be determined later on. The regions $\mathcal{A},\mathcal{B}$ are depicted below
 \begin{figure}[H]
	\centering
		\includegraphics[scale=0.14]{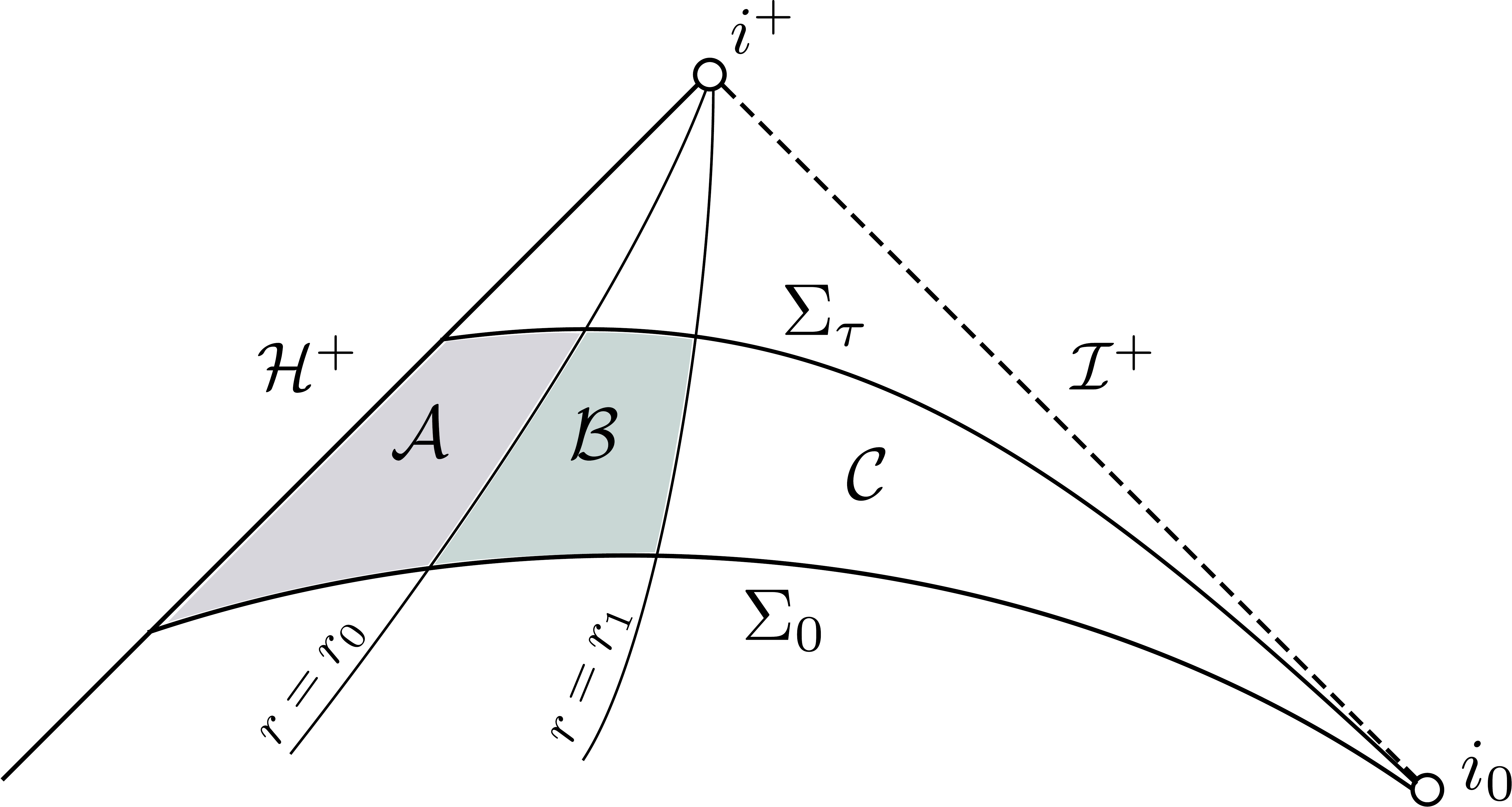}
	\label{fig:ernt7}
\end{figure}
For simplicity we will write $\mathcal{R}$ instead of $\mathcal{R}\left(0,\tau\right)$, $\mathcal{A}$ instead of $\mathcal{A}\left(0,\tau\right)$, etc. 
The ``energy'' identity for the current $J_{\mu}^{L}\left(\partial_{r}\psi\right)$ is 
\begin{equation}
\int_{\Sigma_{\tau}}{J_{\mu}^{L}\left(\partial_{r}\psi\right)n^{\mu}_{\Sigma_{\tau}}}+\int_{\mathcal{R}}{\nabla^{\mu}J_{\mu}^{L}\left(\partial_{r}\psi\right)}+\int_{\mathcal{H}^{+}}{J_{\mu}^{L}\left(\partial_{r}\psi\right)n^{\mu}_{\mathcal{H}^{+}}}=\int_{\Sigma_{0}}{J_{\mu}^{L}\left(\partial_{r}\psi\right)n^{\mu}_{\Sigma_{0}}}.
\label{eiL}
\end{equation}
The right hand side is controlled by the initial data and thus bounded. Also since $L$ is timelike in the compact region $\mathcal{A}$ we have from \eqref{GENERALT} of Appendix \ref{sec:TheHyperbolicityOfTheWaveEquation1}:
\begin{equation}
J_{\mu}^{L}\left(\partial_{r}\psi\right)n^{\mu}_{\Sigma_{\tau}}\sim \left(\partial_{v}\partial_{r}\psi\right)^{2}+\left(\partial_{r}\partial_{r}\psi\right)^{2}+\left|\nabb\partial_{r}\psi\right|^{2},
\label{Lhyper1}
\end{equation}
where the constants in $\sim$ depend on $M$, $\Sigma_{0}$ and $L$. Furthermore, on the horizon (where again $L$ is timelike) we have 
\begin{equation}
J_{\mu}^{L}\left(\partial_{r}\psi\right)n^{\mu}_{\mathcal{H}^{+}}=f_{v}(M)(\partial_{v}\partial_{r}\psi)^{2}-\frac{f_{r}(M)}{2}\left|\nabb\partial_{r}\psi\right|^{2}. 
\label{horizonL1}
\end{equation}
Therefore, the term that remains to be understood is the bulk integral. Since $\partial_{r}\psi$ does not satisfy the wave equation, we have
\begin{equation*}
\begin{split}
\nabla^{\mu}J_{\mu}^{L}\left(\partial_{r}\psi\right)&=K^{L}\left(\partial_{r}\psi\right)+\mathcal{E}^{L}\left(\partial_{r}\psi\right)\\
&=K^{L}\left(\partial_{r}\psi\right)+\left(\Box_{g}\left(\partial_{r}\psi\right)\right)L\left(\partial_{r}\psi\right).
\end{split}
\end{equation*}
We know that
\begin{equation*}
\begin{split}
K^{L}\left(\partial_{r}\psi\right)=F_{vv}\left(\partial_{v}\partial_{r}\psi\right)^{2}+F_{rr}\left(\partial_{r}\partial_{r}\psi\right)^{2}+F_{\scriptsize{\nabb}}\left|\nabb\partial_{r}\psi\right|^{2}+F_{vr}\left(\partial_{v}\partial_{r}\psi\right)\left(\partial_{r}\partial_{r}\psi\right)
\end{split}
\end{equation*}
where
\begin{equation*}
\begin{split}
&F_{vv}=\left(\partial_{r}f_{v}\right),\\
&F_{rr}=D\left[\frac{\left(\partial_{r}f_{r}\right)}{2}-\frac{f_{r}}{r}\right]-D'\cdot\frac{f_{r}}{2},  \\
&F_{\scriptsize{\nabb}}=-\frac{1}{2}\left(\partial_{r}f_{r}\right),\\
&F_{vr}=D\left(\partial_{r}f_{v}\right)-\frac{2f_{r}}{r}.\\
\end{split}
\end{equation*}
In view of equation \eqref{comm1} we have
\begin{equation*}
\begin{split}
\mathcal{E}^{L}\left(\partial_{r}\psi\right)&=\left(\Box_{g}\left(\partial_{r}\psi\right)\right)L\left(\partial_{r}\psi\right)\\
=&\left[D'\partial_{r}\partial_{r}\psi+\frac{2}{r^{2}}\partial_{v}\psi-R'\partial_{r}\psi+\frac{2}{r}\lapp\psi\right]L\left(\partial_{r}\psi\right)\\
=&-D'
f_{r}\left(\partial_{r}\partial_{r}\psi\right)^{2}-D'f_{v}\left(\partial_{v}\partial_{r}\psi\right)\left(\partial_{r}\partial_{r}\psi\right)-R'f_{v}\left(\partial_{v}\partial_{r}\psi\right)\left(\partial_{r}\psi\right)\\
&+2\frac{f_{v}}{r^{2}}\left(\partial_{v}\partial_{r}\psi\right)\left(\partial_{v}\psi\right)+2\frac{f_{r}}{r^{2}}\left(\partial_{r}\partial_{r}\psi\right)\left(\partial_{v}\psi\right)-R'f_{r}\left(\partial_{r}\partial_{r}\psi\right)\left(\partial_{r}\psi\right)\\
&+2\frac{f_{v}}{r}(\partial_{v}\partial_{r}\psi)\lapp\psi+2\frac{f_{r}}{r}(\partial_{r}\partial_{r}\psi)\lapp\psi.
\end{split}
\end{equation*}
Therefore, we can write
\begin{equation*}
\begin{split}
\nabla^{\mu}J_{\mu}^{L}\left(\partial_{r}\psi\right)=&H_{1}\left(\partial_{v}\partial_{r}\psi\right)^{2}+H_{2}\left(\partial_{r}\partial_{r}\psi\right)^{2}+H_{3}\left|\nabb\partial_{r}\psi\right|^{2}+\\
&+H_{4}\left(\partial_{v}\partial_{r}\psi\right)\left(\partial_{v}\psi\right)+H_{5}\left(\partial_{v}\partial_{r}\psi\right)\left(\partial_{r}\psi\right)\\&+H_{6}\left(\partial_{r}\partial_{r}\psi\right)\left(\partial_{v}\psi\right)+H_{7}(\partial_{v}\partial_{r}\psi)\lapp\psi\\&+H_{8}(\partial_{r}\partial_{r}\psi)\lapp\psi+H_{9}\left(\partial_{v}\partial_{r}\psi\right)\left(\partial_{r}\partial_{r}\psi\right)+H_{10}\left(\partial_{r}\partial_{r}\psi\right)\left(\partial_{r}\psi\right),
\end{split}
\end{equation*}
where the coefficients $H_{i},i=1,...,10$ are given by
\begin{equation}
\begin{split}
&H_{1}=\left(\partial_{r}f_{v}\right),\  H_{2}=D\left[\frac{\left(\partial_{r}f_{r}\right)}{2}-\frac{f_{r}}{r}\right]-\frac{3D'}{2}f_{r},\  H_{3}=-\frac{1}{2}\left(\partial_{r}f_{r}\right),\\
&H_{4}=+2\frac{f_{v}}{r^{2}}, \  H_{5}=-f_{v}R', \ H_{6}=+2\frac{f_{r}}{r^{2}},\  H_{7}=2\frac{f_{v}}{r},\  H_{8}=2\frac{f_{r}}{r},\\
&H_{9}=D\left(\partial_{r}f_{v}\right)-D'f_{v}-2\frac{f_{r}}{r} ,\  H_{10}=-f_{r}R'.
\end{split}
\label{listH}
\end{equation}
Since $L$ is a future directed timelike vector field we have $f_{v}\left(r\right)>0$ and $f_{r}\left(r\right)<0$. By taking $\partial_{r}f_{v}\left(M\right)$ sufficiently large we can make $H_{1}$ positive close to the horizon $\mathcal{H}^{+}$. Also since the term $D$ vanishes on the horizon to second order and the terms $R, D'$ to first order and since $f_{r}\left(M\right)<0$, the coefficient $H_{2}$ is positive close to $\mathcal{H}^{+}$ (and vanishes to first order on it). For the same reason we have $H_{9}D\leq\frac{H_{2}}{10}$  and $\left(H_{9}R\right)^{2}\leq\frac{H_{2}}{10}$ close to $\mathcal{H}^{+}$. Moreover, by taking  $-\partial_{r}f_{r}\left(M\right)$ sufficiently large we can also make the coefficient $H_{3}$ positive close to $\mathcal{H}^{+}$ such that $H_{9}<\frac{H_{3}}{10}$. Indeed, it suffices to consider $f_{r}$ such that $-\frac{f_{r}(M)}{M}<-\frac{\partial_{r}f_{r}(M)}{25}$ and then by continuity we have the previous inequality close to  $\mathcal{H}^{+}$. Therefore, we consider $M<r_{0}<2M$ such that in region $\mathcal{A}=\left\{M\leq r\leq r_{0}\right\}$ we have 
\begin{equation}
\begin{split}
& f_{v}> 1,\  \partial_{r}f_{v}> 1,\  -f_{r}>1,\  H_{1}>1,\  H_{2}\geq 0,\  H_{3}> 1,\\
& H_{8}<\frac{H_{3}}{10},\  H_{9}D\leq\frac{H_{2}}{10},\ \left(H_{9}R\right)^{2}\leq\frac{H_{2}}{10},\  H_{9}<\frac{H_{3}}{10}.
\end{split}
\label{listL}
\end{equation}
Clearly, $r_{0}$ depends only on $M$ and the precise choice for $L$ close to $\mathcal{H}^{+}$. In order to define $L$ globally, we just extend $f_{v}$ and $f_{r}$ such that 
\begin{equation*}
\begin{split}
&f_{v}>0 \text{ for all } r<r_{1}  \text{ and } f_{v}=0 \text{ for all } r\geq r_{1},\\
-&f_{r}>0 \text{ for all } r<r_{1} \text{ and } f_{r}=0 \text{ for all } r\geq r_{1},
\end{split}
\end{equation*}
for some $r_{1}$ such that $r_{0}<r_{1}<2M$. Again, $r_{1}$ depends only on $M$ (and the precise choice for $L$). Clearly, $L$ depends only on $M$ and thus all the functions that involve the components of $L$ depend only on $M$.

\subsection{Estimates of the Spacetime Integrals}
\label{sec:EstimatesOfTheBulkTerms}

It suffices to estimate the remaining 7 integrals with coef{}ficients $H_{i}$'s with $i=4,...,10$. Note that all these coef{}ficiens do not vanish on the horizon. We will prove that each of these integrals can by estimated by the $N$ flux for $\psi$ and $T\psi$ and a small (epsilon) portion of the good terms in $K^{L}(\partial_{r}\psi)$. First we prove the following propositions.

\begin{proposition}
For all solutions $\psi$ of the wave equation and  any positive number $\epsilon$ we have
\begin{equation*}
\begin{split}
\int_{\mathcal{A}}{(\partial_{r}\psi)^{2}}\leq &\int_{\mathcal{A}}{\epsilon H_{1}(\partial_{v}\partial_{r}\psi)^{2}+\epsilon H_{2}(\partial_{r}\partial_{r}\psi)^{2}}\\ &+C_{\epsilon}\int_{\Sigma_{0}}{J_{\mu}^{N}(\psi)n^{\mu}_{\Sigma_{0}}}+C_{\epsilon}\int_{\Sigma_{0}}{J_{\mu}^{N}(T\psi)n^{\mu}_{\Sigma_{0}}},
\end{split}
\end{equation*}
where the constant $C_{\epsilon}$ depends  on $M$, $\Sigma_{0}$ and $\epsilon$.
\label{1comprop}
\end{proposition}
\begin{proof}
By applying the third Hardy inequality for the regions $\mathcal{A},\mathcal{B}$ we obtain
\begin{equation*}
\int_{\mathcal{A}}{(\partial_{r}\psi)^{2}}\leq C\int_{\mathcal{B}}{(\partial_{r}\psi)^{2}}+C\int_{\mathcal{A}\cup\mathcal{B}}{D\left[(\partial_{v}\partial_{r}\psi)^{2}+(\partial_{r}\partial_{r}\psi)^{2}\right]},
\end{equation*}
where the constant $C$ depends  on $M$ and $\Sigma_{0}$. Moreover, since $CD$ vanishes to second order at $\mathcal{H}^{+}$, there exists $r_{\epsilon}$ with  $M<r_{\epsilon}\leq r_{0}$ such that  in the region $\left\{M\leq r\leq r_{\epsilon}\right\}$ we have $CD<\epsilon H_{1}$ and $CD\leq \epsilon H_{2}$. Therefore,
\begin{equation*}
\begin{split}
\int_{\mathcal{A}}{(\partial_{r}\psi)^{2}}\leq &C\int_{\mathcal{B}}{(\partial_{r}\psi)^{2}}+\int_{\left\{M\leq r\leq r_{\epsilon}\right\}}{\epsilon H_{1}(\partial_{v}\partial_{r}\psi)^{2}+\epsilon H_{2}(\partial_{r}\partial_{r}\psi)^{2}}\\
&+\int_{\left\{r_{\epsilon}\leq r\leq r_{1}\right\}}{D\left[(\partial_{v}\partial_{r}\psi)^{2}+(\partial_{r}\partial_{r}\psi)^{2}\right]}.
\end{split}
\end{equation*}
The first integral on the right hand side is estimated using \eqref{degX} and the last integral using the local elliptic estimate \eqref{bulkawayH}  since  T is timelike in region $\left\{r_{\epsilon}\leq r \leq r_{1}\right\}$ since $M<r_{\epsilon}$. The result follows from the inclusion $\left\{M\leq r \leq r_{\epsilon}\right\}\subseteq\mathcal{A}$ and the non-negativity of $H_{i},i=1,2$ in $\mathcal{A}$.
\end{proof}
\begin{proposition}
For all solutions $\psi$ of the wave equation and any positive number $\epsilon$ we have 
\begin{equation*}
\left|\int_{\mathcal{H}^{+}}(\partial_{v}\psi)(\partial_{r}\psi)\right|\leq C_{\epsilon}\int_{\Sigma_{0}}{J^{N}_{\mu}(\psi)n^{\mu}_{\Sigma_{0}}}+\epsilon\int_{\Sigma_{0}\cup\Sigma_{\tau}}{J^{L}_{\mu}(\partial_{r}\psi)n^{\mu}_{\Sigma}}
\end{equation*}
where the positive constant $C_{\epsilon}$ depends  on $M$, $\Sigma_{0}$ and $\epsilon$.
\label{2comprop}
\end{proposition}
\begin{proof}
Integrating by parts gives us
\begin{equation*}
\begin{split}
\int_{\mathcal{H}^{+}}{(\partial_{v}\psi)(\partial_{r}\psi)}=-\int_{\mathcal{H}^{+}}{\psi(\partial_{v}\partial_{r}\psi)}+\int_{\mathcal{H}^{+}\cap\Sigma_{\tau}}{\psi(\partial_{r}\psi)}-\int_{\mathcal{H}^{+}\cap\Sigma_{0}}{\psi(\partial_{r}\psi)}.
\end{split}
\end{equation*}
Since $\psi$ solves the wave equation, it satisfies
\begin{equation*}
\begin{split}
-\partial_{v}\partial_{r}\psi=\frac{1}{M}(\partial_{v}\psi)
+\frac{1}{2}\lapp\psi
\end{split}
\end{equation*}
on $\mathcal{H}^{+}$. Therefore,
\begin{equation*}
\begin{split}
-\int_{\mathcal{H}^{+}}{\psi (\partial_{v}\partial_{r}\psi)}=&\frac{1}{M}\int_{\mathcal{H}^{+}}{\psi(\partial_{v}\psi)}+\frac{1}{2}\int_{\mathcal{H}^{+}}{\psi(\lapp\psi)}\\
=&\frac{1}{2M}\int_{\mathcal{H}^{+}\cap\Sigma_{\tau}}{\psi^{2}}-\frac{1}{2M}\int_{\mathcal{H}^{+}\cap\Sigma_{0}}{\psi^{2}}-\frac{1}{2}\int_{\mathcal{H}^{+}}{\left|\nabb\psi\right|^{2}}.
\end{split}
\end{equation*}
All the integrals on the right hand side can be estimated by $\displaystyle\int_{\Sigma_{0}}{J_{\mu}^{N}(\psi)n^{\mu}_{\Sigma_{0}}}$ using the second Hardy inequality  and \eqref{nhh}. Furthermore,
\begin{equation*}
\begin{split}
\int_{\mathcal{H}^{+}\cap\Sigma}{\psi(\partial_{r}\psi)}\leq & \int_{\mathcal{H}^{+}\cap\Sigma}{\psi^{2}}+\int_{\mathcal{H}^{+}\cap\Sigma}{(\partial_{r}\psi)^{2}}.
\end{split}
\end{equation*}
From the first and second Hardy inequality we have 
\begin{equation*}
\begin{split}
\int_{\mathcal{H}^{+}\cap\Sigma}{\psi^{2}}\leq C\int_{\Sigma_{0}}{J_{\mu}^{N}(\psi)n^{\mu}_{\Sigma_{0}}},
\end{split}
\end{equation*}
where $C$ depends  on $M$ and $\Sigma_{0}$. In addition, from the second Hardy inequality we have that for any positive number $\epsilon$ there exists a constant $C_{\epsilon}$ which depends on $M$, $\Sigma_{0}$ and $\epsilon$ such that 
\begin{equation*}
\begin{split}
\int_{\mathcal{H}^{+}\cap\Sigma}{(\partial_{r}\psi)^{2}}\leq C_{\epsilon}\int_{\Sigma}{J_{\mu}^{N}(\psi)n^{\mu}_{\Sigma}}+\epsilon\int_{\Sigma}{J_{\mu}^{L}(\partial_{r}\psi)n^{\mu}_{\Sigma}}.
\end{split}
\end{equation*}
\end{proof}

\begin{proposition}
For all solutions $\psi$ of the wave equation and any positive number $\epsilon$ we have 
\begin{equation*}
\left|\int_{\mathcal{H}^{+}}{(\lapp\psi)(\partial_{r}\psi)}\right|\leq  C_{\epsilon}\int_{\Sigma_{0}}{J^{N}_{\mu}(\psi)n^{\mu}_{\Sigma_{0}}}+\epsilon\int_{\Sigma_{0}\cup\Sigma_{\tau}}{J^{L}_{\mu}(\partial_{r}\psi)n^{\mu}_{\Sigma}},
\end{equation*}
where the positive constant $C_{\epsilon}$ depends  on $M$, $\Sigma_{0}$ and $\epsilon$.
\label{3comprop}
\end{proposition}
\begin{proof}
In view of the wave equation on $\mathcal{H}^{+}$ we have
\begin{equation*}
\int_{\mathcal{H}^{+}}{(\lapp\psi)(\partial_{r}\psi)}=-\frac{2}{M}\int_{\mathcal{H}^{+}}{(\partial_{v}\psi)(\partial_{r}\psi)}-2\int_{\mathcal{H}^{+}}{(\partial_{v}\partial_{r}\psi)(\partial_{r}\psi)}.
\end{equation*}
The first integral on the right hand side can be estimated using Proposition \ref{2comprop}. For the second integral we have
\begin{equation*}
\begin{split}
2\!\!\int_{\mathcal{H}^{+}}{(\partial_{v}\partial_{r}\psi)(\partial_{r}\psi)}=&\int_{\mathcal{H}^{+}}{\partial_{v}\left((\partial_{r}\psi)^{2}\right)}\\
=&\int_{\mathcal{H}^{+}\cap\Sigma_{\tau}}{(\partial_{r}\psi)^{2}}-\int_{\mathcal{H}^{+}\cap\Sigma_{0}}{(\partial_{r}\psi)^{2}}\\
\leq & C_{\epsilon}\!\int_{\Sigma_{0}}{J_{\mu}^{N}(\psi)n^{\mu}_{\Sigma_{0}}}+\epsilon\! \int_{\Sigma_{0}}{J_{\mu}^{L}(\partial_{r}\psi)n^{\mu}_{\Sigma_{0}}}+\epsilon\!\int_{\Sigma_{\tau}}{J^{L}_{\mu}(\partial_{r}\psi)n^{\mu}_{\Sigma_{\tau}}},
\end{split}
\end{equation*}
where, as above, $\epsilon$ is any positive number and $C_{\epsilon}$ depends on $M$, $\Sigma_{0}$ and $\epsilon$.

\end{proof}

\begin{center}
\large{\textbf{Estimate for} $\displaystyle\int_{\mathcal{R}}{H_{4}\left(\partial_{v}\partial_{r}\psi\right)\left(\partial_{v}\psi\right)}$}\\
\end{center}
For any $\epsilon >0$ we have
\begin{equation*}
\begin{split}
\left|\int_{\mathcal{A}}{H_{4}\left(\partial_{v}\partial_{r}\psi\right)\left(\partial_{v}\psi\right)}\right|&\leq \int_{\mathcal{A}}{\epsilon\left(\partial_{v}\partial_{r}\psi\right)^{2}+\int_{\Sigma_{0}}{\epsilon^{-1}H_{4}^{2}(\partial_{v}\psi)^{2}}}\\
&\leq \epsilon\int_{\mathcal{A}}{\left(\partial_{v}\partial_{r}\psi\right)^{2}+C_{\epsilon}\int_{\Sigma_{0}}{J^{N}_{\mu}n^{\mu}_{\Sigma_{0}}}},
\end{split}
\end{equation*}
where the constant $C_{\epsilon}$ depends only on $M$,$\Sigma_{0}$ and  $\epsilon$.

\begin{center}
\large{\textbf{Estimate for} $\displaystyle\int_{\mathcal{R}}{H_{5}\left(\partial_{v}\partial_{r}\psi\right)\left(\partial_{r}\psi\right)}$}\\
\end{center}
As above, for any $\epsilon >0$
\begin{equation*}
\begin{split}
\left|\int_{\mathcal{A}}{H_{5}\left(\partial_{v}\partial_{r}\psi\right)\left(\partial_{r}\psi\right)}\right|&\leq\int_{\mathcal{A}}{\epsilon\left(\partial_{v}\partial_{r}\psi\right)^{2}}+\int_{\mathcal{A}}{\epsilon^{-1}H_{5}^{2} (\partial_{r}\psi)^{2}}\\
&\leq \int_{\mathcal{A}}{\epsilon\left(\partial_{v}\partial_{r}\psi\right)^{2}}+m\int_{\mathcal{A}}{ (\partial_{r}\psi)^{2}},
\end{split}
\end{equation*}
where $m=\operatorname{max}_{\mathcal{A}}\epsilon^{-1}H_{5}^{2}$. Then  from Proposition \ref{1comprop} (where $\epsilon$ is replaced with $\frac{\epsilon}{m}$) we obtain
\begin{equation}
\begin{split}
\left|\int_{\mathcal{A}}{H_{5}\left(\partial_{v}\partial_{r}\psi\right)\left(\partial_{r}\psi\right)}\right|\leq &\int_{\mathcal{A}}{\epsilon\left(\partial_{v}\partial_{r}\psi\right)^{2}}+\int_{\mathcal{A}}{\epsilon H_{1}(\partial_{v}\partial_{r}\psi)^{2}+\epsilon H_{2}(\partial_{r}\partial_{r}\psi)^{2}}\\ &+C_{\epsilon}\int_{\Sigma_{0}}{J_{\mu}^{N}(\psi)n^{\mu}_{\Sigma_{0}}}+C_{\epsilon}\int_{\Sigma_{0}}{J_{\mu}^{N}(T\psi)n^{\mu}_{\Sigma_{0}}},
\end{split}
\label{H6}
\end{equation}
where $C_{\epsilon}$ depends on $M$, $\Sigma_{0}$ and  $\epsilon$.

\begin{center}
\large{\textbf{Estimate for} $\displaystyle\int_{\mathcal{R}}{H_{6}\left(\partial_{r}\partial_{r}\psi\right)\left(\partial_{v}\psi\right)}$}\\
\end{center}
Since
\begin{equation*}
\operatorname{Div} \partial_{r}=\frac{2}{r},
\end{equation*}
 Stokes' theorem yields
\begin{equation*}
\begin{split}
&\int_{\mathcal{R}}{H_{6}\left(\partial_{v}\psi\right)\left(\partial_{r}\partial_{r}\psi\right)}+\int_{\mathcal{R}}{\partial_{r}\left(H_{6}\partial_{v}\psi\right)\left(\partial_{r}\psi\right)}+\int_{\mathcal{R}}{H_{6}\left(\partial_{v}\psi\right)\left(\partial_{r}\psi\right)\frac{2}{r}}\\
=&\int_{\Sigma_{0}}{H_{6}\left(\partial_{v}\psi\right)\left(\partial_{r}\psi\right)\partial_{r}\cdot n_{\Sigma_{0}}}-\int_{\Sigma_{\tau}}{H_{6}\left(\partial_{v}\psi\right)\left(\partial_{r}\psi\right)\partial_{r}\cdot n_{\Sigma_{\tau}}}-\\
&-\int_{\mathcal{H}^{+}}{H_{6}\left(\partial_{v}\psi\right)\left(\partial_{r}\psi\right)\partial_{r}\cdot n_{\mathcal{H}^{+}}}.
\end{split}
\end{equation*}
In view of the boundedness of the non-degenerate energy we have
\begin{equation*}
\left|\int_{\Sigma_{0}}{H_{7}\left(\partial_{v}\psi\right)\left(\partial_{r}\psi\right)\partial_{r}\cdot n_{\Sigma_{0}}}-\int_{\Sigma_{\tau}}{H_{7}\left(\partial_{v}\psi\right)\left(\partial_{r}\psi\right)\partial_{r}\cdot n_{\Sigma_{\tau}}}\right|\leq C\left(\int_{\Sigma_{0}}{J^{N}_{\mu}n^{\mu}_{\Sigma_{0}}}\right)
\end{equation*}
for some uniform positive constant $C$ which depends  on $M$ and $\Sigma_{0}$. The boundary integral over $\mathcal{H}^{+}$ can be estimated using Proposition \ref{2comprop}.

If $Q=\partial_{r}H_{6}+\frac{2}{r}H_{6}$, then it remains to estimate
\begin{equation*}
\begin{split}
\int_{\mathcal{R}}{Q\left(\partial_{v}\psi\right)\left(\partial_{r}\psi\right)}+\int_{\mathcal{R}}{H_{6}\left(\partial_{v}\partial_{r}\psi\right)\left(\partial_{r}\psi\right)}.
\end{split}
\end{equation*}
For the first integral we have 
\begin{equation*}
\begin{split}
\left|\int_{\mathcal{A}}{Q\left(\partial_{v}\psi\right)\left(\partial_{r}\psi\right)}\right|\leq\int_{\mathcal{A}}{ Q^{2}\left(\partial_{v}\psi\right)^{2}}+\int_{\mathcal{A}}{(\partial_{r}\psi)^{2}},
\end{split}
\end{equation*}
which can be estimated using \eqref{nk} and Proposition \ref{1comprop}. The last spacetime integral is estimated by \eqref{H6} (where $H_{5}$ is replaced with $H_{6}$).

\begin{center}
\large{\textbf{Estimate for} $\displaystyle\int_{\mathcal{R}}{H_{7}\left(\partial_{v}\partial_{r}\psi\right)\lapp\psi}$}\\
\end{center}

Stokes' theorem in region $\mathcal{R}$ gives us
\begin{equation*}
\begin{split}
&\int_{\mathcal{R}}{H_{7}(\partial_{v}\partial_{r}\psi)\lapp\psi}+\int_{\mathcal{R}}{H_{7}(\partial_{r}\psi)(\lapp\partial_{v}\psi)}\\
=&\int_{\Sigma_{0}}{H_{7}\left(\partial_{r}\psi\right)(\lapp\psi)\partial_{v}\cdot n_{\Sigma_{0}}}-\int_{\Sigma_{\tau}}{H_{7}\left(\partial_{r}\psi\right)\left(\lapp\psi\right)\partial_{v}\cdot n_{\Sigma_{\tau}}}.
\end{split}
\end{equation*}
For the boundary integrals we have the estimate
\begin{equation*}
\begin{split}
\int_{\mathcal{A}\cap\Sigma}{H_{7}(\partial_{r}\psi)\lapp\psi}=&-\int_{\mathcal{A}\cap\Sigma}{H_{7}\nabb\partial_{r}\psi\cdot\nabb\psi}\\
&\leq\epsilon\int_{\mathcal{A}\cap\Sigma}{J_{\mu}^{L}(\partial_{r}\psi)n^{\mu}_{\Sigma}}+C_{\epsilon} \int_{\Sigma}{J^{N}_{\mu}(\psi)n^{\mu}_{\Sigma}},
\end{split}
\end{equation*}
where $C_{\epsilon}$ depends on $M$, $\Sigma_{0}$ and $\epsilon$. As regards the second bulk integral, after applying Stokes' theorem on $\mathbb{S}^{2}(r)$ we obtain
\begin{equation*}
\begin{split}
\int_{\mathcal{A}}{H_{7}\nabb\partial_{r}\psi\cdot\nabb\partial_{v}\psi}\leq \epsilon\int_{\mathcal{A}}{\left|\nabb\partial_{r}\psi\right|^{2}}+C_{\epsilon}\int_{\mathcal{A}}\left|\nabb\partial_{v}\psi\right|^{2},
\end{split}
\end{equation*}
where $C_{\epsilon}$ depends on $M$, $\Sigma_{0}$ and $\epsilon$. Note that the second integral on the right hand side can be bounded by $\displaystyle\int_{\Sigma_{0}}{J_{\mu}^{N}(T\psi)n^{\mu}_{\Sigma_{0}}}$. Indeed, we commute with $T$ and use Proposition \ref{nk}. (Another way without having to commute with $T$ is by solving with respect to $\lapp\psi$ in the wave equation. As we shall see, this will be crucial in obtaining higher order estimates without losing derivatives.)

\begin{center}
\large{\textbf{Estimate for} $\displaystyle\int_{\mathcal{R}}{H_{8}\left(\partial_{r}\partial_{r}\psi\right)\lapp\psi}$}\\
\end{center}

We have
\begin{equation*}
\begin{split}
&\int_{\mathcal{R}}{H_{8}\left(\lapp\psi\right)\left(\partial_{r}\partial_{r}\psi\right)}+\int_{\mathcal{R}}{\partial_{r}\left(H_{8}\lapp\psi\right)\left(\partial_{r}\psi\right)}+\int_{\mathcal{R}}{H_{8}\left(\lapp\psi\right)\left(\partial_{r}\psi\right)\frac{2}{r}}\\
=&\int_{\Sigma_{0}}{H_{8}\left(\lapp\psi\right)\left(\partial_{r}\psi\right)\partial_{r}\cdot n_{\Sigma_{0}}}-\int_{\Sigma_{\tau}}{H_{8}\left(\lapp\psi\right)\left(\partial_{r}\psi\right)\partial_{r}\cdot n_{\Sigma_{\tau}}}-\\
&\ \ \ \ -\int_{\mathcal{H}^{+}}{H_{8}\left(\lapp\psi\right)\left(\partial_{r}\psi\right)\partial_{r}\cdot n_{\mathcal{H}^{+}}}.
\end{split}
\end{equation*}
The integral over $\mathcal{H}^{+}$ is estimated in Proposition \ref{3comprop}. Furthermore, the Cauchy-Schwarz inequality implies
\begin{equation*}
\begin{split}
\int_{\Sigma_{\tau}\cap\mathcal{A}}{H_{8}\left(\lapp\psi\right)\left(\partial_{r}\psi\right)\partial_{r}\cdot n_{\Sigma_{\tau}}}=&-\int_{\Sigma_{\tau}\cap\mathcal{A}}{H_{8}\left(\nabb\psi\cdot\nabb\partial_{r}\psi\right)\partial_{r}\cdot n_{\Sigma_{\tau}}}\\
\leq & C_{\epsilon}\int_{\Sigma_{0}}{J_{\mu}^{N}(\psi)n^{\mu}_{\Sigma_{0}}}+\epsilon\int_{\Sigma_{0}}{J_{\mu}^{L}(\partial_{r}\psi)n^{\mu}_{\Sigma_{0}}},
\end{split}
\end{equation*}
where $\epsilon$ is any positive number and $C_{\epsilon}$ depends only on $M$, $\Sigma_{0}$ and $\epsilon$. For the remaining two spacetime integrals we have
\begin{equation*}
\begin{split}
&\int_{\mathcal{A}}{\partial_{r}\left(H_{8}\lapp\psi\right)\left(\partial_{r}\psi\right)}+\int_{\mathcal{A}}{\frac{2}{r}H_{8}\left(\lapp\psi\right)\left(\partial_{r}\psi\right)}\\
=&\int_{\mathcal{A}}{H_{8}\left(\lapp\partial_{r}\psi-\left[\lapp,\partial_{r}\psi\right]\right)\left(\partial_{r}\psi\right)}+\int_{\mathcal{A}}{\left(\partial_{r}H_{8}+\frac{2}{r}H_{8}\right)\left(\lapp\psi\right)\left(\partial_{r}\psi\right)}\\
\overset{\eqref{comsphr}}{=}&\int_{\mathcal{A}}{H_{8}\left(\lapp\partial_{r}\psi\right)\left(\partial_{r}\psi\right)}+\int_{\mathcal{A}}{\left(\partial_{r}H_{8}\right)\left(\lapp\psi\right)\left(\partial_{r}\psi\right)}\\
=&\int_{\mathcal{A}}{-H_{8}\left|\nabb\partial_{r}\psi\right|^{2}}+\int_{\mathcal{A}}{-\partial_{r}H_{8}\left(\nabb\psi\cdot\nabb\partial_{r}\psi\right)}
\end{split}
\end{equation*}
The conditions \eqref{listL} assure us that 
\begin{equation*}
\left|H_{8}\right|<\frac{H_{3}}{10}
\end{equation*}
and thus the first integral above can be estimated. For the second integral we apply Cauchy-Schwarz and take
\begin{equation*}
\int_{\mathcal{A}}{-\partial_{r}H_{8}\left(\nabb\psi\cdot\nabb\partial_{r}\psi\right)}\leq C_{\epsilon}\int_{\Sigma_{0}}{J_{\mu}^{N}(\psi)n^{\mu}_{\Sigma_{0}}}+\epsilon\int_{\mathcal{A}}{\left|\nabb\partial_{r}\psi\right|^{2}},
\end{equation*}
where again $\epsilon$ is any positive number and $C_{\epsilon}$ depends  on $M$, $\Sigma_{0}$ and $\epsilon$.

\begin{center}
\large{\textbf{Estimate for} $\displaystyle\int_{\mathcal{R}}{H_{9}\left(\partial_{v}\partial_{r}\psi\right)\left(\partial_{r}\partial_{r}\psi\right)}$}\\
\end{center}
In view of the wave equation we have
\begin{equation*}
\partial_{v}\partial_{r}\psi=-\frac{D}{2}\partial_{r}\partial_{r}\psi-\frac{1}{r}\partial_{v}\psi-\frac{R}{2}\partial_{r}\psi-\frac{1}{2}\lapp\psi.
\end{equation*}
Therefore,
\begin{equation*}
\begin{split}
H_{9}\left(\partial_{v}\partial_{r}\psi\right)\left(\partial_{r}\partial_{r}\psi\right)=&-H_{9}\frac{D}{2}\left(\partial_{r}\partial_{r}\psi\right)^{2}-H_{9}\frac{R}{2}\left(\partial_{r}\psi\right)\left(\partial_{r}\partial_{r}\psi\right)\\
&-\frac{H_{9}}{r}\left(\partial_{v}\psi\right)\left(\partial_{r}\partial_{r}\psi\right)-\frac{H_{9}}{2}\left(\lapp\psi\right)\left(\partial_{r}\partial_{r}\psi\right).
\end{split}
\end{equation*}
Note that in $\mathcal{A}$ we have
\begin{equation*}
H_{9}D\leq\frac{H_{2}}{10}
\end{equation*}
and thus the first term in right hand side poses no problem. Similarly, in $\mathcal{A}$ we have
\begin{equation*}
-H_{9}\frac{R}{2}\left(\partial_{r}\psi\right)\left(\partial_{r}\partial_{r}\psi\right)\leq \left(\partial_{r}\psi\right)^{2}+\left(H_{9}R\right)^{2}\left(\partial_{r}\partial_{r}\psi\right)^{2}.
\end{equation*}
and
\begin{equation*}
(H_{9}R)^{2}\leq \frac{H_{2}}{10}.
\end{equation*}
According to what we have proved above, the integral 
\begin{equation*}
\int_{\mathcal{A}}{-\frac{H_{9}}{r}\left(\partial_{v}\psi\right)\left(\partial_{r}\partial_{r}\psi\right)}
\end{equation*}
can be estimated (provided we replace $H_{6}$ with $-\frac{H_{9}}{r}$). Similarly, we have seen that  the integral
\begin{equation*}
\int_{\mathcal{A}}{-\frac{H_{9}}{2}(\partial_{r}\partial_{r}\psi)\lapp\psi}
\end{equation*}
can be estimated provided we have $H_{9}\leq \frac{H_{3}}{10}$, which holds by the definition of $L$.

\begin{center}
\large{\textbf{Estimate for} $\displaystyle\int_{\mathcal{R}}{H_{10}\left(\partial_{r}\partial_{r}\psi\right)\left(\partial_{r}\psi\right)}$}\\
\end{center}
We have
\begin{equation*}
\begin{split}
& 2\int_{\mathcal{R}}{H_{10}\left(\partial_{r}\partial_{r}\psi\right)\left(\partial_{r}\psi\right)}+\int_{\mathcal{R}}{\left[\partial_{r}H_{10}+\frac{2}{r}H_{10}\right]\left(\partial_{r}\psi\right)^{2}}\\
=&\int_{\Sigma_{0}}{H_{10}\left(\partial_{r}\psi\right)^{2}\partial_{r}\cdot n_{\Sigma_{0}}}-\int_{\Sigma_{\tau}}{H_{10}\left(\partial_{r}\psi\right)^{2}\partial_{r}\cdot n_{\Sigma_{\tau}}}-\int_{\mathcal{H}^{+}}{H_{10}\left(\partial_{r}\psi\right)^{2}}.
\end{split}
\end{equation*}
If $S=\partial_{r}H_{5}+\frac{2}{r}H_{5}$ then 
\begin{equation*}
\int_{\mathcal{A}}{\frac{1}{2}S(\partial_{r}\psi)^{2}}\leq\operatorname*{max}_{\mathcal{A}}\frac{S}{2}\int_{\mathcal{A}}{(\partial_{r}\psi)^{2}}
\end{equation*}
and thus this integral is estimated using Proposition \ref{1comprop}. 
Also
\begin{equation*}
\left|\int_{\Sigma_{0}}{H_{10}\left(\partial_{r}\psi\right)^{2}\partial_{r}\cdot n_{\Sigma_{0}}}-\int_{\Sigma_{\tau}}{H_{10}\left(\partial_{r}\psi\right)^{2}\partial_{r}\cdot n_{\Sigma_{\tau}}}\right|\leq C\left(\int_{\Sigma_{0}}{J^{N}_{\mu}\left(\psi\right)n^{\mu}_{\Sigma_{0}}}\right)
\end{equation*}
for a constant $C$ that depends on $M$ and $\Sigma_{0}$. Finally we need to estimate the integral over $\mathcal{H}^{+}$. It turns out that this integral is the most problematic. In the next section we will see the reason for that. Since
\begin{equation*}
\begin{split}
&R'=D''+\frac{2D'}{r}-\frac{2D}{r^{2}}\Rightarrow R'\left(M\right)=\frac{2}{M^{2}},
\end{split}
\end{equation*}
we have
\begin{equation}
H_{10}\left(M\right)=-f_{r}\left(M\right)R'\left(M\right)=-f_{r}\left(M\right)\frac{2}{M^{2}}>0.
\label{H.10}
\end{equation}
In order to estimate the integral along $\mathcal{H}^{+}$ we apply  the Poincar\'{e} inequality
\begin{equation*}
\left|-\int_{\mathcal{H}^{+}}{\frac{H_{10}(M)}{2}\left(\partial_{r}\psi\right)^{2}}\right|\leq \int_{\mathcal{H}^{+}}{\frac{H_{10}(M)}{2}\frac{M^{2}}{l(l+1)}\left|\nabb\partial_{r}\psi\right|^{2}}.
\end{equation*}
Therefore, in view of  \eqref{horizonL1} it suffices to have
\begin{equation*}
\begin{split}
\frac{M^{2}}{2l(l+1)}H_{10}\left(M\right)&\leq -\frac{f_{r}\left(M\right)}{2}\overset{\eqref{H.10}}{\Leftrightarrow},\\
-f_{r}\left(M\right)\frac{M^{2}}{2l(l+1)}\frac{2}{4M^{2}}&\leq -\frac{f_{r}\left(M\right)}{2}\Leftrightarrow\\
\frac{1}{l(l+1)}&\leq \frac{1}{2}\Leftrightarrow\\
l&\geq 1.
\end{split}
\end{equation*}
Note that for $l=1$ we need to use the whole good term over $\mathcal{H}^{+}$ which appears in identity \eqref{eiL}. This is something we can do, since we have not used this term in order to estimate other integrals. That was possible via successive use of the Hardy inequalities.

\subsection{$L^{2}$ Estimates for the Second Order Derivatives}
\label{sec:UniformBoundednessForTheSecondOrderDerivatives}

We can now prove the following theorem
\begin{theorem}
There exists a  constant $C>0$ which depends on $M$ and $\Sigma_{0}$ such that for all solutions $\psi$ of the wave equation which are supported on  the frequencies $l\geq 1$ we have
\begin{equation}
\begin{split}
&\ \ \  \int_{\Sigma_{\tau}\cap\mathcal{A}}{\left(\partial_{v}\partial_{r}\psi\right)^{2}+\left(\partial_{r}\partial_{r}\psi\right)^{2}+\left|\nabb\partial_{r}\psi\right|^{2}}\\
&+\int_{\mathcal{H}^{+}}{\left(\partial_{v}\partial_{r}\psi\right)^{2}+\chi_{1}\left|\nabb\partial_{r}\psi\right|^{2}}\\
&+\int_{\mathcal{A}}{\left(\partial_{v}\partial_{r}\psi\right)^{2}+\sqrt{D}\left(\partial_{r}\partial_{r}\psi\right)^{2}+\left|\nabb\partial_{r}\psi\right|^{2}}\\\leq &\, C\int_{\Sigma_{0}}{J_{\mu}^{N}(\psi)n^{\mu}_{\Sigma_{0}}}+C\int_{\Sigma_{0}}{J_{\mu}^{N}(T\psi)n^{\mu}_{\Sigma_{0}}}+C\int_{\Sigma_{0}\cap\mathcal{A}}{J_{\mu}^{N}(\partial_{r}\psi)n^{\mu}_{\Sigma_{0}}},
\label{uniformbounden2ndorder}
\end{split}
\end{equation}
where $\chi_{1}=0$ if $\psi$ is supported on $l=1$ and $\chi_{1}=1$ if $\psi$ is supported on $l\geq 2$. 
\label{2ndder}
\end{theorem}
\begin{proof}
First note that $L\sim N$ in region $\mathcal{A}$. Thus, the above estimate where $\Sigma_{\tau}$ is replaced by $\Sigma_{\tau}\cap\mathcal{A}$ follows from the estimates \eqref{hypersurawayH} and \eqref{bulkawayH}, the estimates that we derived in Section \ref{sec:EstimatesOfTheBulkTerms} (by taking $\epsilon$ sufficiently small) and the energy identity \eqref{eiL}. 
\end{proof}
Note that the right hand side of \eqref{uniformbounden2ndorder} is bounded by
\begin{equation*}
C\int_{\Sigma_{0}}{J_{\mu}^{N}(\psi)n^{\mu}_{\Sigma_{0}}}+C\int_{\Sigma_{0}}{J_{\mu}^{N}(N\psi)n^{\mu}_{\Sigma_{0}}}.
\end{equation*}
This means that equivalently we can apply $N$ as multiplier but also as commutator. One can also obtain similar spacetime estimates for spatially compact regions which include $\mathcal{A}$ by commuting with $T$ and $T^{2}$ and applying $X$ as a multiplier. Note that we need to commute with $T^{2}$ in view of the photon sphere. Note also that no commutation with the generators of the Lie algebra so(3) is required.

We are now in position to drop the degeneracy of \eqref{nk}.
\begin{proposition}
There exists a uniform constant $C>0$ which depends on $M$ and $\Sigma_{0}$ such that for all solutions $\psi$ of the wave equation which are supported on the frequencies $l\geq 1$  we have
\begin{equation*}
\begin{split}
\int_{\mathcal{A}}{\left(\partial_{r}\psi\right)^{2}}\leq C\int_{\Sigma_{0}}{J_{\mu}^{N}(\psi)n^{\mu}_{\Sigma_{0}}}+C\int_{\Sigma_{0}}{J_{\mu}^{N}(T\psi)n^{\mu}_{\Sigma_{0}}}+C\int_{\Sigma_{0}\cap\mathcal{A}}{J_{\mu}^{N}(\partial_{r}\psi)n^{\mu}_{\Sigma_{0}}}.
\end{split}
\end{equation*}
\label{r1back}
\end{proposition}
\begin{proof}
Immediate from Proposition \ref{1comprop} and Theorem \ref{2ndder}.
\end{proof}

\section{Conservation Laws on Degenerate Event Horizons}
\label{sec:ConservationLawsOnDegenerateEventHorizons}

We will prove that the lack of redshift gives rise to conservation laws along $\mathcal{H}^{+}$ for translation invariant derivatives. These laws affect the low angular frequencies and explain why we need to additionally assume that the zeroth spherical harmonic vanishes identically in Proposition  $\ref{r1back}$. We will also show that the argument of Kay and Wald (see \cite{wa1}) could not have been applied in our case.

We  use the regular coordinate system $(v,r)$. Our results are not restricted to  extreme Reissner-Nordstr\"{o}m  but apply to a very general class of (spherically symmetric) degenerate black hole spacetimes. We impose no unphysical restriction on the initial data (which, of course, should have finite energy). 

Let us first consider spherically symmetric waves.
\begin{proposition}
For all  spherically symmetric solutions $\psi$ to the wave equation the quantity 
\begin{equation}
\partial_{r}\psi +\frac{1}{M}\psi
\label{0l}
\end{equation}
is conserved along $\mathcal{H}^{+}$. Therefore, for generic initial data  this quantity does not decay.
\label{ndl0}
\end{proposition}
\begin{proof}
Since $\psi$ solves $\Box_{g}\psi =0$ and since $\lapp\psi =0$ we have
\begin{equation*}
\partial_{v}\partial_{r}\psi+\frac{1}{M}\partial_{v}\psi=0
\end{equation*}
and, since $\partial_{v}$ is tangential to $\mathcal{H}^{+}$, this implies that $\partial_{r}\psi +\frac{1}{M}\psi$ remains constant along $\mathcal{H}^{+}$ and clearly for generic initial data it is not equal to zero.
\end{proof}

\begin{proposition}
For all  solutions $\psi$ to the wave equation  that are supported on $l=1$ the quantity
\begin{equation}
\partial_{r}\partial_{r}\psi+\frac{3}{M}\partial_{r}\psi+\frac{1}{M^{2}}\psi
\label{02}
\end{equation} 
is conserved along the null geodesics of $\mathcal{H}^{+}$.
\label{ndl1}
\end{proposition}
\begin{proof}
Since $\lapp\psi=-\frac{2}{r^{2}}\psi$, the wave equation on $\mathcal{H}^{+}$ gives us
\begin{equation}
2\partial_{v}\partial_{r}\psi+\frac{2}{M}\partial_{v}\psi=\frac{2}{M^{2}}\psi.
\label{psi1l}
\end{equation}
Moreover,
\begin{equation*}
\begin{split}
\partial_{r}\left(\Box_{g}\psi\right)=&D\partial_{r}\partial_{r}\partial_{r}\psi+2\partial_{r}\partial_{v}\partial_{r}\psi+\frac{2}{r}\partial_{r}\partial_{v}\psi+R\partial_{r}\partial_{r}\psi+\partial_{r}\lapp\psi\\
&+D'\partial_{r}\partial_{r}\psi-\frac{2}{r^{2}}\partial_{v}\psi+R'\partial_{r}\psi\\
\end{split}
\end{equation*}
and thus by restricting this identity on $\mathcal{H}^{+}$ we take
\begin{equation}
\begin{split}
2\partial_{v}\partial_{r}^{2}\psi+\frac{2}{M}\partial_{v}\partial_{r}\psi-\frac{2}{M^{2}}\partial_{v}\psi+\frac{4}{M^{3}}\psi+\left(-\frac{2}{M^{2}}+R'(M)\right)\partial_{r}\psi=0.
\end{split}
\label{rpsi1}
\end{equation}
However, $R'(M)=\frac{2}{M^{2}}$ and in view of \eqref{psi1l} we have  
\begin{equation*}
\begin{split}
2\partial_{v}\partial_{r}^{2}\psi+\frac{2}{M}\partial_{v}\partial_{r}\psi-\frac{2}{M^{2}}\partial_{v}\psi+\frac{2}{M}\left(2\partial_{v}\partial_{r}\psi+\frac{2}{M}\partial_{v}\psi\right)=0
\end{split}
\end{equation*}
which means that  the quantity
\begin{equation*}
\begin{split}
\partial_{r}\partial_{r}\psi+\frac{3}{M}\partial_{r}\psi+\frac{1}{M^{2}}\psi
\end{split}
\end{equation*}
is constant along the  integral curves of $T$ on $\mathcal{H}^{+}$.
\end{proof}

\begin{proposition}
There exist constants $\a_{i}^{j},i=0,1,2,\, j=0,1$ which depend on $M$ and $l$ such that for all solutions  $\psi$ of the wave equation which are supported on the (fixed) frequency $l$, where $l\geq 2$, we have
\begin{equation*}
\begin{split}
\psi=& \ \a_{0}^{0}\partial_{v}\psi+\a_{1}^{0}\partial_{v}\partial_{r}\psi,\\
\partial_{r}\psi=&\ \a_{0}^{1}\partial_{v}\psi+\a_{1}^{1}\partial_{v}\partial_{r}\psi+\a_{2}^{1}\partial_{v}\partial_{r}^{2}\psi,
\end{split}
\end{equation*}
on $\mathcal{H}^{+}$.
\label{corl2}
\end{proposition}
\begin{proof}
Revisit the proof of Proposition \ref{ndl1} and use the fact that for all $l\geq 2$ we have $-\frac{l(l+1)}{M^{2}}+R'(M)\neq 0$.
\end{proof}
The general form of the conservation laws is the following
\begin{theorem}
There exist constants $\a_{i}^{j},j=0,1,...,l-1,\, i=0,1,...,j+1$, which depend on $M$ and $l$ such that  for all solutions  $\psi$ of the wave equation which are supported on the  (fixed) frequency $l$ we have 
\begin{equation*}
\partial_{r}^{j}\psi=\sum_{i=0}^{j+1}{\a_{i}^{j}\partial_{v}\partial_{r}^{i}\psi},
\end{equation*}
on $\mathcal{H}^{+}$. Moreover, there exist constants $\b_{i},i=0,1,...,l$, which depend on $M$ and $l$ such that the quantity
\begin{equation*}
\partial_{r}^{l+1}\psi+\sum_{i=0}^{l}{\b_{i}\partial_{r}^{i}\psi}
\end{equation*}
is conserved along the null geodesics of $\mathcal{H}^{+}$.
\label{ndk}
\end{theorem}
\begin{proof}
For the first statement, we will use induction on $j$ for fixed $l$. According to Proposition \ref{corl2} the result holds for $j=0,1$. We suppose that it holds for $j=0,1,...,k-1$ and we will prove that it holds for $j=k$ provided $k\leq l-1$. Clearly, 
\begin{equation}
\begin{split}
\partial_{r}^{k}\left(\Box_{g}\psi\right)=&D\left(\partial_{r}^{k+2}\psi\right)+2\partial_{v}\partial_{r}^{k+1}\psi+\frac{2}{r}\partial_{v}\partial_{r}^{k}\psi+R\partial_{r}^{k+1}\psi+\partial_{r}^{k}\lapp\psi+\\
&+\sum_{i=1}^{k}{
\binom{k}{i}\partial_{r}^{i}D\cdot \partial_{r}^{k-i+2}\psi}+\sum_{i=1}^{k}{\binom{k}{i}\partial_{r}^{i}\frac{2}{r}\cdot\partial_{v}\partial_{r}^{k-i}\psi}+\\
&+\sum_{i=1}^{k}{\binom{k}{i}\partial_{r}^{i}R\cdot\partial_{r}^{k-i+1}\psi}.
\end{split}
\label{kcom}
\end{equation}
We observe that the coef{}ficients of $\partial_{r}^{k+2}\psi$ and $\partial_{r}^{k+1}\psi$ vanish on $\mathcal{H}^{+}$. Since $\lapp\psi=-\frac{l(l+1)}{r^{2}}\psi$, the coef{}ficient of $\partial_{r}^{k}\psi$ on $\mathcal{H}^{+}$ is equal to 
\begin{equation*}
\begin{split}
\binom{k}{2}D''+\binom{k}{1}R'-\frac{l(l+1)}{M^{2}}=\frac{k(k+1)}{2}\frac{2}{M^{2}}-\frac{l(l+1)}{M^{2}},
\end{split}
\end{equation*}
which is non-zero if and only if $l\neq k$. Therefore, for all $k\leq l-1$ we  solve with respect to $\partial_{r}^{k}\psi$ and use the inductive hypothesis completing the proof of the first statement. For the second statement, we apply \eqref{kcom} for $k=l$. Then, according to our previous calculation, the coef{}ficients of  $\partial_{r}^{l+2}\psi$,  $\partial_{r}^{l+1}\psi$ and $\partial_{r}^{l}\psi$ vanish on $\mathcal{H}^{+}$. Therefore, we end up with the terms $\partial_{v}\partial_{r}^{k}\psi, k=0,1,...,l+1$ and $\partial_{r}^{j}\psi,j=0,1,...,l-1$. Thus, from the first part of the theorem, there exist 
 numbers $\b_{i},i=0,1,...,l$ which depend on $M$ and $l$ such that
\begin{equation*}
\partial_{v}\partial_{r}^{l+1}\psi+\sum_{i=0}^{l}{\b_{i}\partial_{v}\partial_{r}^{i}\psi}=0
\end{equation*}
on $\mathcal{H}^{+}$, which implies that the quantity
\begin{equation*}
\partial_{r}^{l+1}\psi+\sum_{i=0}^{l}{\b_{i}\partial_{r}^{i}\psi}
\end{equation*}
is constant along the integral curves of $T$ on $\mathcal{H}^{+}$. 
\end{proof}
An important corollary of the above theorem is the following
\begin{corollary}
For generic initial data  there is no Cauchy hypersurface $\Sigma$ crossing $\mathcal{H}^{+}$ and a solution $\tilde{\psi}$ of the wave equation such that 
\begin{equation*}
T\tilde{\psi}=\psi
\end{equation*}
in the future of $\Sigma$.
\label{wald}
\end{corollary}
\begin{proof}
Suppose that there exists a wave $\tilde{\psi}$ such that $T\tilde{\psi}=\psi$. Then we can decompose 
\begin{equation*}
\tilde{\psi}=\tilde{\psi}_{0}+\tilde{\psi}_{\geq 1}
\end{equation*}
as is shown in Section \ref{sec:EllipticTheoryOnMathbbS2}. Then
\begin{equation*}
\begin{split}
T\tilde{\psi}=&T\tilde{\psi}_{0}+T\tilde{\psi}_{\geq 1}\\
=&(T\tilde{\psi})_{0}+(T\tilde{\psi})_{\geq 1}
\end{split}
\end{equation*}
since $T$ is an endomorphism of the eigenspaces of $\lapp$. But $T\tilde{\psi}=\psi$ and thus $\psi_{0}=(T\tilde{\psi})_{0}=T\tilde{\psi}_{0}$. Since $\tilde{\psi}_{0}$ is a spherically symmetric wave we have $\partial_{v}\partial_{r}\tilde{\psi}_{0}+\frac{1}{M}\partial_{v}\tilde{\psi}_{0}=0$ which yields 
\begin{equation*}
\partial_{r}\psi_{0}+\frac{1}{M}\psi_{0}=0.
\end{equation*}
However, for generic intial data the above quantity is non zero and this completes the proof.
\end{proof}
Note that even if we restrict the generic initial data such that the evolving waves $\psi$ are supported on $l\geq L$, for some $L\geq 1$, then the above corollary still holds. Indeed, it suffices to decompose 
\begin{equation*}
\tilde{\psi}=\tilde{\psi}_{L}+\tilde{\psi}_{\geq L+1}
\end{equation*}
and repeat the argument using again the above theorem.

This proves that we can not use the argument of Kay and Wald (see \cite{wald} and \cite{wa1}) even for proving the uniform boundedness of the waves. Indeed, this argument showed that  for any solution of the wave equation $\psi$ there is another wave $\tilde{\psi}$ such that $T\tilde{\psi}=\psi$ in the future of a Cauchy hypersurface of the domain of outer communications of Schwarzschild.

\section{Sharp Higher Order $L^{2}$ Estimates}
\label{sec:HigherOrderEstimates}

In this section we commute the wave equation with $\partial_{r}^{k}$ where $k\in\mathbb{N}$ and $k\geq 2$ aiming at controlling all higher derivatives of $\psi$ (on the spacelike hypersurfaces and the spacetime region up to and including the horizon $\mathcal{H}^{+}$). In view of Theorem \ref{ndk} the weakest condition on $\psi$ would be such that it is supported on the frequencies $l\geq k$.

\subsection{The Commutator $\left[\Box_{g}, \partial_{r}^{k} \right]$}
\label{sec:TheCommutatorLeftBoxGPsiPartialRKRight}

We know that 
\begin{equation*}
\Box_{g}\psi=D\left(\partial_{r}^{2}\psi\right)+2\partial_{v}\partial_{r}\psi+\frac{2}{r}\partial_{v}\psi+R\partial_{r}\psi+\lapp\psi,
\end{equation*}
therefore,
\begin{equation*}
\Box_{g}\left(\partial_{r}^{k}\psi\right)=D\left(\partial_{r}^{k+2}\psi\right)+2\partial_{v}\partial_{r}^{k+1}\psi+\frac{2}{r}\partial_{v}\partial_{r}^{k}\psi+R\partial_{r}^{k+1}\psi+\lapp\partial_{r}^{k}\psi.
\end{equation*}
Furthermore,
\begin{equation*}
\begin{split}
\partial_{r}^{k}\left(\Box_{g}\psi\right)=&D\left(\partial_{r}^{k+2}\psi\right)+2\partial_{v}\partial_{r}^{k+1}\psi+\frac{2}{r}\partial_{v}\partial_{r}^{k}\psi+R\partial_{r}^{k+1}\psi+\partial_{r}^{k}\lapp\psi+\\
&+\sum_{i=1}^{k}{
\binom{k}{i}\partial_{r}^{i}D\cdot \partial_{r}^{k-i+2}\psi}+\sum_{i=1}^{k}{\binom{k}{i}\partial_{r}^{i}\frac{2}{r}\cdot\partial_{v}\partial_{r}^{k-i}\psi}+\\
&+\sum_{i=1}^{k}{\binom{k}{i}\partial_{r}^{i}R\cdot\partial_{r}^{k-i+1}\psi}.
\end{split}
\end{equation*}
Let us compute the commutator $\left[\lapp,\partial_{r}^{k}\right]$. If we denote 
\begin{equation*}
\begin{split}
\lapp_{1}=\frac{1}{\sin\theta}\left[\partial_{\theta}\left(\sin\theta\partial_{\theta}\right)+\frac{1}{\sin\theta}\partial_{\phi}\partial_{\phi}\right]
\end{split}
\end{equation*}
then
\begin{equation*}
\begin{split}
\partial_{r}^{k}\lapp\psi &=\partial_{r}^{k}\left(\frac{1}{r^{2}}\lapp_{1}\right)=\sum_{i=0}^{k}{\binom{k}{i}\partial_{r}^{i}\frac{1}{r^{2}}\cdot\partial_{r}^{k-i}\lapp_{1}}\\
&=\sum_{i=0}^{k}{\binom{k}{i}r^{2}\partial_{r}^{i}\frac{1}{r^{2}}\cdot\lapp\partial_{r}^{k-i}\psi}\\
&=\lapp\partial_{r}^{k}\psi+\sum_{i=1}^{k}{\binom{k}{i}r^{2}\partial_{r}^{i}r^{-2}\cdot\lapp\partial_{r}^{k-i}\psi}.\\
\end{split}
\end{equation*}
Therefore,
\begin{equation}
\left[\lapp,\partial_{r}^{k}\right]\psi=-\sum_{i=1}^{k}{\binom{k}{i}r^{2}\partial_{r}^{i}r^{-2}\cdot\lapp\partial_{r}^{k-i}\psi}
\label{sh1com1ttat}
\end{equation}
and so 
\begin{equation}
\begin{split}
\left[\Box_{g},\partial_{r}^{k}\right]\psi=&-\sum_{i=1}^{k}{
\binom{k}{i}\partial_{r}^{i}D\cdot \partial_{r}^{k-i+2}\psi}-\sum_{i=1}^{k}{\binom{k}{i}\partial_{r}^{i}\frac{2}{r}\cdot\partial_{v}\partial_{r}^{k-i}\psi}\\
&-\sum_{i=1}^{k}{\binom{k}{i}\partial_{r}^{i}R\cdot\partial_{r}^{k-i+1}\psi}-\sum_{i=1}^{k}{\binom{k}{i}r^{2}\partial_{r}^{i}r^{-2}\cdot\lapp\partial_{r}^{k-i}\psi}.
\label{comhigherorder}
\end{split}
\end{equation}

\subsection{Induction on $k$}
\label{sec:TheMultiplierLKAndTheEnergyIdentity}

For any solution $\psi$ of the wave equation we  control the  higher order derivatives in the spacetime region away from  $\mathcal{H}^{+}$:
\begin{equation}
\left\|\partial_{\a}\psi\right\|^{2}_{L^{2}\left(\mathcal{R}\left(0,\tau\right)\cap\left\{M<r_{0}\leq r\leq r_{1}<2M\right\}\right)} \leq C\int_{\Sigma_{0}}{\left(\sum_{i=0}^{k-1}{J^{T}_{\mu}\left(T^{i}\psi\right)n^{\mu}_{\Sigma_{0}}}\right)},
\label{hospacaaH}
\end{equation}
where $C$ depends on $M$, $r_{0}$, $r_{1}$ and $\Sigma_{0}$. This can be proved by commuting the wave equation  with $T^{i},i=1,...,k-1$ and applying $X$ as a multiplier and using local elliptic estimates. Thus, we need to understand the behaviour of the $k^{\text{th}}$-order derivatives of $\psi$ in a neighbourhood of $\mathcal{H}^{+}$. 
\begin{theorem}
There exists a  constant $C>0$ which depends on $M$, $k$ and $\Sigma_{0}$ such that for all solutions $\psi$ of the wave equation which are supported on the frequencies $l\geq k$ we have
\begin{equation*}
\begin{split}
&\int_{\Sigma_{\tau}\cap\mathcal{A}}{\left(\partial_{v}\partial_{r}^{k}\psi\right)^{2}+\left(\partial_{r}^{k+1}\psi\right)^{2}+\left|\nabb\partial_{r}^{k}\psi\right|^{2}}\\
+&\int_{\mathcal{H}^{+}}{\left(\partial_{v}\partial_{r}^{k}\psi\right)^{2}+\chi_{k}\left|\nabb\partial_{r}^{k}\psi\right|^{2}}\\
+&\int_{\mathcal{A}}{\left(\partial_{v}\partial_{r}^{k}\psi\right)^{2}+\sqrt{D}\left(\partial_{r}^{k+1}\psi\right)^{2}+\left|\nabb\partial_{r}^{k}\psi\right|^{2}}\\
\leq  &C\sum_{i=0}^{k}\int_{\Sigma_{0}}{J_{\mu}^{N}\left(T^{i}\psi\right)n^{\mu}_{\Sigma_{0}}}+C\sum_{i=1}^{k}\int_{\Sigma_{0}\cap\mathcal{A}}{J_{\mu}^{N}\left(\partial^{i}_{r}\psi\right)n^{\mu}_{\Sigma_{0}}},
\end{split}
\end{equation*}
where $\chi_{k}=0$ if $\psi$ is supported on $l= k$ and $\chi_{k}=1$ if $\psi$ is supported on $l\geq k+1$. The condition on the spherical decomposition is sharp.
\label{kder}
\end{theorem}
\begin{proof}
We follow an inductive process. We suppose that we have proved that there exists a uniform positive constant $C$ that depends on $M$, $\Sigma_{0}$ and $m$ such that 
\begin{equation}
\begin{split}
&\int_{\Sigma_{\tau}\cap\mathcal{A}}{\left(\partial_{v}\partial_{r}^{m}\psi\right)^{2}+\left(\partial_{r}^{m+1}\psi\right)^{2}+\left|\nabb\partial_{r}^{m}\psi\right|^{2}}\\
+&\int_{\mathcal{H}^{+}}{\left(\partial_{v}\partial_{r}^{m}\psi\right)^{2}+\chi_{m}\left|\nabb\partial_{r}^{m}\psi\right|^{2}}\\
+&\int_{\mathcal{A}}{\left(\partial_{v}\partial_{r}^{m}\psi\right)^{2}+\sqrt{D}\left(\partial_{r}^{m+1}\psi\right)^{2}+\left|\nabb\partial_{r}^{m}\psi\right|^{2}}\\
\leq  &C\sum_{i=0}^{m}\int_{\Sigma_{0}}{J_{\mu}^{N}\left(T^{i}\psi\right)n^{\mu}_{\Sigma_{0}}}+C\sum_{i=1}^{m}\int_{\Sigma_{0}\cap\mathcal{A}}{J_{\mu}^{N}\left(\partial^{i}_{r}\psi\right)n^{\mu}_{\Sigma_{0}}}
\label{hok-1}
\end{split}
\end{equation}
for all natural numbers $m\leq k-1$ (and thus we have $\chi_{m}=1$). Note that for $m=0$ the above estimate is the result of Section \ref{sec:TheVectorFieldN} and for $m=1$ of Section \ref{sec:CommutingWithAVectorFieldTransversalToMathcalH}.

In order to derive the above estimate for $m=k$ we will construct an appropriate multiplier $L^{k}=f_{v}^{k}\partial_{v}+f_{r}^{k}\partial_{r}$ which is a future directed timelike $\phi_{T}$-invariant vector field. Following the spirit of Section \ref{sec:CommutingWithAVectorFieldTransversalToMathcalH}, $L^{k}=0$ in $r\geq r_{1}$ and timelike in the region $M<r_{1}$. Again, we will be interested in the region $\mathcal{A}=\left\{M\leq r\leq r_{0}<r_{1}\right\}$. The control for the higher order derivatives will be derived from the energy identity of the current $J_{\mu}^{L^{k}}\left(\partial_{r}^{k}\psi\right)$
\begin{equation}
\int_{\Sigma_{\tau}}{J^{L^{k}}_{\mu}\!\!\left(\partial_{r}^{k}\psi\right)n_{\Sigma_{\tau}}^{\mu}}+\int_{\mathcal{R}}{\nabla^{\mu}J^{L^{k}}_{\mu}\!\!\left(\partial_{r}^{k}\psi\right)}+\int_{\mathcal{H}^{+}}{J^{L^{k}}_{\mu}\!\!\left(\partial_{r}^{k}\psi\right)n_{\mathcal{H}^{+}}^{\mu}}=\int_{\Sigma_{0}}{J^{L^{k}}_{\mu}\!\!\left(\partial_{r}^{k}\psi\right)n_{\Sigma_{0}}^{\mu}}.
\label{eideho}
\end{equation}
By assumption, the right hand side is bounded. Also since $L^{k}$ is timelike in the compact region $\mathcal{A}$ we have 
\begin{equation}
J_{\mu}^{L^{k}}\!\!\left(\partial_{r}^{k}\psi\right)n^{\mu}_{\Sigma_{\tau}}\sim \left(\partial_{v}\partial_{r}^{k}\psi\right)^{2}+\left(\partial_{r}^{k+1}\psi\right)^{2}+\left|\nabb\partial_{r}^{k}\psi\right|^{2}
\label{hoLst}
\end{equation}
and on the horizon
\begin{equation}
J_{\mu}^{L^{k}}\!\!\left(\partial_{r}^{k}\psi\right)n^{\mu}_{\mathcal{H}^{+}}= f_{v}^{k}(M)\left(\partial_{v}\partial_{r}^{k}\psi\right)^{2}-\frac{f_{r}^{k}(M)}{2}\left|\nabb\partial_{r}^{k}\psi\right|^{2}.
\label{hoLH}
\end{equation}
It suffices to estimate the bulk integral. We have
\begin{equation*}
\begin{split}
\nabla^{\mu}J_{\mu}^{L^{k}}\!\!\left(\partial_{r}^{k}\psi\right)&=K^{L^{k}}\!\left(\partial_{r}^{k}\psi\right)
+\mathcal{E}^{L^{k}}\!\left(\partial_{r}^{k}\psi\right).\\
\end{split}
\end{equation*}
But
\begin{equation*}
\begin{split}
K^{L^{k}}\!\left(\partial_{r}^{k}\psi\right)=&F_{vv}\!\left(\partial_{v}\partial_{r}^{k}\psi\right)^{2}+F_{rr}\!\left(\partial_{r}^{k+1}\psi\right)^{2}+F_{\scriptsize\nabb}\!\left|\nabb\partial_{r}^{k}\psi\right|^{2}+\\
&+F_{vr}\!\left(\partial_{v}\partial_{r}^{k}\psi\right)\!\left(\partial_{r}^{k+1}\psi\right),
\end{split}
\end{equation*}
where the coefficients are given from \eqref{list}. Moreover, equation \eqref{comhigherorder} gives us
\begin{equation*}
\begin{split}
&\mathcal{E}^{L^{k}}\!\left(\partial_{r}^{k}\psi\right)=\left(\Box_{g}\partial_{r}^{k}\psi\right)L^{k}\left(\partial_{r}^{k}\psi\right)\\
=&\left[-\sum_{i=1}^{k}{
\binom{k}{i}\partial_{r}^{i}D\cdot \partial_{r}^{k-i+2}\psi}-\sum_{i=1}^{k}{\binom{k}{i}\partial_{r}^{i}\frac{2}{r}\cdot\partial_{v}\partial_{r}^{k-i}\psi}\right.\\
&\left.-\sum_{i=1}^{k}{\binom{k}{i}\partial_{r}^{i}R\cdot\partial_{r}^{k-i+1}\psi}-\sum_{i=1}^{k}{\binom{k}{i}r^{2}\partial_{r}^{i}r^{-2}\cdot\lapp\partial_{r}^{k-i}\psi}\right]L^{k}\!\left(\partial_{r}^{k}\psi\right)\\
=&-\sum_{i=1}^{k}{\binom{k}{i}f_{v}^{k}\cdot\partial_{r}^{i}D\left(\partial_{r}^{k-i+2}\psi\right)\left(\partial_{v}\partial_{r}^{k}\psi\right)}-\sum_{i=1}^{k}{\binom{k}{i}f_{v}^{k}\cdot\partial_{r}^{i}\frac{2}{r}\left(\partial_{v}\partial_{r}^{k-i}\psi\right)\left(\partial_{v}\partial_{r}^{k}\psi\right)}\\
&-\sum_{i=1}^{k}{\binom{k}{i}f_{v}^{k}\cdot\partial_{r}^{i}R\left(\partial_{r}^{k-i+1}\psi\right)\left(\partial_{v}\partial_{r}^{k}\psi\right)}-\sum_{i=1}^{k}{\binom{k}{i}f_{v}^{k}\cdot r^{2}\partial_{r}^{i}r^{-2}\left(\lapp\partial_{r}^{k-i}\psi\right)\left(\partial_{v}\partial_{r}^{k}\psi\right)}\\
&-\sum_{i=1}^{k}{\binom{k}{i}f_{r}^{k}\cdot\partial_{r}^{i}D\left(\partial_{r}^{k-i+2}\psi\right)\left(\partial_{r}^{k+1}\psi\right)}-\sum_{i=1}^{k}{\binom{k}{i}f_{r}^{k}\cdot\partial_{r}^{i}\frac{2}{r}\left(\partial_{v}\partial_{r}^{k-i}\psi\right)\left(\partial_{r}^{k+1}\psi\right)}\\
&-\sum_{i=1}^{k}{\binom{k}{i}f_{r}^{k}\cdot\partial_{r}^{i}R\left(\partial_{r}^{k-i+1}\psi\right)\left(\partial_{r}^{k+1}\psi\right)}-\sum_{i=1}^{k}{\binom{k}{i}f_{r}^{k}\cdot r^{2}\partial_{r}^{i}r^{-2}\left(\lapp\partial_{r}^{k-i}\psi\right)\left(\partial_{r}^{k+1}\psi\right)}.
\end{split}
\end{equation*}
In order to obtain the sharp result we  need the following lemma
\begin{lemma}
Suppose $\psi$ is a solution to the wave equation which is supported on the (fixed) frequency $l=k$. Then for all $0\leq i\leq k-1$ and any positive number $\epsilon$ we have
\begin{equation*}
\begin{split}
\left|\int_{\mathcal{H}^{+}}{(\partial_{r}^{i}\psi)(\partial_{r}^{k}\psi)}\right|\leq & C_{\epsilon}\sum_{i=0}^{k-1}\int_{\Sigma_{0}}{J_{\mu}^{N}(T^{i}\psi)n^{\mu}_{\Sigma_{0}}}+C_{\epsilon}\sum_{i=0}^{k}\int_{\Sigma_{0}}{J_{\mu}^{N}(\partial_{r}^{i}\psi)n^{\mu}_{\Sigma_{0}}}\\ &+\epsilon\int_{\Sigma_{\tau}\cap\mathcal{A}}{J_{\mu}^{L^{k}}(\partial_{r}^{k}\psi)n^{\mu}_{\Sigma_{\tau}}}+\epsilon\int_{\mathcal{H}^{+}}{(\partial_{v}\partial_{r}^{k}\psi)^{2}},
\end{split}
\end{equation*}
where $C_{\epsilon}$ depends on $M$, $k$ and $\Sigma_{0}$.
\label{lemk}
\end{lemma}
\begin{proof}
Using Theorem \ref{ndk} it suf{}fices to estimate the integrals
\begin{equation*}
\int_{\mathcal{H}^{+}}{(\partial_{v}\partial_{r}^{j}\psi)(\partial_{r}^{k}\psi)},
\end{equation*}
where $0\leq j\leq k$. For $0\leq j\leq k-1$ we have
\begin{equation*}
\begin{split}
\int_{\mathcal{H}^{+}}{(\partial_{v}\partial_{r}^{j}\psi)(\partial_{r}^{k}\psi)}=&\int_{\mathcal{H}^{+}\cap\Sigma_{\tau}}{(\partial_{r}^{j}\psi)(\partial_{r}^{k}\psi)}-\int_{\mathcal{H}^{+}\cap\Sigma_{0}}{(\partial_{r}^{j}\psi)(\partial_{r}^{k}\psi)}\\ &-\int_{\mathcal{H}^{+}}{(\partial_{r}^{j}\psi)(\partial_{v}\partial_{r}^{k}\psi)}.
\end{split}
\end{equation*}
The two boundary integrals can be estimated using the second Hardy inequality. As regards the last integral on the right hand side, the Cauchy-Schwarz and Poincar\'e inequality imply
\begin{equation*}
\begin{split}
\int_{\mathcal{H}^{+}}{(\partial_{r}^{j}\psi)(\partial_{v}\partial_{r}^{k}\psi)}\leq \int_{\mathcal{H}^{+}}{\frac{1}{\epsilon}\left|\nabb\partial_{r}^{j}\psi\right|^{2}+\epsilon(\partial_{v}\partial_{r}^{k}\psi)^{2}}.
\end{split}
\end{equation*}
For $j=k$, we use that $(\partial_{v}\partial_{r}^{k}\psi)(\partial_{r}^{k}\psi)=\frac{1}{2}\partial_{v}(\partial_{r}^{k}\psi)^{2}$ and the second Hardy inequality.
\end{proof}
We are now in position to estimate the bulk integrals. We decompose 
\begin{equation*}
\psi=\psi_{k}+\psi_{\geq k+1},
\end{equation*}
as in Section \ref{sec:EllipticTheoryOnMathbbS2}. This is needed in view of the factor $\chi_{k}$ in the statement of the theorem.
\begin{center}
\large{\textbf{Estimate for} $\displaystyle\int_{\mathcal{R}}{H_{i}^{1}\left(\partial_{r}^{k-i+1}\psi\right)\left(\partial_{r}^{k+1}\psi\right)}, i\geq 0$}\\
\end{center}
For $i=0$ we have 
\begin{equation*}
H_{0}^{1}=-kf_{r}^{k}D'>0
\end{equation*}
and so this coefficient has the ``right'' sign in  \eqref{eideho}.

For $i\geq 1$  we use integration by parts 
\begin{equation*}
\begin{split}
&\int_{\mathcal{R}}{H_{i}^{1}\left(\partial_{r}^{k-i+1}\psi\right)\left(\partial_{r}^{k+1}\psi\right)}+\int_{\mathcal{R}}{\left(\partial_{r}H_{i}^{1}+\frac{2}{r}H_{i}^{1}\right)\left(\partial_{r}^{k-i+1}\psi\right)\left(\partial_{r}^{k}\psi\right)}+\\
&+\int_{\mathcal{R}}{H_{i}^{1}\left(\partial_{r}^{k-i+2}\psi\right)\left(\partial_{r}^{k}\psi\right)}\\
=&\int_{\Sigma_{0}}{H_{i}^{1}\left(\partial_{r}^{k-i+1}\psi\right)\left(\partial_{r}^{k}\psi\right)\partial_{r}\cdot n_{\Sigma_{0}}}-\int_{\Sigma_{\tau}}{H_{i}^{1}\left(\partial_{r}^{k-i+1}\psi\right)\left(\partial_{r}^{k}\psi\right)\partial_{r}\cdot n_{\Sigma_{\tau}}}-\\
&-\int_{\mathcal{H}^{+}}{H_{i}^{1}\left(\partial_{r}^{k-i+1}\psi\right)\left(\partial_{r}^{k}\psi\right)\partial_{r}\cdot n_{\mathcal{H}^{+}}}.
\end{split}
\end{equation*}

From Proposition \ref{1comprop} applied for $\partial_{r}^{k}\psi$ and the inductive hypothesis all the  integrals over $\mathcal{A}$ and $\Sigma_{\tau}$ can be estimated. It only remains to estimate the integral over $\mathcal{H}^{+}$ when $i=1$. In this case, if we apply the Poincar\'e inequality for the $\psi_{k}$ component we notice that we need to absorb a good term  in the divergence identity for $L^{k}$. Indeed,
\begin{equation*}
\begin{split}
\frac{1}{2}\frac{M^{2}}{l\left(l+1\right)}H_{1}^{1}\left(M\right)&= -\frac{f_{r}^{k}\left(M\right)}{2}\Leftrightarrow\\
-f_{r}^{k}\left(M\right)\frac{1}{2}\frac{M^{2}}{l\left(l+1\right)}\left[\binom{k}{2}\frac{2}{M^{2}}+k\frac{2}{M^{2}}\right]&= -\frac{f_{r}^{k}\left(M\right)}{2}\Leftrightarrow\\
\binom{k}{2}+k&= \frac{l\left(l+1\right)}{2}\Leftrightarrow\\
k&= l
\end{split}
\end{equation*}
This implies that we cannot use the Poincar\'e inequality on $\mathcal{H}^{+}$ in order to estimate $(\partial_{r}^{k}\psi_{k})^{2}$ anymore. That is why we proved Lemma \ref{lemk} which we will use for the following integrals. Clearly, for the component $\psi_{\geq k+1}$ we need a fraction of $\left|\nabb\partial_{r}^{k}\psi\right|^{2}$ along $\mathcal{H}^{+}$ and thus we can take small (epsilon) portions of this term later on. 

\begin{center}
\large{\textbf{Estimate for} $\displaystyle\int_{\displaystyle\mathcal{R}}{H_{i}^{2}\left(\partial_{v}\partial_{r}^{k-i+1}\psi\right)\left(\partial_{r}^{k+1}\psi\right)}, i \geq 1$}\\
\end{center}
For $i\geq 2$ we use Stokes' theorem
\begin{equation*}
\begin{split}
&\int_{\mathcal{R}}{H_{i}^{2}\left(\partial_{v}\partial_{r}^{k-i+1}\psi\right)\left(\partial_{r}^{k+1}\psi\right)}+\int_{\mathcal{R}}{\left(\partial_{r}H_{i}^{2}+\frac{2}{r}H_{i}^{2}\right)\left(\partial_{v}\partial_{r}^{k-i+1}\psi\right)\left(\partial_{r}^{k}\psi\right)}\ +\\
&+\int_{\mathcal{R}}{H_{i}^{2}\left(\partial_{v}\partial_{r}^{k-i+2}\psi\right)\left(\partial_{r}^{k}\psi\right)}\\
=&\int_{\Sigma_{0}}{H_{i}^{2}\left(\partial_{v}\partial_{r}^{k-i+1}\psi\right)\left(\partial_{r}^{k}\psi\right)\partial_{r}\cdot n_{\Sigma_{0}}}-\int_{\Sigma_{\tau}}{H_{i}^{2}\left(\partial_{v}\partial_{r}^{k-i+1}\psi\right)\left(\partial_{r}^{k}\psi\right)\partial_{r}\cdot n_{\Sigma_{\tau}}}\\
&-\int_{\mathcal{H}^{+}}{H_{i}^{2}\left(\partial_{v}\partial_{r}^{k-i+1}\psi\right)\left(\partial_{r}^{k}\psi\right)\partial_{r}\cdot n_{\mathcal{H}^{+}}}.
\end{split}
\end{equation*}
From Proposition \ref{1comprop} and the inductive hypothesis we can estimate all the above integrals. Note that in order to estimate the integrals along $\mathcal{H}^{+}$ for the component $\psi_{k}$ we follow the same argument as in the proof of Lemma \ref{lemk}.
For $i=1$ we use the wave equation and thus
\begin{equation*}
\begin{split}
H_{1}^{2}\left(\partial_{v}\partial_{r}^{k}\psi\right)\!\left(\partial_{r}^{k+1}\psi\right)=&H_{1}^{2}\left[\partial_{r}^{k-1}\left(-D\partial_{r}^{2}\psi-\frac{2}{r}\partial_{v}\psi-R\partial_{r}\psi-\lapp\psi\right)\right]\!\left(\partial_{r}^{k+1}\psi\right)\\
=&-\sum_{j=0}^{k-1}{\binom{k-1}{j}H_{1}^{2}\partial_{r}^{j}D\left(\partial_{r}^{k-j+1}\psi\right)\left(\partial_{r}^{k+1}\psi\right)}\\
&-\sum_{j=0}^{k-1}{\binom{k-1}{j}H_{1}^{2}\partial_{r}^{j}\frac{2}{r}\left(\partial_{v}\partial_{r}^{k-j-1}\psi\right)\left(\partial_{r}^{k+1}\psi\right)}\\
&-\sum_{j=0}^{k-1}{\binom{k-1}{j}H_{1}^{2}\partial_{r}^{j}R\left(\partial_{r}^{k-j}\psi\right)\left(\partial_{r}^{k+1}\psi\right)}\\
&-H_{1}^{2}\left(\partial_{r}^{k-1}\lapp\psi\right)\left(\partial_{r}^{k+1}\psi\right).\\
\end{split}
\end{equation*}
The integrals of the first sum can be estimated for $j=0,1$ since their coef{}ficients vanish on the horizon and the case $j\geq 2$ was investigated above. 

The integrals of the second sum were also estimated before. 

For $j=0$ the integral of the third sum can be estimated since its coefficient vanishes on $\mathcal{H}^{+}$. If $j\geq 1$ then again these integrals have been estimated.
It remains to estimate the integral of the last term. Integration by parts gives 
\begin{equation*}
\begin{split}
&\int_{\mathcal{R}}{H_{1}^{2}\left(\partial_{r}^{k-1}\lapp\psi\right)\left(\partial_{r}^{k+1}\psi\right)}+\int_{\mathcal{R}}{\left(\partial_{r}H_{1}^{2}+\frac{2}{r}H_{1}^{2}\right)\left(\partial_{r}^{k-1}\lapp\psi\right)\left(\partial_{r}^{k}\psi\right)}\\
&+\int_{\mathcal{R}}{H_{1}^{2}\left(\partial_{r}^{k}\lapp\psi\right)\left(\partial_{r}^{k}\psi\right)}\\
=&\int_{\Sigma_{0}}{H_{1}^{2}\left(\partial_{r}^{k-1}\lapp\psi\right)\left(\partial_{r}^{k}\psi\right)\partial_{r}\cdot n_{\Sigma_{0}}}-\int_{\Sigma_{\tau}}{H_{1}^{2}\left(\partial_{r}^{k-1}\lapp\psi\right)\left(\partial_{r}^{k}\psi\right)\partial_{r}\cdot n_{\Sigma_{\tau}}}\\
&-\int_{\mathcal{H}^{+}}{H_{1}^{2}\left(\partial_{r}^{k-1}\lapp\psi\right)\left(\partial_{r}^{k}\psi\right)\partial_{r}\cdot n_{\mathcal{H}^{+}}}.
\end{split}
\end{equation*}
If we set $Q=\partial_{r}H_{1}^{2}+\frac{2}{r}H_{1}^{2}$ then 
\begin{equation*}
\begin{split}
&\int_{\mathcal{R}}{Q\left(\partial_{r}^{k-1}\lapp\psi\right)\left(\partial_{r}^{k}\psi\right)}=\int_{\mathcal{R}}{Q\left(\lapp\partial_{r}^{k-1}\psi-\left[\lapp,\partial_{r}^{k-1}\right]\psi\right)\left(\partial_{r}^{k}\psi\right)}\\
=&\int_{\mathcal{R}}{Q\left(\lapp\partial_{r}^{k-1}\psi+\sum_{n=1}^{k-1}{\binom{k-1}{n}r^{2}\partial_{r}^{n}r^{-2}\cdot \lapp\partial_{r}^{k-1-n}\psi}\right)\left(\partial_{r}^{k}\psi\right)}.\\
\end{split}
\end{equation*}
We estimate this integral by applying Stokes' theorem on $\mathbb{S}^{2}$ and Cauchy-Schwarz and using the inductive hypothesis. As regards the last bulk integral we have
\begin{equation*}
\begin{split}
&\int_{\mathcal{R}}{H_{1}^{2}\left(\partial_{r}^{k}\lapp\psi\right)\left(\partial_{r}^{k}\psi\right)}=\int_{\mathcal{R}}{H_{1}^{2}\left(\lapp\partial_{r}^{k}\psi-\left[\lapp,\partial_{r}^{k}\right]\psi\right)\left(\partial_{r}^{k}\psi\right)}\\
=&\int_{\mathcal{R}}{H_{1}^{2}\left(\lapp\partial_{r}^{k}\psi+\sum_{n=1}^{k}{\binom{k}{n}r^{2}\partial_{r}^{n}r^{-2}\lapp\partial_{r}^{k-n}\psi}\right)\left(\partial_{r}^{k}\psi\right)}.\\
\end{split}
\end{equation*}
Therefore, by applying Stokes' theorem on $\mathbb{S}^{2}$ we see that this integral can be estimated provided we have
\begin{equation}
H_{1}^{2}<-\frac{1}{2}\partial_{r}f_{r}^{k}\left(M\right).
\label{ho24}
\end{equation}
Note that $H_{1}^{2}$ does not depend on $\partial_{r}f_{r}^{k}$.  Furthermore, The boundary integral over $\mathcal{H}^{+}$ can be estimated as follows: For the component $\psi_{k}$ we have $\lapp\psi_{k}=-\frac{k(k+1)}{r^{2}}\psi_{k}$  and thus $\partial_{r}^{k-1}\lapp\psi_{k}$ depends on $\partial_{r}^{i}\psi_{k}$ for $0\leq i\leq k-1$ and thus we  use Lemma \ref{lemk}. For the component $\psi_{\geq k+1}$ we commute $\partial_{r}^{k-1}$ and $\lapp$, we use Stokes' theorem on $\mathbb{S}^{2}$ and  Cauchy-Schwarz as above.

\begin{center}
\large{\textbf{Estimate for} $\displaystyle\int_{\mathcal{R}}{H_{i}^{3}\left(\partial_{r}^{k-i+2}\psi\right)\left(\partial_{v}\partial_{r}^{k}\psi\right)}, i\geq 1$}\\
\end{center}
The case $i=1$ was investigated above. 

For $i\geq 2$ we use Proposition \ref{1comprop} and Cauchy-Schwarz.

\begin{center}
\large{\textbf{Estimate for} $\displaystyle\int_{\mathcal{R}}{H_{i}^{4}\left(\partial_{v}\partial_{r}^{k-i}\psi\right)\left(\partial_{v}\partial_{r}^{k}\psi\right)}, i\geq 1$}\\
\end{center}
We use Cauchy-Schwarz and the inductive hypothesis.

\begin{center}
\large{\textbf{Estimate for} $\displaystyle\int_{\mathcal{R}}{H_{i}^{5}\left(\lapp\partial_{r}^{i}\psi\right)\left(\partial_{v}\partial_{r}^{k}\psi\right)}, i\leq k-1$}\\
\end{center}
For $i=0$ we solve with respect to $\lapp\psi$ in the wave equation and then use Cauchy-Schwarz and the inductive hypothesis. We proceed by induction on $i$. Assuming that $(\lapp\partial_{r}^{j}\psi)(\partial_{v}\partial_{r}^{k}\psi)$ is estimated for all $0\leq j\leq i-1$ we will prove that the integral $(\lapp\partial_{r}^{i}\psi)(\partial_{v}\partial_{r}^{k}\psi)$  can also be estimated. We have
\begin{equation*}
\begin{split}
\lapp\partial_{r}^{i}\psi\left(\partial_{v}\partial_{r}^{k}\psi\right)=\partial_{r}^{i}\lapp\psi\left(\partial_{v}\partial_{r}^{k}\psi\right)+[\lapp,\partial_{r}^{i}]\psi\left(\partial_{v}\partial_{r}^{k}\psi\right)
\end{split}
\end{equation*}
The first term on the right hand side can be estimated by solving with respect to $\lapp\psi$ in the wave equation and using Cauchy-Schwarz. The second term can be estimated by our inductive hypothesis.

\begin{center}
\large{\textbf{Estimate for} $\displaystyle\int_{\mathcal{R}}{H_{i}^{6}\left(\lapp\partial_{r}^{k-i}\psi\right)\left(\partial_{r}^{k+1}\psi\right)}, i\geq 1$}\\
\end{center}
Integration by parts yields
\begin{equation*}
\begin{split}
&\int_{\mathcal{R}}{H_{i}^{6}\left(\lapp\partial_{r}^{k-i}\psi\right)\left(\partial_{r}^{k+1}\psi\right)}+\int_{\mathcal{R}}{\left(\partial_{r}H_{i}^{6}+\frac{2}{r}H_{i}^{6}\right)\left(\lapp\partial_{r}^{k-i}\psi\right)\left(\partial_{r}^{k}\psi\right)}\\
&+\int_{\mathcal{R}}{H_{i}^{6}\left(\partial_{r}\lapp\partial_{r}^{k-i}\psi\right)\left(\partial_{r}^{k}\psi\right)}\\
=&\int_{\Sigma_{0}}{H_{i}^{6}\left(\lapp\partial_{r}^{k-i}\psi\right)\left(\partial_{r}^{k}\psi\right)\partial_{r}\cdot n_{\Sigma_{0}}}-\int_{\Sigma_{\tau}}{H_{i}^{6}\left(\lapp\partial_{r}^{k-i}\psi\right)\left(\partial_{r}^{k}\psi\right)\partial_{r}\cdot n_{\Sigma_{\tau}}}\\&-\int_{\mathcal{H}^{+}}{H_{i}^{6}\left(\lapp\partial_{r}^{k-i}\psi\right)\left(\partial_{r}^{k}\psi\right)\partial_{r}\cdot n_{\mathcal{H}^{+}}}.\\
\end{split}
\end{equation*}
The second bulk integral and the boundary integrals over $\Sigma$ are estimated  using Stokes' identity on $\mathcal{S}^{2}$ and the inductive hypothesis. The last bulk integral is estimated by commuting the spherical Laplacian with $\partial_{r}$ and applying again Stokes' identity on $\mathcal{S}^{2}$ and Cauchy-Schwarz. Thus for $i\geq 2$ we use the inductive hypothesis and for $i=1$ it suf{}fices to have
\begin{equation}
H_{1}^{6}<-\frac{1}{2}\partial_{r}f_{r}^{k}\left(M\right).
\label{ho6}
\end{equation}
As regards the integral over $\mathcal{H}^{+}$ we have the following: For the component $\psi_{k}$ we have $\lapp\psi_{k}=-\frac{k(k+1)}{r^{2}}\psi_{k}$  and thus $\lapp\partial_{r}^{k-i}\psi_{k}=-\frac{k(k+1)}{r^{2}}\partial_{r}^{k-i}\psi_{k}$ (see Section \ref{sec:EllipticTheoryOnMathbbS2}) and so we apply Lemma \ref{lemk} to estimate it.  For the component $\psi_{\geq k+1}$ we use Stokes' theorem on $\mathbb{S}^{2}$, Cauchy-Schwarz and the inductive hypothesis.

The construction of $L^{k}$ is now clear for all $k\in\mathbb{N}$. It suf{}fices to take $f_{r}^{k}\left(M\right)<0$, $f_{v}^{r}\left(M\right)>0$ and $-\partial_{r}f_{r}^{k}$ and $\partial_{r}f_{v}^{k}\left(M\right)$ suf{}ficiently large.  As far as the global construction for the vector field $L^{k}$ is concerned, we follow the same idea as in Section \ref{sec:CommutingWithAVectorFieldTransversalToMathcalH}.

Note that again no commutation with the generators of the Lie algebra so(3) is required.

The condition on the spherical decomposition is sharp (see Theorem \ref{hoeblowup}).

\end{proof}

\section{Energy Decay}
\label{sec:EnergyDecay}

In this section we derive the decay for the non-degenerate energy flux of $N$ through an appropriate foliation. Clearly, all the estimates in Section \ref{sec:TheVectorFieldTextbfX} are very useful, however they degenerate at infinity. It turns out that one can obtain non degenerate estimates on regions which connect $\mathcal{H}^{+}$ and $\mathcal{I}^{+}$ (without containing $i^{0}$). This has to do with the fact that a portion of the initial energy remains in  a neighbourhood of $i^{0}$. Such estimates were first derived in the recent \cite{new} along with a new robust method for obtaining decay results. Here we establish several estimates which will allow us to adapt the methods of \cite{new} in the extreme case. These new estimates are closely related with the trapping properties of $\mathcal{H}^{+}$.

We fix $R>2M$ and consider the hypersurface $\tilde{\Sigma}_{0}$ which is spacelike for $M\leq r\leq R$ and crosses $\mathcal{H}^{+}$ and for $r\geq R$ is given by $u=u(p_{0})$, where $p_{0}\in\tilde{\Sigma}_{0}$ is such that $r(p_{0})=R$.
 \begin{figure}[H]
	\centering
		\includegraphics[scale=0.12]{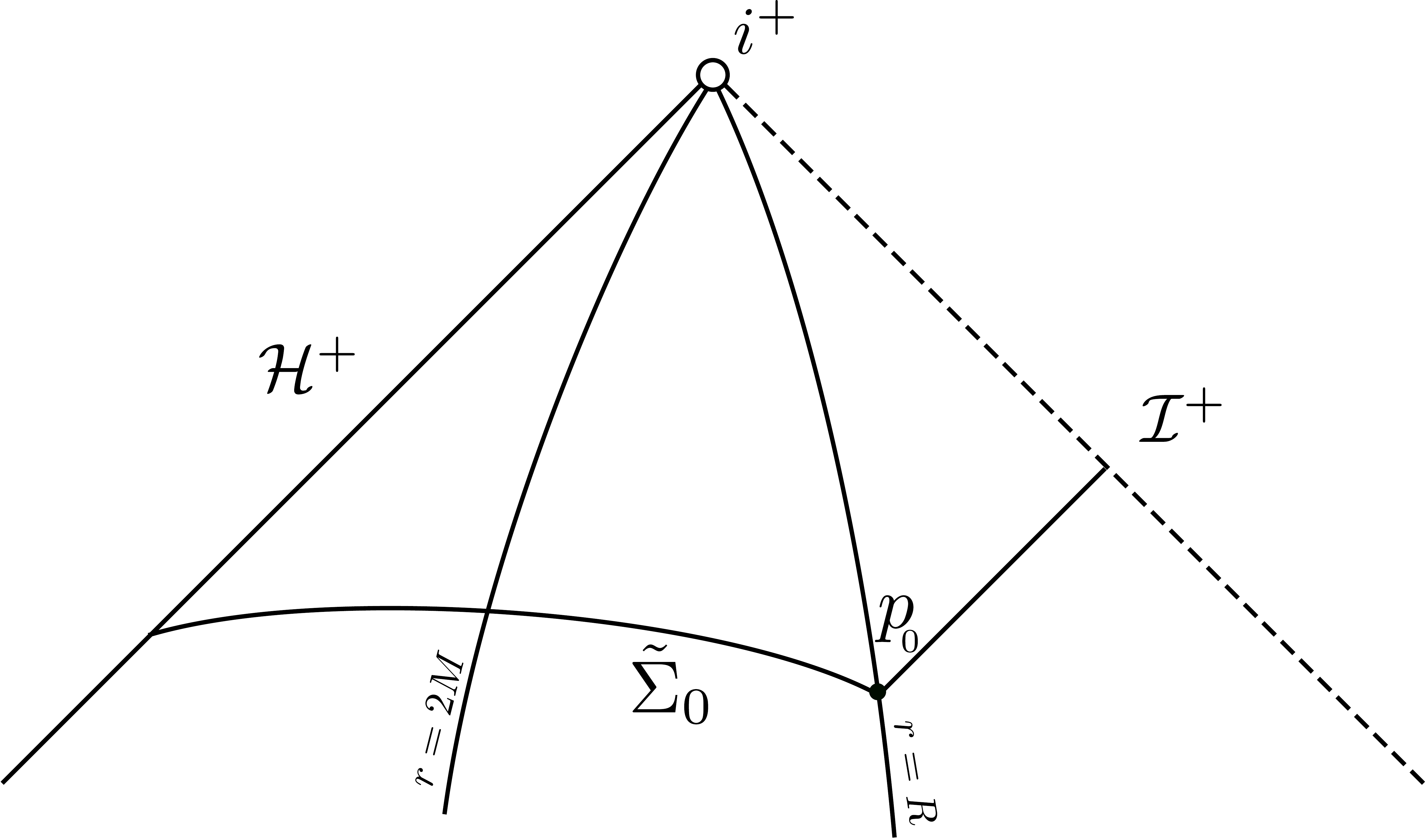}
	\label{fig:pic2ern}
\end{figure}
Consider now $\tilde{\Sigma}_{\tau}=\varphi_{\tau}(\tilde{\Sigma}_{0})$, where $\varphi_{\tau}$ is the flow of $T$. For arbitrary $\tau_{1}<\tau_{2}$ we define 
\begin{equation*}
\begin{split}
&\tilde{\mathcal{R}}_{\tau_{1}}^{\tau_{2}}=\cup_{\tau\in\left[\tau_{1},\tau_{2}\right]}{\tilde{\Sigma}_{\tau}},\\
&\tilde{\mathcal{D}}_{\tau_{1}}^{\tau_{2}}=\tilde{\mathcal{R}}_{\tau_{1}}^{\tau_{2}}\cap\left\{r\geq R\right\},\\
&\tilde{N}_{\tau}=\tilde{\Sigma}_{\tau}\cap\left\{r\geq R\right\},\\
&\Delta_{\tau_{1}}^{\tau_{2}}=\tilde{\mathcal{R}}_{\tau_{1}}^{\tau_{2}}\cap\left\{r= R\right\}.
\end{split}
\end{equation*}
 \begin{figure}[H]
	\centering
		\includegraphics[scale=0.14]{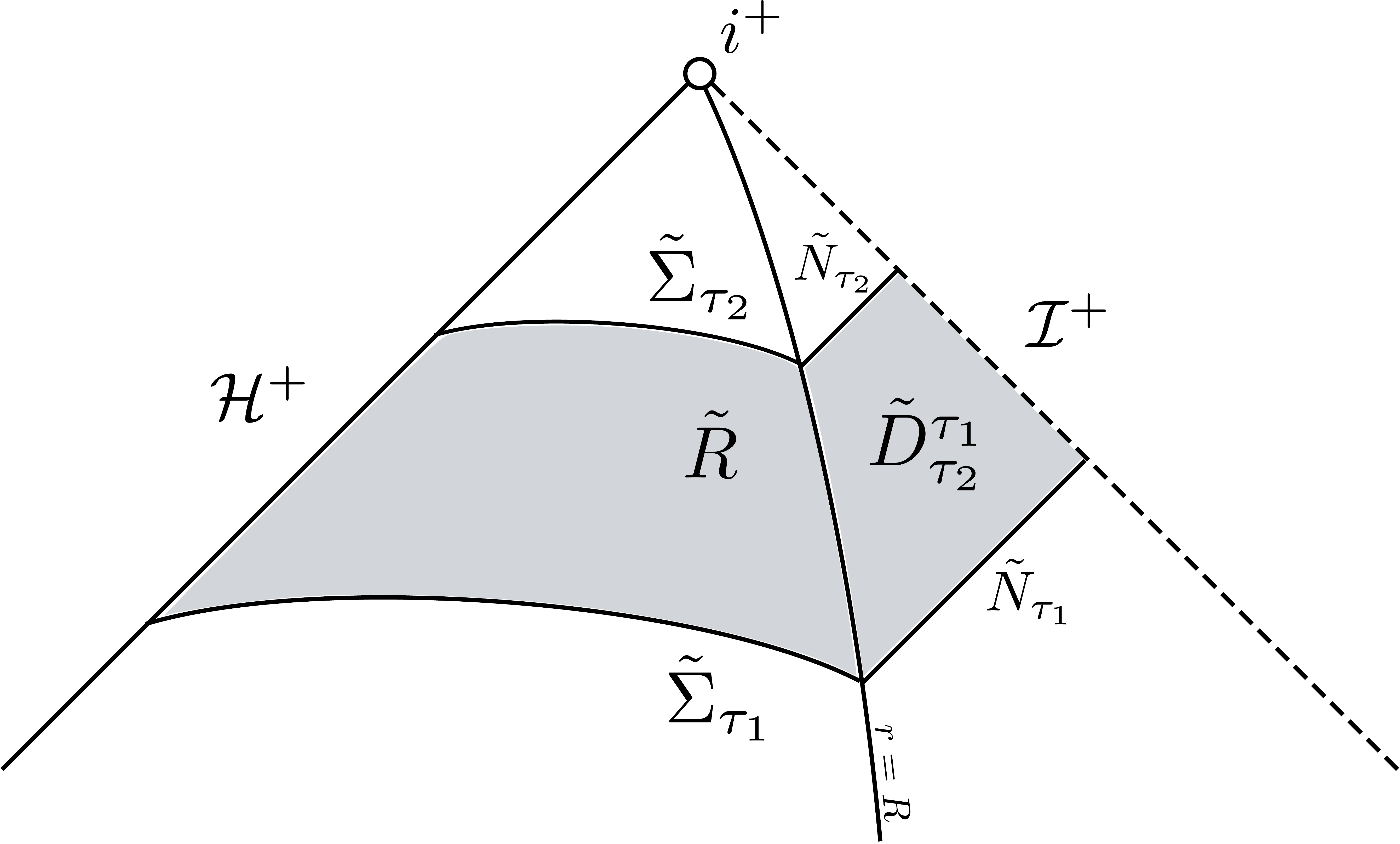}
	\label{fig:pic1ern}
\end{figure}

\subsection{$r$-Weighted Energy Estimates  it a Neighbourhood of $\mathcal{I}^{+}$}
\label{sec:RWeightedEnergyEstimates}
The main idea is to derive a non-degenerate X estimate and then derive similar estimates for its boundary terms.
From now on we work with the null $(u,v)$ coordinates (see Appendix \ref{sec:ConstructingTheExtentionOfReissnerNordstrOM}) unless otherwise stated.  
\begin{proposition}
Suppose $p<3$. There exists a constant $C$ that depends  on $M$ and $\tilde{\Sigma}_{0}$ such that if $\psi$ satisfies the wave equation and $\phi=r\psi$ then
\begin{equation}
\begin{split}
&\int_{\tilde{N}_{\tau_{2}}}{r^{p}\frac{\left(\partial_{v}\phi\right)^{2}}{r^{2}}}+\int_{\tilde{\mathcal{D}}_{\tau_{1}}^{\tau_{2}}}{r^{p-1}\left(p+2\right)\frac{\left(\partial_{v}\phi\right)^{2}}{r^{2}}}+\int_{\tilde{\mathcal{D}}_{\tau_{1}}^{\tau_{2}}}{\frac{r^{p-1}}{4}\left(-pD-rD'\right)\left|\nabb\psi\right|^{2}}\\
&\ \ \ \ \ \ \ \ \ \leq C\int_{\tilde{\Sigma}_{\tau_{1}}}{J_{\mu}^{T}\left(\psi\right)n^{\mu}_{\tilde{\Sigma}_{\tau_{1}}}}+\int_{\tilde{N}_{\tau_{1}}}{r^{p}\frac{\left(\partial_{v}\phi\right)^{2}}{r^{2}}}
\end{split}
\label{rwest}
\end{equation}
\label{rweigestprop}
\end{proposition}

\begin{proof}
We first consider the cut-off function $\left.\zeta:\left[R,\!\!\right.\left.+\infty\right)\right.\rightarrow [0,1]$ such that 
\begin{equation*}
\begin{split}
&\zeta(r)=0\text{   for all } r\in\left[R,R+1/2\right],\\
&\zeta(r)=1\text{   for all } r\in\left[\left.\!R+1,+\infty \right)\right. .\\
\end{split}
\end{equation*}
Let $q=p-2$. We consider the vector field
\begin{equation*}
V=r^{q}\partial_{v}
\end{equation*}
which we apply as multiplier acting on the function $\zeta\phi$ in the region $\tilde{\mathcal{D}}_{\tau_{1}}^{\tau_{2}}$. Then
\begin{equation*}
\int_{\tilde{\mathcal{D}}_{\tau_{1}}^{\tau_{2}}}{K^{V}(\zeta\phi)+\mathcal{E}^{V}(\zeta\phi)}=\int_{\partial\tilde{\mathcal{D}}_{\tau_{1}}^{\tau_{2}}}{J_{\mu}^{V}(\zeta\phi)n^{\mu}}.
\end{equation*}
For the direction of $n^{\mu}$ see Appendix \ref{sec:StokesTheoremOnLorentzianManifolds}. Note that for $r\geq R+1$ we have $K^{V}(\zeta\phi)=K^{V}(\phi)$ and $\mathcal{E}^{V}(\zeta\phi)=\mathcal{E}^{V}(\phi)$. Then,
\begin{equation*}
\begin{split}
K^{V}(\phi)&= \textbf{T}_{\mu\nu}\left(\nabla^{\mu}\left(r^{q}\partial_{v}\right)\right)^{\nu}=\textbf{T}_{\mu\nu}\left(\left(\nabla^{\mu}r^{q}\right)\partial_{v}\right)^{\nu}+\textbf{T}_{\mu\nu}r^{q}\left(\nabla^{\mu}\partial_{v}\right)^{\nu}\\
&=r^{q}K^{V}+\textbf{T}_{\mu v}\left(\nabla^{\mu}r^{q}\right)\\
&=2r^{q-1}(\partial_{u}\phi)(\partial_{v}\phi)+qr^{q-1}(\partial_{v}\phi)^{2}+\frac{r^{q-1}}{4}\left(-qD-rD'\right)\left|\nabb\phi\right|^{2}.
\end{split}
\end{equation*}
Note that since $\psi$ solves the wave equation $\phi$ satisfies
\begin{equation*}
\begin{split}
\frac{4}{D}\partial_{u}\partial_{v}\phi+\frac{D'}{r}\phi-\lapp\phi=0
\end{split}
\end{equation*}
and so 
\begin{equation*}
\begin{split}
\Box_{g}\phi&=-\frac{2}{r}(\partial_{u}\phi-\partial_{v}\phi)-\frac{4}{D}(\partial_{u}\partial_{v}\phi)+\lapp\phi\\
&=-\frac{2}{r}(\partial_{u}\phi-\partial_{v}\phi)+\frac{D'}{r}\phi,
\end{split}
\end{equation*}
which, as expected, depends only on the 1-jet of $\phi$. Therefore,
\begin{equation*}
\begin{split}
\mathcal{E}^{V}(\phi)&=r^{q}(\partial_{v}\phi)(\Box_{g}\phi)\\
&=-2r^{q-1}(\partial_{u}\phi)(\partial_{v}\phi)+2r^{q-1}(\partial_{v}\phi)^{2}+D'r^{q-1}\phi(\partial_{v}\phi).
\end{split}
\end{equation*}
Thus
\begin{equation*}
\begin{split}
K^{V}(\phi)+\mathcal{E}^{V}(\phi)=&(q+2)r^{q-1}(\partial_{v}\phi)^{2}+\frac{r^{q-1}}{4}\left(-qD-rD'\right)\left|\nabb\phi\right|^{2}\\ &+D'r^{q-1}\phi(\partial_{v}\phi).
\end{split}
\end{equation*}
However,
\begin{equation*}
\begin{split}
\int_{\tilde{\mathcal{D}}_{\tau_{1}}^{\tau_{2}}}{D'r^{q-1}\z\phi\left(\partial_{v}\z\phi\right)}=& \int_{\tilde{\mathcal{D}}_{\tau_{1}}^{\tau_{2}}}{r^{q-4}\frac{M}{2}D\left[\sqrt{D}(1-q)-\frac{3M}{r}\right]\left(\z\phi\right)^{2}}\\&-\int_{\Delta}{\frac{r^{q-1}}{4}D'\sqrt{D}\left(\z\phi\right)^{2}}+\int_{\mathcal{I}^{+}}{\frac{D'D}{4}r^{q-1}\left(\z\phi\right)^{2}}.
\end{split}
\end{equation*}
Note that in Minkowski spacetime we would have no zeroth order term in the wave equation. In our case we do have, in such a way, however, such that the terms on the right hand side of the above identity have the right sign for $p<3$ and sufficiently large\footnote{Clearly we need to take $R>2M$.} $R$. 

In view of the cut-off function $\zeta$ all the integrals over $\Delta$ vanish. Clearly, all error terms that arise in the region\footnote{The weights in $r$ play no role in this region.} $\mathcal{W}=\overline{\operatorname{supp} (\z-1)}=\left\{R\leq r\leq R+1\right\}$  are quadratic forms of the 1-jet of $\psi$ and, therefore, in view of \eqref{degX} and \eqref{moraw}, these integrals are bounded by $\displaystyle\int_{\tilde{\Sigma}_{\tau_{1}}}{\!\! J_{\mu}^{T}\left(\psi\right)n^{\mu}_{\tilde{\Sigma}_{\tau_{1}}}}$. Also
\begin{equation*}
\begin{split}
\int_{\partial\tilde{\mathcal{D}}_{\tau_{1}}^{\tau_{2}}}{J_{\mu}^{V}(\zeta\phi)n^{\mu}}=\int_{\tilde{N}_{\tau_{1}}}{r^{q}\left(\partial_{v}\zeta\phi\right)^{2}}-\int_{\tilde{N}_{\tau_{2}}}{r^{q}\left(\partial_{v}\zeta\phi\right)^{2}}-\int_{\mathcal{I}^{+}}{\frac{D}{4}\left|\nabb\phi\right|^{2}}.
\end{split}
\end{equation*}
The last two integrals on the right hand side appear with the right sign. Finally, in view of the first Hardy inequality, the error terms produced by the cut-off $\zeta$ in the region $\mathcal{W}$ are controlled by  the flux of $T$ through $\tilde{\Sigma}_{\tau_{1}}$. 

\end{proof}
The reason we introduced the function $\phi$ is because the weight $r$ that it contains makes it non-degenerate ($\psi=0$ on $\mathcal{I}^{+}$ but $\phi$ does not vanish there in general). The reason we have divided by $r^{2}$ in \eqref{rwest} is because we want to emphasize the weight that corresponds to $\psi$ and not to $\phi$\footnote{This makes this $p$ to be the same as in \cite{new}. The notation $\phi$ and $\psi$ has however been swapped.}.

A first application of the above $r$-weighted energy estimate is the following
\begin{proposition}
There exists a constant $C$ that depends  on $M$ and $\tilde{\Sigma}_{0}$  such that if $\psi$ satisfies the wave equation and $\tilde{\mathcal{D}}_{\tau_{1}}^{\tau_{2}}$  as defined above with $R$ sufficiently large, then 
\begin{equation*}
\int_{\tau_{1}}^{\tau_{2}}{\left(\int_{\tilde{N}_{\tau}}{J^{T}_{\mu}(\psi)n^{\mu}_{\tilde{N}_{\tau}}}\right)d\tau} \,\leq\, C\int_{\tilde{\Sigma}_{\tau_{1}}}{J_{\mu}^{T}\left(\psi\right)n^{\mu}_{\tilde{\Sigma}_{\tau_{1}}}}+C\int_{\tilde{N}_{\tau_{1}}}{r^{-1}\left(\partial_{v}\phi\right)^{2}}.
\end{equation*}
\label{nondegx}
\end{proposition}
\begin{proof}
Applying Proposition \ref{rweigestprop} for $p=1$ and using the fact that for $r\geq R$ and $R$ large enough 
\begin{equation*}
D-rD'>\frac{1}{2},
\end{equation*}
we have that there exists a constant $C$ that depends  on $M$ and $\tilde{\Sigma}_{0}$  such that
\begin{equation}
\int_{\tilde{\mathcal{D}}_{\tau_{1}}^{\tau_{2}}}{\frac{1}{r^{2}}\left(\partial_{v}\phi\right)^{2}+\frac{1}{r^{2}}\left|\nabb\phi\right|^{2}}\leq C\int_{\tilde{\Sigma}_{\tau_{1}}}{J_{\mu}^{T}\left(\psi\right)n^{\mu}_{\tilde{\Sigma}_{\tau_{1}}}}+C\int_{\tilde{N}_{\tau_{1}}}{r^{-1}\left(\partial_{v}\phi\right)^{2}}.
\label{corp-1}
\end{equation}
Note now that since $\left|\nabb\phi\right|^{2}=r^{2}\left|\nabb\psi\right|^{2}$, \eqref{corp-1} yields
\begin{equation*}
\int_{\tilde{\mathcal{D}}_{\tau_{1}}^{\tau_{2}}}{\left|\nabb\psi\right|^{2}}\ \,\leq C\int_{\tilde{\Sigma}_{\tau_{1}}}{J_{\mu}^{T}\left(\psi\right)n^{\mu}_{\tilde{\Sigma}_{\tau_{1}}}}+C\int_{\tilde{N}_{\tau_{1}}}{r^{-1}\left(\partial_{v}\phi\right)^{2}}.
\end{equation*}
Furthermore, for sufficiently large $R$ we have
\begin{equation*}
\begin{split}
\int_{\tilde{\mathcal{D}}_{\tau_{1}}^{\tau_{2}}}{\frac{1}{r^{2}}\left(\partial_{v}\phi\right)^{2}}\,\geq & \int_{\tilde{\mathcal{D}}_{\tau_{1}}^{\tau_{2}}}{\frac{1}{2D^{2}r^{2}}\left(\partial_{v}\phi\right)^{2}}\,=\\
=&\int_{\tilde{\mathcal{D}}_{\tau_{1}}^{\tau_{2}}}{\frac{1}{2D^{2}}\left(\partial_{v}\psi\right)^{2}}+\int_{\tilde{\mathcal{D}}_{\tau_{1}}^{\tau_{2}}}{\frac{1}{4Dr^{2}}\partial_{v}(r\psi^{2})}.
\end{split}
\end{equation*}
However, if $\zeta$ is the cut-off function introduced in the proof of Proposition \ref{rweigestprop}, then
\begin{equation*}
\begin{split}
\int_{\tilde{\mathcal{D}}_{\tau_{1}}^{\tau_{2}}}{\frac{1}{4Dr^{2}}\partial_{v}\left(r(\zeta\psi)^{2}\right)}=
\int_{\mathcal{I}^{+}}{\frac{1}{8r}(\zeta\psi)^{2}}.
\end{split}
\end{equation*}
Therefore, the above integral is of the right sign modulo some error terms in the region $\mathcal{W}$ coming from the cut-off $\zeta$. These terms are quadratic in the 1-jet of $\psi$ and so can be controlled by estimates \eqref{degX} and \eqref{moraw}. Finally, since $n^{\mu}_{\tilde{N}_{\tau}}$ is null we have $J^{T}_{\mu}(\psi)n^{\mu}_{\tilde{N}_{\tau}}\sim \left(\partial_{v}\psi\right)^{2}+\left|\nabb\psi\right|^{2}$ and thus by \eqref{corp-1} and the coarea formula we have the required result.
\end{proof}

This is a spacetime estimate which does \textbf{not} degenerate at infinity. Note the importance of the fact that the region $\tilde{\mathcal{D}}_{\tau_{1}}^{\tau_{2}}$ does not contain $i^{0}$! If we are to obtain the full decay for the energy, then we need to prove decay for the boundary terms in Proposition \ref{nondegx}. The first step is to derive a spacetime estimate of the $r$-weighted quantity $r^{-1}(\partial_{v}\phi)^{2}$.
\begin{proposition}
There exists a constant $C$ which depends  on $M$  and $\tilde{\Sigma}_{0}$  such that
\begin{equation*}
\begin{split}
\int_{\tau_{1}}^{\tau_{2}}{\left(\int_{\tilde{N}_{\tau}}{r^{-1}(\partial_{v}\phi)^{2}}\right)d\tau}\ \leq C\int_{\tilde{\Sigma}_{\tau_{1}}}{J_{\mu}^{T}\left(\psi\right)n^{\mu}_{\tilde{\Sigma}_{\tau_{1}}}}+C\int_{\tilde{N}_{\tau_{1}}}{\left(\partial_{v}\phi\right)^{2}}.
\end{split}
\end{equation*}
\label{rwe1}
\end{proposition}
\begin{proof}
Appying the $r$-weighted energy estimate for $p=2$ we obtain
\begin{equation*}
\begin{split}
\int_{\tilde{\mathcal{D}}_{\tau_{1}}^{\tau_{2}}}{r^{-1}(\partial_{v}\phi)^{2}}\, &\leq\!\!\int_{\tilde{\mathcal{D}}_{\tau_{1}}^{\tau_{2}}}{\frac{M\sqrt{D}}{4r^{2}}\left|\nabb\phi\right|^{2}}+C\int_{\tilde{\Sigma}_{\tau_{1}}}{J_{\mu}^{T}\left(\psi\right)n^{\mu}_{\tilde{\Sigma}_{\tau_{1}}}}\!+\! C\!\!\int_{\tilde{N}_{\tau_{1}}}{\left(\partial_{v}\phi\right)^{2}}\\
&\leq C\int_{\tilde{\mathcal{D}}_{\tau_{1}}^{\tau_{2}}}{\frac{1}{r^{2}}\left|\nabb\phi\right|^{2}}+C\int_{\tilde{\Sigma}_{\tau_{1}}}{J_{\mu}^{T}\left(\psi\right)n^{\mu}_{\tilde{\Sigma}_{\tau_{1}}}}\!+C\!\int_{\tilde{N}_{\tau_{1}}}{\left(\partial_{v}\phi\right)^{2}}.
\end{split}
\end{equation*}
The result now follows from \eqref{corp-1} and the coarea formula.
\end{proof}

\subsection{Integrated Decay of Local (Higher Order) Energy}
\label{sec:DecayOfLocalEnergy}

We have already seen that in order to obtain a non-degenerate spacetime estimate near $\mathcal{H}^{+}$ we need to commute the wave equation with the transversal to the horizon vector field $\partial_{r}$ and assume that the zeroth spherical harmonic vanishes. Indeed, if $\mathcal{A}$ is a spatially compact neighbourhood of $\mathcal{H}^{+}$ (which may contain the photon sphere) then we have:
\begin{proposition}
There exists a constant $C$ that depends  on $M$ and $\tilde{\Sigma}_{0}$   such that  if $\psi$ satisfies the wave equation and is supported on $l\geq 1$, then 
\begin{equation*}
\begin{split}
\int_{\tau_{1}}^{\tau_{2}}{\left(\int_{\mathcal{A}\cap\tilde{\Sigma}_{\tau}}{J^{N}_{\mu}(\psi)n^{\mu}_{\tilde{\Sigma}_{\tau}}}\right)d\tau}\ \leq  \  &C\int_{\tilde{\Sigma}_{\tau_{1}}}{J^{N}_{\mu}(\psi)n^{\mu}_{\tilde{\Sigma}_{\tau}}}+C\int_{\tilde{\Sigma}_{\tau_{1}}}{J^{N}_{\mu}(T\psi)n^{\mu}_{\tilde{\Sigma}_{\tau}}}\\&+C\int_{\mathcal{A}\cap\tilde{\Sigma}_{\tau_{1}}}{J^{N}_{\mu}(\partial_{r}\psi)n^{\mu}_{\tilde{\Sigma}_{\tau}}}.
\end{split}
\end{equation*}
\label{inte1}
\end{proposition}
\begin{proof}
Immediate from  Proposition \ref{r1back}, \eqref{x} and the coarea formula.
\end{proof}
As regards the above boundary terms we have
\begin{proposition}
There exists a constant $C$ that depends  on $M$ and $\tilde{\Sigma}_{0}$   such that  if $\psi$ satisfies the wave equation and is supported on $l\geq 2$, then 
\begin{equation*}
\int_{\tau_{1}}^{\tau_{2}}{\left(\int_{\mathcal{A}\cap\tilde{\Sigma}_{\tau}}{J^{N}_{\mu}(\partial_{r}\psi)n^{\mu}_{\tilde{\Sigma}_{\tau}}}\right)d\tau}\ \leq
C\sum_{i=0}^{2}{\int_{\tilde{\Sigma}_{\tau_{1}}}{J^{N}_{\mu}(T^{i}\psi)n^{\mu}_{\tilde{\Sigma}_{\tau}}}}+C\sum_{k=1}^{2}{\int_{\mathcal{A}\cap\tilde{\Sigma}_{\tau_{1}}}{J^{N}_{\mu}(\partial_{r}^{k}\psi)n^{\mu}_{\tilde{\Sigma}_{\tau}}}}.
\end{equation*}
\label{inte2}
\end{proposition}
\begin{proof}
Immediate from Theorem \ref{kder} and the coarea formula.
\end{proof}

\subsection{Weighted Energy Estimates  in a Neighbourhood of $\mathcal{H}^{+}$}
\label{sec:RWeightedEnergyEstimatesH}
Since for $l=0$ the above estimates do not hold for generic initial data, we are left proving decay for the degenerate energy. For this we derive a hierarchy of (degenerate) energy estimates in a neighbourhood of $\mathcal{H}^{+}$. In this section, we use the $(v,r)$ coordinates.

\begin{proposition}
There exists a $\varphi_{\tau}$-invariant causal vector field $P$ and a constant $C$ which depends only on $M$ such that  for all $\psi$ we have 
\begin{equation*}
\begin{split}
& J_{\mu}^{T}(\psi)n^{\mu}_{\Sigma}\leq  C K^{P}(\psi),\\
& J_{\mu}^{P}(\psi)n^{\mu}_{\Sigma}\leq C K^{N,\delta,-\frac{1}{2}}(\psi)
\end{split}
\end{equation*}
in an appropriate neighbourhood $\mathcal{A}$ of $\mathcal{H}^{+}$.
\label{rstarHenergy}
\end{proposition}
\begin{proof}
Let our ansatz be $P=f_{v}\partial_{v}+f_{r}\partial_{r}$. Recall that 
\begin{equation*}
\begin{split}
K^{P}\left(\psi\right)=&F_{vv}\left(\partial_{v}\psi\right)^{2}+F_{rr}\left(\partial_{r}\psi\right)^{2}+F_{\scriptsize\nabb}\left|\nabb\psi\right|^{2}+F_{vr}\left(\partial_{v}\psi\right)\left(\partial_{r}\psi\right),\\
\end{split}
\end{equation*}
where the coef{}ficients are given by
\begin{equation*}
\begin{split}
&F_{vv}=\left(\partial_{r}f_{v}\right),\\
&F_{rr}=D\left[\frac{\left(\partial_{r}f_{r}\right)}{2}-\frac{f_{r}}{r}\right]-\frac{f_{r}D'}{2},  \\
&F_{\scriptsize\nabb}=-\frac{1}{2}\left(\partial_{r}f_{r}\right),\\
&F_{vr}=D\left(\partial_{r}f_{v}\right)-\frac{2f_{r}}{r}.\\
\end{split}
\end{equation*}
Let us take $f_{r}(r)=-\sqrt{D}$ for $M\leq r\leq r_{0}<2M$, with $r_{0}$ to be determined later. Then
\begin{equation*}
\begin{split}
F_{rr}&=D\left[-\frac{D'}{4\sqrt{D}}+\frac{\sqrt{D}}{r}\right]+\frac{\sqrt{D}D'}{2}\\
&=D\left[\frac{D'}{4\sqrt{D}}+\frac{\sqrt{D}}{r}\right]\\
&\sim D.
\end{split}
\end{equation*}
Since $\frac{D'}{4\sqrt{D}}=\frac{M}{2r^{2}}$, the constants in $\sim$ depend on $M$ and the choice for $r_{0}$. Also,
\begin{equation*}
\begin{split}
F_{vr}=\sqrt{D}\left[\sqrt{D}(\partial_{r}f_{v})+\frac{2}{r}\right]\leq \epsilon D+\frac{1}{\epsilon}\left[\sqrt{D}(\partial_{r}f_{v})+\frac{2}{r}\right]^{2}.
\end{split}
\end{equation*}
If we take $\epsilon$ sufficiently small  and $f_{v}$ such that 
\begin{equation*}
\begin{split}
\frac{1}{\epsilon}\left[\sqrt{D}(\partial_{r}f_{v})+\frac{2}{r}\right]^{2} < \partial_{r}f_{v},
\end{split}
\end{equation*}
then  there exists $r_{0}>M$ such that
\begin{equation}
\begin{split}
K^{P}(\psi)\sim \left((\partial_{v}\psi)^{2}+D(\partial_{r}\psi)^{2}+\left|\nabb\psi\right|^{2}\right)\sim J_{\mu}^{T}n^{\mu}_{\Sigma}
\label{p1} 
\end{split}
\end{equation}
 in $\mathcal{A}=\left\{M\leq r\leq r_{0}\right\}$. Extend now $P$ in $\mathcal{R}$ such that $f_{v}(r)=1$ and $f_{r}(r)=0$ for all $r\geq r_{1}>r_{0}$ for some $r_{1}<2M$. This proves the first part of the proposition.  In region $\mathcal{A}$ we have $-g(P,P)\sim\sqrt{D}$ and so $\omega_{P}\sim \sqrt{D}$. Therefore, from \eqref{GENERALT} (See Appendix \ref{sec:TheHyperbolicityOfTheWaveEquation1}) we obtain 
\begin{equation}
\begin{split}
J_{\mu}^{P}(\psi)n^{\mu}_{\Sigma}&
\sim (\partial_{v}\psi)^{2}+\sqrt{D}(\partial_{r}\psi)^{2}+\left|\nabb\psi\right|^{2}\\
&\sim K^{N,\delta,-\frac{1}{2}}(\psi).
\label{p2}
\end{split}
\end{equation}
\end{proof}
The vector field $P$ will be very crucial for achieving decay results. It is timelike in the domain of outer communications and becomes null on the horizon  ``linearly". This linearity allows $P$ to capture  the degenerate redshift in $\mathcal{A}$  in a weaker way than $N$ but in stronger way than $T$.

\subsection{Decay of Degenerate Energy}
\label{sec:DecayOfDegenerateEnergy}

\subsubsection{Uniform Boundedness of $P$-Energy}
\label{sec:UniformBoundenessOfPEnergy}

First we need to prove that the $P$-flux is uniformly bounded.

\begin{proposition}
There exists a constant $C$ that depends  on $M$  and $\tilde{\Sigma}_{0}$  such that for all solutions $\psi$ of the wave equation we have
\begin{equation}
\int_{\tilde{\Sigma}_{\tau}}{J_{\mu}^{P}(\psi)n^{\mu}_{\tilde{\Sigma}_{\tau}}}\leq C\int_{\tilde{\Sigma}_{0}}{J_{\mu}^{P}(\psi)n^{\mu}_{\tilde{\Sigma}_{0}}}.
\label{pboun}
\end{equation}
\label{pbound}
\end{proposition}
\begin{proof}
Stokes' theorem for the current $J_{\mu}^{P}$ gives us
\begin{equation*}
\int_{\tilde{\Sigma}_{\tau}}{J_{\mu}^{P}n^{\mu}}+\int_{\mathcal{H}^{+}}{J_{\mu}^{P}n^{\mu}}+\int_{\mathcal{I}^{+}}{J_{\mu}^{P}n^{\mu}}+\int_{\tilde{\mathcal{R}}}{K^{P}}=\int_{\tilde{\Sigma}_{0}}{J_{\mu}^{P}n^{\mu}}.
\end{equation*}
Note that since $P$ is a future-directed causal vector field, the boundary integrals over $\mathcal{H}^{+}$ and $\mathcal{I}^{+}$ are non-negative. The same goes also for $K^{P}$ in region $\mathcal{A}$ whereas it vanishes away from the horizon.  In the intermediate region this spacetime integral can be bounded using \eqref{degX}. The result now follows from 
\begin{equation*} 
J_{\mu}^{T}n^{\mu} \leq CJ_{\mu}^{P}n^{\mu}.
\end{equation*}
\end{proof}
We are now in a position to derive local integrated decay for the $T$-energy.
\begin{proposition}
There exists a constant $C$ that depends  on $M$ and $\tilde{\Sigma}_{0}$   such that for all solutions $\psi$ of the wave equation   we have
\begin{equation*}
\int_{\tau_{1}}^{\tau_{2}}{\left(\int_{\mathcal{A}\cap\tilde{\Sigma}_{\tau}}{J_{\mu}^{T}(\psi)n^{\mu}_{\tilde{\Sigma}_{\tau}}}\right)d\tau}\leq C\int_{\tilde{\Sigma}_{\tau_{1}}}{J_{\mu}^{P}(\psi)n^{\mu}_{\tilde{\Sigma}_{\tau_{1}}}}
\end{equation*}
and
\begin{equation*}
\int_{\tau_{1}}^{\tau_{2}}{\left(\int_{\mathcal{A}\cap\tilde{\Sigma}_{\tau}}{J_{\mu}^{P}(\psi)n^{\mu}_{\tilde{\Sigma}_{\tau}}}\right)d\tau}\leq C\int_{\tilde{\Sigma}_{\tau_{1}}}{J_{\mu}^{N}(\psi)n^{\mu}_{\tilde{\Sigma}_{\tau_{1}}}}
\end{equation*}
in an appropriate $\varphi_{\tau}$-invariant neighbourhood $\mathcal{A}$ of $\mathcal{H}^{+}$.
\label{propp1}
\end{proposition}
\begin{proof}
From the divergence identity for the current $J_{\mu}^{P}$ and the boundedness of $P$-energy we have
\begin{equation*}
\int_{\mathcal{A}}{K^{P}}\leq C\int_{\tilde{\Sigma}_{\tau_{1}}}{J_{\mu}^{P}(\psi)n^{\mu}_{\tilde{\Sigma}_{\tau_{1}}}}
\end{equation*}
for a uniform constant $C$. Thus the first estimate follows from \eqref{p1} and the coarea formula. Likewise, the second estimate follows from the divergence identity for the current $J_{\mu}^{N,\delta, -\frac{1}{2}}$, the boundedness of the non-degenerate $N$-energy and \eqref{p2}.
\end{proof}

\subsubsection{The Dyadic Sequence $\rho_{n}$}
\label{sec:TheDyadicSequenceTildeTauN}

In view of \eqref{x} and Propositions  \ref{nondegx} and \ref{propp1} we have
\begin{equation}
\begin{split}
\int_{\tau_{1}}^{\tau_{2}}{\left(\int_{\tilde{\Sigma}_{\tau}}{J^{T}_{\mu}(\psi)n^{\mu}_{\tilde{\Sigma}_{\tau}}}\right)d\tau}\, \leq\,  CI^{T}_{\tilde{\Sigma}_{\tau_{1}}}(\psi),
\label{tinte1}
\end{split}
\end{equation}
where 
\begin{equation*}
\begin{split}
I^{T}_{\tilde{\Sigma}_{\tau}}(\psi)=&\int_{\tilde{\Sigma}_{\tau}}{J^{P}_{\mu}(\psi)n^{\mu}_{\tilde{\Sigma}_{\tau}}}+
\int_{\tilde{\Sigma}_{\tau}}{J^{T}_{\mu}(T\psi)n^{\mu}_{\tilde{\Sigma}_{\tau}}}+\int_{\tilde{N}_{\tau}}{r^{-1}\left(\partial_{v}\phi\right)^{2}}.
\end{split}
\end{equation*}
Moreover, from Propositions \ref{propp1} and \ref{rwe1} we have
\begin{equation}
\int_{\tau_{1}}^{\tau_{2}}{I_{\tilde{\Sigma}_{\tau}}^{T}(\psi)d\tau}\leq CI_{\tilde{\Sigma}_{\tau_{1}}}^{T}(T\psi)+C\int_{\tilde{\Sigma}_{\tau_{1}}}{J_{\mu}^{N}(\psi)n^{\mu}}+C\int_{\tilde{N}_{\tau_{1}}}{(\partial_{v}\phi)^{2}},
\label{tinte2}
\end{equation}
for a  constant $C$ that  depends  on  $M$  and $\tilde{\Sigma}_{0}$. This implies that there exists a dyadic sequence\footnote{Dyadic sequence is an increasing sequence $\rho_{n}$ such that $\rho_{n}\sim\rho_{n+1}\sim(\rho_{n+1}-\rho_{n}).$} $\rho_{n}$ such that 
\begin{equation*}
I^{T}_{\tilde{\Sigma}_{\rho_{n}}}(\psi)\leq \frac{E_{1}}{\rho_{n}},
\end{equation*}
where $E_{1}$ is equal to the right hand side of \eqref{tinte2} (with $\tau_{1}=0$) and depends only on the initial data of $\psi$. We have now all the tools to derive decay for the degenerate energy.
\begin{proposition}
There exists a constant $C$ that depends on $M$   and $\tilde{\Sigma}_{0}$ such that for all solutions $\psi$ of the wave equation we have 
\begin{equation*}
\int_{\tilde{\Sigma}_{\tau}}J^{T}_{\mu}(\psi)n_{\tilde{\Sigma}_{\tau}}^{\mu}\leq CE_{1}\frac{1}{\tau^{2}},
\end{equation*}
where $E_{1}$ is as defined above.
\label{tdecay}
\end{proposition}
\begin{proof}
We apply \eqref{tinte1} for the dyadic interval $[\rho_{n},\rho_{n+1}]$ to obtain
\begin{equation*}
\int_{\rho_{n}}^{\rho_{n+1}}{\left(\int_{\tilde{\Sigma}_{\tau}}{J^{T}_{\mu}(\psi)n^{\mu}_{\tilde{\Sigma}_{\tau}}}\right)d\tau}\, \leq\,  CE_{1}\frac{1}{\rho_{n}}.
\end{equation*}
In view of the energy estimate
\begin{equation*}
\begin{split}
\int_{\tilde{\Sigma}_{\tau}}{J^{T}_{\mu}(\psi)n^{\mu}_{\tilde{\Sigma}_{\tau}}}\ \leq\, C\int_{\tilde{\Sigma}_{\tau'}}{J^{T}_{\mu}(\psi)n^{\mu}_{\tilde{\Sigma}_{\tau'}}},
\end{split}
\end{equation*}
which holds for all $\tau\geq \tau'$, we have
\begin{equation*}
\begin{split}
(\rho_{n+1}-\rho_{n})\int_{\tilde{\Sigma}_{\rho_{n+1}}}{J^{T}_{\mu}(\psi)n^{\mu}_{\tilde{\Sigma}_{\rho_{n+1}}}}\,\leq\, CE_{1}\frac{1}{\rho_{n+1}}.
\end{split}
\end{equation*}
Since there exists a uniform constant $b>0$ such that $b\tau_{n+1}\leq\tau_{n+1}-\tau_{n}$ we have 
\begin{equation*}
\begin{split}
\int_{\tilde{\Sigma}_{\rho_{n+1}}}{J^{T}_{\mu}(\psi)n^{\mu}_{\tilde{\Sigma}_{\rho_{n+1}}}}\,\leq\, CE_{1}\frac{1}{\rho_{n+1}^{2}}.
\end{split}
\end{equation*}
Now, for $\tau\geq\rho_{1}$ there exists $n\in\mathbb{N}$ such that $\rho_{n}\leq\tau\leq\rho_{n+1}$. Therefore,
\begin{equation*}
\begin{split}
\int_{\tilde{\Sigma}_{\tau}}{J^{T}_{\mu}(\psi)n^{\mu}_{\tilde{\Sigma}_{\tau}}}\,\leq\, C\int_{\tilde{\Sigma}_{\rho_{n+1}}}{J^{T}_{\mu}(\psi)n^{\mu}_{\tilde{\Sigma}_{\rho_{n+1}}}}\, \leq\frac{CE_{1}}{\rho_{n}^{2}}\sim\frac{CE_{1}}{\rho_{n+1}^{2}}\leq CE_{1}\frac{1}{\tau^{2}},
\end{split}
\end{equation*}
which is the required decay result for the $T$-energy. 

\end{proof}

\subsection{Decay of Non-Degenerate Energy}
\label{sec:DecayOfNonDegenerateEnergy}

We  now derive decay for the non-degenerate energy. Note that for obtaining such a result we must use Proposition \ref{inte1} which however holds for waves supported on the frequencies $l\geq 1$. In this case, in view of the previous estimates we have
\begin{equation}
\begin{split}
\int_{\tau_{1}}^{\tau_{2}}{\left(\int_{\tilde{\Sigma}_{\tau}}{J^{N}_{\mu}(\psi)n^{\mu}_{\tilde{\Sigma}_{\tau}}}\right)d\tau}\, \leq\, & CI^{N}_{\tilde{\Sigma}_{\tau_{1}}}(\psi),
\end{split}
\label{ninte1}
\end{equation}
where 
\begin{equation*}
\begin{split}
I^{N}_{\tilde{\Sigma}_{\tau}}(\psi)=&\int_{\tilde{\Sigma}_{\tau}}{J^{N}_{\mu}(\psi)n^{\mu}_{\tilde{\Sigma}_{\tau}}}+
\int_{\tilde{\Sigma}_{\tau}}{J^{N}_{\mu}(T\psi)n^{\mu}_{\tilde{\Sigma}_{\tau}}}\\
&+\int_{\mathcal{A}\cap\tilde{\Sigma}_{\tau}}{J^{N}_{\mu}(\partial_{r}\psi)n^{\mu}_{\tilde{\Sigma}_{\tau}}}+\int_{\tilde{N}_{\tau}}{r^{-1}\left(\partial_{v}\phi\right)^{2}}.
\end{split}
\end{equation*}
\begin{proposition}
There exists a constant $C$ that depends  on $M$   and $\tilde{\Sigma}_{0}$ such that for all solutions $\psi$ to the wave equation which are supported on $l\geq 1$ we have
\begin{equation*}
\begin{split}
\int_{\tilde{\Sigma}_{\tau}}{J^{N}_{\mu}(\psi)n^{\mu}_{\tilde{\Sigma}_{\tau}}}\,\leq\, CE_{2}\frac{1}{\tau},
\end{split}
\end{equation*}
where $E_{2}$ depends only on the initial data of $\psi$ and is equal to the right hand side of \eqref{ninte1} (with $\tau_{1}=0$).
\label{l1decay}
\end{proposition}
\begin{proof}
We apply \eqref{ninte1} for the interval $[0,\tau]$ and use the energy estimate
\begin{equation*}
\begin{split}
\int_{\tilde{\Sigma}_{\tau}}{J^{N}_{\mu}(\psi)n^{\mu}_{\tilde{\Sigma}_{\tau}}}\ \leq\, C\int_{\tilde{\Sigma}_{\tau'}}{J^{N}_{\mu}(\psi)n^{\mu}_{\tilde{\Sigma}_{\tau'}}},
\end{split}
\end{equation*}
which holds for all $\tau\geq \tau'$ and for a uniform constant $C$.
\end{proof}

\subsubsection{The Dyadic Sequence $\tau_{n}$}
\label{sec:APriviledgedDyadicSequence}

If we  consider waves which are supported on $l\geq 2$ then from Propositions  \ref{rwe1}, \ref{inte2} and by commuting \eqref{ninte1} with $T$ we take
\begin{equation}
\begin{split}
\int_{\tau_{1}}^{\tau_{2}}{I^{N}_{\tilde{\Sigma}_{\tau}}(\psi)d\tau}\,\leq \, &CI^{N}_{\tilde{\Sigma}_{\tau_{1}}}(\psi)+CI^{N}_{\tilde{\Sigma}_{\tau_{1}}}(T\psi)\\
&+C\int_{\mathcal{A}\cap\tilde{\Sigma}_{\tau_{1}}}{J^{N}_{\mu}(\partial_{r}\partial_{r}\psi)n^{\mu}_{\tilde{\Sigma}_{\tau}}}+C\int_{\tilde{N}_{\tau_{1}}}{(\partial_{v}\phi)^{2}}
\end{split}
\label{ninte2}
\end{equation}
This implies that there exists a dyadic sequence $\tau_{n}$ such that
\begin{equation*}
\begin{split}
I^{N}_{\tilde{\Sigma}_{\tau_{n}}}(\psi) \, \leq\, \frac{E_{3}}{\tau_{n}},
\end{split}
\end{equation*}
where the constant $E_{3}$ is equal to the right hand side of \eqref{ninte2} (with $\tau_{1}=0)$. We can now derive decay for the non-degenerate energy.
\begin{proposition}
There exists a constant $C$ that depends  on $M$   and $\tilde{\Sigma}_{0}$ such that for all solutions $\psi$ to the wave equation which are supported on $l\geq 2$ we have
\begin{equation*}
\begin{split}
\int_{\tilde{\Sigma}_{\tau}}{J^{N}_{\mu}(\psi)n^{\mu}_{\tilde{\Sigma}_{\tau}}}\,\leq\, CE_{3}\frac{1}{\tau^{2}},
\end{split}
\end{equation*}
where $E_{3}$ is as defined above.
\label{energydecay}
\end{proposition}
\begin{proof}
If we apply \eqref{ninte1} for the dyadic intervals $[\tau_{n},\tau_{n+1}]$ we obtain
\begin{equation*}
\begin{split}
\int_{\tau_{n}}^{\tau_{n+1}}{\left(\int_{\tilde{\Sigma}_{\tau}}{J^{N}_{\mu}(\psi)n^{\mu}_{\tilde{\Sigma}_{\tau}}}\right)d\tau}\, \leq\, \frac{CE_{3}}{\tau_{n}}.
\end{split}
\end{equation*}
In view of the boundedness of the $N$-energy  we have
\begin{equation*}
\begin{split}
(\tau_{n+1}-\tau_{n})\int_{\tilde{\Sigma}_{\tau_{n+1}}}{J^{N}_{\mu}(\psi)n^{\mu}_{\tilde{\Sigma}_{\tau_{n+1}}}}\,\leq\, \frac{CE_{3}}{\tau_{n+1}}.
\end{split}
\end{equation*}
Since there exists a uniform constant $b>0$ such that $b\tau_{n+1}\leq\tau_{n+1}-\tau_{n}$ we obtain
\begin{equation*}
\begin{split}
\int_{\tilde{\Sigma}_{\tau_{n+1}}}{J^{N}_{\mu}(\psi)n^{\mu}_{\tilde{\Sigma}_{\tau_{n+1}}}}\,\leq\, \frac{CE_{3}}{\tau_{n+1}^{2}}.
\end{split}
\end{equation*}
Now, for $\tau\geq\tau_{1}$ there exists $n\in\mathbb{N}$ such that $\tau_{n}\leq\tau\leq\tau_{n+1}$. Therefore,
\begin{equation*}
\begin{split}
\int_{\tilde{\Sigma}_{\tau}}{J^{N}_{\mu}(\psi)n^{\mu}_{\tilde{\Sigma}_{\tau}}}\,\leq\, \tilde{C}\int_{\tilde{\Sigma}_{\tau_{n+1}}}{J^{N}_{\mu}(\psi)n^{\mu}_{\tilde{\Sigma}_{\tau_{n+1}}}}\, \leq\frac{CE_{3}}{\tau_{n}^{2}}\sim\frac{CE_{3}}{\tau_{n+1}^{2}}\leq CE_{3}\frac{1}{\tau^{2}}
\end{split}
\end{equation*}
which is the required decay result for the energy. 

\end{proof}
Summarizing, we have finally proved the following theorem
\begin{theorem}
There exists a constant $C$ that depends  on the mass $M$  and $\tilde{\Sigma}_{0}$  such that:
\begin{itemize}
	\item For all solutions $\psi$ of the wave equation we have 
\begin{equation*}
\int_{\tilde{\Sigma}_{\tau}}J^{T}_{\mu}(\psi)n_{\tilde{\Sigma}_{\tau}}^{\mu}\leq CE_{1}\frac{1}{\tau^{2}},
\end{equation*}
where $E_{1}$ depends only on the initial data and is equal to the right hand side of \eqref{tinte2} (with $\tau_{1}=0$).
\end{itemize}
\begin{itemize}
	\item For all solutions $\psi$ to the wave equation which are supported on the frequencies $l\geq 1$ we have
\begin{equation*}
\begin{split}
\int_{\tilde{\Sigma}_{\tau}}{J^{N}_{\mu}(\psi)n^{\mu}_{\tilde{\Sigma}_{\tau}}}\,\leq\, CE_{2}\frac{1}{\tau},
\end{split}
\end{equation*}
where $E_{2}$ depends only on the initial data of $\psi$ and is equal to the right hand side of \eqref{ninte1} (with $\tau_{1}=0$).
\end{itemize}
\begin{itemize}
	\item  For all solutions $\psi$ to the wave equation which are supported on $l\geq 2$ we have
\begin{equation*}
\begin{split}
\int_{\tilde{\Sigma}_{\tau}}{J^{N}_{\mu}(\psi)n^{\mu}_{\tilde{\Sigma}_{\tau}}}\,\leq\, CE_{3}\frac{1}{\tau^{2}},
\end{split}
\end{equation*}
where $E_{3}$ depends only on the initial data of $\psi$ and is equal to the right hand side of \eqref{ninte2} (with $\tau_{1}=0$).
\end{itemize}
\label{decayofenergy}
\end{theorem}

\section{Pointwise Estimates}
\label{sec:PointwiseEstimates}

\subsection{Uniform Pointwise Boundedness}
\label{sec:UniformPointwiseBoundedness}

We first prove that all the waves $\psi$ remain uniformly bounded in $\mathcal{M}$. We work with the foliation $\Sigma_{\tau}$ (or $\tilde{\Sigma}_{\tau}$) and  the induced coordinate system $(\rho,\omega)$. For $r_{0}\geq M$   we have 
\begin{equation*}
\begin{split}
\psi^{2}\left(r_{0},\omega\right)&=\left(\int_{r_{0}}^{+\infty}{\left(\partial_{\rho}\psi\right)d\rho}\right)^{2} \\
&\leq\left(\int_{r_{0}}^{+\infty}{\left(\partial_{\rho}\psi\right)^{2}\rho^{2}d\rho}\right)\left(\int_{r_{0}}^{+\infty}{\frac{1}{\rho^{2}}d\rho}\right)
=\frac{1}{r_{0}}\left(\int_{r_{0}}^{+\infty}{\left(\partial_{\rho}\psi\right)^{2}\rho^{2}d\rho}\right).
\end{split}
\end{equation*}
Therefore,
\begin{equation}
\begin{split}
\int_{\mathbb{S}^{2}}{\psi^{2}(r_{0},\omega)d\omega}&\leq\frac{1}{r_{0}}\int_{\mathbb{S}^{2}}{\int_{r_{0}}^{+\infty}{\left(\partial_{\rho}\psi\right)^{2}\rho^{2}d\rho d\omega}}\\
&\leq \frac{C}{r_{0}}\int_{\Sigma_{\tau}\cap\left\{r\geq r_{0}\right\}}{J_{\mu}^{N}(\psi)n^{\mu}_{\Sigma_{\tau}}},
\end{split}
\label{1pointwise}
\end{equation}
where $C$ is a constant that depends only on $M$ and $\Sigma_{0}$. 
\begin{theorem}
There exists a constant $C$ which depends on $M$ and $\Sigma_{0}$ such that for all solutions $\psi$ of the wave equation we have
\begin{equation}
\begin{split}
\left|\psi\right|^{2}\leq C\cdot E_{4}\frac{1}{r},
\end{split}
\label{2pointwise}
\end{equation}
where
\begin{equation*}
E_{4}=\sum_{\left|k\right|\leq 2}{\int_{\Sigma_{0}}{J_{\mu}^{N}(\Omega^{k}\psi)n^{\mu}_{\Sigma_{0}}}}.
\end{equation*}
\label{unponpsi}
\end{theorem}
\begin{proof}
From the Sobolev inequality on $\mathbb{S}^{2}$ we have
 \begin{equation*}
\left|\psi\right|^{2} \leq C\sum_{\left|k\right|\leq 2}{\int_{\mathbb{S}^{2}}{\left(\Omega^{k}\psi\right)^{2}}}
\end{equation*}
and the theorem follows from \eqref{1pointwise} and \eqref{nener}.
\end{proof}
Note that the above also gives us decay in $r$. Observe also that for the boundedness result, one could avoid commuting with angular derivatives. Indeed, we can decompose
\begin{equation*}
\psi=\psi_{0}+\psi_{\geq 1}
\end{equation*}
and  prove that each projection is uniformly bounded by a norm that depends only on initial data. Indeed, by applying the following Sobolev inequality on the hypersurfaces $\Sigma_{\tau}$ we have
\begin{equation}
\begin{split}
\left\|\psi_{\geq 1}\right\|_{L^{\infty}\left(\Sigma_{\tau}\right)}&\leq C\left(\left\|\left|\nabla\psi_{\geq 1}\right|\right\|_{L^{2}\left(\Sigma_{\tau}\right)}+\left\|\left|\nabla\nabla\psi_{\geq 1}\right|\right\|_{L^{2}\left(\Sigma_{\tau}\right)}+\lim_{x\rightarrow i^{0}}{\left|\psi_{\geq 1}\right|}\right)\\
&=
C\left(\left\|\psi_{\geq 1}\right\|_{\overset{.}{H}^{1}\left(\Sigma_{\tau}\right)}+\left\|\psi_{\geq 1}\right\|_{\overset{.}{H}^{2}\left(\Sigma_{\tau}\right)}+\lim_{x\rightarrow i^{0}}{\left|\psi_{\geq 1}\right|}\right)
\label{Sobolev}
\end{split}
\end{equation}
where again $\nabla\nabla$ is a 2-tensor on the Riemannian manifold $\Sigma_{\tau}$ and $\left|-\right|$ is the induced norm.  Clearly, $C$ depends only on $\Sigma_{0}$. Note also that we use the Sobolev inequality that does not involve the $L^{2}$-norms of zeroth order terms. We observe that the vector field $\partial_{v}-\partial_{r}$ is timelike since $g\left(\partial_{v}-\partial_{r},\partial_{v}-\partial_{r}\right)=-D-2$ and, therefore, by an elliptic estimate (see Appendix \ref{sec:EllipticEstimates}) there exists a uniform positive constant $C$ which depends on $M$ and $\Sigma_{0}$ such that 
\begin{equation*} 
\begin{split}
\left\|\psi_{\geq 1}\right\|^{2}_{\overset{.}{H}^{1}\left(\Sigma_{\tau}\right)}+\left\|\psi_{\geq 1}\right\|^{2}_{\overset{.}{H}^{2}\left(\Sigma_{\tau}\right)}\leq & C\int_{\Sigma_{\tau}}{J_{\mu}^{N}(\psi_{\geq 1})n^{\mu}_{\Sigma_{\tau}}}+C\int_{\Sigma_{\tau}}{J_{\mu}^{N}(T\psi_{\geq 1})n^{\mu}_{\Sigma_{\tau}}}\\ &+C\int_{\Sigma_{\tau}}{J_{\mu}^{N}(\partial_{r}\psi_{\geq 1})n^{\mu}_{\Sigma_{\tau}}}.
\end{split}
\end{equation*}
Pointwise bound for $\psi_{0}$ is immidiately derived by \eqref{1pointwise}.
Therefore, in view of \eqref{nener} and \eqref{uniformbounden2ndorder}  we obtain the following:

There exists a constant $C$ which depends on $M$ and $\Sigma_{0}$ such that for all solutions $\psi$ of the wave equation we have
\begin{equation}
\left|\psi\right|^{2}\leq C\!\!\int_{\Sigma_{0}}{J_{\mu}^{N}(\psi)n^{\mu}_{\Sigma_{0}}}+C\!\!\int_{\Sigma_{0}}{J_{\mu}^{N}(T\psi)n^{\mu}_{\Sigma_{0}}}+C\!\!\int_{\Sigma_{0}}{J_{\mu}^{N}(\partial_{r}\psi)n^{\mu}_{\Sigma_{0}}}.
\label{3pointwise}
\end{equation}
\label{rempoint}

\subsection{Pointwise Decay}
\label{sec:PointwiseDecay}

\subsubsection{Decay away $\mathcal{H}^{+}$}
\label{sec:DecayAwayMathcalH}

We consider the region $\left\{r\geq R_{0}\right\}$, where $R_{0}>M$. From now on, $C$ will be a constant depending only on $M$, $R_{0}$ and $\tilde{\Sigma}_{0}$.
\begin{figure}[H]
	\centering
		\includegraphics[scale=0.15]{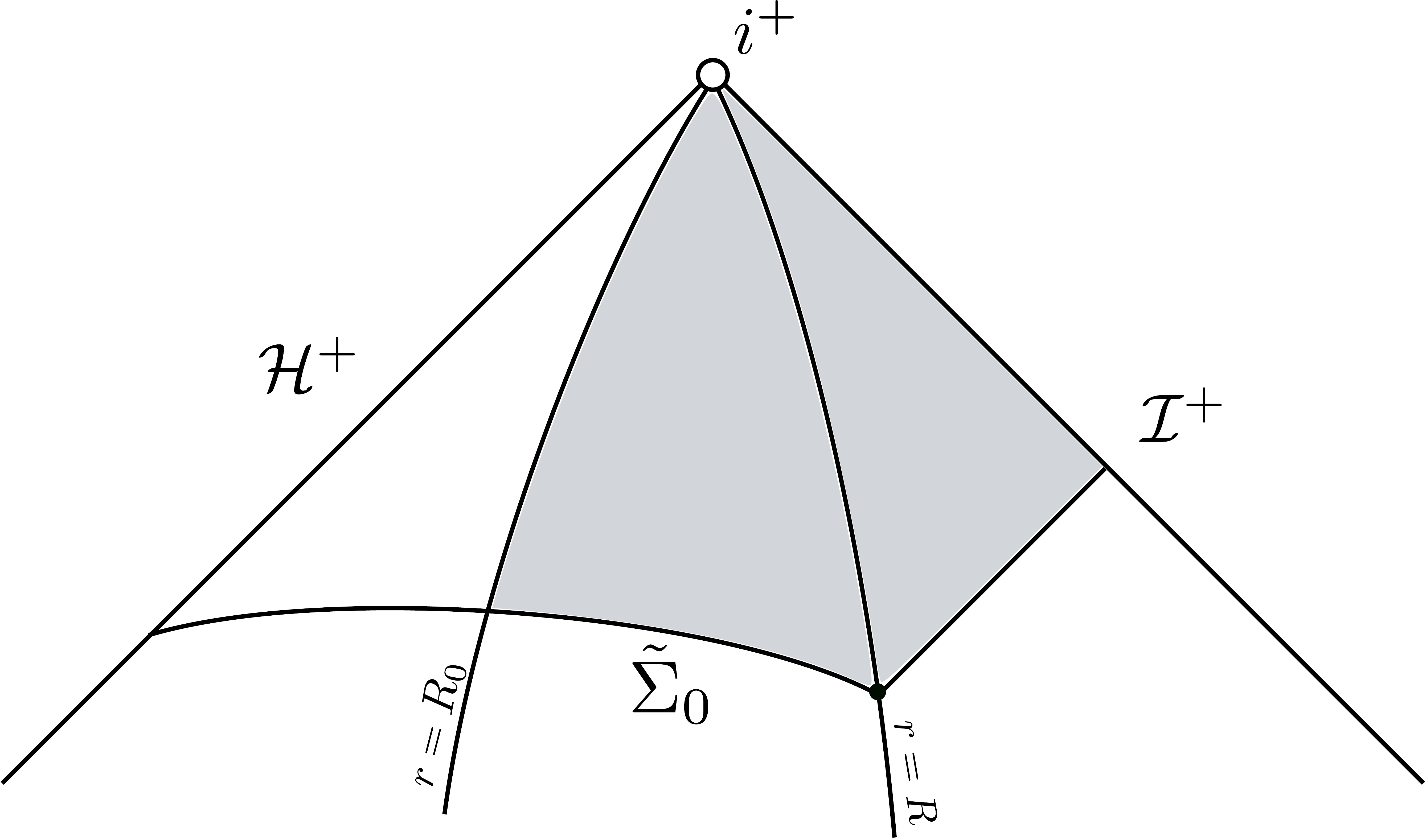}
	\label{fig:1pd}
\end{figure}
Clearly, in this region we have $J_{\mu}^{N}n^{\mu}_{\Sigma}\sim  J_{\mu}^{T}n_{\Sigma}^{\mu}$ and $\sim$ depends on $R_{0}$. Therefore, from \eqref{1pointwise} we have that for any $r\geq R_{0}$  
\begin{equation*}
\int_{\mathbb{S}^{2}}{\psi^{2}(r,\omega)d\omega}\leq \frac{C}{r}\int_{\tilde{\Sigma}_{\tau}}{J_{\mu}^{T}(\psi)n^{\mu}_{\tilde{\Sigma}_{\tau}}}\leq C\cdot E_{1}(\psi)\frac{1}{r\cdot \tau^{2}}.
\end{equation*}
Applying the above for $\psi$, $\Omega_{i}\psi$, $\Omega_{ij}\psi$ and using the  Sobolev inequality for $\mathbb{S}^{2}$ yields
\begin{equation*}
\psi^{2}\leq C E_{5}\frac{1}{r\cdot \tau^{2}},
\end{equation*}
where $C$ depends on $M$, $R_{0}$ and $\tilde{\Sigma}_{0}$ and
\begin{equation*}
E_{5}=\sum_{\left|k\right|\leq 2}{E_{1}\left(\Omega^{k}\psi\right)}.
\end{equation*}
Next we improve the decay with respect to $r$. Observe that for all $r\geq R_{0}$ we have
\begin{equation*}
\begin{split}
\int_{\mathbb{S}^{2}}{(r\psi)^{2}(r,\omega)d\omega}&=\int_{\mathbb{S}^{2}}{(R_{0}\psi)^{2}(R_{0},\omega)d\omega}+2\int_{\mathbb{S}^{2}}\int_{R_{0}}^{r}{\frac{\psi}{\rho}\partial_{\rho}(\rho\psi)\rho^{2}d\rho d\omega}\\
&\leq CE_{1}(\psi)\frac{1}{\tau^{2}}+C\sqrt{\int_{\tilde{\Sigma}_{\tau}\cap\left\{r\geq R_{0}\right\}}\frac{1}{\rho^{2}}\psi^{2}\int_{\tilde{\Sigma}_{\tau}\cap\left\{r\geq R_{0}\right\}}{\left(\partial_{\rho}(\rho \psi)\right)^{2}}}.
\end{split}
\end{equation*}
However, from  the first Hardy inequality (see Section \ref{sec:HardyInequalities}) we have
\begin{equation*}
\int_{\tilde{\Sigma}_{\tau}}{\frac{1}{\rho^{2}}\psi^{2}}\leq C\int_{\tilde{\Sigma}_{\tau}}{J_{\mu}^{T}(\psi)n^{\mu}_{\Sigma_{\tau}}}\leq CE_{1}\frac{1}{\tau^{2}}.
\end{equation*}
Moreover, if $R$ is the constant defined in Section \ref{sec:EnergyDecay} and  recalling that $\rho\psi=\phi$ we have
\begin{equation*}
\begin{split}
\int_{\tilde{\Sigma}_{\tau}\cap\left\{r\geq R_{0}\right\}}{\left(\partial_{\rho}(\rho \psi)\right)^{2}}&=\int_{\tilde{\Sigma}_{\tau}\cap\left\{R\geq r\geq R_{0}\right\}}{\left(\partial_{\rho}(\rho \psi)\right)^{2}}+\int_{\tilde{N}_{\tau}}{(\partial_{v}\phi)^{2}}\\
&\leq C\int_{\tilde{\Sigma}_{0}}{J_{\mu}^{T}(\psi)n_{\tilde{\Sigma}_{0}}^{\mu}}+\int_{\tilde{N}_{\tau}}{(\partial_{v}\phi)^{2}}\\
&\leq  C\int_{\tilde{\Sigma}_{0}}{J_{\mu}^{T}(\psi)n_{\tilde{\Sigma}_{0}}^{\mu}}+\int_{\tilde{N}_{0}}{(\partial_{v}\phi)^{2}},
\end{split}
\end{equation*}
where for the second inequality we used Propositions \ref{rweigestprop} and \ref{nondegx}. Hence for $\tau\geq 1$ we have
\begin{equation*}
\begin{split}
r^{2}\int_{\mathbb{S}^{2}}{\psi^{2}(r,\omega)d\omega}&\leq CE_{1}\frac{1}{\tau^{2}}+C\sqrt{E_{1}}\sqrt{C\int_{\tilde{\Sigma}_{0}}{J_{\mu}^{T}(\psi)n_{\tilde{\Sigma}_{0}}^{\mu}}+\int_{\tilde{N}_{0}}{(\partial_{v}\phi)^{2}}}\frac{1}{\tau}\\
&\leq C E_{1}\frac{1}{\tau},
\end{split}
\end{equation*}
since the quantitiy in the square root is dominated by $E_{1}$. Therefore, by the Sobolev inequality on $\mathbb{S}^{2}$ we obtain
\begin{equation*}
\begin{split}
\psi^{2}\leq CE_{5}\frac{1}{r^{2}\cdot\tau}.
\end{split}
\end{equation*}

\subsubsection{Decay near $\mathcal{H}^{+}$}
\label{sec:DecayNearMathcalH}

We are now investigating the behaviour of $\psi$ in the region $\left\{M\leq r\leq R_{0}\right\}$. We first prove the following
\begin{lemma}
There exists a constant $C$ which depends only on $M, R_{0}$ such that for all $r_{1}$ with $M<r_{1}\leq R_{0}$ and all solutions $\psi$ of the wave equation we have
\begin{equation*}
\begin{split}
\int_{\mathbb{S}^{2}}{\psi^{2}(r_{1},\omega)d\omega}\leq\frac{C}{(r_{1}-M)^{2}}\frac{E_{1}}{\tau^{2}}.
\end{split}
\end{equation*}
\label{1lemmadecay}
\end{lemma}

\begin{proof}
Using \eqref{1pointwise} we obtain 
\begin{equation*}
\begin{split}
\int_{\mathbb{S}^{2}}{\psi^{2}(r_{1},\omega)d\omega}&\leq \frac{C}{r_{1}}\int_{\tilde{\Sigma}_{\tau}\cap\left\{r\geq r_{1}\right\}}{J_{\mu}^{N}(\psi)n^{\mu}_{\tilde{\Sigma}_{\tau}}}=\frac{C}{r_{1}}\int_{\tilde{\Sigma}_{\tau}\cap\left\{r\geq r_{1}\right\}}{
\frac{D(\rho)}{D(\rho)}J_{\mu}^{N}(\psi)n^{\mu}_{\tilde{\Sigma}_{\tau}}}\\ 
&\leq \frac{C}{r_{1}D(r_{1})}\int_{\tilde{\Sigma}_{\tau}\cap\left\{r\geq r_{1}\right\}}{D(\rho)J_{\mu}^{N}(\psi)n^{\mu}_{\tilde{\Sigma}_{\tau}}}\\ &\leq
\frac{C}{(r_{1}-M)^{2}}\int_{\tilde{\Sigma}_{\tau}}{J_{\mu}^{T}(\psi)n^{\mu}_{\tilde{\Sigma}_{\tau}}}\leq\frac{C}{(r_{1}-M)^{2}}\frac{E_{1}}{\tau^{2}}.
\end{split}
\end{equation*}
\end{proof}

\begin{lemma}
There exists a constant $C$ which depends only on $M, R_{0}$ such that for all $r_{0}\in[M, R_{0}]$, $\a >0$ and  solutions $\psi$ of the wave equation, we have
\begin{equation*}
\begin{split}
\int_{\mathbb{S}^{2}}{\psi^{2}(r_{0},\omega)d\omega}&\leq CE_{1}\frac{1}{\tau^{2-2\a}}+C\sqrt{E_{1}}\frac{1}{\tau}\sqrt{\int_{\tilde{\Sigma}_{\tau}\cap\left\{r_{0}\leq r\leq r_{0}+\tau^{-\a}\right\}}{\!\!(\partial_{\rho}\psi)^{2}}}.
\end{split}
\end{equation*}
\label{2lemmadecay}
\end{lemma}
\begin{proof}
We consider the hypersurface $\gamma_{\a}=\left\{r=r_{0}+\tau^{-\a}\right\}$.
\begin{figure}[H]
	\centering
		\includegraphics[scale=0.15]{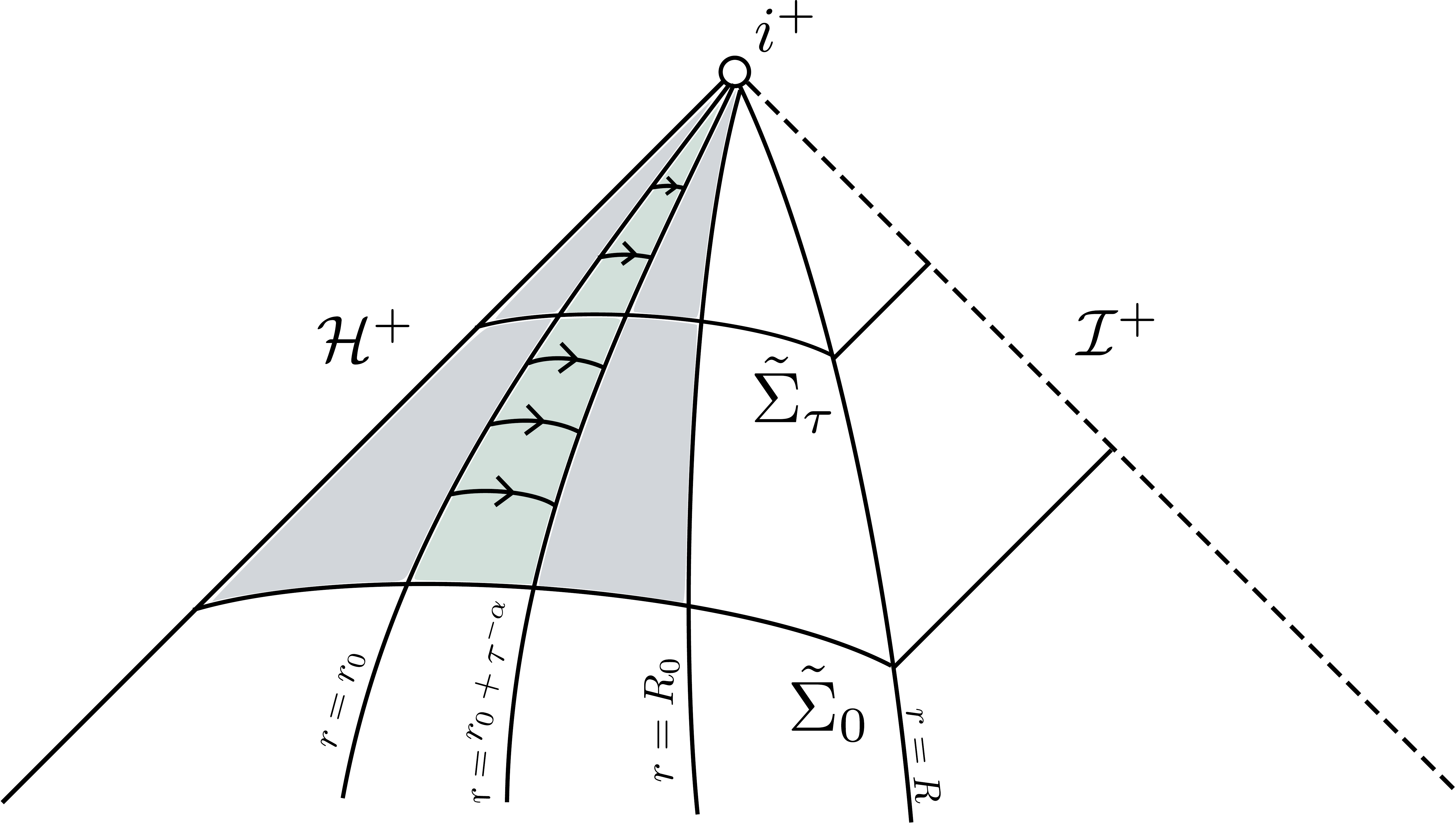}
		\label{fig:2decaypoint}
\end{figure}
Then by applying Stokes' theorem for the hypersurfaces shown in the figure above we obtain
\begin{equation*}
\begin{split}
\int_{\mathbb{S}^{2}}{\psi^{2}(r_{0},\omega)}\leq\int_{\mathbb{S}^{2}}{\psi^{2}(r_{0}+\tau^{-\a},\omega)}+C\int_{\tilde{\Sigma}_{\tau}\cap\left\{r_{0}\leq r\leq r_{0}+\tau^{-\a}\right\}}{\psi(\partial_{\rho}\psi)}.
\end{split}
\end{equation*}
For the first term on the right hand side we apply Lemma \ref{1lemmadecay} (note that $M<r_{0}+\tau^{-\a}$). The lemma now follows from Cauchy-Schwarz applied for the second term, the first Hardy inequality and Theorem   \ref{decayofenergy}.
\end{proof}

\textbf{The case $l=0$}

We first assume that $\psi$ is spherically symmetric. Then we have the pointwise bound
\begin{equation*}
\begin{split}
\left|\partial_{\rho}\psi\right|\leq C\sqrt{\tilde{E}_{6}},
\end{split}
\end{equation*}
in $\left\{M\leq r\leq R_{0}\right\}$, where
\begin{equation*}
\begin{split}
\tilde{E}_{6}=\left\|\partial_{r}\psi\right\|_{L^{\infty}\left(\tilde{\Sigma}_{0}\right)}^{2}+E_{4}(\psi)+E_{4}(T\psi).
\end{split}
\end{equation*}
This can by easily proved by using the method of characteristics and  integrating along the characteristic $u=c$ the wave equation (expressed in null coordinates).
Hence  Lemma \ref{2lemmadecay} for $\a =\frac{2}{5}$ gives
\begin{equation*}
\begin{split}
\int_{\mathbb{S}^{2}}{\psi^{2}(r_{0},\omega)d\omega}&\leq CE_{1}\frac{1}{\tau^{\frac{6}{5}}}+C\sqrt{E_{1}}\sqrt{\tilde{E}_{6}}\frac{1}{\tau^{\frac{6}{5}}}\\
&\leq CE_{6}\frac{1}{\tau^{\frac{6}{5}}},
\end{split}
\end{equation*}
where 
$E_{6}=E_{1}+\tilde{E_{6}}$. Since $\psi$ is spherically symmetric we obtain
\begin{equation}
\begin{split}
\psi^{2}\leq CE_{6}\frac{1}{\tau^{\frac{6}{5}}}.
\end{split}
\label{l=0pointdecay}
\end{equation}

\textbf{The case $l=1$}

Suppose that $\psi$ is supported on $l=1$. Then from Lemma \ref{2lemmadecay} for $\a=\frac{1}{4}$ we obtain
\begin{equation*}
\begin{split}
\int_{\mathbb{S}^{2}}{\psi^{2}(r_{0},\omega)d\omega}&\leq CE_{1}\frac{1}{\tau^{\frac{3}{2}}}+C\sqrt{E_{1}}\frac{1}{\tau}\sqrt{\int_{\tilde{\Sigma}_{\tau}\cap\left\{r_{0}\leq r\leq r_{0}+\tau^{-\a}\right\}}{\!\!(\partial_{\rho}\psi)^{2}}}\\
&\leq CE_{1}\frac{1}{\tau^{\frac{3}{2}}}+C\sqrt{E_{1}}\sqrt{E_{2}}\frac{1}{\tau^{\frac{3}{2}}}\\
&\leq C\tilde{E}_{7}\frac{1}{\tau^{\frac{3}{2}}}
\end{split}
\end{equation*}
where we have used Theorem \ref{energydecay} and $\tilde{E}_{7}=E_{1}+E_{2}$. Therefore, by the Sobolev inequality on $\mathbb{S}^{2}$ we have 
\begin{equation*}
\begin{split}
\psi^{2}\leq CE_{7}\frac{1}{\tau^{\frac{3}{2}}}
\end{split}
\end{equation*}
in $\left\{M\leq r\leq R_{0}\right\}$, where $E_{7}=\sum_{\left|k\right|\leq 2}\tilde{E}_{7}{\left(\Omega^{k}\psi\right)}$.

\textbf{The case $l\geq 2$}

Suppose that $\psi$ is supported on $l\geq 2$. Then from \eqref{1pointwise} and Theorem \ref{energydecay} we have that there exists a constant $C$ which depends only on $M$ and $R_{0}$ such that 
\begin{equation*}
\begin{split}
\int_{\mathbb{S}^{2}}\psi^{2}\leq CE_{3}\frac{1}{\tau^{2}}
\end{split}
\end{equation*}
in $\left\{M\leq r\leq R_{0}\right\}$. By Sobolev we finally obtain
\begin{equation*}
\begin{split}
\psi^{2}\leq CE_{8}\frac{1}{\tau^{2}},
\end{split}
\end{equation*}
where 
\begin{equation*}
\begin{split}
E_{8}=\sum_{\left|k\right|\leq 2}E_{3}{\left(\Omega^{k}\psi\right)}.
\end{split}
\end{equation*}

Summarizing, we have proved the following
\begin{theorem}
Fix $R_{0}$ such that $M<R_{0}$ and let $\tau\geq 1$. Let $E_{1}, E_{2}, E_{3}$ be the quantities as defined in Theorem \ref{energydecay}. Then there exists a constant $C$ that depends  on the mass $M$, $R_{0}$  and $\tilde{\Sigma}_{0}$  such that:
\begin{itemize}
	\item  For all solutions $\psi$ to the wave equation  we have
	\begin{equation*}
\left|\psi\right|^{2}\leq C E_{5}\frac{1}{r\cdot \tau^{2}}
\end{equation*}
in $\left\{R_{0}\leq r\right\}$, where
\begin{equation*}
E_{5}=\sum_{\left|k\right|\leq 2}{E_{1}\left(\Omega^{k}\psi\right)}.
\end{equation*}
Moreover,
\begin{equation*}
\begin{split}
\left|\psi\right|^{2}\leq CE_{5}\frac{1}{r^{2}\cdot\tau}
\end{split}
\end{equation*}
in $\left\{R_{0}\leq r\right\}$.
\end{itemize}
\begin{itemize}
	\item For all spherically symmetric solutions $\psi$ of the wave equation we have 
\begin{equation*}
\begin{split}
\left|\psi\right|^{2}\leq CE_{6}\frac{1}{\tau^{\frac{6}{5}}}
\end{split}
\end{equation*}
in $\left\{M\leq r\leq R_{0}\right\}$, where $E_{6}=E_{1}+E_{4}(\psi)+E_{4}(T\psi)+\left\|\partial_{r}\psi\right\|_{L^{\infty}\left(\tilde{\Sigma}_{0}\right)}^{2}$. 
\end{itemize}
\begin{itemize}
	\item For all solutions $\psi$ to the wave equation which are supported on the frequency $l=1$ we have
\begin{equation*}
\begin{split}
\left|\psi\right|^{2}\leq CE_{7}\frac{1}{\tau^{\frac{3}{2}}}
\end{split}
\end{equation*}
in $\left\{M\leq r\leq R_{0}\right\}$, where $E_{7}=E_{5}+\sum_{\left|k\right|\leq 2}E_{2}{\left(\Omega^{k}\psi\right)}$.
\end{itemize}
\begin{itemize}
\item For all solutions $\psi$ to the wave equation which are supported on the frequencies $l\geq 2$ we have
\begin{equation*}
\begin{split}
\left|\psi\right|^{2}\leq CE_{8}\frac{1}{\tau^{2}},
\end{split}
\end{equation*}
in $\left\{M\leq r\leq R_{0}\right\}$, where 
\begin{equation*}
\begin{split}
E_{8}=\sum_{\left|k\right|\leq 2}E_{3}{\left(\Omega^{k}\psi\right)}.
\end{split}
\end{equation*}
\end{itemize}
\label{pointdecaytheorem}
\end{theorem}

\section{Higher Order  Estimates}
\label{sec:HigherOrderPointwiseEstimates}

We finish this paper by obtaining energy and pointwise results for all the derivatives of $\psi$. We first derive decay for the local higher order (non-degenerate) energy and then pointwise decay, non-decay and blow-up results. We finally use a contradiction argument to obtain non-decay and blow-up results for the local higher order energy.

\subsection{Energy and Pointwise Decay}
\label{sec:EnergyAndPointwiseDecay}

\begin{theorem}
Fix $R_{1}$ such that $R_{1}>M$ and let $\tau\geq 1$. Let also $k,l\in\mathbb{N}$. Then there exists a constant $C$ which depend on $M,l, R_{1}$ and $\tilde{\Sigma}_{0}$ such that the following holds: For all solutions $\psi$ of the wave equation which are supported on the 
 angular frequencies greater or equal to $l$, there exist norms $\tilde{E}_{k,l}$ of the initial data of $\psi$ such that
\begin{itemize}
	\item $\displaystyle\int_{\tilde{\Sigma}_{\tau}\cap\left\{M\leq r\leq R_{1}\right\}}{J_{\mu}^{N}(\partial_{r}^{k}\psi)n_{\tilde{\Sigma}_{\tau}}^{\mu}}\leq C\tilde{E}_{k,l}^{2}\frac{1}{\tau^{2}}$ for all $k\leq l-2$,
	\item $\displaystyle\int_{\tilde{\Sigma}_{\tau}\cap\left\{M\leq r\leq R_{1}\right\}}{J_{\mu}^{N}(\partial_{r}^{l-1}\psi)n_{\tilde{\Sigma}_{\tau}}^{\mu}}\leq C\tilde{E}_{l-1,l}^{2}\frac{1}{\tau}$.
\end{itemize}
\label{hoe}
\end{theorem}
\begin{proof}
By commuting with T and applying local elliptic estimates and previous decay results, the above  integrals decay  on $\tilde{\Sigma}_{\tau}\cap\left\{r_{0}\leq r\leq R_{1}\right\}$ where $r_{0}>M$. So it suffices to prove the above result for $\tilde{\Sigma}_{\tau}\cap\mathcal{A}$, where $\mathcal{A}$ is a $\varphi^{T}_{\tau}$-invariant neighbourhood of $\mathcal{H}^{+}$. For we use the spacetime bound given by  Theorem \ref{theorem3} which implies that there exists a dyadic sequence $\tau_{n}$ such that for all $k\leq l-1$ we have
\begin{equation}
\begin{split}
\int_{\tilde{\Sigma}_{\tau_{n}}\cap\mathcal{A}}{J_{\mu}^{N}(\partial_{r}^{k}\psi)n^{\mu}_{\tilde{\Sigma}_{\tau_{n}}}}\leq CK_{l}\frac{1}{\tau_{n}},
\end{split}
\label{hoed1}
\end{equation}
where
\begin{equation*}
\begin{split}
K_{l}=\sum_{i=0}^{l}\displaystyle\int_{\tilde{\Sigma}_{0}}{J_{\mu}^{N}\left(T^{i}\psi\right)n^{\mu}_{\tilde{\Sigma}_{0}}}+\sum_{i=1}^{l}\int_{\tilde{\Sigma}_{0}\cap\mathcal{A}}{J_{\mu}^{N}\left(\partial^{i}_{r}\psi\right)n^{\mu}_{\tilde{\Sigma}_{0}}}.
\end{split}
\end{equation*}
Then, by Theorem \ref{theorem3} again we have for any $\tau$ such that $\tau_{n}\leq \tau\leq \tau_{n+1}$ 
\begin{equation*}
\begin{split}
\int_{\tilde{\Sigma}_{\tau}\cap\mathcal{A}}{J_{\mu}^{N}(\partial_{r}^{k}\psi)n^{\mu}_{\tilde{\Sigma}_{\tau}}}&\leq C\sum_{i=0}^{k}\displaystyle\int_{\tilde{\Sigma}_{\tau_{n}}}{\!\!J_{\mu}^{N}\left(T^{i}\psi\right)n^{\mu}_{\tilde{\Sigma}_{\tau_{n}}}}\!\!+C\sum_{i=1}^{k}\int_{\tilde{\Sigma}_{\tau_{n}}\cap\mathcal{A}}{\!J_{\mu}^{N}\left(\partial^{i}_{r}\psi\right)n^{\mu}_{\tilde{\Sigma}_{\tau_{n}}}}\leq C E\frac{1}{\tau_{n}}\lesssim CE\frac{1}{\tau},
\end{split}
\end{equation*}
where $E$ depends only on the initial data.  Suppose now that $k\leq l-2$. We apply Theorem \ref{theorem3} for the dyadic intervals $[\tau_{n}, \tau_{n-1}]$ and we obtain
\begin{equation*}
\begin{split}
\int_{\mathcal{A}}{J_{\mu}^{N}(\partial_{r}^{k}\psi)n_{\Sigma}^{\mu}}\leq C\sum_{i=0}^{l-1}\displaystyle\int_{\tilde{\Sigma}_{\tau_{n-1}}}{\!\!\!J_{\mu}^{N}\left(T^{i}\psi\right)n^{\mu}_{\tilde{\Sigma}_{\tau_{n-1}}}}\!\!+\sum_{i=1}^{l-1}\int_{\tilde{\Sigma}_{\tau_{n-1}}\cap\mathcal{A}}{\!\!\!J_{\mu}^{N}\left(\partial^{i}_{r}\psi\right)n^{\mu}_{\tilde{\Sigma}_{\tau_{n-1}}}}.
\end{split}
\end{equation*}
However, the right hand side has been shown to decay like $\tau^{-1}$ and thus a similar argument as above gives us the improved decay for all $k\leq l-2$.
\end{proof}

The next theorem provides pointwise results for the derivatives transversal to $\hh$ of $\psi$.
\begin{theorem}
Fix $R_{1}$ such that $R_{1}>M$ and let $\tau\geq 1$. Let also $k,l,m\in\mathbb{N}$. Then, there exist constants $C$ which depend on $M,l, R_{1}$ and $\tilde{\Sigma}_{0}$ such that the following holds: For all solutions $\psi$ of the wave equation which are supported on  angular frequencies greater or equal to $l$, there exist norms $E_{k,l}$ of the initial data  $\psi$ such that
\begin{itemize}
	\item $\displaystyle\left|\partial_{r}^{k}\psi\right|\leq CE_{k,l}\displaystyle\frac{1}{\tau}$ in $\left\{M\leq r\leq R_{1}\right\}$ for all $k\leq l-2$,
	\item $\displaystyle\left|\partial_{r}^{l-1}\psi\right|\leq CE_{l-1,l}\displaystyle\frac{1}{\tau^{\frac{3}{4}}}$ in $\left\{M\leq r\leq R_{1}\right\}$,
		\item $\displaystyle\left|\partial_{r}^{l}\psi\right|\leq CE_{l,l}\displaystyle\frac{1}{\tau^{\frac{1}{4}}}$ in $\left\{M\leq r\leq R_{1}\right\}$.
		\end{itemize}
\label{hp3}
\end{theorem}
\begin{proof}
 Let $r_{0}$ such that $M\leq r_{0}\leq R_{1}$. We consider the cut-off $\delta:[M,R_{1}+1]\rightarrow [0,1]$ such that $\delta (r)=1,\text{ for } r\leq R_{1}+\frac{1}{4}$ and $\delta(r)=0, \text{ for } R_{1}+1/2\leq r \leq R_{1}+1$.
\begin{figure}[H]
	\centering
		\includegraphics[scale=0.15]{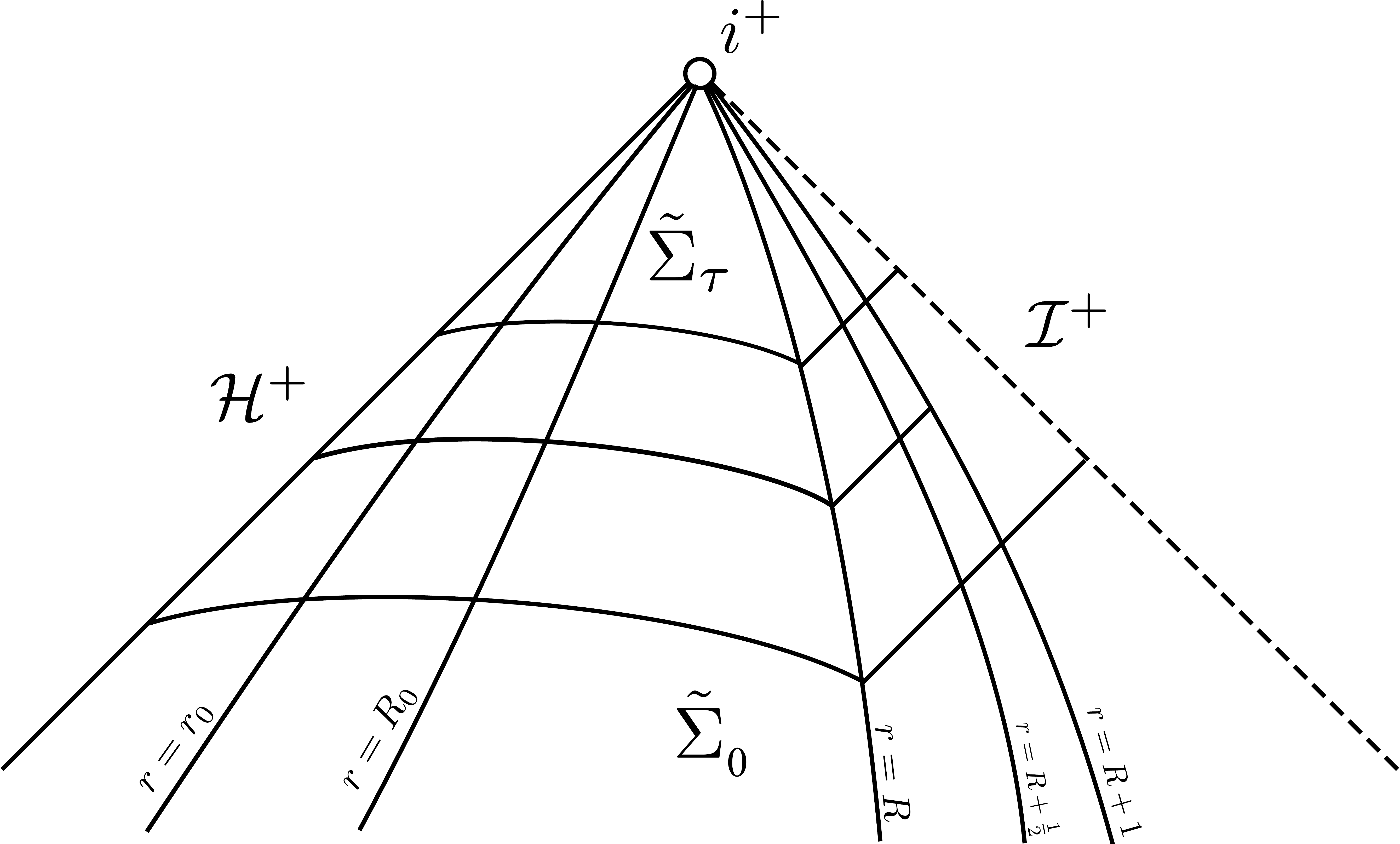}
	\label{fig:hp1}
\end{figure} 
Then,
\begin{equation*}
\begin{split}
\int_{\mathbb{S}^{2}}{\left(\partial_{r}^{k}\psi\right)^{2}(r_{0},\omega)d\omega}&=-2\int_{\tilde{\Sigma}_{\tau}\cap\left\{r_{0}\leq r\leq R_{1}+1\right\}}{\left(\partial_{r}^{k}(\delta\psi)\right)\left(\partial_{\rho}\partial_{r}^{k}(\delta\psi)\right)}\\
&\leq 2\!\left(\int_{\tilde{\Sigma}_{\tau}\cap\left\{r\leq R_{1}+1\right\}}{\!\!\!\left(\partial_{r}^{k}(\delta\psi)\right)^{2}}\right)^{\frac{1}{2}}\!\!\left(\int_{\tilde{\Sigma}_{\tau}\cap\left\{r\leq R_{1}+1\right\}}{\!\!\!\left(\partial_{\rho}\partial_{r}^{k}(\delta\psi)\right)^{2}}\right)^{\frac{1}{2}}\!\!\!.
\end{split}
\end{equation*}
In view of Theorem \ref{hoe} if $k\leq l-2$ then both integrals on the right hand side decay like $\tau^{-2}$. If $k=l-1$ then the first integral decays like $\tau^{-2}$ and the second like $\tau^{-1}$ and if $k=l$ the first integral decays like $\tau^{-1}$ and the second is bounded (Theorem \ref{theorem3}). Commuting with the angular momentum operators and using the Sobolev inequality yield the required pointwise estimates for $\partial_{r}^{k}\psi$ for $k\leq l$.

\end{proof}

\subsection{Non-Decay}
\label{sec:NonDecay}

Let now $H_{l}[\psi]$ be the function on $\mathcal{H}^{+}$ as defined in Theorem \ref{t3}. Since $H_{l}[\psi]$ is conserved along the null geodesics of $\hh$ whenever $\psi$ is supported on the angular frequency $l$, we can simply think of $H_{l}[\psi]$ as a function on $\mathbb{S}^{2}_{\scriptsize 0}=\tilde{\Sigma}_{0}\cap\hh$. We then have the following non-decay result.
\begin{proposition}
For all solutions $\psi$ supported on the angular frequency $l$ we  have  
\begin{equation*}
\partial_{r}^{l+1}\psi(\tau, \theta, \phi)\rightarrow H_{l}[\psi](\theta, \phi)
\end{equation*}
along $\hh$ and generically $H_{l}[\psi](\theta,\phi)\neq 0$ almost everywhere on $\mathbb{S}^{2}_{0}$.
\label{nondecay}
\end{proposition}
\begin{proof}
Since 
\begin{equation*}
\partial_{r}^{l+1}\psi(\tau, \theta, \phi) +\sum_{i=0}^{l}\beta_{i}\partial_{r}^{i}\psi(\tau, \theta, \phi)=H_{l}[\psi](\theta,\phi)
\end{equation*}
on $\hh$ and since all the terms in the sum on the left hand side decay (see Theorem \ref{hp3}) we take  $\partial_{r}^{l+1}\psi(\tau, \theta, \phi)\rightarrow H_{l}[\psi](\theta, \phi)$ on $\hh$. It suffices to show that generically $H_{l}[\psi](\theta,\phi)\neq 0$ almost everywhere on $\mathbb{S}^{2}_{0}$. We will in fact show that for generic solutions $\psi$ of the wave equation the function $H_{l}[\psi]$ is a generic eigenfunction of order $l$ of $\lapp$ on $\mathbb{S}^{2}_{0}$.

Note that the initial data prescribed on $\tilde{\Sigma}_{0}$ do not a priori  determine the function $H_{l}[\psi]$ on $\mathbb{S}^{2}_{0}$ unless $l=0$. Indeed,  $H_{l}[\psi]$ involves derivatives of order $k\leq l+1$ which are not tangential to $\tilde{\Sigma}_{0}$. For this reason we consider another Cauchy problem of the wave equation with initial data prescribed on $\tilde{\Sigma}_{0}^{p}$, where the hypersurface $\tilde{\Sigma}_{0}^{p}$ is as depicted below:
\begin{figure}[H]
	\centering
		\includegraphics[scale=0.11]{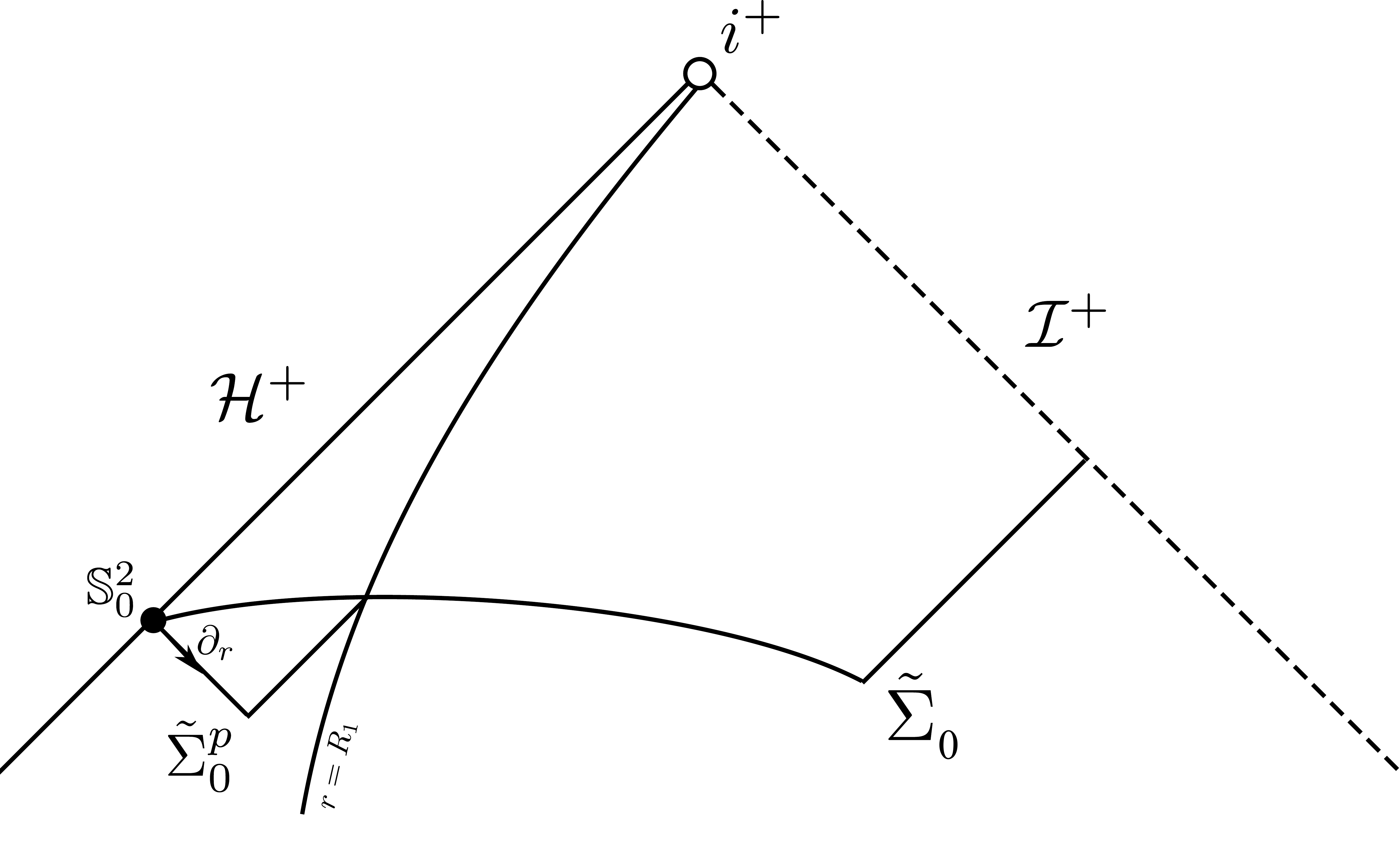}
	\label{fig:s0}
\end{figure} 
Note that the hypersurfaces $\tilde{\Sigma}_{0}$ and $\tilde{\Sigma}_{0}^{p}$ coincide for $r\geq R_{1}$. Any initial data set prescribed on $\tilde{\Sigma}_{0}$ gives rise to a unique initial data set of $\tilde{\Sigma}_{0}^{p}$ and vice versa. The Sobolev norms of the initial data on $\tilde{\Sigma}_{0}$ and $\tilde{\Sigma}_{0}^{p}$ can be compared using the pointwise and energy boundedness. Observe now that given initial data on $\tilde{\Sigma}_{0}^{p}$ the function $H_{l}[\psi]$ is completely determined on $\mathbb{S}^{2}_{0}$, since $H_{l}[\psi]$ involves only tangential to $\tilde{\Sigma}_{0}^{p}$ derivatives at $\mathbb{S}_{0}$. Therefore, generic initial data on $\tilde{\Sigma}_{0}^{p}$ give rise to generic eigenfunctions $H_{l}[\psi]$ of order $l$ of $\lapp$ on $\mathbb{S}_{0}$. Hence, for generic solutions $\psi$ of the wave equation the functions $H_{l}[\psi]$  do not vanish almost everywhere on $\mathbb{S}_{0}$.

\end{proof}

\subsection{Blow-up}
\label{sec:BlowUp}

We next show that the above non-decay results imply that  higher order derivatives of generic solutions $\psi$ blow-up along $\mathcal{H}^{+}$. To make our argument clear we first consider the spherically symmetric case where $l=0$. 

\begin{proposition}
Let $k\in\mathbb{N}$ with $k\geq 2$. Then there exists a positive constant $c$ which depends only on $M$ such that for all spherically symmetric solutions $\psi$ to the wave equation  we have 
\begin{equation*}
\left|\partial_{r}^{k}\psi\right|\geq c \left|H_{0}[\psi]\right|\tau^{k-1}
\end{equation*}
asymptotically on $\mathcal{H}^{+}$.
\label{rrl=0}
\end{proposition}
\begin{proof}
We work inductively. Consider the case $k=2$. By differentiating the wave equation (see for instance \eqref{kcom}) we take
\begin{equation}
2T\partial_{r}\partial_{r}\psi+\frac{2}{M}T\partial_{r}\psi-\frac{2}{M^{2}}T\psi+\frac{2}{M^{2}}\partial_{r}\psi=0
\label{1edw}
\end{equation}
on $\mathcal{H}^{+}$. Note that  $T\partial_{r}^{2}\psi$ and $\partial_{r}\psi$ appear with the same sign. If $H_{0}[\psi]=0$ then there is nothing to prove.  Let's suppose that $H_{0}>0$. Then
\begin{equation*}
\begin{split}
\int_{0}^{\tau}\partial_{r}\psi=\int_{0}^{\tau}H_{0}[\psi]-\frac{1}{M}\psi=H_{0}[\psi]\tau-\frac{1}{M}\int_{0}^{\tau}\psi.
\end{split}
\end{equation*}
We observe 
\begin{equation*}
\begin{split}
\left|\int_{0}^{\tau}\psi\right|\leq \int_{0}^{\tau}\left|\psi\right|\leq CE_{6}\int_{0}^{\tau}\frac{1}{\tau^{\frac{3}{5}}}=CE_{6}\tau^{\frac{2}{5}}.
\end{split}
\end{equation*}
Therefore, 
\begin{equation*}
\begin{split}
\int_{0}^{\tau}\partial_{r}\psi\geq H_{0}[\psi]\tau-CE_{6}\tau^{\frac{2}{5}}\geq cH_{0}[\psi]\tau
\end{split}
\end{equation*}
asymptotically on $\hh$. By integrating \eqref{1edw} along $\mathcal{H}^{+}$ we obtain
\begin{equation*}
\begin{split}
\partial_{r}^{2}\psi(\tau)&=\partial_{r}^{2}\psi(0)+\frac{1}{M}\partial_{r}\psi(0)-\frac{1}{M}\partial_{r}\psi(\tau)-\frac{1}{M^{2}}\psi(0)+\frac{1}{M^{2}}\psi(\tau)-\frac{1}{2M^{2}}\int_{0}^{\tau}\partial_{r}\psi\\
&\leq \partial_{r}^{2}\psi(0)+\frac{1}{M}\partial_{r}\psi(0)-\frac{1}{M}\left(H_{0}[\psi]+\frac{1}{M}\psi\right)-\frac{1}{M^{2}}\psi(0)+\frac{1}{M^{2}}\psi(\tau)-\frac{1}{2M^{2}}\int_{0}^{\tau}\partial_{r}\psi\\
&\leq \partial_{r}^{2}\psi(0)+\frac{1}{M}\partial_{r}\psi(0)-\frac{1}{M}H_{0}[\psi]+CE_{6}\frac{1}{\tau^{\frac{3}{5}}}-cH_{0}[\psi]\tau\\
&\leq -cH_{0}[\psi]\tau
\end{split}
\end{equation*}
asymptotically on $\hh$. A similar argument works for any $k\geq 2$. Indeed, we integrate \eqref{kcom} for $k\geq 1$ (and $l=0$) along $\hh$ and note that  $T\partial_{r}^{k+1}\psi$ and $\partial_{r}^{k}\psi$ appear with the same sign. Therefore, by induction on $k$, the integral $\int_{0}^{\tau}\partial_{r}^{k}\psi$ dominates asymptotically all the remaining terms which yields the required blow-up rates on $\mathcal{H}^{+}$. Note that the sign of $\partial_{r}^{k}\psi$ depends on $k$ and $H_{0}[\psi]$.

\end{proof}

\begin{corollary}
Let $k\geq 2$. For generic initial data which give rise to solutions $\psi$ of the wave equation we have
\begin{equation*}
\begin{split}
\left|\partial_{r}^{k}\psi\right|\rightarrow +\infty
\end{split}
\end{equation*}
along $\mathcal{H}^{+}$.
\end{corollary}
\begin{proof}
Decompose $\psi=\psi_{0}+\psi_{\geq 1}$ and thus
\begin{equation*}
\begin{split}
\int_{\mathbb{S}^{2}}{\left|\partial_{r}^{k}\psi\right|^{2}(M,\omega)d\omega}\geq 4\pi\left|\partial_{r}^{k}\psi_{0}\right|^{2}(M,\omega).
\end{split}
\end{equation*}
Hence the result follows by commuting with $\Omega_{i}$, the Sobolev inequality and  the fact that the right hand side blows up as $\tau\rightarrow +\infty$ as $H_{0}[\psi]\neq 0$ generically. 
\end{proof}
Let us consider the case of a general angular frequency $l$. 
\begin{proposition}
Let $k,l\in\mathbb{N}$ with $k\geq 2$. Then there exists a positive constant $c$ which depends only on $M,l,k$ such that for all  solutions $\psi$ to the wave equation which are supported on the frequency $l$ we have 
\begin{equation*}
\left|\partial_{r}^{l+k}\psi\right|(\tau,\theta,\phi)\geq c \left|H_{l}[\psi](\theta,\phi)\right|\tau^{k-1}
\end{equation*}
asymptotically on $\mathcal{H}^{+}$.
\label{rrl}
\end{proposition}
\begin{proof}
We first consider $k=2$. If $H_{l}[\psi](\theta,\phi)=0$ then there is nothing to prove. Suppose that $H_{l}[\psi](\theta,\phi)>0$. Note 
 \begin{equation*}
 \begin{split}
 \int_{0}^{\tau}\partial_{r}^{l+1}\psi&=H_{l}[\psi]\tau-\int_{0}^{\tau}\sum_{i=0}^{l}\beta_{i}\partial_{r}^{i}\psi\\
 &\geq cH_{l}[\psi]\tau
 \end{split}
 \end{equation*}
 asymptotically on $\hh$, since the integral on the right hand side is eventually dominated by $H_{l}[\psi]\tau$ in view of Theorem \ref{hp3}. If we integrate \eqref{kcom} (applied for $k=l+1)$ along the null geodesic of $\mathcal{H}^{+}$ whose projection on the sphere is $(\theta,\phi)$ we will obtain
  \begin{equation*}
 \begin{split}
 \partial_{r}^{l+2}\psi(\tau,\theta,\phi)\leq -cH_{l}[\psi](\theta,\phi)\tau,
  \end{split}
 \end{equation*}
 since the integral $ \int_{0}^{\tau}\partial_{r}^{l+1}\psi$ eventually dominates all the remaining terms (again in view of the previous decay results). The proposition follows inductively by integrating \eqref{kcom} as in Proposition \ref{rrl=0}. Recall finally that for generic solutions $\psi$ we have $H_{l}[\psi]\neq 0$ almost everywhere on $\mathbb{S}^{2}_{0}$.

\end{proof}

The next theorem provides  blow-up results for the higher order non-degenerate energy. It also shows that our estimates in Section \ref{sec:HigherOrderEstimates} are in fact sharp (regarding at least the restriction on the angular frequencies). 
\begin{theorem}
Fix $R_{1}$ such that $R_{1}>M$. Let also $k,l\in\mathbb{N}$. Then  for generic solutions $\psi$ of the wave equation which are supported on the (fixed) angular frequency $l$ we have
\begin{equation*}
\displaystyle\int_{\tilde{\Sigma}_{\tau}\cap\left\{M\leq r\leq R_{1}\right\}}{J_{\mu}^{N}(\partial_{r}^{k}\psi)n_{\tilde{\Sigma}_{\tau}}^{\mu}}\longrightarrow +\infty 
\end{equation*}
as $\tau\rightarrow +\infty$ for all $k\geq l+1$.
\label{hoeblowup}
\end{theorem}
\begin{proof}

 Consider $M<r_{0}<R_{1}$ and let $\delta$ be the cut-off introduced in the proof of Theorem \ref{hp3}.  Then,
\begin{equation*}
\begin{split}
\int_{\mathbb{S}^{2}}{\left(\partial_{r}^{k}\psi\right)^{2}(r_{0},\omega)d\omega}&=-2\int_{\mathbb{S}^{2}}{\int_{r_{0}}^{R_{1}+1}{(\partial_{r}^{k}(\delta\psi))(\partial_{\rho}\partial_{r}^{k}(\delta\psi))d\rho}d\omega}\\
&\leq C\int_{\tilde{\Sigma}_{\tau}\cap\left\{r_{0}\leq r\leq R_{1}+1\right\}}{\sum_{i=0}^{k}(T\partial_{r}^{i}\psi)^{2}+\sum_{i=0}^{k+1}(\partial_{r}^{i}\psi)^{2}},
\end{split}
\end{equation*}
where $C$ depends on $M$,  $R_{1}$ and $\tilde{\Sigma}_{0}$. Then,
\begin{equation*}
\begin{split}
\int_{\mathbb{S}^{2}}{\left(\partial_{r}^{k}\psi\right)^{2}(r_{0},\omega)d\omega}&\leq \frac{C}{D^{m_{k}}(r_{0})}\int_{\tilde{\Sigma}_{\tau}\cap\left\{r_{0}\leq r\leq R_{1}+1\right\}}{\sum_{i=0}^{k}J_{\mu}^{T}(T^{i}\psi)n^{\mu}_{\tilde{\Sigma}_{\tau}}}\\
&\leq \frac{C}{(r_{0}-M)^{2m_{k}}}\left(\sum_{i=0}^{k}E_{1}(T^{i}\psi)\right)\frac{1}{\tau^{2}},
\end{split}
\end{equation*}
where $m_{k}\in\mathbb{N}$. Note that for the above inequality we used local elliptic estimates (or a more pedestrian way is to use the wave equation and solve with respect to $\partial_{r}^{k}\psi$; this is something we can do since $D(r_{0})>0$). Then using \eqref{kcom} we can inductively replace the $\partial_{r}$ derivatives with the $T$ derivatives.  Therefore, commuting with $\Omega_{i}$ and applying the Sobolev inequality imply that  for any $r_{0}>M$ we have $\left|\partial_{r}^{k}\psi\right|\rightarrow 0$ as $\tau\rightarrow +\infty$ along $r=r_{0}$. Let us assume now that the energy of $\partial_{r}^{k}\psi$ on $\tilde{\Sigma}_{\tau_{j}}\cap\left\{M\leq r\leq R_{1}\right\}$ is uniformly bounded by $B$ (as $\tau_{j}\rightarrow +\infty$). Given $\epsilon >0$ take $r_{0}$ such that $r_{0}-M=\frac{\epsilon^{2}}{4Br_{0}^{2}}$ and let $\tau_{\epsilon}$ be such that for all $\tau\geq \tau_{\epsilon}$ we have $\left|\partial_{r}^{k}\psi(\tau,r_{0})\right|\leq \frac{\epsilon}{8\pi}$. Then,
\begin{equation*}
\begin{split}
\int_{\mathbb{S}^{2}}\left|\partial_{r}^{k}\psi(\tau_{j},M)\right|&\leq \int_{\mathbb{S}^{2}}\left|\partial_{r}^{k}\psi(\tau_{j}, r_{0})\right|+\int_{\tilde{\Sigma}_{\tau_{j}}\cap\left\{M\leq r\leq r_{0}\right\}}\left|\partial_{\rho}\partial_{r}^{k}\psi\right|\\
&\leq \frac{\epsilon}{2}+r_{0}(r_{0}-M)^{\frac{1}{2}}\left(\int_{\tilde{\Sigma}_{\tau_{j}}\cap\left\{M\leq r\leq R_{1}\right\}}{J_{\mu}^{N}(\partial_{r}^{k}\psi)n_{\tilde{\Sigma}_{\tau_{j}}}^{\mu}}\right)^{\frac{1}{2}}\leq \epsilon,
\end{split}
\end{equation*}
for all $\tau\geq \tau_{\epsilon}$.  This proves that $\int_{\mathbb{S}^{2}}\left|\partial_{r}^{k}\psi(\tau_{j},M)\right|\rightarrow 0$ as $t_{j}\rightarrow +\infty$ along $\hh$. However, in view of Propositions \ref{nondecay} and \ref{rrl} we have
\begin{equation*}
\int_{\mathbb{S}^{2}(M)}\left|\partial_{r}^{k}\psi\right|(\tau_{j})\geq c\tau_{j}^{k-1}\int_{\mathbb{S}^{2}(M)}\left|H_{l}[\psi]\right|.
\end{equation*}
We have seen that for generic $\psi$ the function $H_{l}[\psi]$ is non-zero almost everywhere and since it is smooth we have $\int_{\mathbb{S}^{2}(M)}\left|H_{l}[\psi]\right|>0$. This shows that the integral $\int_{\mathbb{S}^{2}}\left|\partial_{r}^{k}\psi(\tau_{j},M)\right|$ can not decay, contradiction.
\end{proof}

\section{Acknowledgements}
\label{sec:Acknowledgements}

I would like to thank Mihalis Dafermos for introducing to me  the problem and for his teaching and advice. I also thank Igor Rodnianski for sharing useful insights. I am supported by a Bodossaki Grant.

\appendix

\section{On the Geometry of Reissner-Nordstr\"{o}m}
\label{sec:OnTheGeometryOfReissnerNordstrOM}

The coupled Einstein-Maxwell equations consist of the system
\begin{equation}
\begin{split}
& R_{\mu\nu}\left(g\right)-\frac{1}{2}R\left(g\right)g_{\mu\nu}=T_{\mu\nu},\\
& T_{\mu\nu}=2\left(F_{\mu}^{\rho}F_{\nu\rho}-\frac{1}{4}g_{\mu\nu}F^{ab}F_{ab}\right),\\
& \nabla ^{\mu}F_{\mu\nu}=0,\\
& dF=0,
\end{split}
\label{eme}
\end{equation}
where $g$ is Lorentzian metric on an appropriate manifold $\mathcal{M}$ and $R_{\mu\nu}\left(g\right),R\left(g\right)$ are the Ricci and scalar curvature of the Levi-Civita connection, respectively and  $F$  a 2-form on $\mathcal{M}$. Here  $T_{\mu\nu}$ denotes  the electromagnetic energy momentum tensor.

The unique family of spherically symmetric  asymptotically flat solutions of these equations is the two parameter  \textit{Reissner-Nordstr\"{o}m} family of 4-dimensional Lorentzian manifolds $\left(\mathcal{N}_{M,e},g_{M,e}\right)$ where the parameters $M$ and $e$ are called mass and (electromagnetic) charge, respectively. The extreme case corresponds to $M=\left|e\right|$.

\subsection{Constructing the Extention of Reissner-Nordstr\"{o}m}
\label{sec:ConstructingTheExtentionOfReissnerNordstrOM}

We first present the Reissner-Nordstr\"{o}m  metric in local coordinates $\left(t,r,\theta,\phi\right)$ which were discovered in 1916 \cite{r} and 1918 \cite{n}. In these coordinates, one metric component blows up for various values of $r$ and it is not a priori obvious what is the appropriate underlying manifold to study the geometry of this solution. Indeed, as we shall see, one can construct another   coordinate system $\left(v,r,\theta,\phi\right)$ which covers a ``bigger'' manifold (which we will denote by $\tilde{\mathcal{M}}$) which is homeomorphic to $\mathbb{R}^{2}\times\mathbb{S}^{2}$.

The Reissner-Nordstr\"{o}m metric $g=g_{M,e}$ in the coordinates $\left(t,r\right)$ is given by
\begin{equation}
g=-Ddt^{2}+\frac{1}{D}dr^{2}+r^{2}g_{\scriptstyle\mathbb{S}^{2}},
\label{rnm}
\end{equation}
where 
\begin{equation}
D=D\left(r\right)=1-\frac{2M}{r}+\frac{e^{2}}{r^{2}}
\label{d}
\end{equation}
and $g_{\scriptstyle\mathbb{S}^{2}}$ is the standard metric on $\mathbb{S}^{2}$. Note that the Maxwell potential $A$ in these coordinates is given by
\begin{equation*}
A=-\frac{Q}{r}dt-B\cos\theta d\phi
\end{equation*}
where $e=\sqrt{Q^{2}+B^{2}}$, $(\theta,\phi)\in\mathbb{S}^{2}$ and $Q$, $B$ are the electric and magnetic charge, respectively.

Clearly,  SO(3) acts by isometry on these spacetimes. We will refer to the SO(3)-orbits as (symmetry) spheres. The coordinate $r$ is defined intrinsically such that the area of the spheres of symmetry is $4\pi r^{2}$ (and thus should be thought of 
as a purely geometric function of the spacetime).  

One could now pose the following question: On what manifold is the metric \eqref{rnm} most naturally defined?
In the above coordinates, it is clear that the metric component $g_{rr}$ is singular at $r=0,r_{-},r_{+}$ where $r_{-},r_{+}$ are the roots of $D$. The computation of the curvature shows that as $r\rightarrow 0$ the curvature blows up and so the singularity of $r=0$ in \eqref{rnm} is essential (for a very detailed description of these phenomena in Schwarzschild case see \cite{md} and \cite{haw}). However,  the points where $r=r_{-},r_{+}$ form coordinate singularities\footnote{It is the function $t$ that is singular at these points.} which can be eliminated  by introducing the so-called \textit{tortoise} coordinate $r^{*}$ 
\begin{equation*}
\frac{\partial r^{*}\left(r\right)}{\partial r}=\frac{1}{D}.
\end{equation*}
We can easily see that if $r_{+}>r_{-}$ then
\begin{equation*}
r^{*}\left(r\right)=r+\frac{1}{2\kappa_{+}}\ln\left|\frac{r-r_{+}}{r_{+}}\right|+\frac{1}{2\kappa_{-}}\ln\left|\frac{r-r_{-}}{r_{-}}\right|+C
\end{equation*}
where $\kappa_{+},\kappa_{-}$ are given by
\begin{equation*}
\kappa_{\pm}=\frac{r_{\pm}-r_{\mp}}{2r^{2}_{\pm}}=\frac{1}{2}\left.\frac{dD\left(r\right)}{dr}\right|_{r=r_{\pm}}.
\end{equation*}
In the extreme case we have $r_{+}=r_{-}=M$ and thus $\kappa_{+}=\kappa_{-}=0$. Then
\begin{equation}
r^{*}(r)=r+2M\ln (r-M)-\frac{M^{2}}{r-M}+C.
\label{exttor}
\end{equation}
The fact that in extreme case $r^{*}$ is inverse linear (instead of logarithmic in the non-extreme case) is crucial. The constant $C$ is taken such that $r^{*}\left(Q\right)=0$, where 
\begin{equation}
Q=\frac{3M}{2}\left(1+\sqrt{1-\frac{8e^{2}}{9M^{2}}}\right).
\label{Q}
\end{equation}
The physical interpretation of the radius $Q$ and the constants $\kappa_{+},\kappa_{-}$  become apparent in Sections  \ref{sec:PhotonSphereAndTrappingEffect} and \ref{sec:RedshiftEffectAndSurfaceGravityOfH}.
By introducing the coordinate system $\left(t,r^{*}\right)$ the metric becomes
\begin{equation}
g=-Ddt^{2}+D\left(dr^{*}\right)^{2}+r^{2}g_{\scriptstyle\mathbb{S}^{2}}.
\label{m*rn}
\end{equation}
\begin{figure}[H]
	\centering
		\includegraphics[scale=0.2]{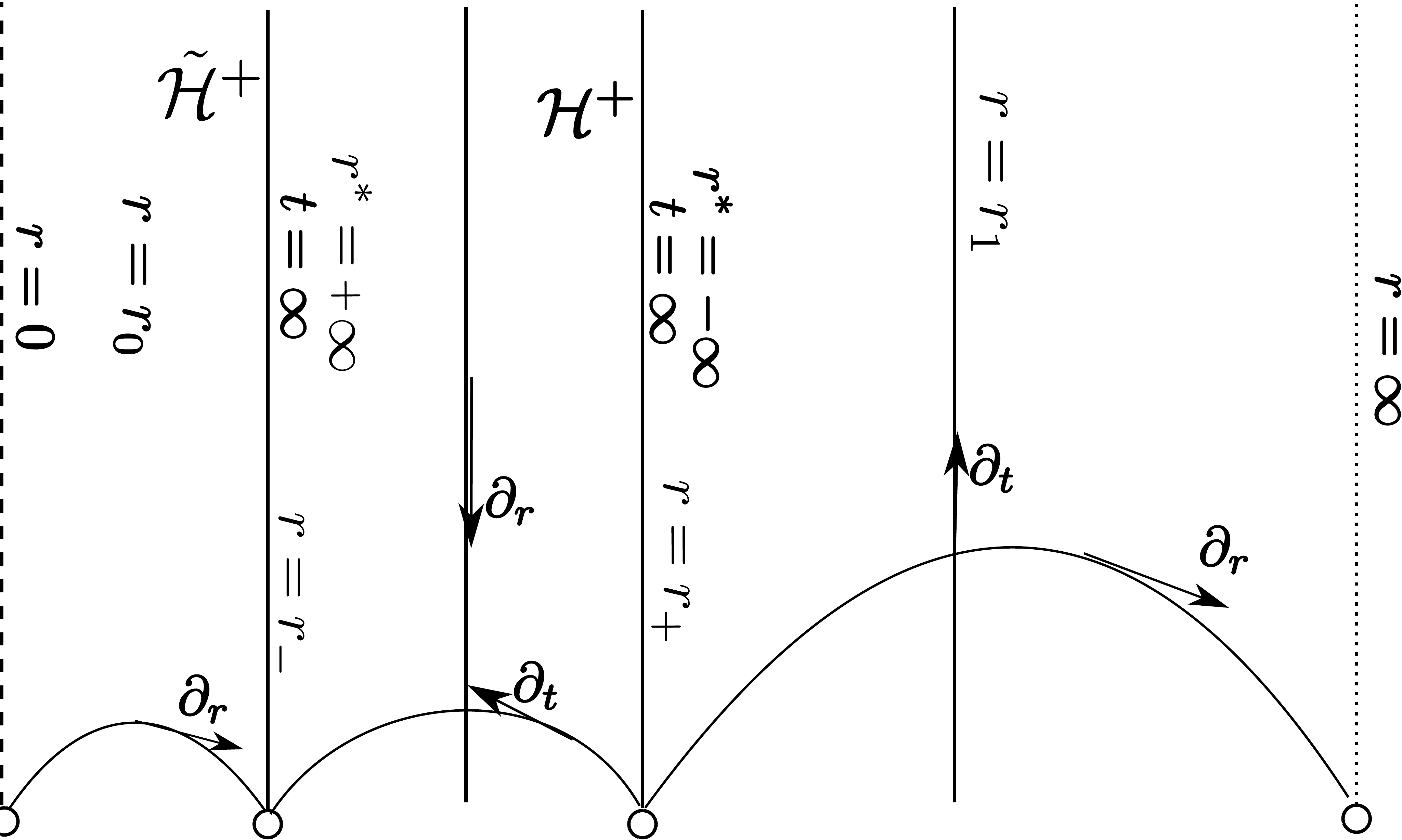}
	\caption{The coordinate systems $\left(t,r\right)$ and $\left(t,r^{*}\right)$. Each point in this diagram represents a sphere of symmetry.}
	\label{fig:tr1}
\end{figure}
The tortoise coordinate $r^{*}$ maps the hypersurface $r=r_{+}$ to $r^{*}=-\infty$. Therefore, although the coordinates $\left(t, r^{*}\right)$ capture the geometry of the region $r>r_{+}$ in a more appropriate way than the coordinates $\left(t,r\right)$ they still fail to reveal the geometric features of the neighbourhoods of the hypersurface $r=r_{+}$. This is done by the \textit{ingoing Eddington-Finkelstein coordinates} $\left(v, r\right)$ where
\begin{equation*}
v=t+r^{*}.
\end{equation*}
In these coordinates the metric is given by
\begin{equation}
g=-Ddv^{2}+2dvdr+r^{2}g_{\scriptstyle\mathbb{S}^{2}}.
\label{RN}
\end{equation}
The coordinate vector field $\partial _{v}$ is Killing and causal in the region 
\begin{equation*}
\left\{0<r\leq r_{-}\right\}\cup\left\{r\geq r_{+}\right\}.
\end{equation*}
In particular, $\partial _{v}$ is everywhere timelike in this region except on the hypersurfaces
\begin{equation*}
\mathcal{H}^{+}=\left\{r=r_{+}\right\},\tilde{\mathcal{H}}^{+}=\left\{r=r_{-}\right\}
\end{equation*}
where it is null. Note that in  extreme Reissner-Nordstr\"{o}m  we have $r_{-}=r_{+}$ and so $\tilde{\mathcal{H}}^{+}\equiv\mathcal{H}^{+}$. The region between $\tilde{\mathcal{H}}^{+}$ and $\mathcal{H}^{+}$ disappears and the Killing vector field $\partial_{v}$ is everywhere causal.  In view of  \eqref{RN} and the fact that $\partial_{v}$ is  tangent and null on $\mathcal{H}^{+}$ we have that the vector $\partial_{v}$ is normal to $\mathcal{H}^{+}$. Recall that if the normal of a null hypersurface is Killing then the hypersurface is called \textit{Killing horizon}. The importance of such  null hypersurfaces will become apparent later. 
\begin{figure}[H]
	\centering
		\includegraphics[scale=0.2]{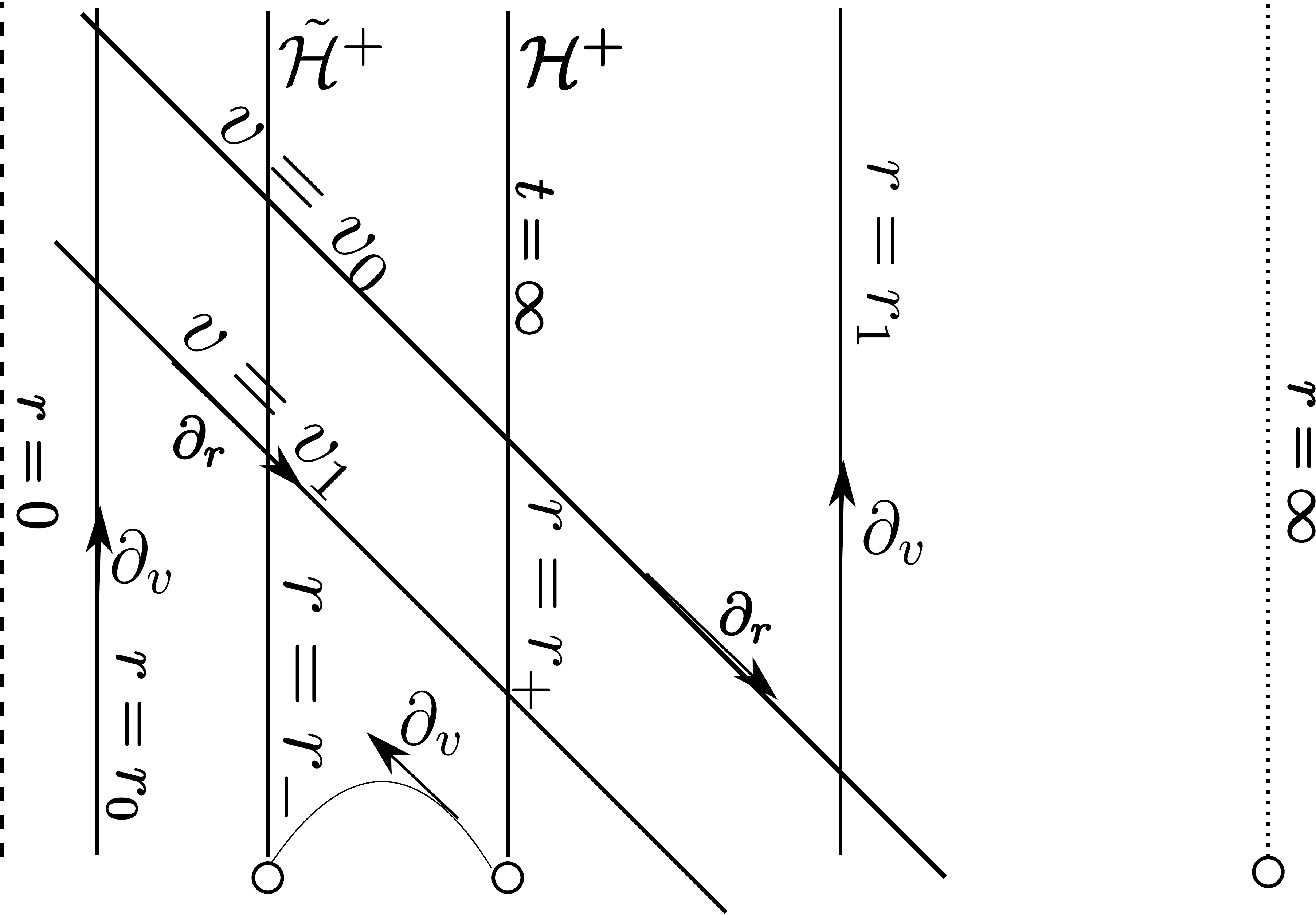}
	\caption{The coordinate system $\left(v,r\right)$}
	\label{fig:tr2}
\end{figure}
The radial curves  $v=c$, where c is a constant, are the ingoing radial null geodesics. This means that the null coordinate vector field $\partial_{r}$ differentiates with respect to $r$ on these null hypersurfaces. This geometric property of $\partial_{r}$ makes this vector field very useful for understanding the behaviour of waves close to $\mathcal{H}^{+}$. 

We see that the coordinates $\left(v,r\right)$ extend the domain that the coordinates $\left(t,r\right)$ cover. Therefore, using the coordinates $\left(v,r\right)$, let us define 
\begin{equation}
\tilde{\mathcal{M}}=\left(-\infty,+\infty\right)\times\left(0,+\infty\right)\times\mathbb{S}^{2}.
\label{Mtilde}
\end{equation}
Figure \ref{fig:tr2} describes the structure of $\tilde{\mathcal{M}}$.

Another coordinate system that partially covers $\tilde{\mathcal{M}}$ is the null system $(u,v)$ where
\begin{equation*}
\begin{split}
&u=t-r^{*},\\
&v=t+r^{*}
\end{split}
\end{equation*}
and with respect to which the metric is
\begin{equation}
g=-Ddudv+r^{2}g_{\scriptstyle\mathbb{S}^{2}}.
\label{metricnull}
\end{equation}
The hypersurfaces $v=c$ and $u=c$ are null and thus this system is useful for applying the method of characteristics or understanding null infinity.

At last, we mention another coordinate
\begin{equation}
t^{*}=t+r^{*}-r=v-r
\label{t*}
\end{equation}
\begin{figure}[H]
	\centering
		\includegraphics[scale=0.2]{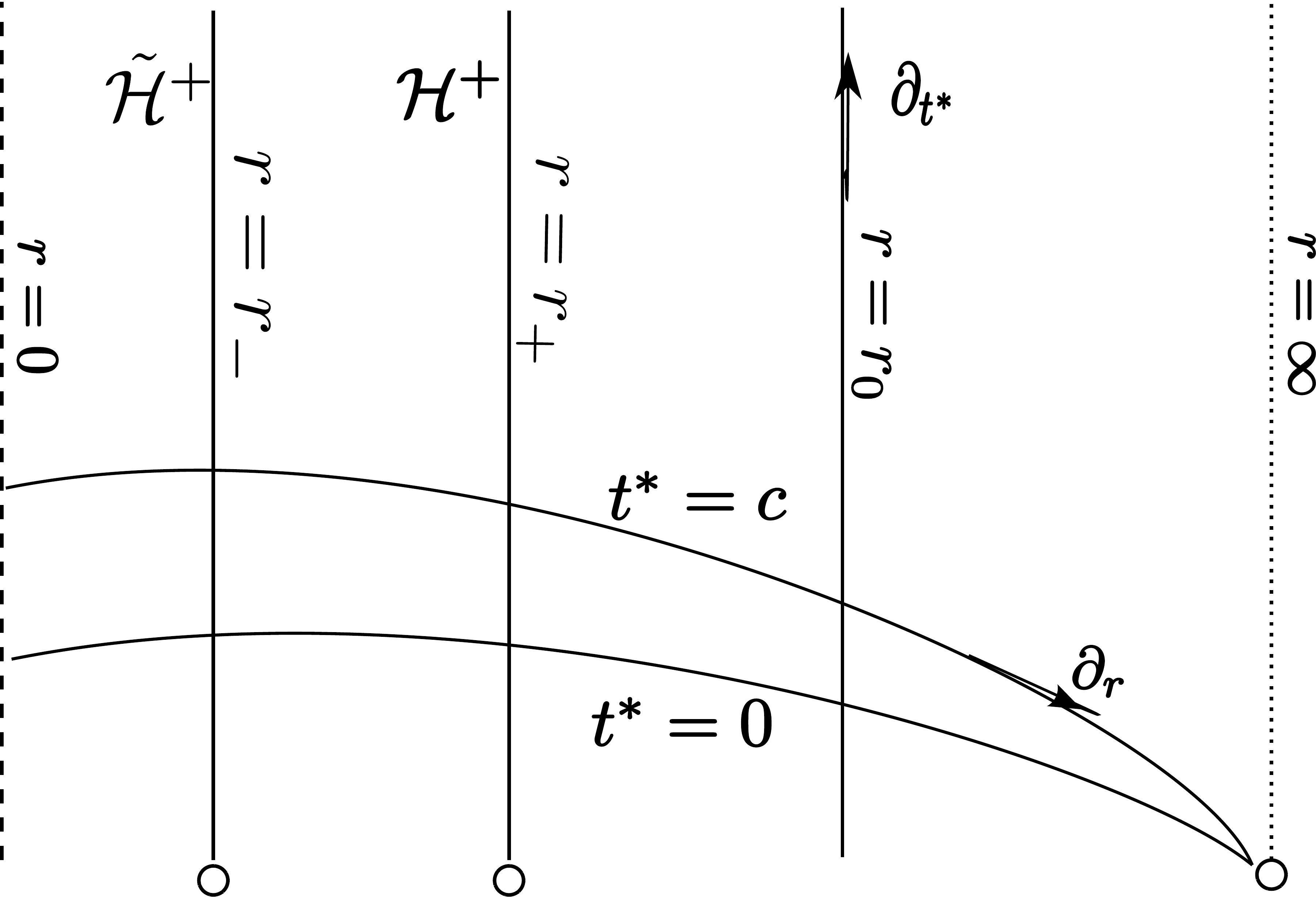}
	\caption{The coordinate system $\left(t^{*},r\right)$}
	\label{fig:tr3}
\end{figure}
As far as the Cauchy problem for the wave equation is concerned, a possible initial hypersurface  is given in these coordinates by $\left\{v=r\right\}=\left\{t^{*}=0\right\}$.

\subsection{Useful Reissner-Nordstr\"{o}m Computations}
\label{sec:UsefulReissnerNordstrOMComputations}

If $N=f_{v}\partial_{v}+f_{r}\partial_{r}$ then the current $K^{N}$ defined in Section \ref{sec:TheVectorFieldMethod} is given by
\begin{equation*}
\begin{split}
K^{N}\left(\psi\right)=\sum_{i,j}{F_{ij}\partial_{i}\psi\partial_{j}\psi}
\end{split}
\end{equation*}
where  the coef{}ficients in $(v,r)$ coordinates are given by
\begin{equation*}
\begin{split}
&F_{vv}=\left(\partial_{r}f_{v}\right),\, F_{rr}=D\left[\frac{\left(\partial_{r}f_{r}\right)}{2}-\frac{f_{r}}{r}\right]-\frac{f_{r}D'}{2}, \\
&F_{vr}=D\left(\partial_{r}f_{v}\right)-\frac{2f_{r}}{r},\, F_{\scriptsize\nabb}=-\frac{1}{2}\left(\partial_{r}f_{r}\right)
\end{split}
\end{equation*}
and in $(t,r^{*})$ coordinates by
\begin{equation*}
\begin{split}
&F_{tt}=\frac{f'}{2D}+\frac{f}{r},\, F_{r^{*}r^{*}}=\frac{f'}{2D}-\frac{f}{r},\, F_{\scriptsize\nabb}=-\frac{f'}{2}-\frac{fD'}{2},\\
\end{split}
\end{equation*}
where $f'=\frac{df}{dr^{*}}$.

\subsubsection{The Wave Operator}
\label{sec:TheWaveOperatorINRN}

The wave operator in $(v,r)$ coordinates is
\begin{equation*}
\Box_{g}\psi=D\partial_{r}\partial_{r}\psi+2\partial_{v}\partial_{r}\psi+\frac{2}{r}\partial_{v}\psi+R\partial_{r}\psi+\lapp\psi,
\end{equation*}
where  $R=D'+\frac{2D}{r}$ and $D'=\frac{dD}{dr}$. Moreover, note that
\begin{equation*}
\begin{split}
&\nabla_{v}\partial_{v}=\left(\frac{D'}{2}\right)\partial_{v}+\left(\frac{D\cdot D'}{2}\right)\partial _{r},\\
&\nabla_{v}\partial_{r}=\left(-\frac{D'}{2}\right)\partial_{r},\\
&\nabla_{r}\partial_{r}=0\\
\end{split}
\end{equation*}
and observe that in degenerate black holes the right hand side of all of them vanishes on the horizon. Also
\begin{equation*}
\operatorname{Div}\partial_{v}=0,\operatorname{Div}\partial_{r}=\frac{2}{r}.
\end{equation*}

The wave operator in $\left(t, r^{*}\right)$ coordinates is
\begin{equation*}
\begin{split}
\Box_{g}\psi=\frac{1}{D}\left[-\partial_{tt}\psi+r^{-2}\partial_{r^{*}}\left(r^{2}\partial_{r^{*}}\psi\right)\right]+\lapp\psi\\
\end{split}
\end{equation*}
and 
\begin{equation*}
\begin{split}
\operatorname{Div}\partial_{r^{*}}=D'+\frac{2D}{r}=R.
\end{split}
\end{equation*}
In $(u,v)$ coordinates
\begin{equation*}
\Box_{g}\psi=-\frac{4}{Dr}\partial_{u}\partial_{v}(r\psi)-\frac{D'}{r}\psi+\lapp\psi
\end{equation*}
and
\begin{equation*}
\operatorname{Div}\partial_{v}=-\operatorname{Div}\partial_{u}=\frac{D'}{2}+\frac{D}{r}.
\end{equation*}

Note that if $\a\in\mathbb{R}$ then 
\begin{equation*}
\Box_{g}\left(\frac{1}{r^{\a}}\right)=D\frac{\a(\a-1)}{r^{\a+2}}-D'\frac{\a}{r^{\a+1}}
\end{equation*}

\subsubsection{The Non-Negativity of the Energy-Momentum Tensor \textbf{T}}
\label{sec:TheHyperbolicityOfTheWaveEquation1}

It is essential to know how the energy momentum tensor $\textbf{T}$  depends on the derivatives of $\psi$. We use the coordinate system $\left(v,r,\theta,\phi\right)$ and suppose that 
\begin{equation*}
V_{1}=V=\left(V^{v},V^{r},0,0\right)
\end{equation*}
and
\begin{equation}
V_{2}=n=\left(n^{v},n^{r},0,0\right).
\label{nS0}
\end{equation}
The reason for the above notation is that everytime we apply the vector field $V$ as a multiplier we have to contract $J_{\mu}^{V}$ with the normal $n$ to the boundary . We proceed by computing
\begin{equation*}
J_{\mu}^{V}n^{\mu}=\textbf{T}_{\mu\nu}V^{\nu}n^{\mu}
\end{equation*}
when $V$ and $n$ are causal future directed vectors. First note that under these assumptions we have that $V^{v}$ and $n^{v}$ are non-negative. Indeed, we have
\begin{equation*}
\begin{split}
&g\left(n,n\right)=-D\left(n^{v}\right)^{2}+2n^{v}n^{r}\leq 0,\\
&g\left(n,T\right)=-Dn^{v}+n^{r}\leq 0.
\end{split}
\end{equation*}
If $n^{v}<0$ then we would have
\begin{equation*}
\begin{split}
-Dn^{v}+2n^{r}\geq &0\geq -Dn^{v}+n^{r}\Leftrightarrow\\
\frac{Dn^{v}}{2}\leq &n^{r}\leq Dn^{v}.
\end{split}
\end{equation*}
These two inequalities imply that $n^{v}>0$  which is contradiction\footnote{Clearly if $n$ is timelike then $n^{v}> 0$.}. Similarly, we take $V^{v}\geq 0$.
In view of equation \eqref{tem} we have
\begin{equation*}
\begin{split}
J^{V}_{\mu}n^{\mu}&=\textbf{T}_{vv}V^{v}n^{v}+\textbf{T}_{vr}\left(V^{v}n^{r}+V^{r}n^{v}\right)+\textbf{T}_{rr}V^{r}n^{r}\\
&=\left[\left(\partial_{v}\psi\right)^{2}+\frac{1}{2}D\left(2\partial_{v}\psi\partial_{r}\psi+D\left(\partial_{r}\psi\right)^{2}+\left|\nabb\psi\right|^{2}\right)\right]V^{v}n^{v}+\\
&\ \ \ \ +\left[-\frac{1}{2}D\left(\partial_{r}\psi\right)^{2}-\frac{1}{2}\left|\nabb\psi\right|^{2}\right]\left(V^{v}n^{r}+V^{r}n^{v}\right)+(\partial_{r}\psi)^{2}V^{r}n^{r}\\
&=V^{v}n^{v}\left(\partial_{v}\psi\right)^{2}+\left[\frac{1}{2}D^{2}V^{v}n^{v}-\frac{1}{2}DV^{v}n^{r}-\frac{1}{2}DV^{r}n^{v}+V^{r}n^{r}\right]\left(\partial_{r}\psi\right)^{2}+\\
&\ \ \ \ +[DV^{v}n^{v}]\partial_{v}\psi\partial_{r}\psi+\left[\frac{1}{2}DV^{v}n^{v}-\frac{1}{2}V^{v}n^{r}-\frac{1}{2}V^{r}n^{v}\right]\left|\nabb\psi\right|^{2}.
\end{split}
\end{equation*}
First observe that 
\begin{equation*}
\begin{split}
\frac{1}{2}DV^{v}n^{v}-\frac{1}{2}V^{v}n^{r}-\frac{1}{2}V^{r}n^{v}&=\frac{V^{v}}{4n^{v}}\left(-g(n,n)\right)+\frac{n^{v}}{4V^{v}}\left(-g(V,V)\right)\\
&=-\frac{1}{2}g(V,n).
\end{split}
\end{equation*}
Clearly the above makes sense even if $n_{v}=0$ since $n_{v}/ g(n,n)$. Furthermore,
\begin{equation*}
\begin{split}
\frac{1}{2}D^{2}V^{v}n^{v}-\frac{1}{2}DV^{v}n^{r}-\frac{1}{2}DV^{r}n^{v}+V^{r}n^{r}=\frac{1}{4}D^{2}V^{v}n^{v}+\frac{g(V,V)}{2V^{v}}\frac{g(n,n)}{2n^{v}}.
\end{split}
\end{equation*}
Therefore, if we denote $\omega_{V}=\frac{1}{2\left(V^{v}\right)^{2}}\left(-g\left(V,V\right)\right)$ and similarly for $n$ then we obtain
\begin{equation*}
\begin{split}
J^{V}_{\mu}n^{\mu}&=V^{v}n^{v}\left[\left(\partial_{v}\psi\right)^{2}+\left[\frac{D^{2}}{4}+\omega_{V}\cdot\omega_{n}\right]\left(\partial_{r}\psi\right)^{2}+D\partial_{v}\psi\partial_{r}\psi\right]+\\
&\ \ \ \ +\left[-\frac{1}{2}g\left(V,n\right)\right]\left|\nabb\psi\right|^{2}.
\end{split}
\end{equation*}
Note also that
\begin{equation*}
\begin{split}
\frac{D}{2}=\sqrt{\frac{D^{2}}{D^{2}+2\omega_{V}\cdot\omega_{n}}}\cdot \sqrt{\frac{D^{2}+2\omega_{V}\cdot\omega_{n}}{4}},
\end{split}
\end{equation*}
where the first fraction on the right hand side is well defined (even on $\mathcal{H}^{+}$) since $\omega_{V}\cdot\omega_{n}$ can vanish at most like $D^{2}$. Indeed, if $V^{v}>0$ then\footnote{If $V^{v}$ vanishes on $\mathcal{H}^{+}$ then we can argue similarly bringing back the factor $V^{v}$ which is outside the brackets. However, in this paper we only need to consider the case where $V^{v}>0$.} $\omega_{V}$ can vanish on $\mathcal{H}^{+}$ at most like $D$ and similarly for $n$. If, in addition, $n$ is timelike (as will be the case in this paper) then the denominator can vanish at most like $D$ and thus the fraction vanishes on $\mathcal{H}^{+}$. Clearly away from $\mathcal{H}^{+}$ the denominator is strictly positive.  Then, 
\begin{equation}
\begin{split}
J^{V}_{\mu}n^{\mu}=&\left[V^{v}n^{v}\left(1-\frac{D^{2}}{D^{2}+2\omega_{V}\cdot\omega_{n}}\right)\right]\left(\partial_{v}\psi\right)^{2}+\left[V^{v}n^{v}\left(\frac{\omega_{V}\cdot\omega_{n}}{2}\right)\right]\left(\partial_{r}\psi\right)^{2}+\\
&+\left(\sqrt{\frac{D^{2}}{D^{2}+2\omega_{V}\cdot\omega_{n}}}\cdot\partial_{v}\psi+\sqrt{\frac{D^{2}+2\omega_{V}\cdot\omega_{n}}{4}}\cdot\partial_{r}\psi\right)^{2}\\
&+\left[-\frac{1}{2}g\left(V,n\right)\right]\left|\nabb\psi\right|^{2}.
\label{GENERALT}
\end{split}
\end{equation}
Note that if $V$ is null and tangent to $\mathcal{H}^{+}$ then the coef{}ficient of the transversal derivative vanishes. Therefore, \eqref{GENERALT} will allow us to understand the rate of the degeneration of this coef{}ficient for several causal multipliers $V$.

\subsection{Penrose Diagrams}
\label{sec:PenroseDiagrams}

The Penrose diagram of a fundamental domain of  subextreme Reissner-Nordstr\"{o}m in the range $0<e< M$ is
\begin{figure}[H]
	\centering
		\includegraphics[scale=0.15]{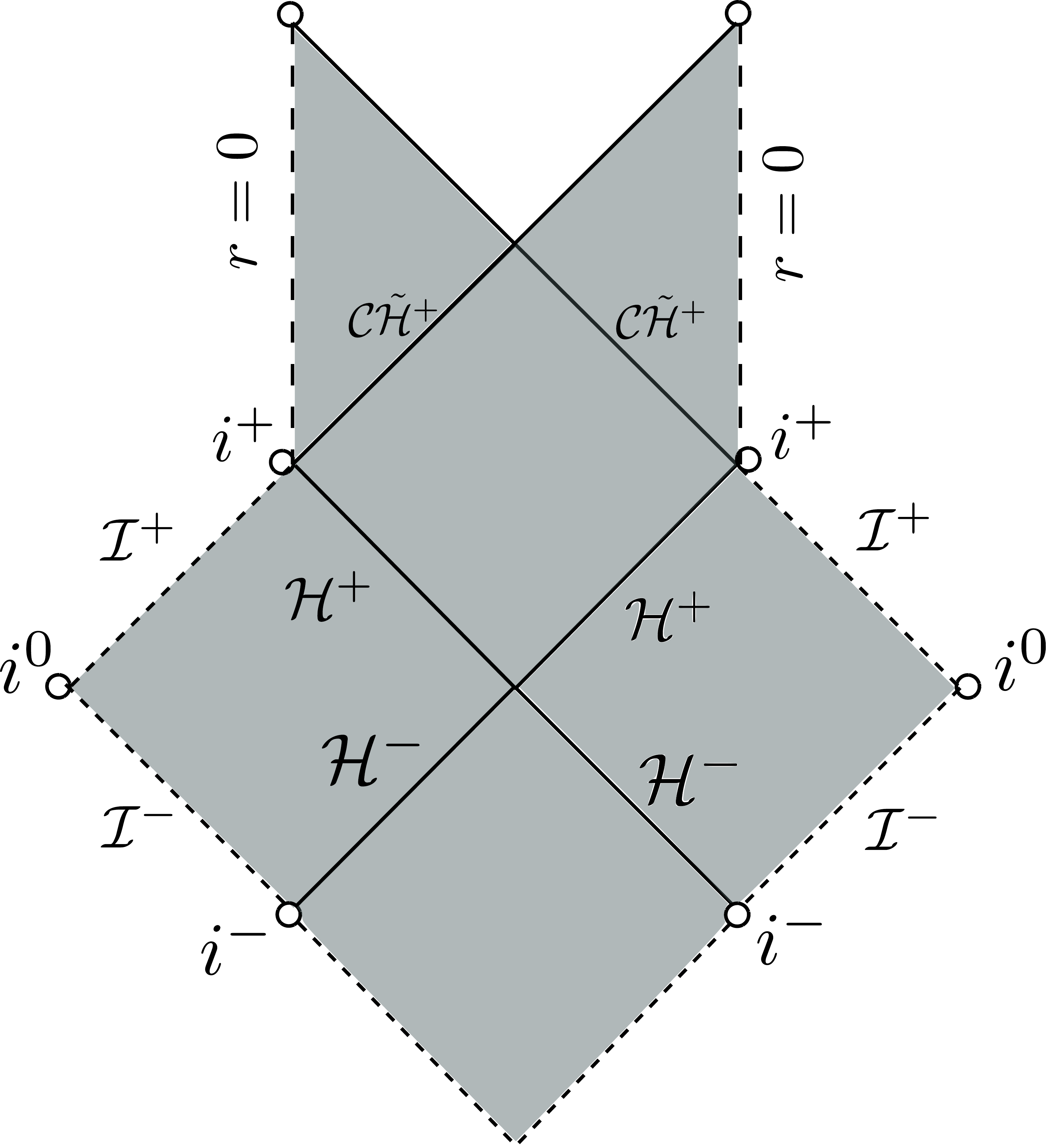}
		\label{fig:Mtildeunion}
\end{figure}
One can glue all these solutions together to obtain the following ``maximal'' solution in the range\footnote{Note that in the schwarzschild case $e=0$, the ``maximal'' solution is rather different and one of its features is that it is globally hyperbolic.} $0<e<M$ 
\begin{figure}[H]
	\centering
		\includegraphics[scale=0.09]{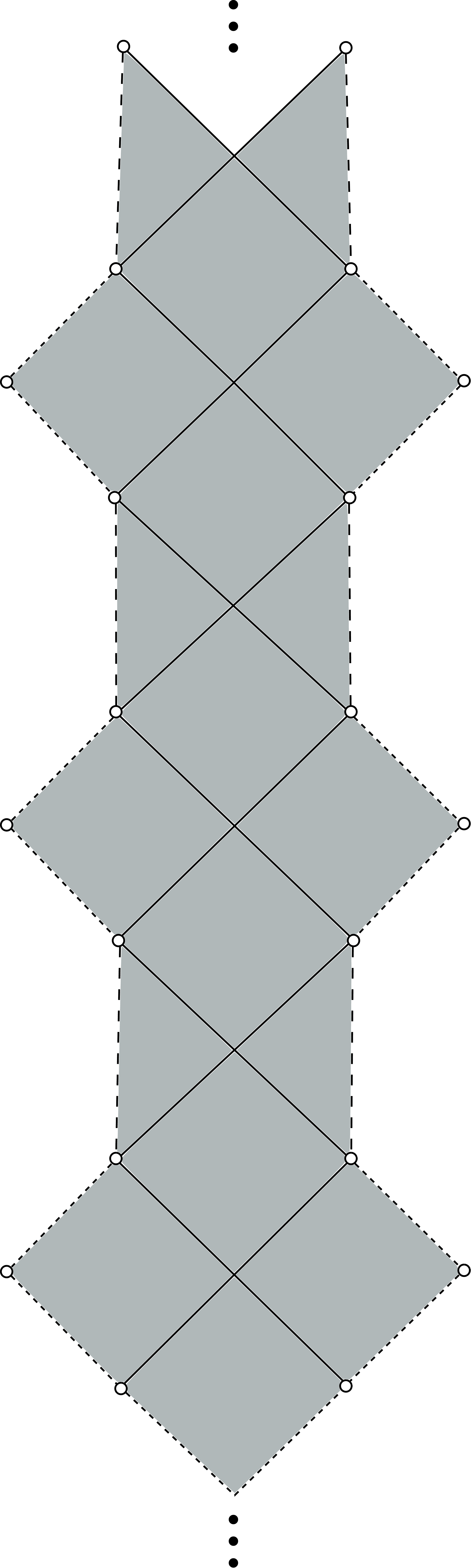}
	\label{fig:RNmax}
\end{figure}
The geometry of the above diagram  was discovered by Graves and Brill \cite{brill} in 1960. Before  presenting the Penrose diagram for  extreme Reissner-Nordstr\"{o}m, it is useful to consider the following subset of the above diagram
\begin{figure}[H]
	\centering
		\includegraphics[scale=0.15]{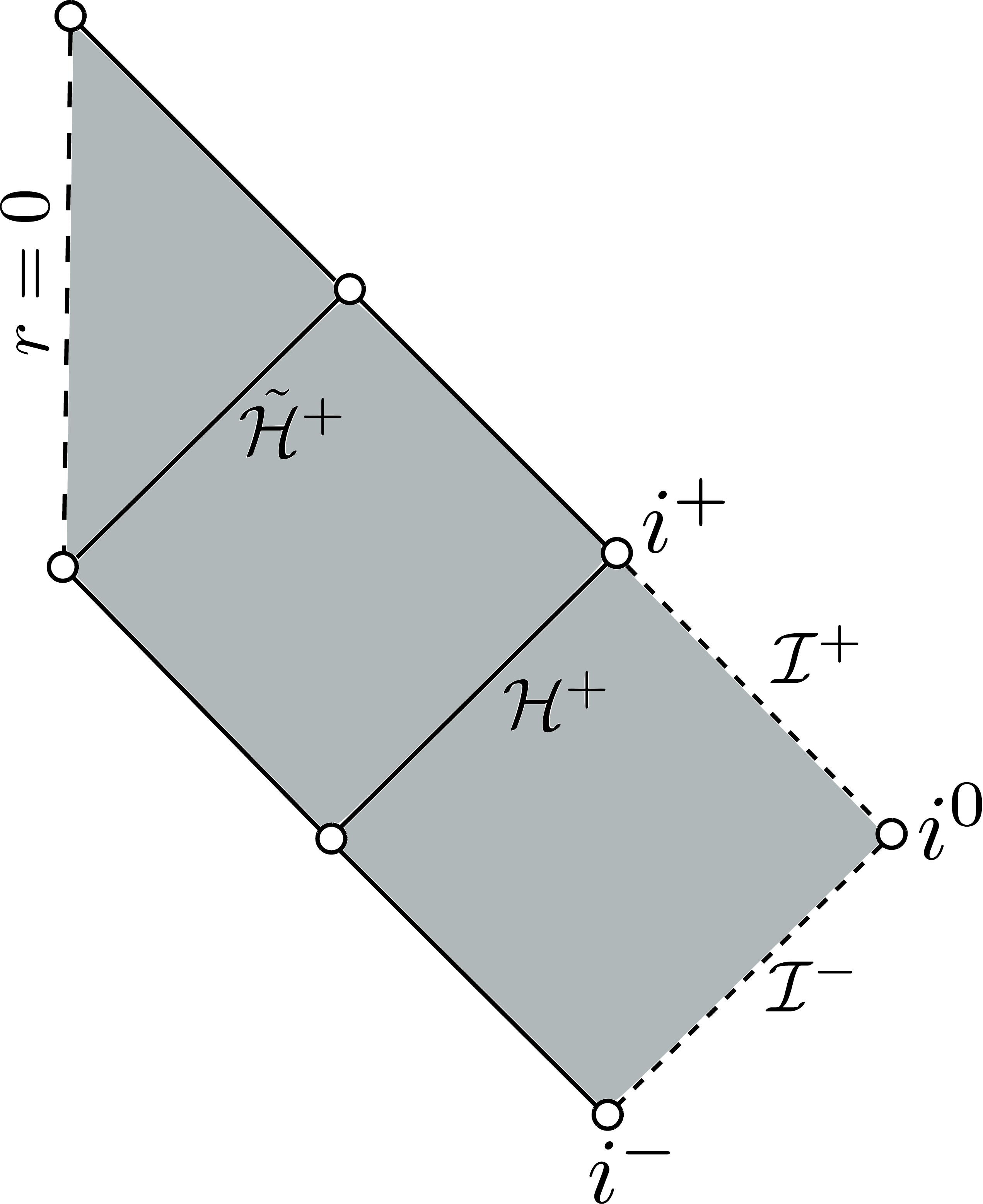}
		\label{fig:Mtilde}
\end{figure}
 In the extreme case the regions between the two horizons collapses to a single   hypersurface $\tilde{\mathcal{H}}^{+}\equiv\mathcal{H}^{+}$. One obtains in fact: 

\begin{figure}[H]
	\centering
		\includegraphics[scale=0.14]{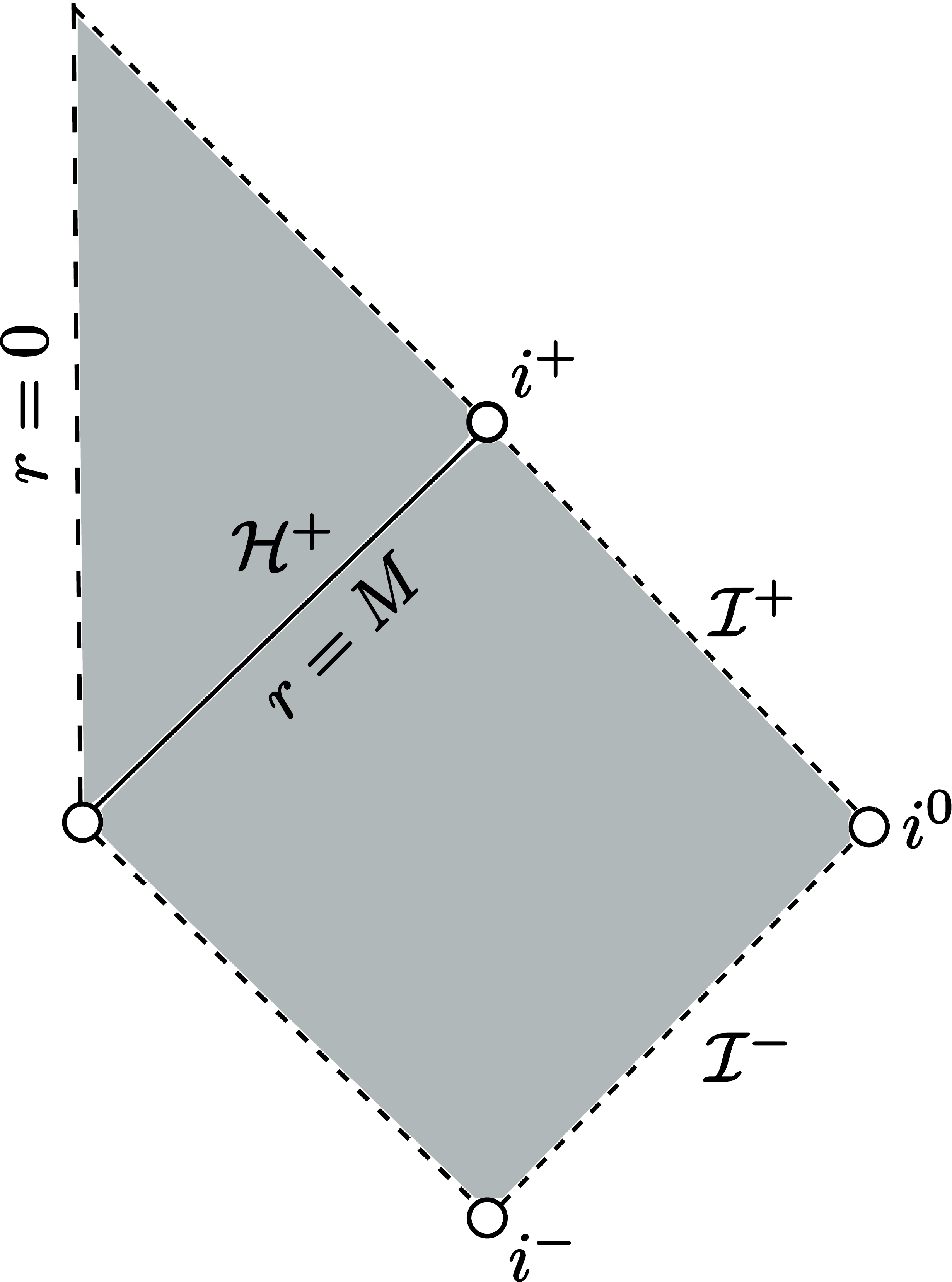}
	\label{fig:exrnsvg1}
\end{figure}
The maximally extended solution is 
\begin{figure}[H]
	\centering
		\includegraphics[scale=0.2]{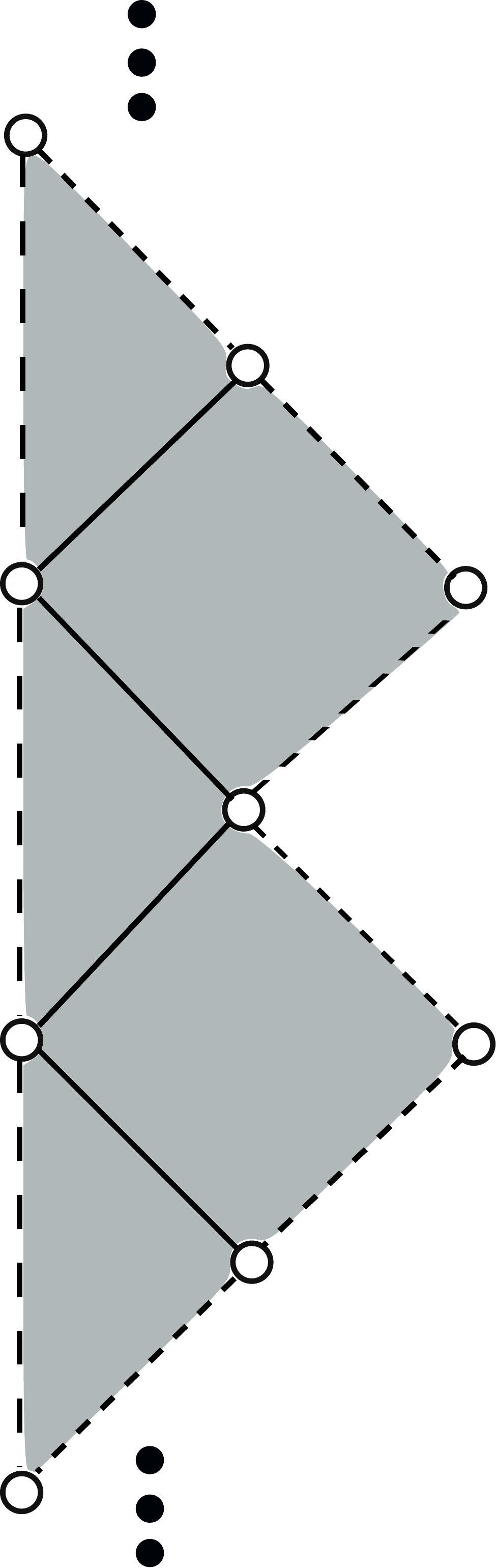}
	\label{fig:exrnsvg}
\end{figure}

\section{Stokes' Theorem on Lorentzian Manifolds}
\label{sec:StokesTheoremOnLorentzianManifolds}

If $\mathcal{R}$ is a pseudo-Riemannian manifold and $P$ is a vector field on it then we have the identity
\begin{equation*}
\int_{\mathcal{R}}{\nabla_{\mu}P^{\mu}}=\int_{\partial\mathcal{R}}{P\cdot n_{\partial\mathcal{R}}},
\end{equation*}
which is an application of Stokes' theorem. Both integrals are taken with respect to the induced volume form. Note that $n_{\partial\mathcal{R}}$ is the unit normal to $\partial\mathcal{R}$ and its direction depends on the convention of the signature of the metric. For example, if $\mathcal{R}$ is a Riemannian manifold then $n_{\partial\mathcal{R}}$ is the outward directed unit normal to $\partial\mathcal{R}$. On the other hand, for Lorentzian metrics with signature  $\left(-,+,+,+\right)$ the vector $n_{\partial\mathcal{R}}$ is the inward directed unit normal to $\partial\mathcal{R}$ in case $\partial\mathcal{R}$ is spacelike and the outward directed unit normal in case $\partial\mathcal{R}$ is timelike. If $\partial\mathcal{R}$ (or a piece of it) is null then    we take a past  (future) directed null normal to $\partial\mathcal{R}$ if it is future (past) boundary. Recall that a piece of the boundary is future (past) if  the past (future) of points on it lie in $\mathcal{R}$. The following diagram is embedded in $\mathbb{R}^{1+1}$
 \begin{figure}[H]
	\centering
		\includegraphics[scale=0.15]{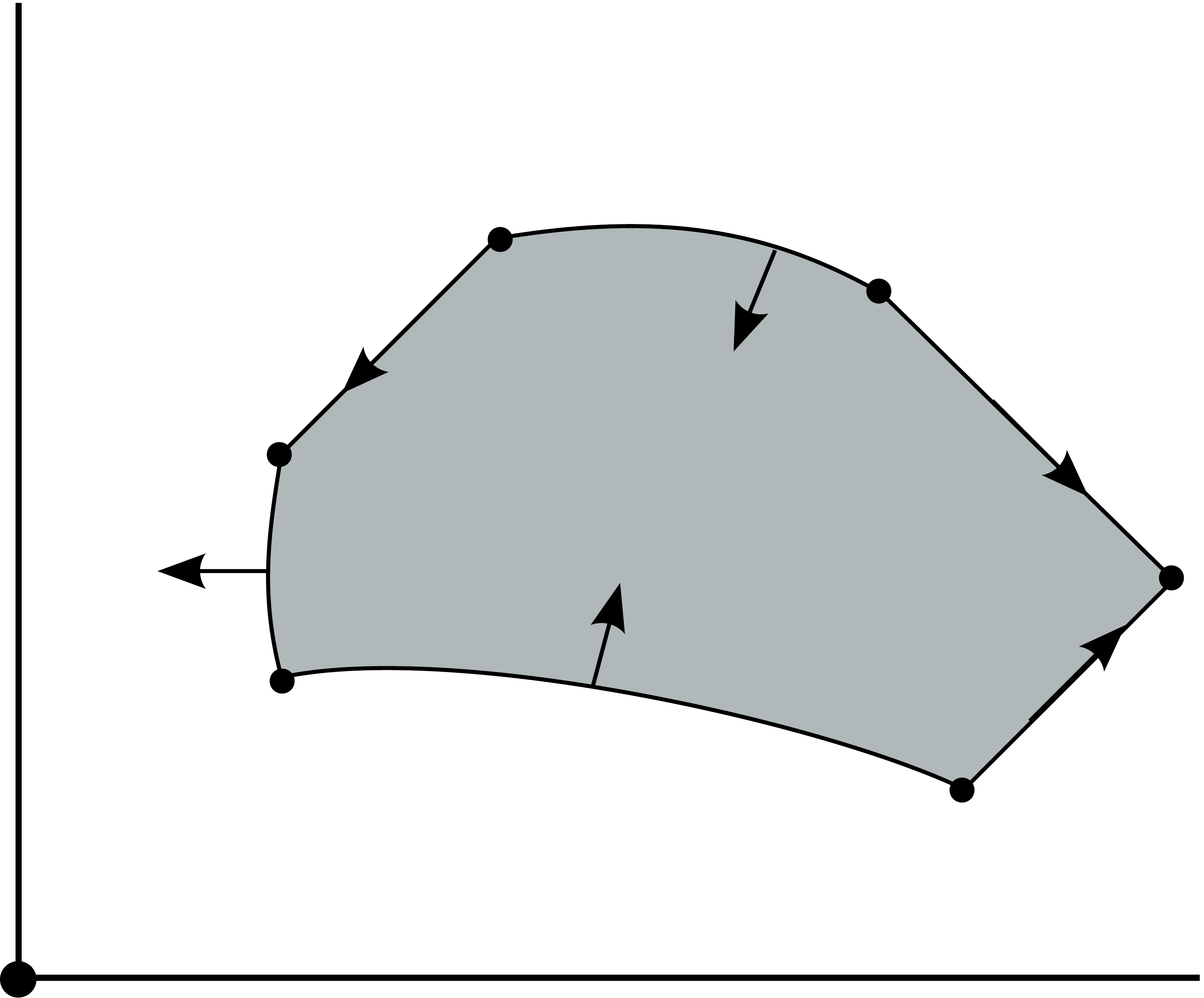}
	\label{fig:appendix}
\end{figure}
If $u$ is a function on $\mathcal{R}$ and $\textbf{V}$ a vector field and if $P=u\textbf{V}$ then
\begin{equation*}
\nabla_{\mu}P^{\mu}=\nabla_{\mu}\left(u\textbf{V}^{\mu}\right)=\left(\nabla_{\mu}u\right)\textbf{V}^{\mu}+u\nabla_{\mu}\textbf{V}^{\mu}= \nabla u\cdot \textbf{V}+u\operatorname{Div}\left(\textbf{V}\right),
\end{equation*}
therefore,
\begin{equation*}
\int_{\mathcal{R}}{ \nabla u\cdot \textbf{V}}+\int_{\mathcal{R}}{u\operatorname{Div}\left(\textbf{V}\right)}=\int_{\partial\mathcal{R}}{u\textbf{V}\cdot n_{\partial\mathcal{R}}}.
\end{equation*}
If we assume that $\textbf{V}=\nabla v$ then 
\begin{equation*}
\int_{\mathcal{R}}{ \nabla u\cdot \nabla v}+\int_{\mathcal{R}}{u\left(\Box v\right)}=\int_{\partial\mathcal{R}}{u\nabla v\cdot n_{\partial\mathcal{R}}}.
\end{equation*}
Note that a special case of the above equality is to consider $\mathcal{R}=\mathbb{S}^{2}$ and $\textbf{V}=\nabb v$ and take
\begin{equation}
\int_{\mathbb{S}^{2}}{ \nabb u\cdot \nabb v}=-\int_{\mathbb{S}^{2}}{u\left(\lapp v\right)}.
\label{ip1sphe}
\end{equation}
Moreover, if we set $\textbf{V}=v\frac{\partial}{\partial x^{i}}$ then 
\begin{equation*}
Div\left(v\frac{\partial}{\partial x^{i}}\right)=\nabla v\cdot \frac{\partial}{\partial x^{i}}+vDiv\left(\frac{\partial}{\partial x^{i}}\right)=\frac{\partial v}{\partial x^{i}}+vDiv\left(\frac{\partial}{\partial x^{i}}\right)
\end{equation*}
yields
\begin{equation}
\int_{\mathcal{R}}{ v\left(\frac{\partial u}{\partial x^{i}}\right)}+\int_{\mathcal{R}}{u\left(\frac{\partial v}{\partial x^{i}}\right)}+\int_{\mathcal{R}}{uvDiv\left(\frac{\partial}{\partial x^{i}}\right)}=\int_{\partial\mathcal{R}}{uv\frac{\partial}{\partial x^{i}}\cdot n_{\partial\mathcal{R}}}.
\label{ip2}
\end{equation}

Note also that another form of Stokes' theorem is 
\begin{equation*}
\int_{\mathcal{M}}{(\partial_{x_{1}}f)dx^{1}\wedge dx^{2}\wedge...\wedge dx^{n}}=\int_{\partial\mathcal{M}}{fdx^{2}\wedge...\wedge dx^{n}}.
\end{equation*}
Indeed, we just have to consider the form $\omega=fdx^{2}\wedge...\wedge dx^{n}$ and notice that $d\omega=(\partial_{x_{1}}f)dx^{1}\wedge dx^{2}\wedge...\wedge dx^{n}$. Clearly, the only components of the boundary that contribute on the right hand side are these where $\left\{x_{1}=c\right\}$, $c$ constant. Of course, this "less geometric'' identity is completely consistent with \eqref{ip2}. Note that the appearance of the divergence in \eqref{ip2} is due to the fact that when a coordinate vector field acts on the factor of the volume form  then we take the divergence of the field times this factor, i.e.
\begin{equation*}
\partial_{i}\sqrt{g}=\left(\operatorname{Div}\partial_{i}\right)\sqrt{g}.
\end{equation*}

\section{Sobolev Spaces}
\label{sec:SobolevSpaces}
Given a function $\psi:\mathcal{M}\rightarrow \mathbb{R}$, where $\left(\mathcal{M},g\right)$ is Riemannian manifold  we  consider the tensors
$\nabla\psi$,$\nabla^{2}\psi$, etc. where $\nabla$ is the associated Levi-Civita connection. For example, the Hessian is given by 
\begin{equation*}  
\left(\nabla^{2}\psi\right)_{ij}=\partial_{i}\partial_{j}\psi-\left(\nabla_{i}\partial_{j}\right)\psi.
\end{equation*}
Given  two tensor fields $T_{i_{1}i_{2}...i_{k}}$ and $S_{j_{1}j_{2}...j_{k}}$ of the same type,  we define the induced inner product on tensors by 
\begin{equation*} 
\left\langle T_{p},S_{p}\right\rangle=g^{i_{1}j_{1}}\cdot g^{i_{2}j_{2}}\cdot ...\cdot g^{i_{k}j_{k}}\cdot T_{i_{1}i_{2}...i_{k}}\cdot S_{j_{1}j_{2}...j_{k}},
\end{equation*}
where $p\in\mathcal{M}$ and similarly we define an inner product on tensor fields by
\begin{equation*} 
\left\langle T,S\right\rangle=\int_{\mathcal{M}}{\left\langle T_{p},S_{p}\right\rangle}.
\end{equation*}
Therefore, if $T=\nabla^{k}\psi$ and $S=\nabla^{k}\phi$, $k\in\mathbb{N},$ then we define the homogeneous Sobolev inner product by
\begin{equation*}
\left\langle \psi,\phi\right\rangle_{\overset{.}{H}^{k}}=\int_{\mathcal{M}}{\left\langle \nabla^{k}\psi,\nabla^{k}\phi\right\rangle}.
\end{equation*}
Note that this is an inner product provided $\psi,\phi$ either decay at infinity of vanish at the boundary of $\mathcal{M}$. The Sobolev inner product is defined by
\begin{equation*}
\left\langle \psi,\phi\right\rangle_{{H}^{k}}=\sum_{i=0}^{k}{\left\langle \psi,\phi\right\rangle_{\overset{.}{H}^{i}}}
\end{equation*}
and thus the Sobolev space $H^{k}(\mathcal{M})$ is the set of functions such that the Sobolev norm
\begin{equation*}
\left\|\psi\right\|_{{H}^{k}}=\sum_{i=0}^{k}{\left\|\psi\right\|_{\overset{.}{H}^{i}}}
\end{equation*}
is finite. Note that if $\psi$ is not sufficiently regular then the derivatives in $\nabla^{k}\psi$ are defined in a weak (distributive) way using $\eqref{ip2}$. Then the Sobolev space $H^{k}(\mathcal{M})$ is complete and thus Hilbert space. In other words, $\psi\in H^{k}(\mathcal{M})$ if and only if all the derivatives up to the k-th order are in $L^{2}(\mathcal{M})$. Finally, we define the space $H^{k}_{\operatorname{loc}}(\mathcal{M})$ to be the set of all functions such that the above expression is finite over all compact subsets of $\mathcal{M}$.

Now according to the Sobolev inequality, if n is the dimension of the space and  $q\geq 1$ and $k<\frac{n}{2}$  such that
\begin{equation*}
\begin{split}
\frac{1}{2}-\frac{1}{q}=\frac{k}{n}
\end{split}
\end{equation*}
then
\begin{equation*}
\begin{split}
\left\|\psi\right\|_{L^{q}}\leq C\left\|\psi\right\|_{H^{k}}.
\end{split}
\end{equation*}
Moreover, if $k>\frac{n}{2}$, then
\begin{equation*}
\begin{split}
\left\|\psi\right\|_{L^{\infty}}\leq C\left\|\psi\right\|_{H^{k}},
\end{split}
\end{equation*}
where the Sobolev constants depend only on the geometry of the space and on $k,q,n$.

\section{Elliptic Estimates on Lorentzian Manifolds}
\label{sec:EllipticEstimates}

Let us suppose that $\left(\mathcal{M},g\right)$ is a globally hyperbolic time-orientable Lorentzian manifold which admits a Killing vector field $T$. We also suppose that $\mathcal{M}$ is   foliated by spacelike hypersurfaces $\Sigma_{\tau}$, where $\Sigma_{\tau}=\phi_{\tau}\left(\Sigma_{0}\right)$. Here, $\Sigma_{0}$ is a Cauchy hypersurface and $\phi_{\tau}$ is the flow of $T$.

Let $N$ be a $\phi_{\tau}$-invariant \textbf{timelike} vector field and constants $B_{1},B_{2}$ such that
\begin{equation*}
0<B_{1}<-g(N,N)<B_{2}.
\end{equation*}
 
We will first derive the required estimate in $\Sigma_{0}$ which for simplicity we denote by $\Sigma$. For each point $p\in\Sigma$ the orthogonal complement in $T_{p}\mathcal{M}$ of the line that contains $N$ is 3-dimensional and contains a 2-dimensional subspace of the tangent space $T_{p}\Sigma$. Let $X_{2},X_{3}$ be an orthonormal basis of this subspace. Let now $X_{1}$ be a vector tangent to $\Sigma$ which is perpendicular to the plane that is spanned by $X_{2},X_{3}$. Note that the line that passes through $X_{1}$ is uniquely determined by $N$ and $\Sigma$. Then, the metric $g$ can be written as 
\begin{equation*}
\begin{split}
g=\begin{pmatrix}
g_{NN}&g_{NX_{1}}&0&0\\
g_{NX_{1}}&g_{X_{1}X_{1}}&0&0\\
0&0&1&0\\
0&0&0&1\\
\end{pmatrix}
\end{split}
\end{equation*}
with respect to the frame $\left(N,X_{1},X_{2},X_{3}\right)$. Then, if we set $\left|g\right|=g_{NN}\cdot g_{X_{1}X_{1}}-g_{NX_{1}}^{2}$, the inverse is 
\begin{equation*}
\begin{split}
g^{-1}=\begin{pmatrix}
\frac{1}{\left|g\right|}g_{X_{1}X_{1}}&-\frac{1}{\left|g\right|}g_{NX_{1}}&0&0\\
-\frac{1}{\left|g\right|}g_{NX_{1}}&\frac{1}{\left|g\right|}g_{NN}&0&0\\
0&0&1&0\\
0&0&0&1\\
\end{pmatrix}.
\end{split}
\end{equation*}
Let $h_{\Sigma}$ be the induced Riemannian metric on the spacelike hypersurface $\Sigma$. Clearly, in general we do \textbf{not} have $h^{ij}_{\Sigma}= g^{ij}$. Indeed
\begin{equation*}
\begin{split}
h_{\Sigma}=\begin{pmatrix}
g_{X_{1}X_{1}}&0&0\\
0&1&0\\
0&0&1\\
\end{pmatrix}
\end{split}
\end{equation*}
and therefore,
\begin{equation*}
\begin{split}
h_{\Sigma}^{-1}=\begin{pmatrix}
\frac{1}{g_{X_{1}X_{1}}}&0&0\\
0&1&0\\
0&0&1\\
\end{pmatrix}.
\end{split}
\end{equation*}
Let $\psi :\mathcal{M}\rightarrow\mathbb{R}$ satisfy the wave equation. Then,
\begin{equation*}
\begin{split}
\Box_{g}\psi &=\text{tr}_{g}\left(\text{Hess}\,\psi\right)=g^{\a\b}\left(\nabla^{2}\psi\right)_{\a\b}=\\
&=g^{0\b}\left(\left(\nabla^{2}\psi\right)_{0\b}+\left(\nabla^{2}\psi\right)_{\b 0}\right)+g^{ij}\left(\nabla^{2}\psi\right)_{ij}.
\end{split}
\end{equation*}
We will prove that the operator 
\begin{equation*}
P\psi=g^{ij}\left(\nabla^{2}\psi\right)_{ij}
\end{equation*}
is strictly elliptic. Indeed, in view of the formula
\begin{equation*}
\left(\nabla^{2}\psi\right)_{ij}=X_{i}X_{j}\psi-\left(\nabla_{X_{i}}X_{j}\right)\psi,
\end{equation*}
the principal part $\sigma$ of $P$ is 
\begin{equation*}
\begin{split}
\sigma\psi=g^{ij}X_{i}X_{j}\psi.
\end{split}
\end{equation*}
If $\xi\in T^{*}\Sigma$, then 
\begin{equation*}
\begin{split}
\sigma\xi&=g^{ij}\xi_{i}\xi_{j}\\
&=\frac{1}{\left|g\right|}g_{NN}\xi_{1}^{2}+\xi_{2}^{2}+\xi_{3}^{3}\\
&>b\left(\frac{1}{g_{X_{1}X_{1}}}\xi_{1}^{2}+\xi_{2}^{2}+\xi_{3}^{3}\right)\\
&=b\left\|\xi\right\|,
\end{split}
\end{equation*}
where the ellipticity constant $b>0$  depends only on $\Sigma$. Moreover, if $\psi$ satisfies $\Box_{g}\psi=0$ then
\begin{equation*}
\begin{split}
\left\|P\psi\right\|^{2}_{L^{2}(\Sigma)}= &\left\|g^{0\b}\left(\left(\nabla^{2}\psi\right)_{0\b}+\left(\nabla^{2}\psi\right)_{\b 0}\right)\right\|_{L^{2}(\Sigma)}^{2}\\
\leq &C\int_{\Sigma}{ \left(\left\|NN\psi\right\|_{L^{2}\left(\Sigma\right)}^{2}+\sum_{i=1}^{3}{\left\|X_{i}N\psi\right\|_{L^{2}\left(\Sigma\right)}^{2}}+\sum_{i=1}^{3}{\left\|X_{i}\psi\right\|_{L^{2}\left(\Sigma\right)}^{2}}+\left\|N\psi\right\|_{L^{2}\left(\Sigma\right)}^{2}\right)}\\
\leq & C\int_{\Sigma}{J_{\mu}^{N}(\psi)n^{\mu}_{\Sigma}+J_{\mu}^{N}(N\psi)n^{\mu}_{\Sigma}},
\end{split}
\end{equation*}
where $C$ is a uniform constant that depends only on the geometry of $\Sigma$ and the precise choice of $N$. Therefore, if $\psi$ can be  shown to appropriately decay at infinity then by a global elliptic estimate on $\Sigma$ we obtain
\begin{equation*}
\begin{split}
\left\|\psi\right\|_{\overset{\!\!.}{H^{1}}\left(\Sigma\right)}^{2}+\left\|\psi\right\|_{\overset{\!\!.}{H^{2}}\left(\Sigma\right)}^{2}&\leq C\cdot\left\|P\psi\right\|_{L^{2}\left(\Sigma\right)}^{2}\leq\int_{\Sigma}{ CJ_{\mu}^{N}(\psi)n^{\mu}_{\Sigma}+CJ_{\mu}^{N}(N\psi)n^{\mu}_{\Sigma}}.
\end{split}
\end{equation*}for some uniform positive constant $C$.

In case our analysis is local and thus we want to confine ourselves in a compact submanifold $\overline{\Sigma}$ of $\Sigma$ then by a local elliptic estimate on $\overline{\Sigma}$ we have
\begin{equation*}
\begin{split}
\left\|\psi\right\|_{{H}^{2}\left(\overline{\Sigma}\right)}^{2}&\leq C\cdot\left\|P\psi\right\|_{L^{2}\left(\overline{\Sigma}\right)}^{2}+\left\|\psi\right\|_{{H}^{1}\left(\overline{\Sigma}\right)}^{2}\\
&\leq \int_{\overline{\Sigma}}{\left(CJ_{\mu}^{N}(\psi)n^{\mu}_{\overline{\Sigma}}+CJ_{\mu}^{N}(N\psi)n^{\mu}_{\overline{\Sigma}}+\psi^{2}\right)}.
\end{split}
\end{equation*}

One can also estimate spacetime integrals by using elliptic estimates. Indeed, if $\overline{\mathcal{R}}\left(0,\tau\right)$ is the spacetime region as defined before, then  
\begin{equation*}
\int_{\overline{\mathcal{R}}\left(0,\tau\right)}{f\left|\nabla u\right|dg_{\overline{\mathcal{R}}}}=\int_{0}^{\tau}{\left(\int_{\overline{\Sigma}_{\tau}}{f}dg_{\overline{\Sigma}_{\tau}}\right)dt},
\end{equation*}
where the integrals are with respect to the induced volume form and $u:\mathcal{M}\rightarrow\mathbb{R}$ is such that $u\left(p\right)=\tau$ iff $p\in\overline{\Sigma}_{\tau}$. Then $\nabla u$ is proportional to $n_{\overline{\Sigma}_{\tau}}$ and since $T(u)=1$, $\nabla u$ is $\phi_{\tau}$-invariant. Therefore, $\left|\nabla u\right|$ is uniformly bounded. If now $f$ is quadratic on the 2-jet of $\psi$ then
\begin{equation*}
\begin{split}
\left|\int_{\overline{\mathcal{R}}\left(0,\tau\right)}{fdg_{\overline{\mathcal{R}}}}\right|&\leq C\int_{0}^{\tau}{\left\|\psi\right\|_{{H}^{2}\left(\overline{\Sigma}_{\tilde{\tau}}\right)}^{2}d\tilde{\tau}}\\
&C\leq \int_{0}^{\tau}{\left(\int_{\overline{\Sigma}_{\tilde{\tau}}}{ J_{\mu}^{N}(\psi)n^{\mu}_{\overline{\Sigma}_{\tilde{\tau}}}+J_{\mu}^{N}(N\psi)n^{\mu}_{\overline{\Sigma}_{\tilde{\tau}}}}+\psi^{2}\right)d\tilde{\tau}}\\
&C\leq \int_{\overline{\mathcal{R}}\left(0,\tau\right)}{J_{\mu}^{N}(\psi)n^{\mu}_{\overline{\Sigma}}+CJ_{\mu}^{N}(N\psi)n^{\mu}_{\overline{\Sigma}}+\psi^{2}}.
\end{split}
\end{equation*}
In applications  we usually use these results away from $\mathcal{H}^{+}$ where we commute with $T$ and we use the $X$ estimate. We can also use this estimate even if $\Sigma$ (and $\mathcal{R}$) crosses $\mathcal{H}^{+}$, provided we have commuted the wave equation with $N$ and $NN$ (recall that we  need commutation with $NN$ only for degenerate black holes).


\begin{thebibliography}{99}
\bibitem{ali} S. Alinhac, \textit{Geometric analysis of hyperbolic differential equations: An introduction}, The London Mathematical Society, Lecture Note Series 374

\bibitem{blukerr} L. Andersson and P. Blue, \textit{Hidden symmetries and decay for the wave equa-
tion on the Kerr spacetime}, arXiv:0908.2265

\bibitem{aretakis1} S. Aretakis, \textit{Stability and instability of extreme Reissner-Nordstr\"{o}m black hole spacetimes for linear scalar perturbations I}, Comm. Math. Phys. \textbf{307} (2011), 17--63 

\bibitem{aretakis2} S. Aretakis, \textit{Stability and instability of extreme Reissner-Nordstr\"{o}m black hole spacetimes for linear scalar perturbations II}, Ann. Henri Poincar\'{e}, online first

\bibitem{other2} J. Bi\v{c}\'{a}k, \textit{Gravitational collapse with charge and small asymmetries I: scalar perturbations}, Gen. Rel. Grav. Vol. \textbf{3}, No. 4 (1972),  331--349

\bibitem{blu0} P. Blue and A. Soffer, \textit{Semilinear wave equations on the Schwarzschild manifold.
I. Local decay estimates}, Adv. Differential Equations \textbf{8} (2003), No. 5,  595--614

\bibitem{blu3} P. Blue and J. Sterbenz, \textit{Uniform decay of local energy and the semi-linear
wave equation on Schwarzschild space}, Comm. Math. Phys. \textbf{268} (2006), No. 2, 481--504

\bibitem{blu1} P. Blue and A. Soffer, \textit{Phase space analysis on some black hole manifolds},  Journal of Functional Analysis (2009), Volume \textbf{256}, Issue 1, 1--90

\bibitem{blu2} P. Blue and A. Soffer, \textit{Improved decay rates with small regularity loss for the wave equation about a Schwarzschild black hole}, arXiv:math/0612168


\bibitem{christab} D. Christodoulou and S. Klainerman, \textit{The Global Nonlinear Stability of the Minkowski Space}, Princeton University Press 1994

\bibitem{schris} D. Christodoulou, \textit{On the global initial value problem and the issue
of singularities}, Classical Quantum Gravity \textbf{16} (1999), No. 12A, A23--A35


\bibitem{chrin} D. Christodoulou, \textit{The instability of naked singularities in the gravitational
collapse of a scalar field}, Ann. of Math. \textbf{149} (1999), No. 1, 183--217
 
\bibitem{christodoulou_actionprinciple}
  D. Christodoulou,
  \emph{The Action Principle and Partial Differential Equations},
  Princeton University Press, New Jersey,  2000.
  
 
  
  \bibitem{formblackholes} D. Christodoulou, \textit{The Formation of Black Holes in General Relativity}, Zurich: European Mathematical Society Publishing House (2009)
  
\bibitem{d1} M. Dafermos, \textit{Stability and instability of the Cauchy horizon for the spherically
symmetric Einstein-Maxwell-scalar field equations}, Ann. of Math.\textbf{ 158}
(2003), No. 3, 875--928

\bibitem{d2} M. Dafermos, \textit{The interior of charged black holes and the problem of uniqueness
in general relativity}, Comm. Pure App. Math. \textbf{58} (2005), No. 4, 445-- 504
  
\bibitem{price} M. Dafermos and I. Rodnianski, \textit{A proof of Price' s law for the collapse of a
self-gravitating scalar field}, Invent. Math. \textbf{162} (2005), 381--457


\bibitem{dr3} M. Dafermos and I. Rodnianski,
\emph{The redshift effect and radiation decay on black hole
spacetimes}, Comm. Pure  Appl. Math. {\bf 62} (2009), 859--919

\bibitem{dr4} M. Dafermos and I. Rodnianski
\emph{The wave equation on  Schwarzschild-de Sitter spacetimes},
 arXiv:0709.2766 

\bibitem{dr5} M. Dafermos and I. Rodnianski,
\emph{A note on energy currents and decay for the wave equation
on a Schwarzschild background},  arXiv:0710.0171 
  
\bibitem{dr7}M. Dafermos and I. Rodnianski, \textit{A proof of the uniform boundedness of solutions to the wave equation on slowly rotating Kerr backgrounds},  arXiv:0805.4309

\bibitem{md}M. Dafermos and I. Rodnianski,
\emph{Lectures on Black Holes and Linear Waves}, arXiv:0811.0354


 
\bibitem{neo} M. Dafermos, \textit{The evolution problem in general relativity}, Current developments in mathematics, 2008, 1--66, Int. Press, Somerville, MA

\bibitem{new}M. Dafermos and I. Rodnianski, \textit{A new physical-space approach to decay for the wave equation with applications to black hole spacetimes}, arXiv:0910.4957

\bibitem{megalaa} M. Dafermos and I. Rodnianski, \textit{Decay for solutions of the wave equation on Kerr exterior spacetimes: The case $\left|a\right|<M$}, preprint

\bibitem{other1} R. Donninger, W. Schlag and A. Soffer, \textit{On pointwise decay of linear waves on a Schwarzschild black hole background}, arXiv:0911.3179

\bibitem{finster1} F. Finster, N. Kamran, J. Smoller and S.T. Yau, \textit{Decay of solutions of the wave equations in the Kerr geometry}, Comm. Math. Phys. \textbf{264} (2), 221--255, 2008

\bibitem{brill} 

J. C. Graves and D. R. Brill, \emph{Oscillatory Character of Reissner-Nordstr\"{o}m Metric for an Ideal Charged Wormhole},
Palmer Physical Laboratory, Princeton University, Princeton, New Jersey (1960)



\bibitem{haw} S.W Hawking, G.F.R. Ellis, \textit{The large scale structure of spacetime}, Cambridge Monographs on Mathematicals Physics, No. 1, Cambridge University Press, London-New York, 1973


\bibitem{muchT} S. Klainerman, \emph{Uniform decay estimates and the Lorentz
invariance of the classical wave equation} Comm. Pure Appl. Math. {\bf 38} (1985),
321--332

\bibitem{kro} J. Kronthaler, \textit{Decay rates for spherical scalar waves in a Schwarzschild geometry}, arXiv:0709.3703

\bibitem{rin} H. Lindblad and I. Rodnianski, \textit{The global stability of Minkowski spacetime
in harmonic gauge}, Ann. of Math. \textbf{171} (2010), No 3, 1401--1477

\bibitem{luk} J. Luk, \textit{Improved decay for solutions to the linear wave equation on a Schwarzschild black hole}, Ann. Henri Poincar\'{e}, \textbf{11} (2010), 805--880

\bibitem{marolf} D. Marolf, \textit{The danger of extremes}, arXiv:1005.2999

\bibitem{tataru1} J. Marzuola, J. Metcalfe, D. Tataru, M. Tohaneanu, \textit{Strichartz estimates
on Schwarzschild black hole backgrounds}, Comm. Math. Phys. \textbf{293 }(2010), No. 1, 37--83

\bibitem{mor1} C. S. Morawetz, \textit{The limiting amplitude principle}, Comm. Pure Appl. Math.
\textbf{15} (1962), 349--361

\bibitem{mor2} C. S. Morawetz, \textit{Time Decay for Nonlinear Klein-Gordon Equation}, Proced. Roy. Soc. London, Vol. \textbf{306}, (1968), 291--296

\bibitem{n} Nordstr\"{o}m, G (1918) \textit{On the Energy of the Gravitational Field in Einstein's Theory}. Verhandl. Koninkl. Ned. Akad. Wetenschap., Afdel. Natuurk., Amsterdam 26: 1201--1208

\bibitem{price72} R. Price, \textit{Nonspherical perturbations of relativistic gravitational collapse.
I. Scalar and gravitational perturbations}, Phys. Rev. D (3) \textbf{5} (1972), 2419--2438

\bibitem{poisson} E. Poisson, \textit{A Relativist's Toolkit: The Mathematics of Black-Hole Mechanics}, Cambridge University Press, 2004

\bibitem{RW} T. Regge and J. Wheeler, \emph{Stability of a Schwarzschild singularity}
Phys. Rev. {\bf 108} (1957), 1063--1069

\bibitem{r} Reissner, H (1916) \textit{\"{U}ber die Eigengravitation des elektrischen Feldes nach der Einstein'schen Theorie}. Annalen der Physik \textbf{50}

\bibitem{jared} I. Rodnianski and J. Speck, 
     \textit{The Stability of the Irrotational Euler-Einstein System with a Positive Cosmological Constant}, arXiv:0911.5501 

\bibitem{volker} V. Schlue,\textit{ Linear waves on higher dimensional Schwarzschild black holes}, Rayleigh Smith Knight Essay, January 2010, University of Cambridge

\bibitem{Sogge} C. Sogge, \emph{Lectures on nonlinear wave equations},
International Press, Boston, 1995

\bibitem{tataru2} D. Tataru and M. Tohaneanu, \textit{Local energy estimate on Kerr black hole
backgrounds}, arXiv:0810.5766

\bibitem{tataru3} D. Tataru, \textit{Local decay of waves on asymptotically flat stationary space-times}, arXiv:0910.5290 

\bibitem{tay} M.E. Taylor, \textit{Partial Differential Equations} I, Springer Publications

\bibitem{tw} F. Twainy, \textit{The Time Decay of Solutions to the Scalar Wave Equation in
Schwarzschild Background}, Thesis. San Diego: University of California 1989

\bibitem{drimos} R. M. Wald, \emph{Note on the stability of the 
Schwarzschild metric} J. Math. Phys. {\bf 20} (1979), 1056--1058


\bibitem{wald} R. M. Wald, \textit{General Relativity}, The University of Chicago Press, 1984

\bibitem{wa1}  R. Wald and B. Kay, \textit{Linear stability of Schwarzschild under perturbations
which are nonvanishing on the bifurcation 2-sphere}, Classical Quantum
Gravity \textbf{4} (1987), No. 4, 893--898



 
\end{thebibliography}
\end{document}